\documentclass[11pt]{article}
\usepackage{amssymb}
\usepackage{amsmath,amsthm}
\usepackage{amsfonts}
\usepackage{leftidx}
\usepackage{mathrsfs}
\usepackage{authblk}
\usepackage{enumerate}  
\usepackage{abstract}
\usepackage{stmaryrd}
\usepackage{ulem}
\usepackage{slashed}
 \usepackage{float}
 \usepackage{color}
\usepackage{graphicx}
 \usepackage{hyperref}
  \usepackage{xcolor}
\usepackage{slashed}
\usepackage{adjustbox}
\usepackage{multirow}
\usepackage[top=1in, bottom=1.2in, left=1in, right=1in]{geometry}

\def\R{\mathbb{R}}

\def\Z{\mathbb{Z}}

\def\d{|\nabla|}

\def\p{\partial}

\def\vo{\vspace{1\baselineskip}}

\def\be{\begin{equation}}
\def\ee{\end{equation}}

\newtheorem{theorem}{Theorem}[section]

\newtheorem{lemma}{Lemma}[section]
\newtheorem{proposition}{Proposition}[section]

\theoremstyle{definition}
\newtheorem{definition}{Definition}[section]
\newtheorem{convention}{Convention}[section]

\theoremstyle{remark}
\newtheorem{remark}{Remark}[section]

\numberwithin{equation}{section}
\title{Large data global solution of the 3D RVM system with cylindrical symmetry I: Iterative smoothing scheme}
\author[$\star$]{Xuecheng Wang}
\affil[$\star$]{\small Tsinghua University \& BIMSA}

\date{  }

\begin{document}

 \maketitle 

\begin{abstract}
This is the first part of a two-paper sequence  establishing the global existence of $3D$ relativistic Vlasov-Maxwell system (RVM) for \emph{arbitrarily large} smooth localized initial data with   cylindrical symmetry.

 This paper employs Part II's pointwise estimates \cite{PartII}—developed independently of Part I—as fundamental tools to demonstrate the iterative smoothing scheme (ISS), which is originated and inspired by   the work of Klainerman-Staffilani\cite{Klainerman3}.  Using ISS and a novel singular weighted space-time estimate for the distribution function, we prove upper bounds for the projection and full velocity characteristics via a standard bootstrap argument.  Using the classic momentum method, we find that   the high order momentum at most grows polynomially in time. This further  implies  that the $L^\infty_{x}$-norm of the electromagnetic field, $\|(E(t),B(t)\|_{L^\infty_x}$, also grows at most polynomially over time. Consequently, this verifies the continuation criteria obtained by Luk-Strain \cite{luk2}. As a result, the energy function of $3D$ RVM system doesn't blow up in finite time,  thereby establishing  global existence.

\end{abstract}

\setcounter{tocdepth}{2}

\tableofcontents
\section{Introduction}
  
  The  $3D$ relativistic  Vlasov-Maxwell system is one of the   fundamental models in the plasma physics, which reads as follows, 
\be\label{mainequation}
(\textup{RVM})\qquad \left\{\begin{array}{l}
\p_t f + \hat{v} \cdot \nabla_x f + (E+ \hat{v}\times B)\cdot \nabla_v f =0,\\
\nabla \cdot E = 4 \pi \displaystyle{\int_{\R^3}  f(t, x, v) d v}, \qquad \nabla \cdot B =0, \\
\p_t E = \nabla \times B - 4\pi \displaystyle{\int_{\R^3} f(t, x, v) \hat{v} d v }, \\ 
 \p_t B  =- \nabla\times E,
\end{array}\right.
\ee
where $f:\R_t\times \R_x^3\times \R_v^3\longrightarrow [0, \infty)$ denotes the distribution function of particles, $ E, B  :\R_t\times \R_x^3 \longrightarrow \R^3 $ denote  the electromagnetic fields, and  $\hat{v}:=v/\sqrt{1+|v|^2}$.

  For generic \emph{finite energy regular} initial data, the global existence of $3D$ RVM remains an outstanding open problem after multiple attempts  over many years.    In this paper, we focus on the large data Cauchy problem of the    RVM system,  in which   we don't understand much.  While the small  data case is well understood, we  refer readers to \cite{glasseys1,luk2,luk,wang} and references therein for more details.    

Similar to the well-known problem regarding the $3D$ Euler equations in fluid dynamics—specifically, whether a global  solution exists for any finite energy smooth initial data—this remains an outstanding open question despite some exciting progress made thus far.

Interestingly, the $3D$ RVM system's continuation criteria exhibit a striking parallel to the Beale-Kato-Majda blowup criterion for the $3D$ Euler equations. Specifically, the $L^\infty_x$ norm of the electromagnetic field, $\|(E, B)\|_{L^\infty_x}$, plays an analogous role to the $L^\infty_x$ norm of vorticity (see Section \ref{previousresultscon}). The large data problem for RVM remains open in part because its supercritical behavior, a consequence of the $L^2_x$ level of the conservation law (\ref{conservationlaw}), presents significant challenges.  From the Sobolev embedding $H^{\frac{n+}{2}}(\R^n)\hookrightarrow L^\infty_x(\R^n)$, the   regularity gap between $L^\infty_x(\R^n)$  and $L^2_x(\R^n)$   increases with respect to the space dimension. Therefore, the large data problem is more challenging in higher dimension.

 Before addressing the full problem directly, we aim to impose some symmetry assumptions on the 3D RVM system to enhance our understanding of the large data problem. Unfortunately, the options for imposing symmetry on the RVM system are limited.  If we impose the translation symmetry, also known as the dimension reduction,  the initial data will have  \emph{infinite} energy. In this case, global existence is established, thanks to the celebrated works of Glassey and Schaeffer \cite{glasseys1,glasseys3,glasseys2}.   If we impose the radial symmetry,  the divergence-free condition combined with the finite energy assumption for the magnetic field necessitates that the magnetic field be trivial, i.e., 
$B=0$. For this scenario, global existence is also established by another notable result from Glassey and Schaeffer \cite{glasseys11}. Thus, the only feasible symmetry assumption we can impose for a general class of \emph{finite energy regular} initial data is cylindrical symmetry.

Even with cylindrical symmetry, the large data problem remains significant.   The behavior of 3D Euler smooth solutions with large swirl, for example, is still poorly understood.  The use of cylindrical symmetry, despite reducing the number of independent variables to two, does not fundamentally alter the $3D$ nature. Comparing the $2D$ case\footnote{Another reason is that the electromagnetic field is more restricted in $2D$, e.g., we have $E=(E_1,E_2,0), B=(0,0,B_3)$ for some scalar functions $E_1,E_2,B_3$, see \cite{luk}. } (see \cite{glasseys3,glasseys2}), the regularity gap and the challenge  are amplified in $3D$.  Furthermore, the potential for $z$-axis particle trajectories negates the benefits of cylindrical symmetry. The avoidance of blow-up in such cases constitutes the principal and exceptionally challenging aspect of the problem.

   In a two-paper sequence, we establish the global existence of $3D$ RVM system for \emph{arbitrarily large} smooth localized initial data, which has \emph{cylindrical symmetry} with respect to the $z$-axis.  

   The main theorem of the  two-paper sequence is stated as follows.

\begin{theorem}\label{maintheorem}
  Given  any    cylindrical  symmetric initial data $E_0, B_0\in H^{s}(\R^3)$ ,  $  f_0(x,v)\in H^s(\R_x^3 \times \R_v^3)$, $s\in \mathbb{Z}_{+},s\geq 6$. Without loss of generality, we assume the  initial data are cylindrical symmetric with respect to the $z$-axis  in the following sense, 
\be\label{april25eqn1}
\begin{split}
   & \forall  \theta\in [0,2\pi], x,v\in \R^3,  \quad R=\begin{bmatrix}
\cos \theta & -\sin \theta & 0 \\ 
\sin \theta &  \cos\theta & 0\\ 
0 & 0 & 1\\  
\end{bmatrix},\\
 & f_0(Rx, Rv) = f_0( x,v), \quad  E_0(  Rx)= RE_0(x), \quad B_0( Rx)= RB_0(x). \\
 \end{split}
\ee
If the initial distribution function decays polynomially in the following sense,  
\be\label{assumptiononinitialdata}
\sum_{\alpha \in \mathbb{Z}_{+}^6,|\alpha|\leq s} \|(1+| x|+|v|)^{N_0}\nabla_{x,v}^\alpha f_0(x,v)\|_{L^2_{x,v}}< +\infty, \quad  N_0:= 10^{10},
\ee
then the relativistic Vlasov-Maxwell system \eqref{mainequation}  admits a global solution $ (f(t,\cdot, \cdot), E(t,\cdot), B(t,\cdot))\in  H^s(\R_x^3 \times \R_v^3)\times H^s(\R_x^3)\times H^s(\R_x^3)$. Moreover, the $L^\infty_x$-norm of the electromagnetic field grows at most polynomially over time. 

\end{theorem} 

The objective of Part I is to prove the stated theorem. To this end, we utilize the pointwise estimates established in Part II \cite{PartII}, which are presented in detail in Section \ref{mainresultsPartIIdetail}. As noted in the abstract, we emphasize that these estimates are independent of the results derived in Part I.

A few remarks are in order. 
\begin{remark}

A key difference between the $3D$ RVM and  $3D$ Euler systems is the presence of a smoothing effect in RVM, although at the expense of a loss function.  Understanding how to effectively exploit and quantify this smoothing effect appears to be crucial for the large data problem. 

\end{remark}
\begin{remark}
The benefit of the  cylindrical symmetry  is used crucially to show a non-trivial fact about the distribution of particles, which is not expected to be true in the general cases without symmetry. 

For illustrative purposes, we assume that the initial data has compact support in both $x$ and $v$.  Roughly speaking, instead of a ball of radius $r$, the distribution function $f$ is supported in a cylinder $\{(v_1,v_2): |(v_1,v_2)| \leq r^{2+/3}\} \times [-r, r]\subset \R_v^3$ in velocity space, where $r$ denotes the maximum of   $|v|$, s.t., $f(t,x,v)\neq 0$. Consequently, for $v$ near its maximum, the direction vector $v/|v|$ is nearly aligned with the $z$-axis. 
\end{remark}

\begin{remark}
Our improved understanding of the distribution function, leading to a better understanding of the electromagnetic field, allows us to demonstrate, via the iterative smoothing scheme, that the problematic $z$-axis particle trajectories (which do not benefit from cylindrical symmetry directly) are well-controlled. 
\end{remark}

\begin{remark}
The plausible  goal of optimizing the size of $N_0, s $, which is less interesting, is not pursued here. 
\end{remark}

 The rest of this section is organized as follows.
\begin{enumerate}
  \item[$\bullet$] In subsection \ref{conservationlawssubsection}, we examine three conservation laws of the RVM system, which serve as the primary driving force in studying the large data problem.

\item[$\bullet$] In subsection \ref{previousresultscon}, we begin by reviewing previous results to enhance our understanding of the state of the art, followed by a discussion of the continuation criteria that will be employed in the proof of Theorem \ref{maintheorem}.

\item[$\bullet$] In subsection \ref{notationsubsection}, we introduce the technical notations that will be essential for discussing the key ideas in the proof of Theorem \ref{maintheorem}. 

\item[$\bullet$] In subsection \ref{skeletonintroduction},   we outline  the proof of Theorem \ref{maintheorem}, which is reduced to proving three independent propositions (Propositions \ref{finalproposition}, \ref{bootstraplemma1}, and \ref{bootstraplemma2}). The key ideas of each step are discussed.
\item[$\bullet$] In subsection \ref{mainresultsPartIrough},  a preliminary, informal exposition of the main results of Part II is provided, together with an explanation of their function in the proof of Theorem \ref{maintheorem}.

\item[$\bullet$] In subsection \ref{ISSintroduction}, we  elaborate the  iterative smoothing scheme and explain its application in the Propositions \ref{bootstraplemma1} and \ref{bootstraplemma2}. 
\item[$\bullet$] In subsection \ref{fullprobandopenprob}, we first explore some proof ideas that are independent of the cylindrical symmetry assumption. These include observations about the structural properties of the RVM system, which may provide insights into the study of the general large data case. Following this, we outline several related open problems of interest. In particular, we will also discuss the closely related relativistic Vlasov-Poisson (RVP) system. 
\item[$\bullet$] In subsection \ref{outlinefirstpaper}, we provide a table of essential notations and give the outline of this paper. 
\end{enumerate}

   \subsection{Main driving force: conservation laws}\label{conservationlawssubsection}
There are three conservation laws available to us, which are   the main driving force of studying the large data problem.   The first conservation law is  due to the transport nature of the Vlasov equation. More precisely, we have
  \be\label{2024nov9eqn111}
   \forall p \in[1,\infty], \qquad \| f(t,x,v)\|_{L^p_{x,v}}= \| f_0(x,v)\|_{L^p_{x,v}}. 
   \ee

   The second conservation law comes from the following  result of direct computations, see \eqref{mainequation},  
\be\label{dec3eqn1}
\frac{d}{dt}\big[\frac{1}{2}(|E|^2+|B|^2) + 4\pi \int_{\R^3} \sqrt{1+|v|^2}f(t,x,v)   d v  \big] = - \nabla_x\cdot\big(  (E\times B) + 4\pi \int_{\R^3} vf(t,x,v) d v  \big).
\ee
After integrating  the above  equality in the whole space $\R^3$, we have 
 \be\label{conservationlaw}
\mathcal{H}(t):= \int_{\R^3}|E(t,x)|^2 + |B(t,x)|^2 d x + 8\pi  \int_{\R^3} \int_{\R^3} \sqrt{1+|v|^2} f(t,x,v)  d x d v = \mathcal{H}(0).
\ee
Therefore,  $L^2$-norm of the electromagnetic fields $(E, B)$  and the   first momentum of the distribution function remain  bounded over time  within the lifespan of solution. 

Lastly, from   the work of Luk-Strain\cite{luk2}[Proposition 2.2],  the conservation law, as established in Lemma \ref{conservationlawlemma}, is satisfied by integrating \eqref{dec3eqn1} over the spacetime region bounded by the hyperplane $t=0$ and the backward light cone emanating from the point $(x,t)$.

 The improved properties of certain electric and magnetic field combinations (see \eqref{2024Dec7eqn21}) are reminiscent of the  ``electric-magnetic decomposition''  central to Christodoulou and Klainerman's seminal work \cite{ChristKlainerman}[Section 7.2] on the global stability of Minkowski spacetime.

\begin{lemma}[Luk-Strain\cite{luk2},Proposition 2.2]\label{conservationlawlemma}
For any $t\in [0,T^{ })$, we have
\be\label{march18eqn31}
\begin{split}
\sup_{x\in \R^3} &  \int_0^t \int_{\mathbb{S}^2} (t-s)^2 K_g^2(s, x+(t-s)\omega, \omega) d \omega ds \\
& + \int_0^t \int_{\R^3} \int_{\mathbb{S}^2} (t-s)^2 \langle v \rangle(1+\hat{v}\cdot \omega) f(s, x+(t-s)\omega, v) d \omega d v ds  \lesssim_{data} 1,\\
\end{split}
\ee
where
\be\label{2024Dec7eqn21}
K_g^2:= |E\cdot \omega|^2 + |B\cdot \omega |^2 +|E-\omega \times B|^2 +|B+\omega \times E|^2. 
\ee
\end{lemma}

The above conservation law plays a major role in obtaining pointwise estimates   in Part II\cite{PartII}. However, we will not use this conservation law directly in this paper.  
\subsection{Previous results: continuation criteria}\label{previousresultscon}

The global existence of solutions for the $3D$ RVM system with suitably regular, finite energy initial data remains an outstanding open problem. 
Much of the existing work on the $3D$ RVM system has focused on the continuation criteria.  A classic result  of Glassey-Strauss \cite{glassey3} says that the  classical  solution can be globally  extended  as long as the distribution function has compact support in $v$ for all the time. 

    An important work by Klainerman-Staffilani \cite{Klainerman3}   provides a new perspective to study the large data problem on the Fourier side. Since it's  closely related to our work, for the benefit of readers, we state their main theorem as follows.
\begin{theorem}[Klainerman-Staffilani \cite{Klainerman3}, \textbf{2002}]
Consider the  IVP of 3D RVM system \eqref{mainequation} with the initial data $f_0\in C_0^1(\R^3\times \R^3)$ and $E_0, B_0\in C^1(\R^3)$. Assume that  the electromagnetic field is bounded on any fixed time interval $[0,T]$, i.e., there exists an absolute constant $C$ s.t., 
\[
\sup_{t\in [0,T]}\|(E(t), B(t))\|_{L^\infty_x}\leq C.
\]
Then the 3D RVM system \eqref{mainequation} admits a unique solution, s.t., $E, B\in C^{1}([0,T]\times\R^3)$ and $f\in C^1([0,T]\times\R^3\times \R^3).$ 
\end{theorem}
\begin{remark}
Our iterative smoothing scheme is inspired by the smoothing effect identified in Klainerman-Staffilani \cite{Klainerman3}. See Section \ref{ISSintroduction} for further discussion of this point.
\end{remark}

 Since we don't work with the compactly supported initial data, see \eqref{assumptiononinitialdata},  we won't use the above theorem directly. Instead, in this two-paper sequence, we use the following continuation criteria by Luk-Strain \cite{luk}, which is a by-product of \cite{luk}[Theorem 5.7 \& Theorem 1.5]. We reformulate it as follows. 

\begin{theorem}[Luk-Strain(\textbf{2016}),\cite{luk}]\label{continuationcriteria}
Let $(f_0, E_0, B_0)$ be initial data in $\R^3$ satisfying certain technical assumptions, which are guaranteed by assumptions in Theorem \ref{maintheorem}. If the nonlinear solution exists in the time interval $[0, T)$ and the following bound holds, 
\be\label{continuationcrit}
\sup_{t\in [0, T)} \|E(t, x)\|_{L^\infty_x} +\|B(t, x)\|_{L^\infty_x}< \infty,
\ee
then the solution can be extended uniquely to $[0, T+\delta]$ for some $\delta >0$ s.t., $E,B\in L^\infty_{[0, T+\delta] }(H^s(\R^3))$ and $f\in L^\infty_{[0, T+\delta] }(H^s(\R^3\times \R^3))$.
\end{theorem} 
\begin{remark}
In \cite{luk},  the authors proved     stronger results. Firstly, they do not require a rapid decay rate for the initial data. Secondly,  they only need the following assumption 
\[
\sup_{t\in[0, T), (x,v)\in \R^3\times \R^3}\int_0^T (\big|E(s,X( x,v,s,t))\big| + \big|E(s,X(x,v,s,t))\big|  ) d s < \infty,
\]
 to extend the lifespan of the solution, where $X( x,v,s,t ))$ denotes the backward space characteristics, see also \eqref{backward}. Lastly, they proved that a regular solution can be extended as long as $\|\langle v \rangle^\theta f\|_{L_x^q L_v^1}$ remains bounded for some $\theta >2/q$ and $q\in (2,\infty].$
\end{remark}

We refer readers to Kunze\cite{kunze}, Pallard\cite{pallard3}, Luk-Strain \cite{luk2,luk},   Patel\cite{patel}, and references therein for more recent improvements  on the continuation criteria.

Lastly, it's   worth to mention a closely related model,   
 relativistic Vlasov-Poisson system in the  plasma physics case (RVP-PP), which reads as follows. 
\be\label{vlasovpo}
(\textup{RVP-PP})\quad \left\{\begin{array}{l} 
\p_t f + \hat{v} \cdot \nabla_x f + E \cdot \nabla_v f =0,\quad E=\nabla\phi, \\
\Delta \phi = \rho(t), \quad \rho(t):= \displaystyle{\int_{\R^3} f(t,x,v) d v},\quad f(0,x,v)=f_0(x,v). \\ 
\end{array}\right. 
\ee

 While the physical significance of this system may be debated, it remains mathematically intriguing. In contrast to the well-established non-relativistic case in $3D$, as demonstrated by the work of Lions-Perthame \cite{Lions} and Pfaffelmoser \cite{Pfaffelmoser}, the global regularity of the $3D$ RVP-PP system with sufficiently regular finite energy initial data remains an open problem. Although the RVM and RVP-PP systems are equivalent in the radial setting for smooth localized initial data in which $B=0$, they exhibit subtle differences in the non-radial setting, even when $B=0$. For a more detailed discussion of the 3D RVP-PP system, we refer readers to \cite{wang2,wang3}.
\subsection{Notation and Preliminaries}\label{notationsubsection}

\begin{definition}\label{relationdef}
 For any two numbers $A$ and $B$, we use  $A\lesssim B$, $A\approx B$,  and $A\ll B$ to denote  $A\leq C B$, $|A-B|\leq c A$, and $A\leq c B$ respectively, where $C$ is an absolute constant and $c$ is a sufficiently small absolute constant. We use $A\sim B$ to denote the case when $A\lesssim B$ and $B\lesssim A$.   For an integer $k\in\mathbb{Z}$, we use ``$k_{+}$'' to denote $\max\{k,0\}$ and  use ``$k_{-}$'' to denote $\min\{k,0\}$. For $x_0\in \R^3$, $r\in \R_{+}$, we use both the notation $B(x_0, r)$ and $B_r(x_0)$ to denote the set $\{x:|x-x_0|<r, x\in \R^3\}$. 
\end{definition}

\begin{convention}\label{conventionconst}
  We use the convention that all constants which only depend on the  \textbf{initial data}, e.g., the conserved quantities in   \eqref{conservationlaw}, will be treated as absolute constants. There are several absolute constants we use  throughout this two-paper sequence, which are $N_0:=10^{10},\epsilon:=10^{-7}, \iota:=10^{-4}. $
\end{convention}

Let $T$ denotes the maximal time of existence.   Our goal is to control  $E, B\in L^\infty([0,T^{})\times \R_{x}^3)$, i.e., to prove \eqref{continuationcrit}, to ensure that the lifespan  of RVM  \eqref{mainequation}  can be extended  to $[0, T+\epsilon]$ because of the continuation criteria in Theorem \ref{continuationcriteria}. Therefore, due to the definition of the maximal time of existence, we know that  the solution exists globally.

We  fix an even smooth function $\tilde{\psi}:\R \rightarrow [0,1]$, which is supported in $[-3/2,3/2]$ and equals to one  in $[-5/4, 5/4]$. For any $k, k_1,k_2\in \mathbb{Z}$, we define $\psi_k, \psi_{\leq k},\psi_{\geq k}, \psi_{[k_1,k_2]} :\cup_{n\in \Z_+}\R^n\longrightarrow\R$, 
\[
\begin{split}
&\psi_{k}(x) := \tilde{\psi}(|x|/2^k) -\tilde{\psi}(|x|/2^{k-1}), \quad \psi_{\leq k}(x):= \sum_{l\leq k}\psi_{l}(x), \\
 &\psi_{\geq k}(x):= 1-\psi_{\leq k-1}(x), \quad \psi_{[k_1,k_2]}(x)=\sum_{k\in[k_1,k_2]\cap \Z}\psi_k(x). \\
\end{split}
\]
Moreover, we define the cutoff function $\psi_{l;\bar{l}}:\cup_{n\in \Z_+}\R^n\longrightarrow\R$ with the threshold $\bar{l}\in \Z$,  as follows,
  \be\label{cutoffwiththreshold}
\varphi_{l;\bar{l}}(x):=\left\{\begin{array}{ll}
\psi_{\leq \bar{l}}(x) & \textup{if\,\,} l=\bar{l}\\
\psi_l(x) & \textup{if\,\,} l>\bar{l}\\
\end{array}
\right., \quad \forall l_1,l_2\in \Z, \quad  \varphi_{[l_1, l_2];\bar{l}}(x):= \sum_{l\in [\max\{l_1,\bar{l}\}, l_2]\cap \Z } \varphi_{l;\bar{l}}(x).
  \ee
In particular, if the threshold $\bar{l}=0$, we use the following notation, 
\be\label{cutoffwiththreshold100}
\forall k, k_1, k_2 \in \mathbb{Z}_{+},\quad  \varphi_k(x):=\varphi_{k;0}(x), \quad \varphi_{[k_1,k_2]}(x):=  \varphi_{[k_1,k_2];0}(x).
\ee

For a function $F(x)\in L^1$, we use $F^{+}:=F$ and $F^{-}:=\bar{F}$. Moreover, we use both $\widehat{F}(\xi)$ and $\mathcal{F}(F)(\xi)$ to denote the Fourier transform of $F$, which is defined as follows,
\[
\mathcal{F}(F)(\xi)= \int e^{-ix \cdot \xi} F(x) d x.
\]
 
We use $\mathcal{F}^{-1}(G)$ to denote the inverse Fourier transform of $G(\xi)$.
Moreover, for any $k\in \Z,$ we use  $P_{k}$, $P_{\leq k}$ and $P_{\geq k}$ to denote the projection operators  by the Fourier multipliers $\psi_{k}(\cdot),$ $\psi_{\leq k}(\cdot)$ and $\psi_{\geq k }(\cdot)$ respectively. For convenience in notation, we also use  $f_{k}(x)$ to abbreviate $P_{k} f(x)$.
 
We define the following class of symbol, 
\be\label{symbolclassdefinition}
\mathcal{S}^\infty:=\{m(\xi): m: \R^3\longrightarrow \R, \quad \|\mathcal{F}^{-1}[m](x)\|_{L^1_x}< +\infty\}.
\ee
Moreover, the $\mathcal{S}^\infty$-norm of symbols is defined as follows, 
\be
\| m(\cdot)\|_{\mathcal{S}^\infty} := \|\mathcal{F}^{-1}(m) (x)\|_{L^1_x}.
\ee
\begin{definition}\label{varioureldef}
For any vector $u=(u_1, u_2,u_3)\in \R^3, v\in \R^3/\{0\}$, we define 
\[
 \tilde{v}:=v/|v|,  \quad  \langle u \rangle :=(1+|u|^2)^{1/2}, \quad   \hat{u}:=u/\langle u \rangle, \quad     {u}_{\bot}:=(u_1, u_2) . 
\] 
In particular, we have $ {\hat{u}}_{\bot}=(u_1, u_2)/\langle u\rangle$. Moreover, we define the following projection maps, 
\be\label{definitionprojection}
\mathbf{P}:\R^3\longrightarrow \R^2,  \mathbf{P}_i:\R^3\longrightarrow \R, \quad  \mathbf{P}(u)=  {u}_{\bot}, \quad \mathbf{P}_i(u)=u_i, i\in\{1,2,3\}.
\ee
\end{definition}

\vo

 We define the unit vectors of the Cartesian coordinate  system in $\R^3$ as follows, 
\be\label{2020feb18eqn1}
e_1:=(1,0,0), \quad e_2:=(0,1,0), \quad e_3:=(0,0,1).
\ee

Let $v\in \R^3/\{0\}$ be fixed. With the above notation, for any  $u \in \mathbb{R}^3,$ we have an important decomposition of  $u $ with respect to $v$,  which follows from a straightforward computation,
 \be\label{nov24eqn1}
 \begin{split}
&  u= (\tilde{v}\cdot u) \tilde{v}     + \sum_{i=1,2,3}  ((\tilde{v}\times e_i) \cdot u)  (\tilde{v}\times e_i),\\
\end{split}
\ee 
where ``$\cdot$'' denotes the standard inner product and  ``$\times$'' denotes the standard cross product in $\R^3$. 

The above decomposition is motivated from the following fact that the radial   derivative of $\hat{v}$ is much better than the rotational derivatives of $\hat{v}$ when $v$ is large. This fact follows from the following explicit computation, 
\be
\tilde{v}\cdot\nabla_v \hat{v}= \frac{\tilde{v}}{(1+|v|^2)^{3/2}},  \quad   (\tilde{v}\times e_i) \cdot\nabla_v \hat{v}= \frac{ (\tilde{v}\times e_i) }{(1+|v|^2)^{1/2}}, \quad i\in\{1,2,3\}.
\ee

For the rest of this paper, we will use the   decomposition in  \eqref{nov24eqn1}    constantly in the estimate of  derivatives of $\hat{v}$ and the relativistic velocity characteristics $\hat{V}(t)$without further explanation. 

\subsubsection{Backward characteristics}
 
Recall  \eqref{mainequation}.   If $s\leq t$ ($s\geq t$),    the following system of equations satisfied by the backward (forward)  characteristics, 
\be\label{backward} 
\left\{\begin{array}{rl}
\p_s X(x,v,s,t) &= \widehat{V}(x,v,s,t),\\ 
 \p_s V(x,v, s,t) &= E(s, X(x,v,s,t)) + \widehat{V}(x,v,s,t)\times B(s, X(x,v,s,t))\\
& := K(s,X(x,v,s,t), V(x,v,s,t)),\\ 
X(x,v,t,t)&=x, \quad V(x,v,t,t)=v.\\ 
\end{array}\right. 
\ee
 In subsequent discussions, we will  refer $K(s,X(x,v,s,t), V(x,v,s,t))$ as the acceleration force of characteristics.  Note that, due to the transport nature of the Vlasov equation,    for any $s, t\in [0, T^{  })$, we have
\be\label{conservation}
f(t,x,v) = f(s, X(x,v,s,t), V(x,v,s,t)),
\ee
which also gives us the conservation law for the $L^\infty_{x,v}$-norm of the distribution function $f$.

\subsubsection{Dyadically measuring the velocity characteristics}

To control the  electromagnetic field $(E, B)$ over time, see also Theorem \ref{maintheorem1part1},  we use the classic momentum method. More precisely, we define 
 \be\label{may2eqn1}
\mathfrak{M}_{}(t ):= \int_{\R^3} \int_{\R^3} (1+ |v|)^{N_0/10} f(t,x,v) d v d x,\quad    \overline{\mathfrak{M}}(t) := (1+t)^{(N_0)^3}+\sup_{s\in[0,t]} \mathfrak{M}(s), 
 \ee
 where $N_0=10^{10}$ is defined in \eqref{assumptiononinitialdata}. From the above definition, it's clear that $\overline{\mathfrak{M}}(t) $ is an increasing function with respect to time $t$ and the size of $t$ is much smaller than   $\overline{\mathfrak{M}}(t)$.

Since we will do dyadic decomposition for many variables, for convenience of comparing different objects,   we  define the following dyadic quantity to capture the size of $\overline{\mathfrak{M}}(t)$.
\begin{definition}\label{scaleofv}
For any $\forall t\in [0, T ),$ we define 
\be\label{dec2eqn1}
 M_t:= \inf\{k: k\in \Z, 2^k\geq (  \overline{\mathfrak{M}}(t) )^{1/  (N_0/10-1 )} \}.
\ee
\end{definition}
\begin{remark}\label{scaleofvphilosophical}
We provide a philosophical  interpretation for the important quantity $M_t$  defined above.  Given that we lack any a priori information about the distribution function $f$, we  assume that $f$ is supported in a annulus of size $R$, i.e., $supp(f)\subset\{v: |v|\in [R, 2R]\}$. Our goal is to extract information about $R$ from the moment $\mathfrak{M}_{}(t )$.  Note that, in view of the conservation laws \eqref{2024nov9eqn111} and \eqref{conservationlaw}, the first momentum has both the   upper bound and the lower bound, which depend only on the initial data.   Therefore, we have
\[
\mathfrak{M}_{}(t ) \sim R^{N_0/10-1}\int_{\R^3} \int_{\R^3} (1+ |v|)^{ } f(t,x,v) d v d x \sim R^{N_0/10-1}, \quad \Longrightarrow R\sim \big(\mathfrak{M}_{}(t ) \big)^{1/(N_0/10-1)},
\]
where $\sim$ is defined in \textbf{Definition} \ref{relationdef}. Recall the definition of $M_t$ in \eqref{dec2eqn1}. Roughly speaking, $M_t$ captures the information about the dyadic interval such that $v$ lives. 

\end{remark}
\begin{remark}
  Note that, for any $t_1, t_2\in [0, T^{ }),$ s.t., $t_1\leq t_2,$ we have $M_{t_1}\leq M_{t_2}$ because $ \overline{\mathfrak{M}}(t) $ is an increasing function with respect to $t$. 
\end{remark}

In view of the definition $M_t$ in \eqref{dec2eqn1}, to control the moment $\overline{\mathfrak{M}}(t) $, it suffices to control $M_t$ over time. Since we are working with the initial data with   polynomial decay, we would like to show   that 
the tail part doesn't play an essential role. To this end, we are motivated to define a time dependent majority set  as follows.

\begin{definition}[$t$-majority set]\label{tmajorityset}
For any $t\in [0, T^{  })$,   we define  \textbf{$t$-majority} set  of particles, which initially localize around zero, \textbf{at time $s$}$\in [0, T^{  })$ as follows,  
\be\label{may10majority}
R_t(s):= \{(X(x,v,s,0),V(x,v,s,0)): |V (x,v,0,0)|+| X(x,v,0,0) |\leq  2^{ M_t/2}  \}.
\ee
In particular, we have
\[
  R_t(0):=\{(x,v): |(x,v)|\leq 2^{ M_t/2} \}. 
\]
Moreover, to measure the $t$-majority set,  we define 
\be\label{may9en21}
\begin{split}
\alpha_t(s,x,v)&:=  \sup_{\tau \in[0,s] } \inf\{ k: k\in \R_+, |  {V}_{\bot}(x,v,\tau,0)| \leq 2^{k M_t  }   \}  ,\quad \alpha_t(s):=\sup_{(x,v)\in R_t(0)} \alpha_t(s,x,v), \\ 
\beta_t(s, x,v)&:= \sup_{\tau\in[0,s] } \inf\{ k: k\in \R_+, | V(x,v,\tau,0)|\leq 2^{k M_t}  \}, \quad \beta_t(s):=\sup_{(x,v)\in R_t(0)} \beta_t(s,x,v), \\
 \alpha_t&:=\alpha_t(t), \quad  \tilde{\alpha}_t=\min\{\alpha_t, (1+2\epsilon) \},\quad  \beta_t:=\beta_t(t), \quad  \tilde{\beta}_t=\min\{\beta_t, (1+2\epsilon) \},
\end{split}
\ee
 where ${V}_{\bot}$ is defined in \textbf{Definition} \ref{varioureldef} and the velocity characteristics $V(x,v,s,t)$ is defined in \eqref{backward}. In subsequent discussions, we will refer to the region of characteristics  that start from the $t$-majority set as the majority part and refer the region of characteristics  that start from the  complement  of the  $t$-majority set as the tail part.  
\end{definition}

 \begin{remark}\label{sizeofvphilosophical}

We offer a geometric interpretation of $\alpha_t$ and $\beta_t$, as much of this two-paper sequence aims to elucidate these quantities. Broadly speaking, if we assume that we are working with compactly supported initial data, we can identify a ball in $\R^3$
  that encompasses all forward velocity characteristics $V(x,v,s,0)$  originating from the compact support; in this case, 
 $\beta_t$
  represents the radius of this ball. Additionally, we can find another ball in 
$\R^2$
  that covers all projections of the forward velocity characteristics, denoted as  $\mathbf{P}\big(V(x,v,s,0)\big)$; here, 
 $\alpha_t$
  measures the radius of this ball in 
$\R^2$
 . To compare these radii in relation to 
$M_t$
 , we are motivated to adjust the scales in the definitions of 
 $\alpha_t$ and $\beta_t$. A priori, we don't know the size of $\alpha_t$ and $\beta_t$.  It turns out that $\tilde{\alpha}_t$ and $\tilde{\beta}_t$, which are bounded from the  above, are very useful  during the estimates. A posteriori, $\tilde{\alpha}_t$ ($\tilde{\beta}_t$) is same as $\alpha_t$ ($\beta_t$).
 
 \end{remark}
\begin{remark}

Since $M_t$ is an increasing function, $\forall t\in [0, T^{  }), s\in [0, t], $ from the above definition, we know that 
\be\label{2021dec18eqn1}
 R_s(0)\subset R_t(0), \quad  \alpha_s M_s \leq \alpha_t(s) M_t\leq \alpha_t M_t, \quad \beta_s M_s\leq \beta_t(s) M_t\leq \beta_t M_t. 
\ee
\end{remark}
\begin{remark}
Note that, $\forall s\in[0, t]\subset [0, T), x, v \in \R^3,$ we have 
\[
  X( X(x,v,0,s), V(x,v,0,s),s,0)= x, \quad V( X(x,v,0,s), V(x,v,0,s),s,0)= v. 
\]
Therefore, in view of the above definition,  for any $v\in \R^3$, s.t., either $| v_{\bot}|\geq 2^{\alpha_t M_t+\epsilon M_t}$ or $| v|\geq 2^{\beta_t M_t+\epsilon M_t}$, we have
\be\label{nov24eqn27}
\forall s\in [0, t], x \in \R^3,  \quad (X(x,v,0,s), V(x,v,0,s))\notin R_t(0). 
\ee
 Thanks to the rapid polynomial decay rate of the initial data in   \eqref{assumptiononinitialdata}, from \eqref{nov24eqn27} and \eqref{may10majority}, the following estimate holds for any $x,v\in \R^3$, s.t., either $|  v_{\bot}|\geq 2^{\alpha_t M_t+\epsilon M_t}$ or $| v|\geq 2^{\beta_t M_t+\epsilon M_t}$, 
 \be\label{nov24eqn41}
 \forall s\in [0, t],  \quad \big|f(s,x,v)\big| = \big|f_0( X(x,v,0,s),V(x,v,0,s) )\big|\lesssim 2^{- 2N_0 M_t/5}. 
 \ee
Thanks to the above rapid decay estimate,   the contribution from the case $|  v_{\bot}|\geq 2^{\alpha_t M_t+\epsilon M_t}$ and the case $| v|\geq 2^{\beta_t M_t+\epsilon M_t}$ becomes negilible\footnote{In the compactly supported initial data case, an analogue of the estimate \eqref{nov24eqn41} would simply be $f(s,x,v)=0.$}.  As a result, it suffices  to consider the case $| v_{\bot}|\leq  2^{\alpha_t M_t+\epsilon M_t}$ and $| v|\leq 2^{\beta_t M_t+\epsilon M_t}$  when dealing with the distribution function  in subsequent discussions.
\end{remark}

\subsubsection{Representation formula: Duhamel's formula}

Recall \eqref{mainequation}. After doing dyadic decomposition for the size of $v$, substituting the $\p_t f$ by using the the Vlasov equation in  \eqref{mainequation}   and doing integration by parts in $v$, 
we can reduce the Maxwell system into the standard wave equations as follows, 
\be\label{electromagnetic}
\forall U\in \{E,B\}, \quad (\p_t^2-\Delta) U =\sum_{j\in \Z_+} \mathcal{N}^j_U(t, x),
\ee
where
\be\label{sep5eqn1}
\begin{split}
\mathcal{N}^j_E(t, x)&:= 4\pi\big[  \int_{\R^3}\int_{\R^3}  \big(\hat{v} (\hat{v}\cdot \nabla_x) f(t, x, v) -\nabla_x f(t, x , v) \big)\varphi_j(v) \\
&\qquad - f(t,x,v)(E(t,x)+\hat{v}\times B(t,x))\cdot \nabla_v\big(\varphi_j(v) \hat{v}\big) d x d v\big], \\
 \mathcal{N}^j_B(t, x)&:=-4\pi\int_{\R^3} \hat v\times \nabla_x f(t, x, v) \varphi_j(v)d v.\\
 \end{split}
\ee
From the Duhamel's formula, $\forall  U\in\{E, B\},$ the following representation formula of the electromagnetic field holds, 
\be\label{march14eqn1}
\begin{split}
  U(t)  
  & =  \cos(t\d)U_0 + \frac{\sin(t\d)}{\d}  U_1 + \sum_{  j\in\Z_{+}}    U_j(t), \\ 
  U_j(t)
&= \sum_{ \mu\in \{+,-\}}\int_0^t \frac{e^{i\mu(t-s)\d}}{ 2i \mu  \d}     \mathcal{N}^j_U(s) d s,
  \end{split}
\ee
where $U_0= U(t)\big|_{t=0} $ and $U_1:=(\p_t  U )\big|_{t=0}$. 
\subsection{An outline of  the  proof of Theorem \ref{maintheorem}}\label{skeletonintroduction}

Let \textbf{ISS} abbreviates the iterative smoothing scheme. With notation and set-up introduced above, intuitively speaking, an outline of  the  proof of Theorem \ref{maintheorem} can be summarized as follows, 
 \be\label{skeletonofproof}
\begin{split}
&\textup{global existence} \qquad  \overset{\textup{continuation criteria} }{\Longleftarrow  } \qquad (E(t), B(t))\in L^\infty([0, T)\times \R^3 ),\\
& (E(t), B(t))\in L^\infty([0, T)\times \R^3 )  \qquad  \Longleftarrow   \qquad  \overline{\mathfrak{M}}(t) \lesssim  (1+t)^{(N_0)^3}  \qquad  \Longleftarrow   \qquad  \beta_t\leq 1-2\epsilon,  \\
&\beta_t\leq 1-2\epsilon \qquad  \stackrel{\textup{bootstrap} }{\underset{ \textbf{ISS}}{\Longleftarrow} }\qquad  \alpha_t\leq \alpha^\star:=2/3+\iota ,\\
& \alpha_t\leq \alpha^\star:=2/3+\iota \qquad     \stackrel{\textup{bootstrap} }{\underset{ \textbf{ISS}}{\Longleftarrow} } \qquad   \textup{control of space characteristics}, \\
&     \textup{control of space characteristics} \qquad  \stackrel{\textup{bootstrap}}{\Longleftarrow}\qquad  \textup{iterative smoothing scheme}\, (\textbf{ISS}).   \\
\end{split} 
\ee
where $   \iota:= 10^{-4}$ and $\epsilon:=10^{-7}.$

\begin{remark}
 The bootstrap argument in \eqref{skeletonofproof} is a two-step process.  Recall the  definition of $\alpha_t$ and $\beta_t$ in \eqref{may9en21}. Assuming initial upper bounds for both $\alpha_t$ and $\beta_t$,     we first refine the bound for $\alpha_t$ by controlling the space characteristics  through a bootstrap argument. We then refine the bound for $\beta_t$. Crucially, lines 4 and 5 of \eqref{skeletonofproof} also depend on $\beta_t$.
\end{remark}

\begin{remark}\label{remarktwothird}

The upper bound of 
$\alpha_t$, set at 
$2/3+\iota$, is motivated by previous studies on the $3D$ RVP-PP system with cylindrical symmetry, particularly those by Glassey and Schaeffer \cite{glassey6} and Wang \cite{wang2,wang3}, where a similar value was observed. This choice facilitates the application of the iterative smoothing scheme. While it is likely not the optimal bound for 
$\alpha_t$, the advantages of further optimizing this value appear to be limited.

 \end{remark}
 By ``control of space characteristics'', we refer to the phenomenon where   \textbf{space characteristics will {eventually travel  away from $z$-axis} if the size of the projection of velocity characteristics,  $\mathbf{P}(V(x,v,s,t))$, approaches $\alpha_t$}.   This implies that space characteristics cannot remain close to the 
$z$-axis—the most unfavorable scenario—over an extended period.   Furthermore, controlling the space characteristics leads to the important observation that the duration during which the space characteristics stay within a small dyadic annulus is limited. Consequently, this brief duration mitigates the singularity of the distance relative to the $z$-axis. 

The iterative smoothing scheme (ISS), briefly outlined in \eqref{skeletonofproof} (with further details in Section \ref{ISSintroduction}), aims to control the increment of either velocity or the distance to the $z$-axis over the characteristic time. This control, achieved by iteratively exploiting the smoothing effect of the electromagnetic field, is a central component of the bootstrap argument.

Now, we explain more about the outline and  we proceed in steps as follows. 

\medskip

\noindent \textbf{Step 1}.\qquad  The first two lines in   \eqref{skeletonofproof}.

\medskip

The first two lines can be summarized as the following Proposition. 

\begin{proposition}
For any $t\in [0, T), $ if $\beta_t\leq 1-2\epsilon$, then we have 
\be
\overline{\mathfrak{M}}(t) +\|(E(t), B(t)\|_{L^\infty_x} \lesssim   (1+t)^{(N_0)^3}.
\ee
Therefore, the continuation criteria in Theorem \ref{continuationcriteria} is verified and the global existence is established.
\end{proposition}
\begin{proof}

Roughly speaking, $\beta_t\leq 1-2\epsilon$ implies   the \textit{\textbf{sub-linearity}} of    $ \overline{\mathfrak{M}}(t) $. Recall  \eqref{may9en21}  and the definition of the $t$-majority set in  \eqref{may10majority}. From the estimate $\beta_t \leq 1-2\epsilon$, we know that  $|  X(x,v,0,t) |+|  V(x,v,0,t) |\geq 2^{  M_t /2}$ if $|v|\geq 2^{(1- \epsilon )M_t }$. From the decay assumption of the initial data in  \eqref{assumptiononinitialdata}  and the following estimate holds if $|v|\geq 2^{(1- \epsilon)M_t }$, 
\be\label{may11eqn31}
\begin{split}
|f(t,x,v)|&=| f_0(X(x,v,0,t),  X(x,v,0,t))| \\
& \lesssim (1+|X(x,v,0,t)|+|V(x,v,0,t)|)^{-N_0+10}.\\
\end{split}
\ee
Recall  \eqref{backward}. From the rough   $L^\infty_x$ estimate of electromagnetic field in  \eqref{maintheoremroughest}  in Theorem \ref{maintheorem1part1}, we have
\be\label{may11eqn32}
\begin{split}
 & \textup{If\,\,} |v|\gtrsim 2^{ (3+20\epsilon) M_t},\quad \big|v-V(x,v,0,t)|\big| \lesssim    2^{ (3+10\epsilon) M_t}, \quad \Longrightarrow |v|\sim |V(x,v,0,t)|. \\
 & \textup{If\,\,} |x|\gtrsim 2^{2\epsilon M_t}, \quad \big|x-    X(x,v,0,t) \big| \lesssim    2^{\epsilon M_t}, \quad \Longrightarrow |x|\sim |X(x,v,0,t)|.\\
\end{split}
\ee  
 
Hence, if $|v|\geq 2^{(1- \epsilon)M_t }$, after combining the above estimates  \eqref{may11eqn31}--\eqref{may11eqn32}, we have  
\be
|f(t,x,v)|=| f_0(X(x,v,0,t),  X(x,v,0,t))|\lesssim (1+|x|)^{-4}(1+|v |)^{-N_0/10 -4}.  
\ee
From the above estimate, we have
\be\label{may11eqn34}
 \big| \int_{\R^3}\int_{|v|\geq 2^{(1- \epsilon)M_t }}(1+|v|)^{  N_0/10} f(t,x,v)   d x dv\big|\lesssim 1.
\ee

It remains to consider the case   $|v|\leq 2^{(1- \epsilon)M_t }$,  from the conservation law  \eqref{conservationlaw}, we have 
\be\label{march20eqn11}
 \big| \int_{\R^3}\int_{|v|\leq 2^{(1- \epsilon )M_t }}(1+|v|)^{N_0/10} f(t,x,v)   d x dv\big|\lesssim  2^{(N_0/10-1)(1- \epsilon)M_t}\leq (\overline{\mathfrak{M}}(t))^{ 1- \epsilon }. 
\ee
Recall \eqref{may2eqn1}. From the estimates  \eqref{may11eqn34}  and  \eqref{march20eqn11},   for any  $t\in[0,T)$, we have 
\be 
\mathfrak{M}(t)\lesssim \big( \overline{\mathfrak{M}}(t)  \big)^{  1- \epsilon }.  
\ee
From the above estimate and 
the fact that $ \mathfrak{M}(t) $ is an increasing function with respect to $t $,   we have 
\be\label{2024oct12eqn31}
 \overline{\mathfrak{M}}(t)=\sup_{s\in [0, t]} {\mathfrak{M}}(s)+  (1+t)^{(N_0)^3}  \lesssim \big(   \overline{\mathfrak{M}}(t) \big)^{  1- \epsilon}+  (1+t)^{(N_0)^3} , 
 \quad \Longrightarrow   \overline{\mathfrak{M}}(t) \lesssim   (1+t)^{(N_0)^3}. 
\ee

Therefore, $ \overline{\mathfrak{M}}(t)  $ grows at most at rate $(1+t)^{(N_0)^3}$ over time.    From the rough   $L^\infty_x$ estimate of electromagnetic field in  \eqref{maintheoremroughest}  in Theorem \ref{maintheorem1part1}, we know that $\| (E(t), B(t))\|_{L^\infty_x} $ grows at most at rate $(1+t)^{(N_0)^3}$ over time as well. This concludes the fact that $(E(t), B(t))\in L^\infty([0, T)\times \R^3 ) $. Therefore, the continuation criteria in  Theorem \ref{continuationcriteria} is verified.

\end{proof}

\medskip

\noindent \textbf{Step 2}.\qquad  
 The last three lines  in   \eqref{skeletonofproof}.

\medskip

We begin by outlining the bootstrap arguments used in the final three lines of (\ref{skeletonofproof}), then present the proof as three independent Propositions.    Desipte $x,v\in \R^3,t\in[0,T)$ are arbitrary, they are fixed in the argument. For simplicity of notation, we drop the dependence of characteristics $(X(x,v,s,t), V(x,v,s,t))$, w.r.t., $x,v,t$.   We proceed in sub-steps as follows. 

\medskip

  \textbf{Step 2A}.   \qquad Firstly, we explain  the fourth line and the last line in   \eqref{skeletonofproof}.

 \medskip

By using a bootstrap argument, we start at a time  ``$s_1$'' such that  $|  V_{\bot}(s_1)|$ is close to $2^{(\alpha^{\star}-2\epsilon)M_t}$, and $|V(s_1)|$ is bounded from above by $2^{(1-3\epsilon)M_t}$. Now, we assume that these estimates also hold up to absolute constants within the time interval $[s_1,s_2]$. Hence, $\forall s\in[s_1,s_2]$, $|  V_{\bot}(s)|/|V(s)|$ has a good lower bound, more precisely, $|  V_{\bot}(s)|/|V(s)|\gtrsim 2^{(\alpha^{\star}-1+\epsilon)M_t}$.

Given the cylindrical symmetry of the distribution function, we anticipate that particles will not be strongly localized when they are away from the 
 $z$-axis. The advantages of cylindrical symmetry become more pronounced as the distance from the $z$-axis increases. Therefore, it is reasonable to consider the distance from the $z$-axis as a significant factor. Motivated by this, we proceed to compute the following quantity:
 \be\label{2024Dec5eqn1}
   \frac{d^2}{ds^2}  |  X_{\bot}(s)|^2 =2\big( \frac{|  V_{\bot}(s)|^2 }{1+|V(s)|^2}  +C_0(  X_{\bot}(s), V(s))\cdot K(s, X(s), V(s) )\big),
\ee
where
\[
C_0( X_{\bot}(s), V(s))=\frac{\big(  X_{\bot}(s), 0\big)}{\sqrt{1+|V(s)|^2}}- \frac{  X_{\bot}(s)\cdot   V_{\bot}(s) }{  \big({1+|V(s)|^2}\big)^{3/2}} V(s).
\]

 Roughly speaking,  thanks to the lower bound of  $| V_{\bot}(s)|/|V(s)|$ due to the bootstrap assumption, we demonstrate that
 \be\label{2022march2eqn1}
 \underbrace{ \int_{t_1}^{t_2}\frac{|  V_{\bot}(s)|^2 }{1+|V(s)|^2} ds}_{\textup{leading term}}   +  \underbrace{ \int_{t_1}^{t_2} C_0(  X_{\bot}(s), V(s))\cdot K(s, X(s), V(s) ) ds}_{\textup{just a perturbation}} .
 \ee

To prove that the second term in \eqref{2022march2eqn1} is perturbative, we use   the iterative smoothing scheme (ISS) together with the pointwise estimates established in Part II \cite{PartII},  which are tailored for the purpose of closing the bootstrap argument.  A detailed exposition of  the ISS and  the pointwise estimates   is presented in sections \ref{ISSintroduction} and \ref{mainresultsPartIIdetail}, respectively.

 Since the first term in \eqref{2022march2eqn1} is dominating,  from  \eqref{2024Dec5eqn1}, our key observation is that the \textbf{ space characteristics will {eventually travel  away from $z$-axis}}. Roughly speaking, the trajectory of  $ |   X_{\bot}(s)|^2 $ is almost like a parabola.    
As a result,   the shorter the distance interval with respect to the $z$-axis, the shorter the time lapse the trajectory $|  X_{\bot}(s)|$  travels.

With this improved understanding about the space characteristics, we study the variations of the full velocity characteristics and their projections with the time interval $[s_1,s_2]$. As a result of direct computation, from  \eqref{backward}, for any $t_1,t_2\in [s_1,s_2],$  we have
\be\label{2024Dec6eqn11}
\begin{split}
 & \big||  V_{\bot}(t_2)|  -  | V_{\bot}(t_1)| \big| \lesssim \big|\int_{t_1}^{t_2}  \big(  V_{\bot}(s)/|V(s)|, 0\big) \cdot K(s, X(s), V(s) )  ds\big| , \\
 &    \big| |  V(t_2)| - |V(t_1)|  \big|  \lesssim  \big|  \int_{t_1}^{t_2}  {V}(s)/|V(s)|  \cdot K(s, X(s), V(s) ) d s  \big|. \\
  \end{split}
\ee

Using the iterative smoothing scheme and the pointwise estimates from Part II, along with an estimate of the time spent by $|  {X}_{\bot}(s)|$ in each dyadic interval, we show that variations in both the full velocity characteristics and their projections are perturbative, thus completing the bootstrap argument.

 \medskip

   \textbf{Step 2B}.   \qquad  
Lastly, we     explain  the third line in \eqref{skeletonofproof}.

 \medskip

By using a bootstrap argument,  it suffices to start at a time $s_0$ such that $|V(s_0)|=2^{(1-3\epsilon)M_t}$ , which is smaller but close to $2^{(1-2\epsilon)M_t}$.   Now, we assume that the estimate of $|V(s)|$ also hold up to absolute constants within the time interval $[s_0,s_0']$ for some $s_0'$.  Recall \eqref{backward}.  As in \eqref{2024Dec6eqn11}, to control the increment of the velocity characteristics over time, it suffices to estimate the following quantity,
\be\label{2024Dec5eqn2}
\int_{t_1}^{t_2} \tilde{V}(s)\cdot K(s,X(s), V(s)) d s, \quad \tilde{V}(s):=V(s)/|V(s)|,
\ee
where $t_1, t_2\in [s_0,s_0']$. 

The gap between the upper bounds of $\alpha_t$ and $\beta_t$, combined with the iterative smoothing scheme and Part II's pointwise estimates, allows us to demonstrate that the integral in \eqref{2024Dec5eqn2} is perturbative.

\medskip

To sum up,  the last three lines in  \eqref{skeletonofproof} are encoded in the following Proposition.

 \begin{proposition}\label{finalproposition}
Assuming the  validities of Proposition \ref{bootstraplemma1} and Proposition \ref{bootstraplemma2}, for any $t\in [0, T^{})$, s.t., $M_t\gg 1,$ we have 
\be\label{finalestimate}
\alpha_t \leq  \alpha^{\star}:= (2/3+\iota), \quad
\beta_t \leq (1-2\epsilon), \quad \iota:= 10^{-4}. 
\ee
 \end{proposition}
\begin{proof}
Postponed to section \ref{bootstraparg} for better presentation. The logic of proof and  key observations are mentioned in \eqref{skeletonofproof}. 
\end{proof}

 The assumed two  Propositions in the     Proposition \ref{finalproposition} are motivated from the  main bootstrap arguments outlined in  \eqref{skeletonofproof}. To quantify the effect of acceleration force, we dyadic localize quantities of relevance, e.g., the distance to the $z$-axis, $|  X_{\bot}(s)|$, the projected velocity $|  {V}_{\bot}(s)|$, and the full velocity $|V(s)|$. The explicit coefficients appear in these two propositions are the results of computing $\p_s^2|   X_{\bot}(s)|^2$, $\p_s| V_{\bot}(s)|$, and $\p_s| V(s)|$ respectively. More precisely, assumed two  Propositions are stated as follows.  
\begin{proposition}\label{bootstraplemma1}
Let $t\in [0, T), a_p\in \Z, \alpha^{\star}=2/3+\iota, \iota:=10^{-4 }$,  $t_1, t_2\in [0, t]$ be fixed, s.t., $M_t\gg 1,$ $\forall s\in [t_1, t_2], |   X_{\bot}(s)|\sim 2^{a_p }, {\alpha}_s M_s\leq  \alpha^{\star} M_t,  |  V_{\bot}(s)|\sim  2^{\gamma_1 M_t},  |  V(s)|\sim  2^{\gamma_2 M_t} $,  where  $\gamma_1\in  [ \alpha^{\star} -4\epsilon , \alpha^{\star}]  $, $\gamma_2 \leq (1-2\epsilon)  $.  Let 
\be\label{2025oct16eqn21}
\begin{split}
C_1(  X_{\bot}(s), V(s))&:=\frac{\big(  X_{\bot}(s), 0\big)}{\sqrt{1+|V(s)|^2}}- \frac{ X_{\bot}(s)\cdot   V_{\bot}(s) }{  \big({1+|V(s)|^2}\big)^{3/2}} V(s),\quad \mathcal{M}(C_1):=2^{a_p-\gamma_2 M_t},\\
 C_2( X_{\bot}(s), V(s))&:= (\frac{V_1(s)}{|   V_{\bot}(s)|},\frac{V_2(s)}{|  V_{\bot}(s)|},0 ),\quad \mathcal{M}(C_2):=  1,\\ 
 C_3(  X_{\bot}(s), V(s))&:=   (\frac{V_1(s)}{|  V(s)|},\frac{V_2(s)}{|  V(s)|}, \frac{V_3(s)}{|  V(s)|} ),\quad \mathcal{M}(C_3):= 2^{\gamma_2M_t-\gamma_1M_t}.\\  
\end{split}
\ee
 Then the following estimate holds for any $i\in\{1,2,3\}$,
\be\label{nov17eqn31}
\begin{split}
&\big|\int_{t_1}^{t_2} C_i(  X_{\bot}(s), V(s)) \cdot   K(s, X(s), V(s)) d s \big| \\
& \lesssim  \mathcal{M}(C_i) \big[\big(\sum_{b\in \mathcal{T}+\mathcal{T}}   2^{-ba_p  } 2^{ b(\gamma_1-\gamma_2)M_t+ (\gamma_1-2\epsilon )  M_t}\big)(t_2-t_1)+  2^{(3\alpha^{\star} +20\epsilon)M_t/4}\big],
\end{split}
\ee
where the index set $\mathcal{T}:= \{0, 1/8,1/6,1/4,1/3,3/8,1/2\}.$
\end{proposition}
\begin{proof}
See section \ref{mainimprovedfull}. 
\end{proof}

\begin{proposition}\label{bootstraplemma2}
Let $t\in [0, T),  \alpha^{\star}=2/3+\iota, \iota:=10^{-4}$,  $t_1, t_2\in [0, t]$ be fixed, s.t., $M_t\gg 1,$  $\forall s\in [t_1, t_2],   \alpha_s M_s \leq \alpha^{\star}M_t  ,  |  V(s)|\sim 2^{\gamma  M_t}$, where  $ \gamma \geq  1-4\epsilon  $.
 Then the following estimate holds, 
\be\label{oct29eqn70}
\big|\int_{t_1}^{t_2}  \big(\frac{V_1(s)}{| V(s)|}, \frac{V_2(s)}{| V(s)|}, \frac{V_3(s)}{| V(s)|}\big) \cdot   K(s, X(s), V(s)) d s \big|\lesssim    2^{   (\gamma-5\epsilon) M_t }  .
\ee
\end{proposition}
\begin{proof}
See section \ref{fullimproved}. 
\end{proof}

Firstly, we briefly outline the relevance of the two propositions mentioned above.  Proposition \ref{bootstraplemma1} will be used in estimating ${\alpha}_t$, see \eqref{may9en21}, specifically corresponding to the last two lines in the outline \eqref{skeletonofproof}. This allows us to control both the velocity and, crucially, the spatial characteristics. Proposition \ref{bootstraplemma2} will be used in  estimating  ${\beta}_t$, see \eqref{may9en21}. It corresponds to the third  line  in the outline \eqref{skeletonofproof}. In essence, both propositions indicate that the electromagnetic field serves a perturbative role in the bootstrap argument.

Secondly, we explain the relevance of   coefficients in \eqref{2025oct16eqn21} and \eqref{oct29eqn70}. The coefficients we address in the bootstrap argument are quite specific. For Proposition \ref{bootstraplemma1}, the   coefficients listed in \eqref{2025oct16eqn21} come  from the estimate of $| X_{\bot}(s)|,$   $ |  V_{\bot}(s)|$, and $|V(s)|$ respectively, see also \eqref{2024nov4eqn1}. For Proposition \ref{bootstraplemma2}, the  the   coefficient  $ ({V_1(s)}/{| V(s)|},   {V_2(s)}/{| V(s)|},{V_3(s)}/{| V(s)|}\big)$ comes from the estimate of  $|V(s)|$ because it's   the sole objective of the third  line  in the outline \eqref{skeletonofproof}.

Lastly, we explain the different sets of assumptions in  Proposition \ref{bootstraplemma1} and  Proposition \ref{bootstraplemma2}. Moreover, we highlight the key difference between these two sets of assumptions. 
\begin{enumerate}
\item[$\bullet$]    Proposition \ref{bootstraplemma1} will only be used to show that   ${\alpha}_t$ is bounded above by $\alpha^{\star}$. By the standard bootstrap argument, we start with a characteristics such that $|   V_{\bot}(s)|$ is sufficiently close to $2^{\alpha^{\star}M_t}$ and $|V(s)|$ is bounded by $2^{(1-\epsilon) M_t}$. Due to the continuity of characteristics, within a small short interval of time,  the size of $|  V_{\bot}(s)|$, $|V(s)|$, and $|  X_{\bot}(s)|$ can be measured in a dyadic interval. This explains the assumption on different parameters in Proposition \ref{bootstraplemma1}.

  As explained in \textbf{Step 2A}, the shorter the distance interval with respect to the $z$-axis, the shorter the time lapse the trajectory travels. Specifically, the time spent within the dyadic spatial interval $[2^{a_p-1}, 2^{a_p}]$ (where $|  X_{\bot}(s)|$ resides) is on the order of $2^{a_p-(\gamma_1-\gamma_2)M_t}$. Crucially, this short time duration offsets the singularity of $|   X_{\bot}(s)|^{-1}\sim 2^{-a_p}$. This justifies the singularity of $| {X}_{\bot}(s)|^{-1}$ permitted in estimate (\ref{nov17eqn31}).

\item[$\bullet$]   

Proposition \ref{bootstraplemma2} serves solely to establish an upper bound of $1-2\epsilon$ for $\beta_t$. The standard bootstrap argument begins with a characteristic trajectory where $|V(s)|$ is approximately $2^{(1-3\epsilon) M_t}$. A key difference from Proposition \ref{bootstraplemma1} is the possibility of $|  V_{\bot}(s)|$ and $| X_{\bot}(s)|$ being simultaneously zero, allowing for characteristics propagating solely along the $z$-axis.  Crucially, the gap between the upper bounds on $\alpha_t$ (less than $2/3 + \iota$) and $\beta_t$ enables us to demonstrate that the resonance part, where the smoothing effect is absent, plays a perturbative role within the bootstrap argument.
\end{enumerate}

\subsection{Overview of Part II's main results}\label{mainresultsPartIrough}

This section provides a non-technical overview of Part II's main results, offering a preliminary understanding. A detailed presentation of the main results is provided in Section \ref{mainresultsPartIIdetail}

The first result in Part II concerns the    effectiveness of the classic moment method. Furthermore, it provides a rough upper bound for the electromagnetic field. 
 
\begin{theorem}[Rough statement about the electromagnetic field]
For any $t\in [0, T],$ $\|(E(t), B(t))\|_{L^\infty_x}$ can be controlled by the high moment  defined in \eqref{may2eqn1}. 
Furthermore, 
\begin{enumerate}
\item[(i)] As the distance from the $z$-axis  increases, the upper bound on the electromagnetic field improves.  
\item[(ii)] Using an atomic type decomposition of the electromagnetic field, we find that only  finite number of atoms  are significant.
\end{enumerate} 
\end{theorem} 
\begin{remark}
See Theorem \ref{maintheorem1part1} and Theorem \ref{roughesttailpart} for the precise statements. 
\end{remark}

Very often, we find that the quantity we need to control, such as $E+\hat{v}\times B$, differs from the estimates we have available, such as  the conservation law for $E-\omega \times B$
  (see \eqref{march18eqn31}). Importantly,  the difference  only lies in the magnetic field. Next result in Part II concerns the fine structure of the magnetic field.

\begin{theorem}[Rough statement about the fine structure of magnetic field]
Using an atomic type decomposition of the magnetic field, the $L^2_x$-norm of the atom is very small (better than the conservation law in \eqref{conservationlaw}) if the $L^\infty_x$-norm of the atom is very big. Furthermore, the third component, $\mathbf{P}_3(B)$, of the magnetic field has a better estimate than the magnetic field as a whole.
\end{theorem} 
\begin{remark}
See section \ref{magneticfielddictomy} for   precise statements. 
\end{remark}

Unfortunately, the above results are too rough to prove the desired estimates in Proposition \ref{bootstraplemma1} and Proposition \ref{bootstraplemma2} directly. To exploit and   quantify the smoothing effect, we decompose the acceleration force $K(t, X(t), V(t))$ (see \eqref{backward}) into localized pieces.  It's important to  consider the acceleration force as a whole, which possesses a double null structure, yields better results compared to analyzing the electromagnetic field separately. See section \ref{refinedec} for more details about the double null structure.

  In Part II, we  establish two sets of pointwise estimates for the localized  acceleration force by exploiting the benefits of the cylindrical symmetry and conservation laws, the smoothing effect and structural observations related to the electromagnetic field, including the null structure, the double null structure, and the dichotomy of the magnetic field.    The first set provides a uniform upper bound for the localized electromagnetic field. The second set offers an upper bound that depends on the distance from the $z$-axis, providing better control when the electromagnetic field is located further away from the $z$-axis.   A defining feature of these estimates is that they are assessed concerning the magnitude of the oscillation phase over characteristic time, effectively quantifying the smoothing effect.

These two sets of estimates are fundamental to the iterative smoothing scheme (ISS), providing both the  starting point for its execution and the basis for analyzing the nonlinear interactions between different generations of the electromagnetic field.

\begin{theorem}[Rough statement about the  pointwise estimates for the localized  acceleration]\label{roughstatementloc}
For the localized acceleration force, the following holds:
\begin{enumerate}
\item[(i)] Localized acceleration forces without smoothing effects constitute error terms in the bootstrap argument. If the localized acceleration force does not contribute to these error terms, the quantified smoothing effect has a positive lower bound.

\item[(ii)] If the localized acceleration force does not contribute to these error terms, the angle of localization has a lower bound. 
\item[(iii)] As the distance from the $z$-axis  increases, the upper bound of  the localized acceleration force    improves.   
\end{enumerate}
\end{theorem}
 
 \begin{remark}
See Theorem  \ref{maintheoremellipitic} and Theorem \ref{mainresultsfirstpart} for   precise statements. 
\end{remark}

\subsection{The iterative smoothing scheme}\label{ISSintroduction}

  Roughly speaking, the iterative smoothing scheme(ISS), applied to the $3D$ RVM system with cylindrical symmetry,  provide a good control of the following time integral, 
\[
\int_{t_1}^{t_2} C(  X_{\bot}(t), V(t))\cdot K(t,X(t), V(t)) d t, 
\]
where $C(  X_{\bot}(t), V(t))$ denotes some given smooth function and $K(t,X(t),V(t))$ is the acceleration force defined in \eqref{backward}.

We begin with a general overview of the iterative smoothing scheme, followed by a detailed explanation of its application to Propositions \ref{bootstraplemma1} and \ref{bootstraplemma2}.

\subsubsection{General ideas}

 Firstly, we should emphasize that  the smoothing effect in the ISS is different from the smoothing effect in the classic Glassey-Strauss decomposition \cite{glassey3}. Recall \eqref{march14eqn1}. On the Fourier side, for $U\in \{E, B\}$, we have
\[
U_j(t,x)= \int_0^t \int_{\R^3} e^{ix\cdot \xi+ i\mu(t-s)|\xi|} \frac{1}{ 2i \mu |\xi|}    \widehat{ \mathcal{N}^j_U } (s, \xi) d \xi d s.
\]
The basis of the classic Glassey-Strauss decomposition is primarily a result of integrating with respect to $s$, while keeping $t$ fixed.   

The smoothing effect in the ISS concerns the integration of electromagnetic field along characteristics time. For example, for the following typical quantity,
\be\label{2024nov5eqn1}
\begin{split}
&\int_{t_1}^{t_2}   U_j(t,X(t))   d t = \int_{t_1}^{t_2} \int_0^t \int_{\R^3} e^{iX(t) \cdot \xi+ i\mu(t-s)|\xi|} \frac{1}{ 2i \mu |\xi|}     \widehat{ \mathcal{N}^j_U } (s, \xi) d \xi d s d t,
\end{split}
\ee
the smoothing effect in   the ISS  mainly comes from the integration with respect to the characteristic time $t$.    More precisely, we exploit the following equality to do integration by parts in characteristic time $t$,
\be\label{2024nov5eqn21}
 e^{iX(t) \cdot \xi+ i\mu(t-s)|\xi|}= \frac{\p_t\big( e^{iX(t) \cdot \xi+ i\mu(t-s)|\xi|} \big)}{ i \mu |\xi|+i\hat{V}(t)\cdot \xi}.
\ee

If $|V(t)|$ is bounded, a full derivative can be obtained. This key result, presented in Klainerman-Staffilani \cite{Klainerman3} (Theorem 1.4, Fact 2), represents a significant departure from the smoothing effect exploited in the Glassey-Strauss decomposition. The present ISS approach is motivated by this result. To the best of our knowledge, this undervalued observation has not been extensively explored in the existing literature.  

One potential obstacle to utilizing this process is the appearance of a new, potentially large electromagnetic field term when integrating by parts with respect to $t$. This added term complicates the analysis. Additionally, the smoothing effect itself becomes less pronounced as the magnitude of $V(t)$ increases.

Because the smoothing process affects different parts of the electromagnetic field differently, we decompose it into finer pieces to quantify these varying gains and losses. This is the primary motivation for the fixed-time pointwise estimates for the localized electromagnetic field presented in Part II.

Intuitively speaking, after representing the acceleration force in terms of the half-wave and  localizing the size of  frequency and the angle between $\xi$ and $\mu \tilde{V}(t)$,  we have the following approximation on the Fourier side,
\be\label{2024oct5eqn1}
\begin{split}
& \int_{t_1}^{t_2} \tilde{V}(t)\cdot K(t,X(t), V(t)) d t\\
&  \approx \sum_{\begin{subarray}{c}
 \mu\in\{+, -\}\\
k\in \Z_+, n\in[-M_t, 2]\cap \Z\\
\end{subarray}}\int_{t_1}^{t_2} e^{i X(t)\cdot \xi +it \mu |\xi|} \widehat{\mathcal{K}^{\mu}}(t,\xi, V(t)) \varphi^{\mu}_{k,n}(\xi, V(t))  d t,
  \end{split}
\ee
where $ \widehat{\mathcal{K}^{\mu}}(t,\xi, V(t))$ denotes the profile of the half wave associated with the acceleration force on the Fourier side,  $ \varphi^{\mu}_{k,n}(\xi, V(t)) $ denotes a smooth cutoff function s.t., $|\xi|\sim 2^k, \angle(\xi, 
\mu \tilde{V}(t))\sim 2^n.$ 

Very importantly, note that, $ \forall \xi\in supp(\varphi_{k,n}(\xi, V(t))),$ we have 
\be\label{2024oct6eqn1}
  | \hat{V}(t)\cdot \xi +  \mu |\xi||\sim  2^{k+2n}. 
\ee
From the above estimate and the equality in \eqref{2024nov5eqn21}, it's clear that the strength of the smoothing effect is measured by the quantity $k+2n.$

The bootstrap argument's initial difficulty arises when $k + 2n = 0$. We must determine whether the bootstrap argument can be closed in the absence of smoothing.

As stated in Theorem \ref{roughstatementloc}, Part II's pointwise estimates resolve the potential issue of the absence of a smoothing effect. Specifically,   the first estimate in (\ref{2024oct8eqn1}) yields the following pointwise estimate: 
\[
\big|\int_{t_1}^{t_2} e^{i X(t)\cdot \xi +i t \mu |\xi|}  \widehat{\mathcal{K}^{\mu}}(t,\xi, V(t)) \varphi^{\mu}_{k,n}(\xi, V(t)) d t \big|\lesssim 2^{(k+2n)/2+(\alpha_t+10\epsilon) M_t   }.
\]

In light of the preceding pointwise estimate, the bootstrap assumption ensures control over the case  $k+2n\geq  (2-10^{-3})M_t/3$. The contribution from resonant terms is demonstrably negligible owing to the established crucial gap between the upper bounds of  $\alpha_t$ and $\beta_t$.  For the non-resonance case, i.e., $k+2n\geq  (2-10^{-3})M_t/3$, we    do integration by parts in ``$s$'' once for \eqref{2024oct5eqn1} by using the equality \eqref{2024oct6eqn1}. This procedure is now commonly called the normal form transformation, see the classic works of Germain-Masmoudi-Shatah \cite{GerMasSha09,GerMasSha12}. Further details regarding this procedure and the general Fourier method, both of which have significantly contributed to recent advances in the study of global stability for nonlinear dispersive PDEs, may be found in the  pioneering works of   Ionescu-Pausader \cite{IP1,IP2,IP3,IP}.

For convenience of reference, we call the above argument as \textbf{comparing-smoothing argument}: firstly, comparing the size of phase with some positive constant, which are   $k+2n$ and $ (2-10^{-3})M_t/3$ in above example; lastly, doing nothing for the small phase case since it's  already controlled and doing integration by parts for the large  phase case to exploit  the smoothing effect.

Our ISS iteratively applies a comparing-smoothing argument, demonstrating the dominance of the smoothing effect within a finite number of iterations.

 Because each iteration step creates a new electromagnetic field, we must typically consider the interactions of multiple nonlinear waves, unlike the linear wave case in \eqref{2024nov5eqn1}. The following typical term from the second iteration exemplifies the smoothing effect in this multilinear form:
\[
\begin{split}
 & \int_{t_1}^{t_2} \underbrace{K^{\nu}_{k_1,n_1}(t,X(t), V(t))}_{\textup{second generation of localized acceleration force}} \cdot  \underbrace{\tilde{K}^{\mu}_{k,n}(t,X(t), V(t))}_{\textup{first generation of localized acceleration force}}  d s \\
 &\approx  \int_{t_1}^{t_2} \int_{\R^3} \int_{\R^3} e^{iX(t)\cdot (\xi+\eta) + it \mu|\xi| + it \nu |\eta|} \varphi^{\mu}_{k,n}(\xi, V(t))\\
 &\quad \times \widehat{K}_{k_1,n_1}(t,\eta , V(t)) \cdot   \widehat{K}_{k,n}(t,\xi, V(t)) \varphi^{\nu}_{k_1,n_1}(\eta, V(t)) d \eta d \xi d t, \\
 \end{split}
\]
where $\mu, \nu\in\{+, -\}$. 

In particular, since the frequencies are localized by the cutoff functions, we have
\[
 | \hat{V}(t)\cdot \eta +  \nu |\eta||\sim  2^{k_1+2n_1}, \quad  | \hat{V}(t)\cdot \xi +  \mu |\xi||\sim  2^{k+2n},
\]
As a result, if  $k_1+2n_1 \gg k+2n$,   we have 
\be\label{2024nov7eqn21}
\begin{split}
e^{ iX(t)\cdot (\xi+\eta) + it \mu|\xi| + it \nu |\eta| } & = \frac{\p_t \big( e^{iX(t)\cdot (\xi+\eta) + i t\mu|\xi| + i t \nu |\eta|}\big)}{i \hat{V}(t)\cdot (\xi+\eta) + i\mu|\xi| + i \nu |\eta|}, \\
  \big| \hat{V}(t)\cdot (\xi+\eta) +  \mu|\xi| +   \nu |\eta| \big| &\sim 2^{k_1+2n_1}. 
\end{split}
\ee
As in \eqref{2024nov5eqn21}, now it's clear that there is a smoothing effect available if $k_1+2n_1 \gg k+2n$. 

An important question arises: what happens if  $k_1+2n_1 = k+2n$? For this case, we don't gain anything from using the equality \eqref{2024nov7eqn21} because $\hat{V}(t)\cdot (\xi+\eta) +  \mu|\xi| +   \nu |\eta|$ could be zero. 

Given the strong pointwise estimates from Part II, the case $k_1 + 2n_1 \leq k + 2n + c_1 M_t$, for some absolute constant $c_1 > 1/10$, is demonstrably negligible in the first stage. Its impact is confined to error terms, thus playing a minor role in the bootstrap argument.

Through $l$ iterations, it is established that the dominant contribution arises from multilinear wave interaction chains where, for each $i \in {1, \dots, l}$, $k_{i+1} + 2n_{i+1} \geq k_i + 2n_i + c_i M_t$ with $c_i \geq 1/10$. All other interaction chains represent negligible error terms. The cumulative smoothing effect, quantified by $k_{l+1} + 2n_{l+1}$, exceeds $(c_1 + \dots + c_l)M_t$. Given the upper bound on the electromagnetic field  established in Theorem \ref{maintheoremroughest} (see \eqref{maintheoremroughest}), the smoothing effect becomes dominant after a finite number of iterations.

The preceding discussion provides an initial understanding of how ISS operates. The following sub-sections will elaborate on the details, clarifying its role in the proofs of Propositions \ref{bootstraplemma1} and \ref{bootstraplemma2}. 
\subsubsection{Result of integration by parts in characteristics time }

Given that we will perform integration by parts in characteristic time multiple times, this process will generate numerous terms. Therefore, it is helpful to first outline our general approach to classifying the different terms produced during integration by parts in characteristic time.

From \eqref{2024nov5eqn1}, performing integration by parts in time 
$t$ will generate several terms. While some of these terms can be immediately controlled, others cannot. We refer to the terms that can be controlled right away as error type terms. Specifically, the endpoint terms resulting from integration by parts and those arising when 
 ``$\p_t$''
  acts upon the spatial characteristics, as seen in the coefficient 
$C( X_{\bot}(t),V(t))$, are classified as error type terms.

Unfortunately, there are two   types of terms that we cannot control immediately. More precisely, 
\begin{enumerate}
\item[(i)] The first type of terms, referred to as   \textbf{Type I terms}, arise  when the time derivative $\p_t$ acts upon the velocity characteristic $ {V}(t)$, creating an independent electromagnetic field as shown in \eqref{backward}. 
\item[(ii)]The second type of terms arise  when $\p_t$ acts on $\int_0^t e^{-is\mu|\xi|} \frac{1}{2i\mu|\xi|} \widehat{\mathcal{N}^j_U}(s,\xi) ds$. This introduces the nonlinearities of the Maxwell equations, $\widehat{\mathcal{N}^j_U}(s,\xi)$, as seen in \eqref{sep5eqn1}, leading to terms of the following form
\be\label{2024nov5eqn11}
\int_{\R^3}  f(t,x,v)(E(t,x)+\hat{v}\times B(t,x))\cdot \nabla_v\big(\varphi_j(v) \hat{v}\big)  d v.
\ee
It is evident from the preceding formula that a new electromagnetic field is present. We accordingly designate the associated terms as \textbf{Type II terms}.

\end{enumerate}

The following three-node tree diagram (Figure \ref{fig:IBPtree}) provides a visual representation of the terms generated after integrating by parts in time once. 
\begin{figure}[ht]
 \centering
 \includegraphics[width=0.3\textwidth]{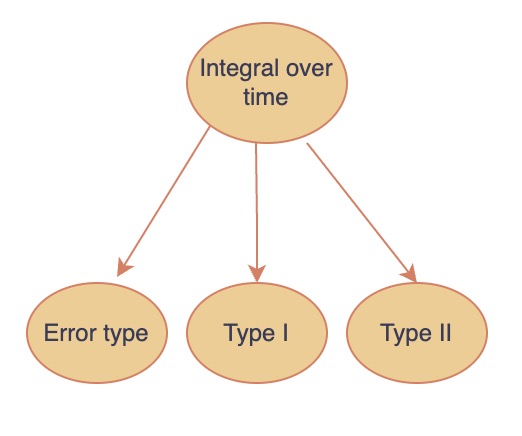}
 \caption{Classifying terms from  integration by parts in time.}
 \label{fig:IBPtree}
\end{figure}

Subsequent application of integration by parts with respect to time to terms of type I or type II will generate additional branches emanating from the corresponding nodes in the aforementioned tree.

Generally speaking, \textbf{Type II terms} are relatively easier than \textbf{Type I terms}. Because the strategy of exploiting  the $L^2_x$-conservation law of the electromagnetic field by  putting $\widehat{ \mathcal{N}^j_U } (s, \xi)$ in $L^2_\xi$ is available, see \eqref{conservationlaw}.  This strategy is not available to the first type of terms because the new generation of electromagnetic field evaluates at the location of space characteristics, which is a specific   value. For this scenario, the $L^2_x$ conservation law becomes irrelevant.

Consequently, new type II terms can be managed as error terms after at most two iterations\footnote{There is no a priori reason to expect the number of iterations to be less than two. Our argument's summary supports this conclusion. The elliptic estimate (Section \ref{ellipticpartPartII}) requires two iterations, while the hyperbolic estimate (Section \ref{hyperbolicTypeIIwithsymm}) requires none, as it's directly controllable without branching.  }, eliminating the need for additional time integration. In the tree structure, this translates to a maximum of two branches per node.

To sum up, the iterative smoothing scheme  (ISS) is representable as a finite-depth tree structure. Each node within this tree possesses either zero, two, or three child nodes. Importantly, any node  located at a depth greater than two from the root node will possess a maximum of two child nodes. A representative example of this tree structure is provided in the Figure \ref{fig:IBPtree2} below.

\begin{figure}[ht]
 \centering
 \includegraphics[width=0.55\textwidth]{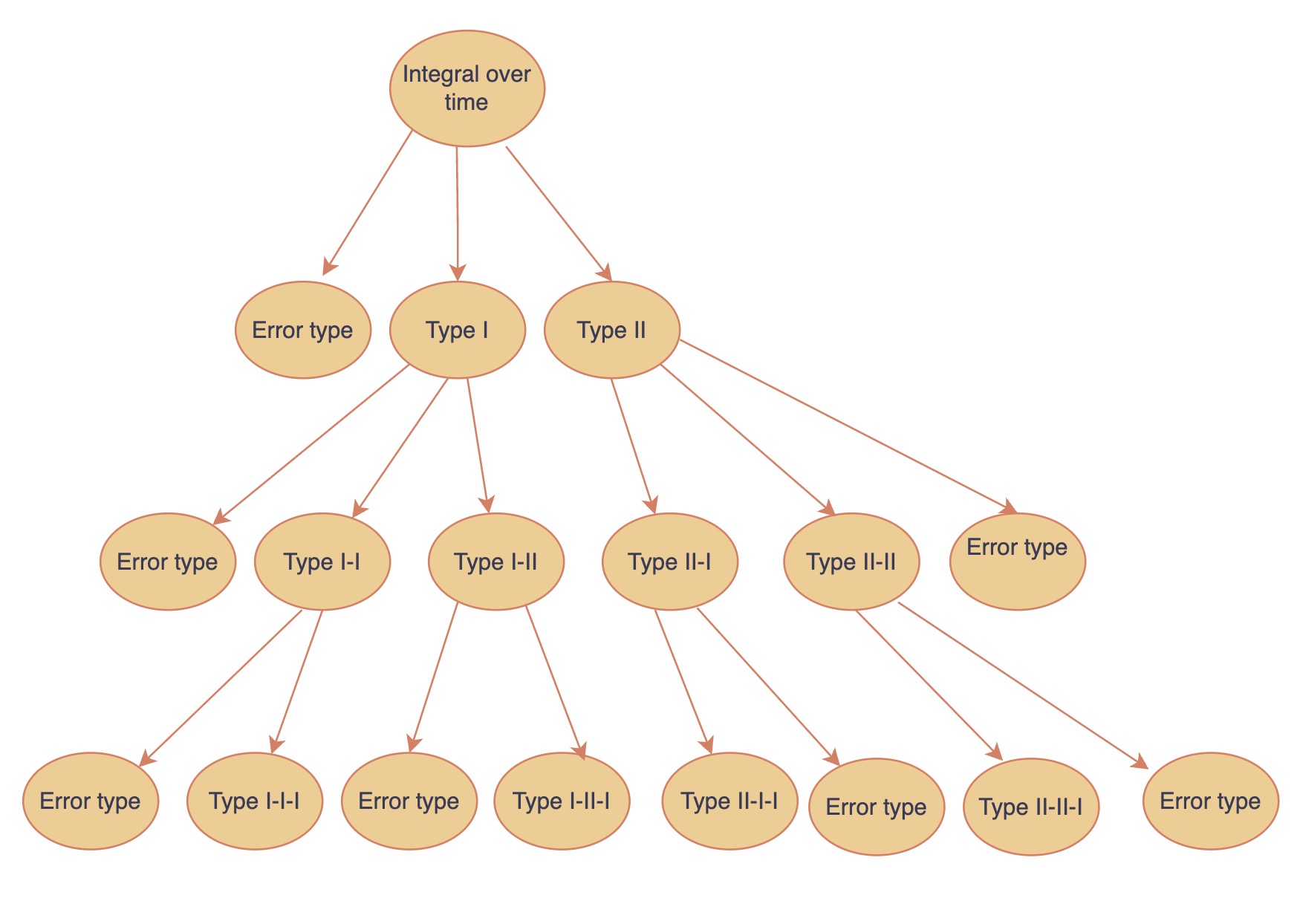}
 \caption{Classifying terms after triple integration by parts in time: An example}
 \label{fig:IBPtree2}
\end{figure}

Although the abstract ISS operates in a conceptually similar manner, variations arise in the number of children per branching node (two versus three). These differences stem from differing assumptions regarding the characteristics and distinct estimation goals outlined in Propositions \ref{bootstraplemma1} and \ref{bootstraplemma2}.

 In particular, the bootstrap assumption of Proposition \ref{bootstraplemma1} allows us to mitigate the singularity of $|  X_{\bot}(s)|^{-1}$ through exploitation of the smallness of the duration of time (cf. Section \ref{skeletonintroduction}). This, in conjunction with the cylindrical symmetry, permits the control of Type II terms without the introduction of branching nodes, effectively classifying them as error terms. Consequently, a simplified expansion tree structure is obtained for the proof of Proposition \ref{bootstraplemma1}.

\subsubsection{The first reduction}

In this  section,   we first   decompose the integral in estimates \eqref{nov17eqn31} and \eqref{oct29eqn70} into hyperbolic and elliptic components (modulo error terms).  These parts are distinguished by their    oscillatory behavior in characteristic time. The elliptic part is a consequence of the linear density terms present in the nonlinearity of the Maxwell system.

Roughly speaking, the elliptic parts have the following form, 
\[
\int_{t_1}^{t_2} \int_{\R^3}\int_{\R^3} e^{i X(t)\cdot \xi} \widehat{f}(t, \xi, v) d \xi d v  d t =\int_{t_1}^{t_2} \int_{\R^3}\int_{\R^3} e^{i X(t)\cdot \xi- i t\hat{v}\cdot \xi} \widehat{g}(t, \xi, v) d \xi d v  d t,
\]
where $g(t,x,v):=f(t,x+t\hat{v}, v)$ denote the profile of the distribution function\footnote{It's the pullback of the nonlinear solution along the linear flow. In the case of the linear equation, the profile simply becomes the initial data}.  Note that
\[
e^{i X(t)\cdot \xi- i t\hat{v}\cdot \xi}= \frac{\p_t \big( e^{i X(t)\cdot \xi- i t\hat{v}\cdot \xi}\big)}{i (\hat{V}(t)-\hat{v})\cdot \xi }.  
\]
The corresponding phase of oscillation in time is determined by the  function $(\hat{V}(t)-\hat{v})\cdot \xi$. It is very different from the phase function  $   \mu |\xi|+ \hat{V}(t)\cdot \xi$  that appears in the hyperbolic parts (see \eqref{2024nov5eqn21}) because $(\hat{V}(t)-\hat{v})\cdot \xi$ can vanish for a large set of frequencies.

More specifically, for the first reduction, we proceed in the following steps:
\begin{enumerate}
\item[] \textbf{Step 1}\quad  We first localize the acceleration force as in \eqref{sep5eqn66} and then decompose the localized acceleration force further into five pieces by using the decomposition \eqref{sep17eqn32}. 
\item[] \textbf{Step 2} \quad For the localized part, we employ the hyperbolic-elliptic decomposition given in \eqref{oct7eqn1}. By applying the results of Theorems \ref{maintheoremellipitic} (for the elliptic component) and \ref{mainresultsfirstpart} (for the hyperbolic component), we can exclude certain values of the parameters $k$ and $n$. This leaves us with a specific range of $k$ and $n$ (detailed in \eqref{oct15eqn1}), where the inequality $k + 2n \ge 2(1/3-5\iota)M_t$ holds. This inequality ensures sufficiently rapid phase oscillation, leading to the desired initial smoothing effect.

\item[] \textbf{Step 3} \quad For $k$ and $n$ within the previously defined range, and for $i \in \{0, 1, 2, 3\}$, the hyperbolic-elliptic decomposition \eqref{oct7eqn1} is directly applicable. This is because the decomposition already accounts for the first iteration of the smoothing effect. The resulting decomposition yields the hyperbolic and elliptic parts of the integral over the characteristic time.

The treatment of the $i=4$ case differs from the previous cases. The normal form transformation applied in \eqref{oct7eqn1} produces the symbol $(\mu |\xi| + \hat{v} \cdot \xi)^{-1}$, which can be highly singular if $v$ lies in the support of $\varphi_{j,n}^4(\cdot, \zeta)$ (see \eqref{sep4eqn6}). To avoid this singularity, instead of doing normal form transformation, we integrate by parts once in the characteristic time $t$ using the equality \eqref{2024nov5eqn21}. Type I terms are classified as hyperbolic. Considering the nonlinearities $\mathcal{N}^j_U(t, x)$ (defined in \eqref{sep5eqn1}), the linear density terms in Type II terms are classified as elliptic, and the nonlinear terms are classified as hyperbolic.

\end{enumerate}
 
 As a result of the above first reduction, to sum up, we conclude the following estimate, 
\be\label{2024nov9eqn211}
\begin{split}
\forall i\in\{0,1,2,3,4\},\quad & \big|\int_{t_1}^{t_2}     C (  X_{\bot}( t), V(t))  \cdot T_{k,j;n}^{\mu,i } (t, X(t), V(t)) d t \big|\\
&\lesssim 1 +   \big|  \mathfrak{E}^{\mu,i}_{k,j;n}(t_1, t_2)\big|  + \sum_{a \in\{0,1 \} }   \big|{}_a^{1}\mathfrak{H}^{\mu,i}_{k,j;n}(t_1, t_2)\big| +\big| {}_a^{1}Err^{\mu,i}_{k,j;n}(t_1, t_2)\big|,\\
\end{split}
\ee
where $ \mathfrak{E}^{\mu,i}_{k,j;n}(t_1, t_2)$ denotes the elliptic parts, and ${}_a^{1}\mathfrak{H}^{\mu,i}_{k,j;n}(t_1, t_2)$ denotes two different kinds of hyperbolic parts, and ${}_a^{1}Err^{\mu,i}_{k,j;n}(t_1, t_2)$ denotes the error part, see \eqref{oct25eqn1}--\eqref{nov12eqn61} for their detailed formulas.

\subsubsection{The estimate of the hyperbolic parts  ${}_a^{1}\mathfrak{H}^{\mu,i}_{k,j;n}(t_1, t_2)$, $a\in\{0,1\}$.}

The detailed estimates of the hyperbolic parts are presented in Section \ref{hyperbolicPartIest} as part of the proof of Proposition \ref{bootstraplemma1}, and in Section \ref{hyperbolipartISSpartII}  as part of the proof of Proposition \ref{bootstraplemma2}. This section provides a high-level overview of the main ideas and techniques used in those detailed estimates.

We consider two cases, depending on the proposition being proved.

\medskip 

\noindent  \textbf{Case 1.}\quad In the proof of Proposition \ref{bootstraplemma1}.  

\medskip

The singularity of $| {X}_{\bot}(t)|^{-1}$ in the proof of Proposition \ref{bootstraplemma1} is handled by exploiting the short time interval, which is possible due to the bootstrap assumption in Proposition \ref{bootstraplemma1}. This observation, together with the cylindrical symmetry, allows for direct control of the Type II terms ${}_1^{1}\mathfrak{H}^{\mu,i}_{k,j;n}(t_1, t_2)$ without the need for iterative smoothing.

 The Type I terms ${}_0^{1}\mathfrak{H}^{\mu,i}_{k,j;n}(t_1, t_2)$  are estimated using the methods described earlier in this section. Specifically, we apply Theorem \ref{mainresultsfirstpart} (estimates \eqref{2024oct8eqn2}--\eqref{2024Dec6eqn31}) to eliminate certain cases, then perform five time integrations by parts (due to the oscillatory phase). The estimate is completed using the electromagnetic field bound in \eqref{2024oct28eqn61}.

Due to the fact that the lapse of time can only contributes the smallness of size   $| X_{\bot}(t)|^{ }$, we can at most use the estimates in  \eqref{2024oct8eqn5} and \eqref{2024Dec6eqn31} twice for two generations of acceleration force. For other generations of acceleration force, we need to use the rough $L^\infty_{x}$-type estimate in \eqref{2024oct8eqn1}.

\medskip 

\noindent  \textbf{Case 2.}\quad In the proof of Proposition \ref{bootstraplemma2}. 

\medskip

The main difference here stems from Proposition \ref{bootstraplemma2}'s bootstrap assumption, which allows for the possibility of characteristics traveling along the $z$-axis. This scenario prevents the direct use of the pointwise estimates \eqref{2024oct8eqn2} from Theorem \ref{mainresultsfirstpart} and \eqref{2024oct28eqn61} from Theorem \ref{maintheorem1part1} in controlling the Type II terms ${}_1^{1}\mathfrak{H}^{\mu,i}_{k,j;n}(t_1, t_2)$. Therefore, the ISS becomes necessary. To handle this, we decompose the $E + \hat{v} \times B$ term into a new generation of acceleration force, $E + \hat{V}(t) \times B$, and an error term, $(\hat{v} - \hat{V}(t)) \times B$. We control the error term using the dichotomy of the magnetic field  in Proposition \ref{meanLinfest} and apply four steps of iterative smoothing to the new acceleration force term.

The estimate for  ${}_0^{1}\mathfrak{H}^{\mu,i}_{k,j;n}(t_1, t_2)$  mirrors the iterative steps of Case 1. However, there are two key distinctions: first, the parameter sets differ; second, only estimates \eqref{2024oct8eqn1} from Theorem \ref{mainresultsfirstpart} and the coarser estimate \eqref{maintheoremroughest} from Theorem \ref{maintheorem1part1} are utilized in the final iteration.

 \subsubsection{The estimate of the elliptic parts $ \mathfrak{E}^{\mu,i}_{k,j;n}(t_1, t_2)$.}

A complete analysis of the elliptic parts is provided in Section \ref{ellipticPartIest} for the proof of Proposition \ref{bootstraplemma1} and in Section \ref{ellipticpartPartII} for the proof of Proposition \ref{bootstraplemma2}. This section focuses on the main ideas and methods used in those detailed  estimates.

We first discuss about the general strategy. We first rule out certain set of parameters of $k,n$ by using the estimates in Theorem \ref{maintheoremellipitic} and Theorem \ref{mainresultsfirstpart}. For the rest of parameters of $k,n$, in which there is a smoothing effect,  we do integration by parts in time for the case  $i\in\{0,1,2,3\}$ because the oscillation function  $(\hat{V}(t)-\hat{v})\cdot \xi$ has a lower bound. The case $i=4$ is more delicate because   the oscillation function $(\hat{V}(t)-\hat{v})\cdot \xi$  can  vanish, see \eqref{sep4eqn6}.  For this case,    we use the   inhomogeneous type partition of unity  based on the size of $(\hat{V}(t)-\hat{v})\cdot \xi$ as follows,
\[
1= \sum_{\kappa\in [\bar{\kappa},2]\cap\Z} \varphi_{\kappa;\bar{\kappa}}\big(\frac{(\hat{V}(t)-\hat{v})\cdot \xi}{|\xi||\hat{V}(t)-\hat{v}|} \big).
\] 
 We select the threshold $\bar{\kappa}$ to ensure the threshold case—using the cutoff function $\varphi_{\bar{\kappa};\bar{\kappa}}(\cdot) = \psi_{\leq \bar{\kappa}}(\cdot)$—is an error term. When $\kappa > \bar{\kappa}$, the frequency is localized away from the vanishing set, allowing for doing  integration by parts in time.

The following analysis is divided into two cases, with the specific strategy employed depending on the proposition under consideration.

\medskip 

\noindent  \textbf{Case 1.}\quad In the proof of Proposition \ref{bootstraplemma1}.  

\medskip 

The strategy for estimating the elliptic parts mirrors that used for the hyperbolic parts: the short time interval is exploited to compensate for the singularity of $| {X}_{\bot}(s)|^{-1}$ in the proof of Proposition \ref{bootstraplemma1}. By carefully analyzing the terms generated during the ISS process and leveraging cylindrical symmetry, the  estimate of elliptic parts is obtained within two iterations of smoothing.

\medskip 

\noindent  \textbf{Case 2.}\quad In the proof of Proposition \ref{bootstraplemma2}. 

\medskip 

The ISS process is most delicate here. The bootstrap assumption in Proposition \ref{bootstraplemma2} prevents us from ruling out characteristics traveling along the $z$-axis, thus precluding the direct application of cylindrical symmetry.

The general iteration scheme applies to Type I terms, as in the hyperbolic part estimates. Type II terms, however, require a more complex approach. For the hyperbolic part estimates of Type II terms, we decomposed $E + \hat{v} \times B$ into $E + \hat{V}(t) \times B$ and $(\hat{v} - \hat{V}(t)) \times B$. The term $(\hat{v} - \hat{V}(t)) \times B$ is readily controlled in the hyperbolic part estimates.

 Unfortunately, the control of the corresponding term in the elliptic component is not immediately achievable due to the presence of an \textbf{additional derivative} when compared to its counterpart in the hyperbolic component. This disparity is a direct consequence of the linear density term within the nonlinearity $\mathcal{N}^j_U$ possessing a higher derivative order than the quadratic term, see \eqref{sep5eqn1}.

Beside the general iteration scheme we discussed, there are two \textbf{additional} main ingredients used to control the type II terms in the elliptic part. We elaborate them as follows.
\begin{enumerate}
\item[(i)] The role of $(\hat{v}-\hat{V}(t))$.

In Section \ref{singularweight}, we derive a singular weighted space-time estimate for the distribution function. This result, of intrinsic value, affords additional control of the distribution function beyond that provided by the conservation laws (Equations \eqref{2024nov9eqn111}--\eqref{march18eqn31}). The derivation relies on the $L^2$-conservation law for the electromagnetic field and the advantageous properties of cylindrical symmetry. It is noteworthy that this estimate is employed only once in the estimate of elliptic parts.

  Although the singular weighted space-time estimate from Section \ref{singularweight} is coarse, it allows us to rule out   the case when $|\hat{v} - \hat{V}(t)|$ is not very small. This simplifies the analysis to the case $|\hat{v} - \hat{V}(t)|$ is very small. Since $\mathbf{P}(\hat{V}(t)) \le 2^{(\alpha^{\star} - \gamma)M_t} \le 2^{-M_t/3 + \iota M_t}$ is very small due to the bootstrap assumption, we know that $\mathbf{P}(\hat{v})$ is small as well.
To simplify the subsequent exposition,   symbols with size $|\hat{v} - \hat{V}(t)|  $ are termed  ``improved symbols'', while those with size $|\hat{v} - \hat{V}(t)|(|\mathbf{P}(\hat{v})| + |\mathbf{P}(\hat{V}(t))|)  $ are termed ``doubly improved symbols''. 

\item[(ii)] Refined decomposition  of $ \big((\hat{v}-\hat{V}(t))\times B\big)\cdot\nabla_v$.

Ingredient (i) implies that $(\hat{v} - \hat{V}(t))$ is an improved symbol, and furthermore, that $\mathbf{P}_3(\hat{v} - \hat{V}(t))$ is doubly improved. Additionally, since $\mathbf{P}(\hat{v})$ is small, $\partial_{v_3}(\text{symbol})$ exhibits better behavior than $\partial_{v_i}(\text{symbol})$ for $i \in \{1, 2\}$.

Exploiting the observations above and the structure of the cross product $((\hat{v} - \hat{V}(t)) \times B)\cdot \nabla_v$, we find that terms without doubly improved symbols depend only on $\mathbf{P}_3(B)$. 
  Fortunately,  $\mathbf{P}_3(B)$ exhibits better behavior than $\mathbf{P}(B)$ (see Section \ref{magneticfielddictomy}), we utilize the refined decomposition of $\mathbf{P}_3(B)$ in   \eqref{nov6eqn47} and the improved estimates in Propositions \ref{goodpartprojmagn} and \ref{set1goodPartP3B}.
\end{enumerate} 

With the aforementioned results and the refined decomposition as described in   \eqref{findecompellipticpart2024oct}, the estimation of the elliptic parts is achieved within four iterative smoothing steps.

 \subsubsection{The estimate of the error parts}

There are three sources of error terms. First source is the error terms in  the first reduction, see \eqref{2024nov9eqn211}. The other two sources   are   the ISS process of estimating the hyperbolic parts and the elliptic parts respectively. The error parts can be estimated without exploiting the smoothing effect again. Generally, the estimate of the error parts can be obtained easily as by-products in the ISS process. 

\subsubsection{The summability issue }
 
Since we will apply the dyadic decomposition many times and will have many independent parameters, it's beneficial to  discuss the  issue of summability with respect to these parameters here. Generally speaking, the tail parts, i.e., $k\geq 50 M_t$, $j\geq (1+2\epsilon)M_t$ etc, can be ruled out easily. As a result, we are left with at most $(M_t)^C$ cases of parameters  to be considered, where $C$ is some absolute constant. Since the final estimates we obtain at least have room  of size $2^{-\epsilon M_t}$, which compensates the  loss of size  $(M_t)^C$ from the total number of cases. Hence, for most of time, we can simply let these parameters to be fixed.

\subsection{On the full problem without symmetry and open problems}\label{fullprobandopenprob}

In this section, we explore ideas and methods from the two-paper sequence that may be beneficial for further investigations into the global regularity problem without symmetry. We highlight that the abstract iterative smoothing scheme   is invariant to symmetry, yielding identical terms in the iterative process; however, the overall smoothing effect differs significantly.

 Firstly, we present key elements of our arguments that are independent of cylindrical symmetry.
\begin{enumerate}
\item[(i)] The moment method and the sublinear  framework of the moment function as in \eqref{2024oct12eqn31}, i.e., the first two lines in  \eqref{skeletonofproof}. 

\item[(ii)]  Two methods of exploiting the smoothing effect: (i) Dyadically  localized Glassy-Strauss   decomposition; (ii)  Normal form transformation. 
\item[(iii)] The  Structure observations: null structure, double null structure, and the decomposition \eqref{july1eqn11}, are also valid. 
\end{enumerate}

Secondly, it is important to explicitly identify key facts that are not expected to hold without symmetry. These represent the primary challenges that need to be overcome.
\begin{enumerate}
\item[(i)] We don't expect the upper bound of $\alpha_t$ is better than the   upper bound of $\beta_t$. 
\item[(ii)] Because the north pole and the south pole will not be special without the cylindrical symmetry,  we don't expect that $\mathbf{P}_3(B)$ is better than $ B $. 

\item[(iii)] The pointwise estimates of the localized acceleration force  in Theorem \ref{mainresultsfirstpart} depends heavily on the cylindrical symmetry.  While it is possible to establish analogous estimates, the resulting upper bounds may be significantly weaker, potentially hindering the initiation of the iterative smoothing scheme; specifically, the resonance part at the initial stage may not be eliminated. A possible solution is to demonstrate an upper bound of the form 
   $2^{\alpha_1(k+2n) +\alpha_2 M_t} $, where    $\alpha_1,\alpha_2\in (0,1)$, for the  the localized acceleration force. If this holds true, we can eliminate the resonance part, indicating that the iterative smoothing scheme could still be effective.
\end{enumerate}

Finally, we outline several interesting open problems as follows. It is important to emphasize that the primary challenge continues to be the full problem without symmetry. Additionally, there appears to be no definitive evidence that global regularity of  
$3D$ RVM system holds for any regular  finite energy    initial data  without symmetry.
\begin{enumerate}
   \item[(i)] Assume that $B=0$ as long as the solution exists for the $3D$ RVM system \eqref{mainequation}, does the   global regularity  holds for   any regular  finite energy   initial data   without   symmetry? 
    
\item[(ii)] Does the   global regularity  of $3D$ RVP-PP   system  \eqref{vlasovpo} holds for     any regular  finite energy   initial data  without symmetry?  
\item[(iii)] If we consider an initial data of the form $Data_{sym}^1+Data_{gen}^2$, where $Data_{sym}^1$ is large and possesses cylindrical symmetry and $Data_{gen}^2$ is general, lacks symmetry, and is small in size, does  the $3D$ RVM  system still have global solution? This problem was suggested to me by Sergiu Klainerman, to whom I am very grateful. 
 \item[(iv)] Now that global existence is established, is it possible to demonstrate decay estimates for the  electromagnetic field in  the $2D$ RVM system,  the $2.5D$RVM system with general data, and  the $3D$ RVM system with cylindrical symmetric initial data?
\end{enumerate}
A few remarks are in order. 
\begin{remark}
Certainly, the magnetic field may not always be zero for generic initial data, making the question in (i) irrelevant. This inquiry arises from the curiosity about whether it’s possible to solve the full problem without considering the contribution from $EB^2$ (see \eqref{july1eqn11}), which is the most challenging aspect in various estimates.

\end{remark}

\begin{remark}
The question in (i) is simpler than the question in (ii) since $3D$ RVM system with $B=0$ has an additional equation, making it over-determined compared to the $3D$  RVP-PP system, see \eqref{mainequation} and \eqref{vlasovpo}. 

\end{remark}

 \subsection{Table of notations and the outline of this paper}\label{outlinefirstpaper}\label{outlinepaper}

The universally applicable notations and definitions central to this two-paper series are systematically delineated in Table   \ref{essential_global_notations_}.

\begin{table}[H]
\centering
\resizebox{\columnwidth}{!}{%
\begin{tabular}{ |c|c|c|c| } 
 \hline
 Notation & Definition & Appearance  & Remarks \\ 
 \hline
$N_0$ & $10^{10}$ & Theorem \ref{maintheorem} &  \\ 
  \hline
  $\epsilon$ & $10^{-7}=100/N_0$ & Convention \ref{conventionconst}  & The smallest absolute constant\\ 
  \hline
   $\iota$ & $10^{-4}=1000\epsilon  $ & Convention \ref{conventionconst}  & The next level of small constant\\ 
  \hline
 $\tilde{v}$ & $\tilde{v}:=v/|v|$ & Definition \ref{varioureldef} & The direction of $v=(v_1, v_2,v_3)$\\ 
 \hline
 $\hat{v}$ & $\hat{v}:=v/\sqrt{1+|v|^2} $ & Equation \ref{mainequation} & The relativistic velocity\\ 
 \hline
 $v_{\bot}$ & $v_{\bot}:=(v_1,v_2)$ & Definition \ref{varioureldef} & The first two projection of $v$\\ 
 \hline
  $\mathbf{P}(v)$  &$\mathbf{P}(v):=(v_1,v_2)$ & Definition \ref{varioureldef} &  $\mathbf{P}$ is mainly used for long notations\\ 
 \hline
 $\mathbf{P}_i(v)$  &$\mathbf{P}_i(v):=v_i$, $i\in\{1,2,3\}$ & Definition \ref{varioureldef} & $\mathbf{P}_i$ is mainly used for long notations \\ 
 \hline
 $  X(x,v,s,t ) $ & The space characteristics  & \eqref{backward} & Dependence of $x,v,t$ are often   \\
  $V(x,v,s,t )$& The velocity characteristics   & & dropped since they are fixed\\ 
 \hline
$K(s,X(s), V(s)) $ & The acceleration force  & \eqref{backward} &  \\
  \hline
$ M_t$ & Definition \ref{scaleofv} & \eqref{dec2eqn1} &  Remark \ref{scaleofvphilosophical} \\ 
 \hline
$\alpha_t, \beta_t  $ &  Definition \ref{tmajorityset} & \eqref{may9en21} & Remark \ref{sizeofvphilosophical}\\ 
 \hline
 $\tilde{\alpha}_t  $ & $\min\{\alpha_t, 1+2\epsilon\}$ & Definition \ref{tmajorityset}  & A posteriori, $\tilde{\alpha}_t$ ($\tilde{\beta}_t$) is same  \\ 
  $\tilde{\beta}_t  $ & $\min\{\beta_t, 1+2\epsilon\}$ &    & as $\alpha_t$ ($\beta_t$). It was not clear a priori.\\ 
 \hline
 $\alpha^{\star}$ & $2/3+\iota$ & \eqref{skeletonofproof} & A posteriori,  it's the upper bound 
 \\
 & & & of  $ \alpha_t$, independent of time; Remark \ref{remarktwothird}\\ 
 \hline
\end{tabular}%
}
\caption{Essential global notations.}\label{essential_global_notations_}
\end{table}

To facilitate comprehensive understanding of the technical content,  we accompany each section with a notation table, which is systematically positioned at the conclusion of each section. Readers are advised to consult these tables during detailed analysis to ensure contextual consistency.

 The rest of this paper is organized as follows, 
 \begin{enumerate}
    \item[$\bullet$] In section \ref{mainresultsPartIIdetail},   we present the main results in  Part II —two sets of pointwise estimates—and discuss   the motivations behind them.
\item[$\bullet$] In section \ref{singularweight}, we prove a singular weighted space-time estimate for the distribution function, which plays an important role in the estimate of elliptic parts in the proof of Proposition \ref{bootstraplemma2}.
\item[$\bullet$] In section \ref{bootstraparg},  by assuming the validity of  Proposition \ref{bootstraplemma1} and  Proposition \ref{bootstraplemma2}, we prove   Proposition \ref{finalproposition} by using a standard bootstrap argument. Hence finishing the proof of our main theorem \ref{maintheorem}. 
\item[$\bullet$] In section \ref{mainimprovedfull}, we give a postponed proof of  Proposition \ref{bootstraplemma1}. 

\item[$\bullet$] In section \ref{fullimproved}, we  give a postponed proof of  Proposition \ref{bootstraplemma2}.
\end{enumerate}

\noindent \textbf{Acknowledgment}\qquad 
The author thanks Yu Deng, Pin Yu, and Fan Zheng for helpful comments and suggestions on the early version of draft.  The author acknowledges support from  NSFC-12322110, 12141102, 12326602, and MOST-2020YFA0713003, 2024YFA1015000.

\section{Main results in Part II}\label{mainresultsPartIIdetail}

In this section, we begin by presenting the main results in Part II, followed by an explanation of the rationale behind seeking such estimates. Finally, we provide a brief overview of the key observations and ingredients utilized in deriving these results.

 Our first result  in Part II concerns controlling the electromagnetic field by the moments function $\overline{\mathfrak{M}}(t)$, see \eqref{may2eqn1},  i.e., the effectiveness of the moment method.  More precisely, in Part II, we have shown that 
\begin{theorem}\label{maintheorem1part1}
 Let $\epsilon=10^{-7}$.For  any  $t\in[0, T^{})$,   we have 
\be\label{maintheoremroughest}
 \| E(s,x)\|_{L^\infty_{[0,t]}L^\infty_x} +   \| B(s,x)\|_{L^\infty_{[0,t]}L^\infty_x}   \lesssim 2^{  (1+2\tilde{\alpha}_t +6\epsilon)M_t}.
\ee
Moreover, for any $x\in \R^3$, s.t., $| x_{\bot}|\neq 0$, the following pointwise estimate holds, 
\be\label{2024oct28eqn61}
 \sup_{s\in [0, t]}| E(s,x)| +   | B(s,x)|   \lesssim 1   +|  x_{\bot}|^{-1/2} 2^{ 5 \epsilon M_t}(2^{M_t}+ 2^{2 \tilde{\alpha}_t M_t }) + |  x_{\bot}|^{-1/4}  2^{5M_t/4+\tilde{\alpha}_t M_t/4+5\epsilon M_t}   .
\ee
 \end{theorem}
\begin{proof}
See \cite{PartII}[Theorem 1.4]
\end{proof}
\subsection{Refined decompositions and estimates of the localized acceleration force}\label{refinedec}

 Unfortunately, the above estimate is too rough to conclude any useful conclusion immediately. To see better the structure of electromagnetic field, we decompose them into fine pieces as follows.

 \begin{definition}[Angular localization for the electromagnetic field]\label{angularlocalizationelect}

  For $ U\in \{E, B\} $, any fixed $ \mu\in \{+,-\}, k\in \Z_+, j\in   \Z_+, n\in [-M_t, 2]\cap \Z, \zeta\in \R^3/\{0\}$, and Fourier symbol $\mathfrak{m}(\xi, \zeta)\in \mathcal{S}^\infty$ (see \eqref{symbolclassdefinition}), we define the following localized electromagnetic field,   in which the   frequency   and the angle with respect to a fixed direction $\tilde{\zeta}$\footnote{  The quantity $\zeta$ will serve as the velocity characteristic $V(s)$ in subsequent arguments..} are localized,  
  \be\label{sep5eqn66}
  \begin{split}
  &T_{k,j;n}^{\mu}(\mathfrak{m}, U)(t,x, \zeta):= \int_0^t \int_{\R^3} e^{i x\cdot \xi + i \mu(t-s)|\xi | } |\xi|^{-1}  \mathfrak{m}(\xi, \zeta)\widehat{\mathcal{N}^j_U}(s, \xi) \varphi_k(\xi)\varphi_{n;-M_t}\big( \tilde{\xi} + \mu \tilde{\zeta}\big) d \xi d s,  \\
    &T_{k;n}^{\mu}(\mathfrak{m}, U)(t,x, \zeta):=\sum_{j\in\Z_+}  T_{k,j;n}^{\mu}(\mathfrak{m}, U)(t,x, \zeta),\\
  \end{split}
  \ee
  where $\forall u\in \R^3/\{0\}, \tilde{u}:=u/|u|$.
\end{definition}

Recall \eqref{march14eqn1}. Actually, we  only need to consider the case $j\in[0, (1+2\epsilon)M_t]$  because  the   tail part of the electromagnetic field, i.e., $j\geq (1+2\epsilon)M_t$, is well controlled. More precisely, we have
\begin{theorem}\label{roughesttailpart}
 Let $\epsilon=10^{-7}$.For  any  $t\in[0, T^{})$,   we have 
\be\label{dec2eqn31}
\sum_{j\in \Z, j\geq (1+2\epsilon)M_t}\sum_{U\in \{E, B\}}\| U_j(t,x  )\|_{L^\infty_x} \lesssim \sum_{j\in \Z, j\geq (1+2\epsilon)M_t} 2^{-9j + (3+10\epsilon)M_t} \lesssim 2^{-5M_t}. 
\ee
\end{theorem}
  \begin{proof}
 See \cite{PartII}[Theorem 1.5].
 \end{proof}

Despite the localization in \eqref{sep5eqn66} is good enough for the purpose of measuring the smoothing effect, it turns out the angle between $v$ and $\zeta$ also play  an important role.

 Roughly speaking, for any $U\in \{E, B\}$,  the symbol of the linear part of  $\mathcal{N}_U^j$ is of size $\tilde{v}\times \tilde{\xi}$, see \eqref{sep5eqn1}. In the case where $v$ and $\xi$  are parallel, which constitutes the  worst case scenario, the symbol vanishes. This is commonly referred to as null structure. For a deeper understanding of null structure, see Klainerman's classic works \cite{Kla83,Kl85,Kla86} on the null structure of nonlinear wave equations.

    In particular,  the symbol of the linear part of  $\mathcal{N}_E^j+\hat{\zeta}\times \mathcal{N}_B^j$ is of size $(\hat{v}-\hat{\zeta})\times (\tilde{v}\times \tilde{\xi})$, which provides extra smallness when $v$ and $\zeta$ are very close. For this reason, we refer this structure of the acceleration force  as the \textbf{double null structure}. Hence, instead of estimating the electromagnetic field separately, we  estimate  the combination of  the electric field and the magnetic field as in  the acceleration force $K(s,X(x,v,s,t), V(x,v,s,t))$, see \eqref{backward}.

To fully exploit the benefit of the double null structure, in Part II\cite{PartII}, we obtained a refined decomposition for the nonlinearity  $\mathcal{N}^j_U(t,x)$,  $U\in \{E, B\}$,   see \eqref{sep5eqn1},
based on the possible size of $v$ and  the angle between $v$ and $\zeta$, which is fixed, as follows, 
\be\label{sep5eqn10}
\begin{split}
 \mathcal{N}^j_U(t,x)&=\sum_{i=0,1,2,3,4 } \mathcal{N}^{j,n}_{U;i}(t,x, \zeta)  , \\ 
\mathcal{N}^{j,n}_{E;i}(t,x,\zeta)&:= - 4\pi\big[ \int_{\R^3} \big( \hat{v} \p_t f(t, x, v)   +  \nabla_x f(t, x , v)\big)\varphi_{j,n}^i(v, \zeta) d v\big], \\ 
\mathcal{N}^{j,n}_{B;i}(t,x,\zeta) &:=-4\pi\int_{\R^3} \hat v\times \nabla_x f(t, x, v) \varphi_{j,n}^i(v, \zeta)d v,
\end{split}
\ee
where the cutoff functions $\varphi_{j,n}^i(v, \zeta),i\in\{0,1,2,3,4 \},$ are defined as follows, 
\be\label{sep4eqn6}
\begin{split}
\varphi_{j,n}^0(v, \zeta)& := \varphi_j(v)\psi_{\geq  \vartheta^{\star}_0}(\tilde{\zeta}- \tilde{v})\psi_{ \leq    (1/2+3\iota + 55\epsilon ) M_{t^{\star  }} +2}(v), \\
  \varphi_{j,n}^1(v, \zeta)&:= \varphi_j(v)\psi_{\geq  \vartheta^{\star}_0 }(\tilde{\zeta}- \tilde{v})\psi_{ >    (1/2+3\iota + 55\epsilon ) M_{t^{  \star}}+2 }(v),\\ 
\varphi_{j,n}^2(v, \zeta)&:= \varphi_j(v) \psi_{ [\vartheta^{\star}_1, \vartheta^{\star}_0) }( \tilde{\zeta}- \tilde{v} )  ,  \\
 \varphi_{j,n}^3(v, \zeta)&:= \varphi_j(v)  \psi_{<  \vartheta^{\star}_2   }( \tilde{\zeta} - \tilde{v} )
,\\
  \varphi_{j,n}^4(v, \zeta)&:= \varphi_j(v)  \psi_{ [\vartheta^{\star}_2, \vartheta^{\star}_1)   }( \tilde{\zeta} - \tilde{v} ), \\ 
  \vartheta^{\star}_0 &:=\max\{  n + \epsilon M_{t^\star}, -  10^{-3}  M_{t^\star} \},  \quad   \vartheta^{\star}_1:=  n + \epsilon   M_{t^\star}, \quad \vartheta^{\star}_2:=  n -\epsilon  M_{t^\star},
\end{split}
\ee
where $t^{\star}\in [0,T)$ is some fixed time s.t., $t\in [0, t^{\star}].$

We briefly explain  the motivation of decomposing into five pieces as in \eqref{sep4eqn6}.  We compare   the angle between $\zeta$ and $v$, which is a free variable, with respect to $2^n$, which measures $\angle(\xi, -\mu\zeta)$ due to the cutoff function in \eqref{sep5eqn66}.  Roughly speaking, there are four cases: (i) either  $\angle(v, \zeta)$ is not too small or the case $n$ is not too small; (ii) $n$ is small, $\angle(v, \zeta)$ and $2^n$ are not comparable and  $\angle(v, \zeta)$ is bigger than  $2^n$; (iii) $n$ is small, $\angle(v, \zeta)$ and $2^n$ are not comparable and  $\angle(v, \zeta)$ is smaller than  $2^n$; (iv) $n$ is small, $\angle(v, \zeta)$ and $2^n$ are almost comparable. This intuition corresponds exactly the partition of unity in \eqref{sep4eqn6} modulo a technical issue caused by the size of $v$, which motivates the definition of $\varphi_{j,n}^0(v, \zeta)$ and $\varphi_{j,n}^1(v, \zeta)$.

 We choose to split into pieces based on the size of the angle $\angle(v, \zeta)$
  because this quantity is part of  the double null structure. Additionally, since 
  $\zeta$ is fixed, it is much easier to make measurements with a constant reference.

  Correspondingly, we decompose the localized electromagnetic field into five parts as follows, 
  \be\label{sep17eqn32}
 \begin{split}
T_{k,j;n}^{\mu}( \mathfrak{m}, U)(t,x, \zeta)&:=\sum_{ i\in\{0,1,2,3,4 \}} T_{k,j;n}^{\mu,i}(  \mathfrak{m}, U)(t,x, \zeta), \\ 
T_{k,j;n}^{\mu,i}( \mathfrak{m}, U)(t,x, \zeta)& :=  \int_0^t \int_{\R^3} e^{i x\cdot \xi + i \mu(t-s)|\xi | } |\xi|^{-1}  \mathfrak{m}(\xi, \zeta)\widehat{ \mathcal{N}^{j,n}_{U;i} }(s, \xi, \zeta) \\
&\qquad \times \varphi_k(\xi)\varphi_{n;-M_t}\big(  \tilde{\xi} + \mu \tilde{\zeta} \big) d \xi d s. 
 \end{split}
 \ee

 Since the nonlinearities $ \mathcal{N}^j_U(t,x), U\in\{E, B\},$ contain linear terms, which depend on the distribution function $f$ and lose one derivative. After exploiting the smoothing effect, we have the following elliptic-hyperbolic type decomposition for the localized electromagnetic field, 
\be\label{oct7eqn1}
\begin{split}
  & T_{k,j;n}^{\mu,i}(  \mathfrak{m}, E)(t,x, \zeta) + \hat{\zeta}\times  T_{k,j;n}^{\mu,i}(  \mathfrak{m}, B)(t,x, \zeta)  \\
   & =  Ini_{k,j,n}^{\mu,i}( \mathfrak{m})(t, x, \zeta) + \mathfrak{H}_{k,j;n}^{\mu,i}( \mathfrak{m})(t, x, \zeta )     +  \mathfrak{E}^{\mu, i}_{k,j;n}( \mathfrak{m})(t,  x, \zeta) \mathbf{1}_{i\in\{0,1,2,3\} },
\end{split}
\ee
where $Ini_{k,j,n}^{\mu}( \mathfrak{m})( s, x, \zeta)$ depends only on initial data, which don't play a role in the argument.

We  refer  terms  $\mathfrak{E}^{\mu, i}_{k,j;n}( \mathfrak{m})( s, x, \zeta)$  as elliptic parts because they resemble  the structure of the electric field of the Vlasov-Poisson system \eqref{vlasovpo}, in which the electric field solves an elliptic equation. Moreover, we refer terms  $\mathfrak{H}_{k,j;n}^{\mu,i}( \mathfrak{m})( s, x, \zeta)$ as hyperbolic parts because they solve nonlinear wave equations.

 The elliptic parts  $\mathfrak{E}^{\mu, i}_{k,j;n}( \mathfrak{m})( t, \cdot, \zeta), i\in\{0,1,2,3\}, $ are defined as follows, 
\be\label{2022feb24eqn81} 
\begin{split}
   \mathfrak{E}^{\mu, i}_{k,j;n}( \mathfrak{m})(t, x, \zeta)& :=  \int_{\R^3}\int_{\R^3} e^{i x\cdot \xi} \hat{f}( t, \xi, v)   \mathfrak{m}(\xi, \zeta)  m_i(\xi, v, \zeta) \\
   &\quad\times \psi_k(\xi) \varphi_{j,n}^a(v,\zeta)  \varphi_{n;-M_t}( \tilde{\xi} + \mu \tilde{\zeta})    d\xi d v,\\
 m_i(\xi, v, \zeta)&:= \frac{    4\pi\big( {(\hat{v}-\hat{\zeta} )\times (\hat{v}\times \xi)+\xi(1-|\hat{v}|^2)} \big)   }{   |\xi|(\mu |\xi|+ \hat{v}\cdot \xi  )  },  
\end{split}
\ee
and   the hyperbolic parts $\mathfrak{H}_{k,j;n}^{\mu,i}( \mathfrak{m})(s, x, \zeta), i \in\{0,1,2,3,4\},  $ are defined as follows,
\be
    \mathfrak{H}_{k,j;n}^{\mu,i}( \mathfrak{m})(t, x, \zeta) := \int_{\R^3} e^{ix\cdot \xi + i \mu t|\xi|} \mathfrak{m}(\xi, \zeta) \mathcal{F}[  \mathfrak{H}_{k,j;n}^{\mu,i}] (t, \xi, \zeta) d \xi,
\ee
where
\be\label{2022feb22eqn81}
 \begin{split}
   \mathcal{F}[  \mathfrak{H}_{k,j;n}^{\mu, a }] (t, x, \zeta)& :=\int_0^t \int_{\R^3}  e^{  - i\mu \tau |\xi| } \mathcal{F}\big((E+\hat{v}\times B)f\big)(\tau, \xi, v)  \cdot \nabla_v \big[  m_{j,n}^a  (v, \zeta, \xi) \big]  \\
&\quad \times   \varphi_{n;-M_t} ( \tilde{\xi} + \mu \tilde{\zeta} )  \varphi_k(\xi) d v  d \tau,\quad a\in\{0,1,2,3\}, \\ 
  m_{j,n}^a  (v , \xi , \zeta)&:=   \big(
 \frac{   4\pi\big( {(\hat{v}-\hat{\zeta} )\times (\hat{v}\times \xi)+\xi(1-|\hat{v}|^2)} \big) }{    |\xi|(\mu |\xi|+ \hat{v}\cdot \xi  )  }+  \frac{\hat{v}}{|\xi|} \big) \varphi_{j,n}^a  (v, \zeta),\quad a\in\{0,1,2,3\}, \\ 
 \mathcal{F}[  \mathfrak{H}_{k,j;n}^{\mu,4  }] (t, x, \zeta)&:=  \int_0^t\int_{\R^3} e^{  - i \mu \tau |\xi | } \frac{\varphi_k(\xi)}{|\xi|} \varphi_{n;-M_t}( \tilde{\xi} + \mu \tilde{\zeta}  ) \\
&\quad \times \big(\widehat{ \mathcal{N}^{j,n}_{E;4} }(\tau, \xi, \zeta) + \hat{\zeta} \times  \widehat{ \mathcal{N}^{j,n}_{B;4} }(\tau, \xi, \zeta)\big) d \xi d \tau , \\ 
\end{split}
\ee
where   $   \mathfrak{H}_{k,j;n}^{\mu,a  }(s, \xi, \zeta)$ with $a\in \{0,1,2,3\}$   have two sources, one source is  from the normal form transformation when $\p_s $ hits   the profile $g(s,x,v):=f(s,x+s\hat{v},v)$ and the other source is from the nonlinearity of the electric field $\mathcal{N}^j_E(t, x)$. 
Meanwhile,  $  \mathfrak{H}_{k,j;n}^{\mu,4  }(s, \xi, \zeta)$ only comes from the nonlinearity of the electric field $\mathcal{N}^j_E(t, x)$ since we didn't do normal form transformation at the initial stage.

The elliptic part is absent in \eqref{oct7eqn1} for the case $i=4$ because we did not apply a normal form transformation in this case.  If we were to apply the normal form transformation, the resulting symbol, $(\mu |\xi|+ \hat{v}\cdot \xi)^{-1}$, would exhibit significant singularities when $v\in supp( \varphi_{j,n}^4(\cdot, \zeta))$ (see \eqref{sep4eqn6}). Thus, performing the normal form transformation for $i=4$ is inadvisable at this point.

 As summarized in the following Theorem,   we obtain the following fixed time pointwise estimates for  the elliptic parts in Part II \cite{PartII},
\begin{theorem}\label{maintheoremellipitic}
Let  Fourier symbol $m(\xi, \zeta)\in  L^\infty_\zeta \mathcal{S}^\infty$,   $\mu\in\{+,-\},$  $ t^{\star}\in [0, T), \alpha^{\star}:= 2/3+\iota, \iota:=10^{-4},   \zeta \in\R^3/\{0\}, t \in [0,  t^{\ast}]$ be fixed     s.t., $\alpha_t M_{t_{ }}\leq  \alpha^{\star} M_{t^{\star}}, M_t\gg 1$. 
For any fixed   $k\in \Z_{+}, j\in[0, (1+2\epsilon)M_t]\cap \Z,   n\in [-  M_t, 2]\cap \Z$, $i\in\{0,1,2,3\},$  
\begin{enumerate}
   \item[(i)] The following $L^\infty_x$-type rough estimate holds, 
\be\label{2022feb25eqn1}
\begin{split}
    & \big\|\mathfrak{E}^{\mu, i}_{k,j;n}(\mathfrak{m})(t, \cdot, \zeta)\big\|_{L^\infty_x}\\
      &\lesssim \| \mathfrak{m}(\cdot, \zeta)\|_{\mathcal{S}^\infty}  \min\big\{2^{ 4{\alpha}^{\star}  M_{t^{\star}}/3 + 5\epsilon M_{t^{\star}} } \big(2^{{\alpha}^{\star}  M_{t^{\star}}/3 } + (2^{M_{t^{\star}} }\frac{| \zeta_{\bot}|}{|\zeta|})^{1/3}\big),    2^{(1-20\epsilon)M_{t^{\star}} } \\
     &\quad  +2^{50\epsilon M_{t^{\star}}}\mathbf{1}_{n\geq  (1-2 {\alpha}^{\star})M_{t^{\star}}-40\epsilon M_{t^{\star}} }    \min\{2^{(k+2n)/2+  {\alpha}^{\star}  M_{t^{\star}} },2^{(k+4n)/2+(3{\alpha}^{\star}  -1)   M_{t^{\star}} }\}     \big\}. \\
  \end{split}
\ee
  \item[(ii)]  The following pointwise estimates hold for any $x, \zeta\in \R^3$ s.t., $|  x_{\bot}|\in (0, 2^{M_{t^{\star}}/2}]$,  $| \zeta_{\bot}|\sim   2^{\gamma_1 M_{t^{\star}}  } , |\zeta|\sim  2^{\gamma_2 M_{t^{\star}}   } ,$  where $\gamma_1\geq   \alpha^{\star} -4\epsilon    $, $\gamma_2 \leq (1-2\epsilon) $, 
\begin{itemize}
  \item[(ii-a)]If $ | x_{\bot}|\leq  -k-n +2\epsilon M_{t}$, we have 
  \be\label{2024oct27eqn1}
   \big|     \mathfrak{E}^{\mu, i}_{k,j;n} (t, x,\zeta)\big| \lesssim   \| \mathfrak{m}(\cdot, \zeta)\|_{\mathcal{S}^\infty}  |  x_{\bot}|^{-1} 2^{(\gamma_1-\gamma_2)M_t} 2^{ 5 {\alpha}^{\star} M_{t^{\star}}/6} \| \mathfrak{m}(\cdot, \zeta)\|_{\mathcal{S}^\infty}.\\
  \ee
    \item[(ii-b)]If  
 $ |  x_{\bot}|\geq  -k-n +3\epsilon M_{t}/2$, we have 
\be\label{2022feb24eqn1}
\begin{split}
  \big|     \mathfrak{E}^{\mu, i}_{k,j;n} (t, x,\zeta)\big| &\lesssim  \| \mathfrak{m}(\cdot, \zeta)\|_{\mathcal{S}^\infty}   |  x_{\bot}|^{-1/2} 2^{(  {\alpha}^{\star}+10\epsilon)M_{t^{\star}}}, \\
  \big|    \mathfrak{E}^{\mu, i}_{k,j;n} (t, x,\zeta)\big| &\lesssim \| \mathfrak{m}(\cdot, \zeta)\|_{\mathcal{S}^\infty}   |   x_{\bot}|^{-1/2}2^{  (\gamma_1-\gamma_2)M_{t^{  }}/2}   \big[ 2^{(  {\alpha}^{\star}-10\epsilon)M_{t^{\star}}}   + 
  \mathbf{1}_{n\geq  -\alpha^\star M_{t^{  }}/2  }\\
  &\quad \times  \min\{2^{(k+2n)/2+ 2{\alpha}^{\star}M_{t^{  }}/3}, 2^{(k+4n)/2+{\alpha}^{\star}M_{t^{  }}-(\gamma_1-\gamma_2)M_{t^{  }}/3}\}\big].
  \end{split}
\ee
\end{itemize}
 
\end{enumerate}
\end{theorem}
 \begin{proof}
See \cite{PartII}[Theorem 1.6].
\end{proof}
 
For the hyperbolic parts $\mathfrak{H}_{k,j;n}^{\mu,i}( \mathfrak{m})(s, x, \zeta ), i\in\{0,1,2,3,4\}$, we have 
 \begin{theorem}\label{mainresultsfirstpart}
Let  Fourier symbol $m(\xi, \zeta)\in  L^\infty_\zeta \mathcal{S}^\infty$,   $\mu\in\{+,-\},$  $ t^{\star}\in [0, T), \alpha^{\star}:= 2/3+\iota, \iota:=10^{-4},   \zeta \in\R^3/\{0\}, t \in [0,  t^{\ast}]$ be fixed     s.t., $\alpha_t M_{t_{ }}\leq  \alpha^{\star} M_{t^{\star}}, M_t\gg 1$. 
For any fixed   $k\in \Z_{+}, j\in[0, (1+2\epsilon)M_t]\cap \Z,   n\in [-  M_t, 2]\cap \Z$,  
 \begin{enumerate}
 \item[(i)] The following $L^\infty_x$-norm type estimates hold for any $\zeta\in \R^3/\{0\}$,
\be\label{2024oct8eqn1}
\begin{split}
 \sum_{i=0,2}\| \mathfrak{H}_{k,j;n}^{\mu,i}( \mathfrak{m})(t, x, \zeta )\|_{L^\infty_x} 
 &\lesssim \| \mathfrak{m}(\cdot, \zeta)\|_{\mathcal{S}^\infty}   \big[2^{(1-19\epsilon)M_{t^{\star}} } 
+   2^{  128\epsilon M_{t^{\star}} }    \mathbf{1}_{n\geq  \mathfrak{n}_1}\\
&\quad \times \min\{ 2^{(k+2n)/2+ (\alpha^{\star}+3\iota) M_{t^{\star}} }  , 2^{(k+4n)/2 +(1+6\iota)M_{t^{\star}}} \} \big], \\
 \| \mathfrak{H}_{k,j;n}^{\mu,1}( \mathfrak{m})(t, x, \zeta )\|_{L^\infty_x} 
 &\lesssim \| \mathfrak{m}(\cdot, \zeta)\|_{\mathcal{S}^\infty}   \big[2^{(1-19\epsilon)M_{t^{\star}} } 
+     2^{   121\epsilon M_{t^{\star}} }    \mathbf{1}_{n\geq  \mathfrak{n}_1 } \\
  &\quad \times  \min\{   2^{(k+2n)/2 + (2{\alpha}^{\star} -1)M_{t^{\star}}}  , 2^{(k+4n)/2 +(1+6\iota)M_{t^{\star}}} \} \big],\\
  \sum_{i=3,4}\| \mathfrak{H}_{k,j;n}^{\mu,i}( \mathfrak{m})(t, x, \zeta )\|_{L^\infty_x} 
 & \lesssim    \| \mathfrak{m}(\cdot, \zeta )\|_{\mathcal{S}^\infty}   \big[2^{(1-19\epsilon) M_{t^\star} } + 2^{128\epsilon M_{t^\star} } \mathbf{1}_{n\geq \mathfrak{n}_2 } \\
  &\quad \times \min\{2^{(k+2n)/2 +  {\alpha}^{\star} M_{t^\star}}, 2^{(k+4n)/2 +(7/6+5\iota/2)M_{t^\star}}\} \big],  \\
\end{split}
\ee
 where 
 \be
  \mathfrak{n}_1:= -(\alpha^{\star}+3\iota+60\epsilon) M_{t^{\star}}, \quad \mathfrak{n}_2:= - ( (1 +3\iota)/2+40\epsilon) M_{t^{\star}}. 
 \ee
 \item[(ii)] The following pointwise estimates hold for any $x, \zeta\in \R^3,$ s.t., $|  x_{\bot}|\in (0, 2^{M_{t^{\star}}/2}]$,  $|  \zeta_{\bot}|\sim   2^{\gamma_1 M_{t^{\star}}  } , |\zeta|\sim  2^{\gamma_2 M_{t^{\star}}   } ,$  where $\gamma_1\geq   \alpha^{\star} -4\epsilon    $, $\gamma_2 \leq (1-2\epsilon) $, 
 \begin{itemize}
  \item[(ii-a)] If $| x_{\bot} |\leq 2^{-k-n+2\epsilon M_{t^{\star}}}$, we have 
\be\label{2024oct8eqn2}
\begin{split}
&\sum_{i=0,1,2,3,4} 2^{(\gamma_1-\gamma_2)M_{t^{\star}}}\big|\mathfrak{H}_{k,j;n}^{\mu,i}( \mathfrak{m})(t, x, \zeta ) \big|+\big|\mathbf{P}\big(\mathfrak{H}_{k,j;n}^{\mu,i}( \mathfrak{m})(t, x, \zeta ) \big)\big|\\  
&\lesssim   \| \mathfrak{m}(\cdot, \zeta )\|_{\mathcal{S}^\infty}  \big[\sum_{a\in \mathcal{T}+\mathcal{T}} |  x_{\bot}|^{-a}2^{a(\gamma_1-\gamma_2)M_{t^{\star}} } 2^{( \alpha^{\star}   -10\epsilon)M_{t^{\star}}}  \big],\\
\end{split}
\ee
  \item[(ii-b)] If $|  x_{\bot} |\geq 2^{-k-n+\epsilon M_{t^{\star}}}$, we have 
  \be\label{2024oct8eqn5}
\begin{split}
 &\sum_{i=0,1,2 } 2^{(\gamma_1-\gamma_2)M_{t^{\star}}}\big|\mathfrak{H}_{k,j;n}^{\mu,i}( \mathfrak{m})(t, x, \zeta ) \big|+\big|\mathbf{P}\big(\mathfrak{H}_{k,j;n}^{\mu,i}( \mathfrak{m})(t, x, \zeta ) \big)\big|\\
&\lesssim  \| \mathfrak{m}(\cdot, \zeta )\|_{\mathcal{S}^\infty} \Big[   2^{    40\epsilon M_{t^{\star}}}     \big( \min\{ |  x_{\bot}|^{-1/2} 2^{7 \alpha^{\star}  M_{t^{\star}} /8}, |  x_{\bot}|^{-1} 2^{3 \alpha^{\star}  M_{t^{\star}} /8  } \}\big)  \\
& \quad  +  \sum_{a\in \mathcal{T}} |  x_{\bot}|^{-a}2^{a(\gamma_1-\gamma_2)M_{t^{\star}} } \big[2^{( \alpha^{\star}   -10\epsilon)M_{t^{\star}}}+    \mathbf{1}_{n\geq -(\alpha^\star/2 + 30\epsilon)M_{t^\star}}   \\
&\quad   \times  \min\big\{2^{(k+2n)/2+ 2\alpha^{\star}M_{t^\star}/3- (\gamma_1-\gamma_2)M_{t^{\star} }/6},   2^{(k+2n)/2+  \alpha^{\star}M_{t^\star} - (\gamma_1-\gamma_2)M_{t^{\star} }/3 }\big\}  \big] \Big], \\ 
\end{split}
\ee
\be\label{2024Dec6eqn31}
\begin{split}
 &\sum_{i=3,4 } 2^{(\gamma_1-\gamma_2)M_{t^{\star}}}\big|\mathfrak{H}_{k,j;n}^{\mu,i}( \mathfrak{m})(t, x, \zeta ) \big|+\big|\mathbf{P}\big(\mathfrak{H}_{k,j;n}^{\mu,i}( \mathfrak{m})(t, x, \zeta ) \big)\big|\\
 &\lesssim  \| \mathfrak{m}(\cdot, \zeta )\|_{\mathcal{S}^\infty} \Big[   2^{    40\epsilon M_{t^{\star}}}     \big( \min\{ |  x_{\bot}|^{-1/2} 2^{7 \alpha^{\star}  M_{t^{\star}} /8}, |  x_{\bot}|^{-1} 2^{3 \alpha^{\star}  M_{t^{\star}} /8  } \} \big) \\
& \quad  +  \sum_{a\in \mathcal{T}} |  x_{\bot}|^{-a}2^{a(\gamma_1-\gamma_2)M_{t^{\star}} } \big[2^{( \alpha^{\star}   -10\epsilon)M_{t^{\star}}}+    \mathbf{1}_{n\in  \mathcal{N}_{t^{\star} }^1} \min\big\{2^{(k+2n)/2}   \\
&\qquad   \times  2^{ 2\alpha^{\star}M_{t^\star}/3- (\gamma_1-\gamma_2)M_{t^{\star} }/6},   2^{(k+2n)/2+  \alpha^{\star}M_{t^\star} - (\gamma_1-\gamma_2)M_{t^{\star} }/3 }\big\}  \\
& \quad +     \mathbf{1}_{n\in \mathcal{N}_{t^{\star}}^2  } \min\{   2^{(k+3n)/2 +3\alpha^{\star} M_{t^{\star} }/4  },   2^{(k+4n)/2 + \alpha^{\star} M_{t^{\star} }   - (\gamma_1-\gamma_2)M_{t^{\star} }/3}  \} \big]    \Big], \\
\end{split}
\ee
\end{itemize}
where $\mathbf{P}$ denotes the projection operator, see \eqref{definitionprojection}, and the sets $\mathcal{T}$, $\mathcal{N}_{t^{\star} }^1$, and $\mathcal{N}_{t^{\star} }^2$ are defined as follows, 
\be\label{2024oct8eqn8}
\begin{split}
&\mathcal{T}:=\{0,1/8, 1/6, 1/4,1/3, 3/8, 1/2\},\\
&\mathcal{N}_{t^{\star} }^1:=\{n: n\in [ (-  \alpha^{\star} +\gamma_1-\gamma_2) M_{t^{\star} }/2 -30\epsilon M_{t^{\star}}    , 2], \\ 
&\qquad n\notin  [    (\gamma_1-\gamma_2-2\epsilon)M_{t^{\star}},  (\gamma_1-\gamma_2+2\epsilon)M_{t^{\star}}] \}\cap \Z,  \\
&\mathcal{N}_{t^{\star} }^2:=\{n:  n\in  [    (\gamma_1-\gamma_2-2\epsilon)M_{t^{\star}},  (\gamma_1-\gamma_2+2\epsilon)M_{t^{\star}}] \}\cap \Z. 
\end{split}
\ee

 \end{enumerate}
\end{theorem}
\begin{proof}
See \cite{PartII}[Theorem 1.7]. 
\end{proof}

At first glance, the decompositions and estimates outlined above may seem intricate and complex. It is not immediately clear why we pursue these specific estimates.    To provide more context for the preceding results, a few remarks are in order.
\begin{remark}
We emphasize the importance of the sharpness of the upper bound. For instance, consider the following upper bound for the localized acceleration force,
\[
\begin{split}
\textup{(An example of not sufficient estimate)}\quad &\|T_{k,j;n}^{\mu}(\mathfrak{m}, E)(t,x, \zeta) + \hat{\zeta}\times  T_{k,j;n}^{\mu}(\mathfrak{m}, B)(t,x, \zeta)\|_{L^\infty_{x,v}} \\
&\quad \lesssim 2^{(k+2n)/2+ (1+3\iota) M_t },
\end{split}
\]
then this presents a challenge for cases where there is no smoothing effect, such as when $k+2n=0$. Additionally, ISS is not applicable here since we lack a starting point for our argument. Consequently, we are unable to establish the sub-linearity of the momentum function, which is central to our reasoning. 
\end{remark}
 
\begin{remark}
The reason we single out the projection part $\mathbf{P}(\mathfrak{H}_{k,j;n}^{\mu,i}( \mathfrak{m}))$
  in \eqref{2024oct8eqn2}, \eqref{2024oct8eqn5}, and \eqref{2024Dec6eqn31}, is that we observe it has a more favorable symbol compared to $\mathfrak{H}_{k,j;n}^{\mu,i}( \mathfrak{m})$.
  Specifically, the symbol of $\mathbf{P}(\mathfrak{H}_{k,j;n}^{\mu,i}( \mathfrak{m}))$
 introduces an additional smallness factor of  $| v_{\bot}|/|v|$
 , which is particularly small in the worst-case scenario, given that 
  $|  v_{\bot}|$ is bounded above by 
   $2^{\alpha_t M_t}$
 , while $|v|$
  can reach up to $2^{M_t}$.

\end{remark}

\begin{remark}
We adopt different assumptions for $\zeta$ in the two sets of estimates outlined in $(i)$ and $(ii)$ of Theorem \ref{mainresultsfirstpart} due to their distinct applications in bootstrap arguments, where the available information about velocity characteristics varies. Specifically, the estimates in $(ii)$ are intended solely to demonstrate that $\alpha_t \leq 2/3 + \iota$. In contrast, the estimates in $(i)$ can be utilized across all bootstrap arguments, as they do not impose any specific assumptions on $\zeta$. 
\end{remark}

\begin{remark}
In    estimates \eqref{2024oct8eqn2}, \eqref{2024oct8eqn5}, and \eqref{2024Dec6eqn31}, we consider two cases based on the size of $|   x_{\bot}|$ relative to $2^{-k-n+\epsilon M_{t^{\star}}/2}$, which defines the essential support\footnote{ In the two-paper sequence, ``essential support" refers to the region outside of which the tail of the kernel is negligible and well-controlled. } of the kernel due to frequency localization. We also include a margin of $2^{\epsilon M_{t^{\star}}}$ in our case division to allow flexibility in applying the smooth cutoff function.

Roughly speaking,  the estimate in  \eqref{2024oct8eqn2} indicates that the contribution from the case $|  x_{\bot}| \leq 2^{-k-n+2\epsilon M_{t^{\star}}}$ behaves like an error term in the bootstrap argument. Thus, it suffices to focus on the case $|  x_{\bot}| \geq 2^{-k-n+2\epsilon M_{t^{\star}}}$, where we have $|  x_{\bot} -   y_{\bot}| \sim |  x_{\bot}|$ for all $y \in B_{2^{-k-n+\epsilon M_{t^{\star}}}}(0)$. This means that the frequency localization process does not significantly alter the spatial distance of $x$ relative to the $z$-axis, which will be a key characteristic in the bootstrap argument.
\end{remark}
\begin{remark}
We emphasize  the importance of the lower bound of $n$. This explains the presence of characteristic functions $\mathbf{1}_{n\geq \cdot}$ in Theorem \ref{maintheoremellipitic} and Theorem \ref{mainresultsfirstpart}. The reason behind this is that there is an extra loss of $2^{-n}$ in the integration by parts process. It  happens when the time derivative hits the symbol $(\mu|
\xi|+\hat{V}(s)\cdot\xi)^{-1}$, which is introduced in the exploiting the  smoothing effect process. As a result,   we actually lose more than what we can gain from the smoothing effect if $n$ is very small.
\end{remark}
\begin{remark}
In   estimates \eqref{2024oct8eqn1},   \eqref{2024oct8eqn5}, and \eqref{2024Dec6eqn31}, we present two distinct estimates that highlight different aspects of the localized acceleration force and are applied at various stages of the comparing-smoothing argument, as detailed in \textbf{Step 2} of subsection \ref{skeletonintroduction}. During the comparison stage, we utilize the first estimate, as the critical quantity at this stage is $2^{k+2n}$, which represents the size of the oscillation phase (see \eqref{2024oct6eqn1}).

  Since the second estimate in \eqref{2024oct8eqn1} (and also in \eqref{2024oct8eqn5} and \eqref{2024Dec6eqn31}) is significantly better than the first when $n$ is very small, we employ the second estimate  when  we encounter an additional angular loss of size $2^{-n}$ that arises when the time derivative $\partial_s$ is applied on    the denominator (see \eqref{2024oct6eqn1}).

\end{remark}
\begin{remark}
The estimates derived in the above theorem, particularly \eqref{2024oct8eqn5} and \eqref{2024Dec6eqn31}, appear complex due to the presence of several independent parameters, such as $k$ for frequency, $n$ for angle, and $j$ for the size of $v$. Optimizing across different regions leads to varying results. We cannot sacrifice precision for simplicity in these estimates, as the iterative smoothing scheme relies on the small margins available between different estimates to function effectively.
\end{remark}

 Now, we briefly discuss the  main ingredients used in  obtaining the above main results. 
\begin{enumerate}
\item[(i)] Two ways of exploiting the smoothing effect. 
\begin{itemize}
\item[(1)] Dyadically  localized Glassy-Strauss type decomposition.
 \item[(2)]  Normal form transformation.
\end{itemize}
\item[(ii)] Structure observations: \begin{itemize}
\item[(1)] Null structure in both $E$ and $B$.
\item[(2)]  The double null structure for the acceleration force $E+\hat{\zeta}\times B$; 
\item[(3)] The following useful decomposition of the acceleration force in the S-part of the representation formula allows us to exploit the conservation laws,  
\be\label{july1eqn11}
\begin{split}
 &E(s, x )  + \hat{v}\times B(s, x  )
  = \sum_{i=1,2}EB^i(t,s,x ,\omega, v) ,
 \end{split}
\ee
where the precise formulas of $EB^i(t,s,x ,\omega, v), i\in\{1,2\},$ are not need for the discussion. Interested readers are refereed to part II \cite{PartII}[Section 1.6] for their precise formulas.     For $EB^1(t,s,x ,\omega, v)$, we can exploit the conservation law in Lemma \ref{conservationlawlemma} to control the electromagnetic field. Meanwhile, for $EB^2(t,s,x ,\omega, v)$, which depends only on the magnetic field,  there is an extra gain of smallness in terms of  $\angle(v, -\omega)$.
 
\end{itemize}
\item[(iii)]   A  dichotomy  for the localized magnetic field (see Proposition \ref{meanLinfest}), which provides a better control than the $L^2_x$-conservation law. It will be used to control the contribution from the $EB^2$-part of the S-part.

\end{enumerate}

\subsection{Refined decompositions and estimates of the    magnetic field}\label{magneticfielddictomy}

We often observe that the quantity we need to control, such as $E+\hat{v}\times B$
 in the hyperbolic parts (see \eqref{2022feb22eqn81}), differs from the estimates we have available, such as the estimate for the acceleration force $E+\hat{\zeta}\times B$ 
  or the conservation law for $E-\omega \times B$
  (see \eqref{march18eqn31}). It is important to note that the only difference between these quantities lies in the magnetic field.

To control these differences, in Part II, we established a dichotomy for the localized magnetic field. In essence, after decomposing the magnetic field into finer components, we demonstrated that the 
 $L^\infty_x$-norm
 and the $L^2_x$-norm of each piece cannot be large simultaneously.

More precisely,  for the magnetic field, we have the following atomic type decomposition, 
\be\label{2024oct30eqn2}
B(s,x)= \sum_{(m,k,j,l)\in \mathcal{S}_1(t)\cup \mathcal{S}_2(t)} B^{  {m}}_{  {k};  {j},    {l}}(s,x), 
\ee
where the precise formulas of $B^{  {m}}_{  {k};  {j},    {l}}(s,x)$ and index sets  $ \mathcal{S}_1(t)$ and $ \mathcal{S}_2(t)$
 don't play much role in Part I, we omit them here. Interested readers are referred to \cite{PartII}[Section 3.4] for details.  What we need to understand are the distinct properties of the different components. 

 For the localized magnetic field, the following dichotomy holds, 
 \begin{proposition}\label{meanLinfest}
For any $t\in [0, T^{})$,  we have 
\be\label{2024oct30eqn1}
\begin{split}
\sum_{(m,k,j,l)\in \mathcal{S}_1(t)} \sup_{s\in [0,t]}\|B^{  {m}}_{  {k};  {j},    {l}}(s,\cdot)\|_{L^\infty}  & \lesssim   2^{11\epsilon M_t} \big( 2^{ 2 \tilde{\alpha}_t  M_t }   + 2^{7M_t/6+\tilde{\alpha}_t M_t/4} \big),  \\
\sum_{(m,k,j,l)\in \mathcal{S}_2(t)}\big(\sup_{s\in [0,t]}\|B^{  {m}}_{  {k};  {j},    {l}}(s,\cdot)\|_{L^\infty} &  \big)^{\frac{1}{2}  } \big( \sup_{s\in [0,t]}\|B^{ {m}}_{  {k};  {j},  {l}}(s,\cdot)\|_{L^2}\big)^{ \frac{1}{2}}\\
&  \lesssim  2^{10\epsilon M_t} \big( 2^{\tilde{\alpha}_t M_t}   + 2^{7M_t/12+ \tilde{\alpha}_t M_t/8} \big). 
 \end{split}
\ee
\end{proposition}
\begin{proof}
See \cite{PartII}[Proposition 3.2].
\end{proof}

 Furthermore, in Part II, we demonstrated that the projection of the magnetic field onto the $z$-axis,
 $\mathbf{P}_3(B(t,x))$, is better than $ B(t,x)$, with certain components exhibiting similar estimates to the localized acceleration force, which has a double null structure. This advantage arises because 
   $\mathbf{P}_3(\hat{v}\times \xi)$ is better than $\hat{v}\times \xi$ 
 when 
 $\tilde{v}$
  and 
 $\tilde{\xi}$
  are near the north or south pole, allowing this symbol to fulfill a similar role as the double null structure.

To be more precise, we first decompose $\mathbf{P}_3(B(t,x))$ as follows based on  the size of frequency and  the angle between $\xi$ and $\pm \zeta$, 
\be\label{nov6eqn47} 
\begin{split}
\mathbf{P}_3(B(t,x))&= \sum_{\begin{subarray}{c}
k\in \Z_+, n\in [-M_t, 2]\cap \Z \\ 
  \mu\in \{+,-\}, n\geq \tilde{n}\\ 
\end{subarray}}   {}_{}^zT_{k,n}^{\mu}(B)(t,x,\zeta) \\
&\quad +   \sum_{\begin{subarray}{c}
k\in \Z_+, n\in [-M_t, 2]\cap \Z \\ 
  \mu\in \{+,-\}, n\leq \tilde{n}\\ 
\end{subarray}} {}_{}^zT_{k,n}^{\mu;1}(B)(t,x,\zeta) + {}_{}^zT_{k,n}^{\mu;2}(B)(t,x,\zeta), \\ 
{}_{}^zT_{k,n}^{\mu;2}(B)(t,x,\zeta)&= \sum_{(m,k,j,l)\in \mathcal{S}_1(t)\cup \mathcal{S}_2(t)} {}_{}^zT_{k,j;n,l}^{\mu;m,2}(B)(t,x,\zeta), 
\end{split}
\ee
where  $ \tilde{n}\in [-M_t, 2]\cap \Z$ is some fixed number s.t.,   $|  \zeta_{\bot}|/|\zeta|\leq 2^{\tilde{n}-10}$, and the detailed formulas of ${}_{}^zT_{k,n}^{\mu}(B)(t,x,\zeta)$ and  ${}_{}^zT_{k,n}^{\mu;i}(B)(t,x,\zeta), i\in\{1,2\},$ are given as follows, 
\be\label{P3Bdecomdef}
\begin{split}
{}_{}^zT_{k,n}^{\mu}(B)(t,x,\zeta) & : =\int_0^t \int_{\R^3}\int_{\R^3} e^{i x\cdot \xi +i \mu(t-s)|\xi|} \mathbf{P}_3(\hat{v}\times \xi )|\xi|^{-1} \\
&\quad \times  \widehat{f}(s, \xi, v) \varphi_k(\xi)  \varphi_{n;-M_t}( \tilde{\xi}+ \mu \tilde{\zeta})  d \xi d v d s,\\
{}_{}^zT_{k,n}^{\mu;i}(B)(t,x,\zeta)&: =\int_0^t \int_{\R^3} \int_{\R^3}e^{i x\cdot \xi  + i \mu(t-s)|\xi|} \mathbf{P}_3(\hat{v}\times \xi )|\xi|^{-1} \\
  &\quad \times  \widehat{f}(s, \xi, v) \varphi_k(\xi)   \varphi_{n;-M_t}( \tilde{\xi}+ \mu \tilde{\zeta}) \varphi_{\tilde{n}}^i(v, \zeta)  d \xi d v d s, \\ 
  \varphi_{\tilde{n}}^1(v, \zeta)&: =\psi_{\geq \tilde{n}+5}(\tilde{v}- \tilde{\zeta} ), \quad \varphi_{\tilde{n}}^2(v, \zeta): =\psi_{ < \tilde{n}+5}( \tilde{v}- \tilde{\zeta}).\\
\end{split}
\ee
Again, since the precise formulas of  ${}_{}^zT_{k,j;n,l}^{\mu;m,2}(B)(t,x,\zeta)$ don't play a   role in this paper, we omit them here.  Interested readers are referred to Part II\cite{PartII}[section 4.6] for details.

 As summarized in the following lemma, the quantity  ${}_{}^zT_{k,n}^{\mu;1}(B)(t,x,\zeta)$
 , which is part of 
 $ \mathbf{P}_3(B(t,x))$, enjoys an improved estimate due to the role of the symbol  $\mathbf{P}_{3}(\hat{v}\times \xi)$
 as the double null structure in the localized region. Additionally, the distinction between the null structure and the double null structure can be compensated by the smallness factor 
 $2^n$
 . For any localized piece $ {}_{}^zT_{k,n}^{\mu}(B)(t,x,\zeta)$ 
  that only incorporates the null structure instead of the double null structure, we obtain a slightly less favorable estimate (by a factor of 
 $2^n$
 ) compared to the localized acceleration force.
 
\begin{proposition}\label{goodpartprojmagn}
Let     $\mu\in\{+,-\},  k, j\in \Z_+, n, \tilde{n}\in [-M_t, 2]\cap \Z, $  $ t^{\star}\in [0, T), \alpha^{\star}:= 2/3+\iota, \iota:=10^{-4},   \zeta \in\R^3/\{0\}, t \in [0,  t^{\ast}]$ be fixed     s.t., $\alpha_t M_{t_{ }}\leq  \alpha^{\star} M_{t^{\star}}, M_t\gg 1$, and $|  \zeta_{\bot}|/ |\zeta|\leq 2^{\tilde{n}-10}$.   If  $n\leq \tilde{n}$,  then  the following estimates holds, 
\be\label{nov5eqn1}
\begin{split}
\big\|  {}_{}^zT_{k,n}^{\mu;1}(B)(t,x,\zeta) \big\|_{L^\infty_x}  &\lesssim   2^{(1-19\epsilon)M_{t^{\star}} }  + 2^{  130\epsilon M_{t^{\star}} }  \mathbf{1}_{n\geq  -(\alpha^{\star}+3\iota+60\epsilon) M_t    }    \\
&\qquad \times  \min\{ 2^{(k+2n)/2+ (\alpha^{\star}+3\iota) M_t }  , 2^{(k+4n)/2 +(1+6\iota)M_t} \}, \\
\big\|  {}_{}^zT_{k,n}^{\mu }(B)(t,x,\zeta) \big\|_{L^\infty_x}    &  \lesssim   2^{-n}     \big[   2^{(1-19\epsilon)M_{t^{\star}} } +   2^{  130\epsilon M_{t^{\star}} }  \mathbf{1}_{n\geq  -(\alpha^{\star}+3\iota+60\epsilon) M_t} 
\\
 &\qquad\times      \min\{ 2^{(k+2n)/2+ (\alpha^{\star}+3\iota) M_t }  , 2^{(k+4n)/2 +(1+6\iota)M_t} \}  \big]. \\
\end{split}
\ee 
 
\end{proposition}
\begin{proof}
See \cite{PartII}[Lemma 4.13].
\end{proof}

Although we do not have a similar double null structure for the other component of $ \mathbf{P}_3(B(t,x))$, namely ${}_{}^zT_{k,n}^{\mu;2}(B)(t,x,\zeta)$
 (see \eqref{nov6eqn47}), there remains an advantage in that the symbol 
    $\mathbf{P}_{3}(\hat{v}\times \xi)$ is better than $  (\hat{v}\times \xi)$.

   To exploit this advantage, comparing with the dichotomy     for the magnetic field $B$ in Proposition \ref{meanLinfest}, we have the following  improved dichotomy (by a factor of $2^{\tilde{n}}$) for the ${}_{}^zT_{k,n}^{\mu;2}(B)(t,x,\zeta)$ part of $ \mathbf{P}_3(B(t,x))$.

 \begin{proposition}\label{set1goodPartP3B}
For any fixed $\zeta\in \R^3/\{0\},  t\in [0, T)$,  $ \mu\in \{+,-\}, k, j\in \Z_+, n, \tilde{n}\in [-M_t, 2]\cap \Z, m\in [-10M_t, \epsilon M_t]\cap \Z$,  s.t.,  $n\leq \tilde{n}$,  $|  \zeta_{\bot}|/ |\zeta|\leq 2^{\tilde{n}-10}$, we have 
\be\label{2024oct30eqn21}
\begin{split}
&\sum_{ (m,k,j,l)\in \mathcal{S}_1(t) }  \big\|     {}_{}^zT_{k,j;n,l}^{\mu;m,2}(B)(t,x,\zeta) \big\|_{L^\infty_x} \\
&\lesssim  2^{ \tilde{n} + 11\epsilon M_t }  \big( 2^{ 2 \tilde{\alpha}_t  M_t }   + 2^{7M_t/6+\tilde{\alpha}_t M_t/4} \big),\\ 
&\sum_{ (m,k,j,l)\in \mathcal{S}_2(t) }\big(\big\|     {}_{}^zT_{k,j;n,l}^{\mu;m,2}(B)(t,x,\zeta) \big\|_{L^\infty_x}\big)^{1/2} \big(\big\|     {}_{}^zT_{k,j;n,l}^{\mu;m,2}(B)(t,x,\zeta) \big\|_{L^2_x}\big)^{1/2} \\
& \lesssim   2^{ \tilde{n} + 10\epsilon M_t }  \big( 2^{  \tilde{\alpha}_t  M_t }   + 2^{7M_t/12+\tilde{\alpha}_t M_t/8} \big). 
\end{split}
\ee
\end{proposition}
\begin{proof}
See \cite{PartII}[Lemma 4.14].
\end{proof}

The essential notations employed in this section are systematically detailed in Table \ref{tablesection2}.
\begin{table}[H]
\centering
\resizebox{\columnwidth}{!}{%
\begin{tabular}{ |c|c|c|c| } 
 \hline
 Notation & Definition & First appearance  & Remarks \\ 
 \hline
$T_{k,j;n}^{\mu}(\mathfrak{m}, U)(t,x, \zeta)$ & Definition \ref{angularlocalizationelect}  & \eqref{sep5eqn66} & $U\in\{E, B\}$; Angular localization \\
 & & & of the electromagnetic field\\
 \hline
$\varphi_{j,n}^i(v, \zeta)$ & \eqref{sep4eqn6}  & \eqref{sep5eqn10} & Partition of unity based on the size of $v$\\
 & & & and the angle between $v$ and $\zeta$\\
 \hline
 $ \mathfrak{E}^{\mu, i}_{k,j;n}( \mathfrak{m})(t, x, \zeta)$  & \eqref{2022feb24eqn81} & \eqref{oct7eqn1} &  The elliptic parts of the localized acceleration force\\
   \hline
    $\mathfrak{H}_{k,j;n}^{\mu,i}( \mathfrak{m})(t, x, \zeta )$  & \eqref{2022feb22eqn81} & \eqref{oct7eqn1} &  The hyperbolic parts of the localized acceleration force\\
   \hline
 $B^{  {m}}_{  {k};  {j},    {l}}(s,x)$ & Atomic piece of  & \eqref{2024oct30eqn2} &The precise formula of  $B^{  {m}}_{  {k};  {j},    {l}}(s,x)$\\
 &  the magnetic field $B$ & & doesn't play a role in part I. \\
  \hline
  $ {}_{}^zT_{k,n}^{\mu}(B)(t,x,\zeta)$ & \eqref{P3Bdecomdef}  & \eqref{nov6eqn47} & Angular localization of $P_3(B)$\\
    \hline
      $ {}_{}^zT_{k,n}^{\mu;i}(B)(t,x,\zeta)$ & \eqref{P3Bdecomdef}  & \eqref{nov6eqn47} & Further decomposition of  $ {}_{}^zT_{k,n}^{\mu}(B)(t,x,\zeta)$\\
      & & & based on the angle between $v$ and $\zeta$\\
    \hline
      $  {}_{}^zT_{k,j;n,l}^{\mu;m,2}(B)(t,x,\zeta)$ & precise definition doesn't  & \eqref{nov6eqn47} & Atomic decomposition for $ {}_{}^zT_{k,n}^{\mu;2}(B)(t,x,\zeta)$; an improved \\
    &  play a role in Part I& &    dichotomy holds for $ {}_{}^zT_{k,n}^{\mu;2}(B)(t,x,\zeta)$, see Proposition \ref{set1goodPartP3B}\\
    \hline
   $\gamma_1$ & Theorem \ref{maintheoremellipitic}[(ii)] & \eqref{2024oct27eqn1} & Measuring $|\zeta_{\bot}|$; local notation only \\
    $\gamma_2$ & & & Measuring $|\zeta |$; local notation only\\
       \hline
\end{tabular}%
}
\caption{Essential notations in section \ref{mainresultsPartIIdetail}.}\label{tablesection2}
\end{table}

\section{A singular weighted space-time estimate for the distribution function}\label{singularweight} 

  The primary purpose of this section is to demonstrate the validity of estimate \eqref{nov1eqn1} from Lemma \ref{singularweigh}. This estimate is essential for the subsequent analysis of the elliptic parts in Section \ref{fullimproved}. Specifically, it provides better control than the first momentum conservation law (see \eqref{conservationlaw}) in the regime where $|  x_{\bot}|$ is small and the ratio $|  v_{\bot} |/|v|$ is close to unity.

The space-time estimate \eqref{nov1eqn1} for the distribution function uses a singular weighted approach similar to that developed in our previous work on the Vlasov-Poisson system \cite{wang2}. The increased complexity of the hyperbolic structure in the current problem, however, prevents us from obtaining an estimate as sharp as the one in \cite{wang2} [Proposition 3.1]. Nevertheless, this estimate is sufficient for our subsequent analysis. Readers interested in   more details about the corresponding estimate on the  Vlasov-Poisson system are referred to \cite{wang2}.

\begin{lemma}\label{singularweigh}
  For any $t\in [0, T^{}), n\in (-\infty,2]\cap \Z$, the following estimate holds, 
\be\label{nov1eqn1}
 A_n(t):=\int_0^t \int_{\R^3} \int_{\R^3}  \frac{|  v_{\bot}|^{2+2\epsilon  }}{| x_{\bot}|^{1-2\epsilon  } \langle v \rangle} f(s,x,v) \psi_{n}\big(\frac{| v_{\bot}|}{|v|}\big) d x d v d s \lesssim 2^{6\epsilon M_t}\min\{2^{j+n}, 2^{\alpha_t M_t}\}.
\ee
\end{lemma}
\begin{proof}
For the sake of readers, we record the construction of weighted function used in  \cite{wang2}  in details as follows, 
\be\label{cutoff}
\phi(x):=\left\{\begin{array}{cc} 
2 &x\in [2^{},\infty)\\ 
 2+  (x-2)^3& x\in [1, 2]\\
x^3 & x\in [0,1)\\
0 & x\in (-\infty, 0]\\ 
\end{array}\right. ,\quad  \phi_l(x):= \phi(2^{-l}x).
\ee
From the above explicit formula of cutoff function, we have
\be\label{jan13eqn27}
\begin{split}
& \phi'(x)\geq 0, \quad \forall x\in (0, \infty),\quad \big| \frac{x\phi'(x)}{\phi(x)}\big|\lesssim  1, \quad \phi_l'(x):=2^{-l}\phi'(2^{-l}x),\quad  \big| \frac{x\phi_l'(x)}{\phi_l(x)}\big|\lesssim  1.\\
 \end{split}
\ee

For $\mu\in \{+,-\}$, we choose a weight function as follows,
 \be\label{may13eqn3}
\omega_{\mu}(x,v):= \big(\mu   |  v_{\bot} | | x_{\bot} |   x_{\bot} \cdot   v_{\bot} \phi_{-1 0M_t}(\mu \frac{ x_{\bot}\cdot v_{\bot}}{|  x_{\bot}||  v_{\bot} |}) +(  x_{\bot} \times   v_{\bot} )^2\big)^{\epsilon  }\phi\big(\mu\frac{ x_{\bot} \cdot v_{\bot}}{| x_{\bot} ||  v_{\bot}|} + \frac{1}{2}\big). 
\ee
As a result of direct computation, we have
\be\label{may13eqn1} 
 \begin{split}
 & {v}_{\bot} \cdot \nabla_{ {x}_{\bot}} \omega_{\mu}(x,v)\\ &=\epsilon   \big(\mu   | {v}_{\bot} | | {x}_{\bot} |  {x}_{\bot} \cdot  {v}_{\bot} \phi_{-10M_t}(\mu \frac{ {x}_{\bot}\cdot  {v}_{\bot}}{|  {x}_{\bot}|| {v}_{\bot} |}) +( {x}_{\bot} \times  {v}_{\bot} )^2\big)^{\epsilon -1}\phi\big(\mu \frac{  {x}_{\bot} \cdot {v}_{\bot}}{| {x}_{\bot} || {v}_{\bot}|} + \frac{1}{2}\big)\\
 &\quad \times \big[  \mu   |  {x}_{\bot} |  |  {v}_{\bot}|^3\big( 1+ \frac{( {x}_{\bot} \cdot   {v}_{\bot})^2}{|  {x}_{\bot} |^2|  {v}_{\bot}|^2}\big)   \phi_{-10M_t}(\mu \frac{ {x}_{\bot}\cdot   {v}_{\bot}}{|  {x}_{\bot}||{v}_{\bot} |}) \\
 &\quad +    |{v}_{\bot} | |{x}_{\bot} | {x}_{\bot} \cdot  {v}_{\bot} \phi_{-10M_t}'(\mu \frac{ {x}_{\bot}\cdot  {v}_{\bot}}{|  {x}_{\bot}||{v}_{\bot} |})\frac{ ( {x}_{\bot} \times {v}_{\bot})^2}{|{x}_{\bot} |^3 |{v}_{\bot}| }\big] \\
  &\quad+  \big(\mu   |{v}_{\bot} | |{x}_{\bot} |  {x}_{\bot} \cdot {v}_{\bot} \phi_{-10 M_t}(\mu \frac{ {x}_{\bot}\cdot {v}_{\bot}}{|{x}_{\bot}||{v}_{\bot} |}) +  ({x}_{\bot} \times {v}_{\bot} )^2\big)^{\epsilon  }\phi'\big(\mu\frac{{x}_{\bot} \cdot {v}_{\bot}}{|{x}_{\bot} ||{v}_{\bot}|} + \frac{1}{2}\big) \frac{\mu({x}_{\bot}\times {v}_{\bot})^2}{|{x}_{\bot} |^3 |{v}_{\bot}|}. 
  \end{split}
 \ee

 From the above equality, we conclude
  \be\label{may13eqn2}
 \begin{split}
  \mu {v}_{\bot} \cdot \nabla_{ {x}_{\bot}} \omega_{\mu}(x,v)&\gtrsim \displaystyle{\frac{ |  x_{\bot} |  | v_{\bot}|^3  \phi_{-10M_t}(\mu \frac{  x_{\bot}\cdot   v_{\bot}}{|  x_{\bot}||  v_{\bot} |}) }{\big(\mu   |  v_{\bot} | |  x_{\bot} |  x_{\bot} \cdot   v_{\bot} \phi_{-10M_t}(\mu \frac{ x_{\bot}\cdot  v_{\bot}}{| x_{\bot}||  v_{\bot} |}) +(  x_{\bot} \times   v_{\bot} )^2\big)^{ 1-\epsilon   }} }\\
  & \gtrsim \frac{|  v_{\bot}|^{1+2\epsilon  }}{|  x_{\bot}|^{1-2\epsilon  }}\phi_{-10M_t}(\mu \frac{ x_{\bot}\cdot   v_{\bot}}{|  x_{\bot}||  v_{\bot} |}) \geq 0. \\
  \end{split}
 \ee
Define
\be\label{may23eqn3}
I_n^{\mu}(t):= \int_{\R^3} \int_{\R^3}| v_{\bot}|  \omega_{\mu}(x,v) f(t,x,v)  \psi_{n}\big(\frac{|  v_{\bot}|}{|v|}\big)  d xd v.
\ee
As a result of direct computation, we have
\[
 \begin{split}
\frac{d }{dt}  I_n^{\mu}(t)&= \int_{\R^3} \int_{\R^3} \big(E(t,x)+ \hat{v}\times B(t,x)\big)  \cdot\nabla_v\big(|  v_{\bot}|  \omega_{\mu}( x,   v)  \psi_{n}\big(\frac{|  v_{\bot}|}{|v|}\big)  \big) f(t,x,v)  \\
 &\quad + |  v_{\bot}|  \hat{v}\cdot\nabla_x\big( \omega_{\mu}(x,v)\big)   \psi_{n}\big(\frac{|  v_{\bot}|}{|v|}\big)  f(t,x,v)   d xd v.
 \end{split}
\]
From the above equality and the estimate  \eqref{may13eqn2}, we have
\be\label{may13eqn71}
\begin{split}
& \big| \int_0^t \int_{\R^3}\int_{\R^3}   \frac{|  v_{\bot}|^{2+2\epsilon  }}{|  x_{\bot}|^{1-2\epsilon }\langle v \rangle } \phi_{-10M_t}( \mu \frac{  x_{\bot}\cdot   v_{\bot}}{|  x_{\bot}||  v_{\bot} |} )   \psi_{n}\big(\frac{|  v_{\bot}|}{|v|}\big) f(s,x,v)   d x d v d s \big|\\
 & \lesssim | I_n^{\mu}(t) | + | I_n^{\mu}(0) | + \big| \int_{0}^t\int_{\R^3} \int_{\R^3}  \big(E(s,x) \\
 &\quad + \hat{v}\times B(s,x)\big)  \cdot\nabla_v\big(|  v_{\bot}|  \omega_{\mu}(   x,    v)  \psi_{n}\big(\frac{| v_{\bot}|}{|v|}\big)  \big) f(s,x,v)  d xd v d s \big|.
 \end{split}
\ee
Note that, from the estimate \eqref{maintheoremroughest}, and  the conversation law  \eqref{conservationlaw}, we have
\be\label{may13eqn72}
  |  I_n^{\mu}(t)|\lesssim 1+  \int_{|x|\leq 2^{ M_t}} \int_{|v|\leq 2^{3.1M_t}} |  x_{\bot}|^{2\epsilon  }  |  v_{\bot}|^{1+2\epsilon }   f(t,x,v) d xd v\lesssim 2^{\epsilon M_t}.
\ee
Moreover, 
\be\label{may13eqn73}
\begin{split}
 &\int_{0}^t \int_{\R^3} \int_{\R^3} \frac{| v_{\bot}|^{2+2\epsilon  }}{|  x_{\bot}|^{1-2\epsilon  }\langle v \rangle }\big(1- \sum_{\mu\in\{+,-\}} \phi_{-10M_t}(  \mu \frac{  x_{\bot}\cdot  v_{\bot}}{| x_{\bot}||  v_{\bot}  |}  )  \big) f(s,x,v) d  x d v d s \\
&\lesssim 1+  \int_{0}^t \int_{|x|\leq 2^{ M_t}} \int_{|v|\leq 2^{3.1M_t}}  \frac{|  v_{\bot}|^{2+2\epsilon  }}{|  x_{\bot}|^{1-2\epsilon }\langle v \rangle } f(s,x,v)   \big(1- \sum_{\mu\in\{+,-\}} \phi_{-10M_t}(  \mu \frac{  x_{\bot}\cdot   v_{\bot}}{| x_{\bot}||  v_{\bot} |}  )  \big) d xd v d s\\
&\lesssim 1 + 2^{-10M_t + 9.5 M_t + 10\epsilon M_t}\lesssim 1. 
\end{split}
\ee

Lastly, we estimate the contribution from the nonlinear effect. Recall  \eqref{may13eqn3}. As a result of direct computation,   from the estimate \eqref{jan13eqn27},  for any $(x,v)\in supp(\phi\big(\mu\frac{  x_{\bot} \cdot  v_{\bot}}{|  x_{\bot} ||  v_{\bot}|} + \frac{1}{2} \big))\cup supp(\phi'\big(\mu\frac{ x_{\bot} \cdot  v_{\bot}}{|  x || v|} + \frac{1}{2}\big))$, we have
\be
\mu   |  v_{\bot} | |  x_{\bot} |   x_{\bot} \cdot   v_{\bot} \phi_{-10M_t}(\mu \frac{  x_{\bot}\cdot  v_{\bot}}{|  x_{\bot}|| v_{\bot} |}) +(  x_{\bot} \times v_{\bot} )^2 \sim | x_{\bot} |^{2} | v_{\bot}|^2.
\ee
Moreover, recall \eqref{may13eqn3}, from  the result of  computation and the estimate of the cutoff function in \eqref{jan13eqn27}, we have 
\be\label{may13eqn6}
\begin{split}
 \big|\nabla_v \omega_{\mu}(x,v)\big|& \lesssim \phi\big(\mu \frac{ x_{\bot} \cdot  v_{\bot}}{| x_{\bot} ||  v_{\bot}|} + \frac{1}{2} \big) (|  x_{\bot}||   v_{\bot}|)^{2\epsilon  -2}
  \big[ |  x_{\bot} |^2 |  v_{\bot}|+     | x_{\bot} |  x_{\bot} \cdot  v_{\bot} \phi_{-10M_t}'(\mu \frac{  x_{\bot}\cdot  v_{\bot}}{| x||  v_{\bot} |})   \big]\\ 
  &\quad + \phi'\big(\mu(\frac{  x_{\bot} \cdot  v_{\bot}}{|  x_{\bot} || v_{\bot}|} + \frac{1}{2})\big) \frac{ (| x_{\bot}||  v_{\bot}|)^{2\epsilon  }     }{|  v_{\bot} |}  \lesssim   \frac{|  x_{\bot}|^{2\epsilon  } }{  |  v_{\bot}|^{1-2\epsilon  }}. 
\end{split}
\ee
Recall  \eqref{backward}.  Note that,   if $ |x|\gtrsim 2^{2\epsilon M_t}$,
\[
\begin{split}
\big||x|-|X(  x,v,0,t)|\big| \leq    2^{\epsilon M_t}, \quad \Longrightarrow |x|\sim | X(  x,v,0,t)|.  
\end{split}
\]
Moreover, from \eqref{may2eqn1}, we have
\be
\begin{split}
\sum_{j\geq (1+\epsilon) M_t }\big| \int_{\R^3}\int_{\R^3} f(t,x,v) \psi_{j}(v) d x d v \big| & \lesssim \sum_{j\geq (1+\epsilon)  M_t} 2^{-N_0 j/10 } \mathfrak{M}_{}(t )  \\
& \lesssim 2^{-N_0(1+\epsilon)M_t/10 + (N_0/10-1)M_t}\\
& \lesssim 2^{-100M_t}.\\
\end{split}
\ee
  After localizing the sizes of   $|v|$,  from  the above estimates,  the decay rate of initial data in  \eqref{assumptiononinitialdata},   and the rough estimate of the electromagnetic field  \eqref{maintheoremroughest} in Theorem \ref{maintheorem1part1}, we have 
\be\label{nov1eqn36}
\begin{split}
  & \big| \int_{0}^t\int_{\R^3} \int_{\R^3}  \big(E(s,x)+ \hat{v}\times B(s,x)\big)  \cdot\nabla_v\big(|  v_{\bot}|  \omega_{\mu}(   x_{\bot},   v_{\bot})  \psi_{n}\big(\frac{| v_{\bot}|}{|v|}\big)  \big) f(s,x,v)  d xd v d s \big| \\
 & \lesssim 1+ \sum_{ j \in [0, (1+2\epsilon)M_t]\cap \Z_+} 2^{3\epsilon M_t}  H_j^n(t), \\
 \end{split}
\ee
where
\be\label{nov1eqn11}
\begin{split}
 H_j^n(t)&: =  \int_{0}^t\int_{|x|\leq 2^{M_t/2}}  \int_{|v|\leq 2^{  (1+2\epsilon) M_t}} \big(|E(s,x)| + \big|B(s,x)\big| \big)   f(s, x,v)\varphi_j(v) \psi_{n}\big(\frac{| v_{\bot}|}{|v|}\big)  d x d v d  s. 
 \end{split}
\ee 

After using the Cauchy-Schwarz inequality and the conservation law \eqref{conservationlaw},  we have
\be
\begin{split}
| H_j^n(t)| & \lesssim \int_0^t \big( \int_{\R^3} \big(\int_{\R^3}  f(s, x,v)\varphi_j(v) \psi_{n}\big(\frac{|  v_{\bot}|}{|v|}\big) d v \big)^2 d x\big)^{1/2} d s \\ 
& \lesssim \int_0^t \big( \int_{\R^3}  \int_{\R^3}  f(s, x,v)\varphi_j(v) \psi_{n}\big(\frac{|  v_{\bot}|}{|v|}\big) d v   d x\big)^{1/2} \|\int_{\R^3}  f(s, x,v)\varphi_j(v) \psi_{n}\big(\frac{|  v_{\bot}|}{|v|}\big) d v\|_{L^\infty_x}^{1/2} d s \\ 
&\lesssim 2^{ 2\epsilon M_t}\big(\min\{2^{3j+2n},2^{j+2\alpha_t M_t}\}\big)^{1/2} 2^{-j/2}\\
&\lesssim 2^{2\epsilon M_t}\min\{2^{j+n}, 2^{\alpha_t M_t}\} .
\end{split}
\ee
From the above estimate  and the obtained estimates  \eqref{may13eqn71},   \eqref{may13eqn72}, \eqref{may13eqn73}, and  \eqref{nov1eqn36},   we have 
\be
A_n(t)\lesssim  2^{6\epsilon M_t}\min\{2^{j+n}, 2^{\alpha_t M_t}\}.
\ee
Hence finishing the proof of our desired estimate  \eqref{nov1eqn1}. 
\end{proof}
 
The essential notations employed in this section are systematically detailed in Table \ref{tablesection3}.
\begin{table}[H]
\centering
\begin{tabular}{ |c|c|c|c| } 
 \hline
 Notation & Definition   & Remarks \\ 
 \hline
 $ A_n(t)$&\eqref{nov1eqn1}  &  Space-time singular weighted estimate for $f$;  \\
 &   &only used in section \ref{ellipticpartPartII} \\
\hline 
$ \phi(x)$ &  \eqref{cutoff}    &    Special cutoff function;local notation only.\\
 \hline
$\omega_{\mu}(x,v)$ & \eqref{may13eqn3}  & Special weight  function;local notation only.\\
 \hline
$I_n^{\mu}(t)$  & \eqref{may23eqn3}   & Weighted $L^1_{x,v}$ of $f$; controlled by  \eqref{conservationlaw} and characteristics\\
 \hline
 $ H_j^n(t)$ & \eqref{nov1eqn11} & Localized nonlinear source term; providing control for \eqref{may13eqn71}\\
\hline
\end{tabular} 
\caption{Essential notations in section \ref{singularweight}.}\label{tablesection3}
\end{table}

\section{Bootstrap argument: Proof of Proposition \ref{finalproposition}}\label{bootstraparg}

Let  $t\in [0,T)$   be fixed.   Recall  \eqref{may9en21}  and  \eqref{2021dec18eqn1}.  We define the first time that characteristics almost reaches the threshold as follows, 
\be\label{2021dec18eqn21}
t^{\ast}:=\inf\{s: s\in [0, t], \beta_t(s)\geq  (1-3\epsilon)  \}.
\ee
 Due to the continuity of characteristics,  and the rough estimate of the electromagnetic field \eqref{maintheoremroughest} in Theorem  \ref{maintheorem1part1}, we know that $t^{\ast}>0. $ Moreover,  $\forall s \in [0, t^{\ast}],$ we have $|V(s)|\leq 2^{(1-3\epsilon)M_t}. $  We make the following bootstrap assumption, 
\be\label{2021dec18eqn22}
\tau:= \sup\{ \kappa: \kappa \in [0,t], \forall s\in [t^{\ast}, \kappa],    2^{ \beta_t(s) M_t}\in  2^{(1-3 \epsilon)M_t }    [99/100, 101/100] \} .
\ee
We aim to improve the above bootstrap assumption and  prove that $\tau=t$.

We employ a two-step bootstrap argument. Initially, we focus on estimating the projection of the velocity characteristics, $|  V_{\bot}(x,v,s,0)|$. This estimate, obtained via a bootstrap argument, provides the first estimate in \eqref{finalestimate}. Subsequently, leveraging the upper bound on $| V_{\bot}(x,v,s,0)|$ established for $s \in [0, \tau]$, we proceed to estimate the full characteristics. To simplify notation, \textbf{we omit the explicit $(x,v)$ dependence of the characteristics}, assuming $(x,v) \in R_t(0)$ throughout.

\medskip

\noindent \textbf{Step 1.}  \quad Estimating the projection of the velocity characteristics within the time interval $[0, \tau]$.

\medskip

Let $\gamma_1 = \alpha^{\star} - 2\epsilon$. Define $\tau_{\ast}$ and $\tau^{\star}$ to be  the first times $\alpha_t(\kappa)$ reaches $\gamma_1$ and $\gamma_1 + 1/M_t$    respectively.
\be\label{may10eqn109}
\begin{split}
\tau_{\ast}&:= \sup\{ s: s\in [0,\tau], \forall\kappa\in [0, s],    {\alpha_t (\kappa) }\leq  \gamma_1   \},\\
 \tau^{\star}&:= \sup\{  s: s\in [0,\tau], \forall\kappa\in [0, s],    {\alpha_t (\kappa) }\leq    \gamma_1   +1/M_t \}. \\
\end{split}
\ee

Our goal is to show $\tau^{\star} = \tau$. If $\tau_{\ast} = \tau$ or $\tau^{\star} = \tau$, the result is immediate. Therefore, we consider the case $0 < \tau_{\ast} < \tau^{\star} < \tau$, where continuity of characteristics implies $|  V_{\bot}(x,v,\tau_{\ast}, 0)| = 2^{\gamma_1 M_t}$. Moreover, we characterize the size of the full characteristics at the time $\tau_{\ast}$ as follows
\be\label{gamma2definition}
\gamma_2:= \inf\{k: k\in \R_+, |V(x,v,\tau_{\ast}, 0)||\leq 2^{k M_t}\}. 
\ee
We make the    bootstrap assumption for the velocity characteristics  $  V(s) $ as follows, 
\be\label{bootstrap1}
\begin{split}
  \tau^{\ast}:=\sup\{s: s\in[\tau_{\ast},\tau^{\star} ], \forall \kappa\in [\tau_{\ast}, s],  &  | V_{\bot}(s)| \in   2^{  \gamma_1  M_t}  [99/100, 101/100],\\
  & |V(s)|  \in   2^{   \gamma_2 M_t}  [99/100, 101/100] \}.\\
\end{split}
\ee
Now, our goal is to show that $\tau^{\ast}= \tau^{\star}$ independent of $(x,v)\in R_t(0)$. As a consequence, we have $2^{\alpha_t (\tau^{\star}) M_t}<  2^{   \gamma_1 M_t + 1} $. This further yields contradiction to the definition of $\tau^{\star}$. Therefore, we conclude that  $\tau^{\star}$ equals to $\tau$.

Note that $[\tau_{\ast}, \tau^{\ast}]\subset[\tau_{\ast}, \tau^{\star}]\subset [0, \tau]$. Recall  \eqref{2021dec18eqn1},  \eqref{2021dec18eqn21},    \eqref{2021dec18eqn22}, and  \eqref{bootstrap1},  for any $s\in  [\tau_{\ast}, \tau^{\ast}]$, we have
\[
 \alpha_s M_s \leq  \alpha_t(s) M_t \leq \gamma_1  M_t +1, \quad \gamma_2 M_t\leq (1-3\epsilon) M_t+1. 
\]
Moreover, note that, from the equation satisfied by the space characteristics in  \eqref{backward}, the following rough estimate holds, 
\[
\forall s\in [0,t], \quad |X(s)|\leq |X(0)|+\int_{0}^t  |\hat{V}(s)| d s\leq 2^{M_t/2 +1 }.
\]

Based on the possible size of $| X_{\bot}(s)|$, we decompose the interval $[0,M_t/2+1]$ into a finite union of intervals, which overlap with at most with four other intervals. More precisely, we have 
\begin{multline}\label{overlapintervals}
[0,2^{M_t/2+1}]=\cup_{i=0}^K I_i, \quad I_0=[0,  2^{{\kappa_0 }}], \quad I_k=[2^{a_k},2^{a_k+10}], \quad a_{0}=\kappa_0-6, \kappa_0=-100M_t,  \\ 
 \forall k\in\{1,\cdots, K-2\}, a_{k+1}= a_{k}+6,\quad a_{K}= M_t/20+1-10,\quad   K=\lceil (M_t/2+1-\kappa_0)/6\rceil.
\end{multline}
Within the time interval $[\tau_{\ast}, \tau^{\ast}]$, under the bootstrap assumption  \eqref{bootstrap1},  we   show the following two facts about the characteristics, 
\begin{enumerate}
\item[$\bullet$]  \textbf{Fact (i)}:\qquad Given any $t_1, t_2\in  [\tau_{\ast}, \tau^{\ast}]$ and any $i\in\{0,\cdots, K-1\}$, if $\forall s\in [t_1, t_2], |  X_{\bot}(s)|\in I_i\cup I_{i+1}$, then the bounds of speed characteristics $|  V_{\bot}(s)|, |V(s)|$ can be improved for any $s\in [t_1,t_2].$ More, precisely, the increment of  $|  V_{\bot}(s)|$     is at most  $2^{ \gamma_1 M_t- \epsilon M_t }$, and  the increment of   $|V(s)|$   is at most  $2^{ \gamma_2 M_t- \epsilon M_t }$. 
\item[$\bullet$] \textbf{Fact (ii)}:\qquad Let $I_{-1}:=\emptyset$. For any  $i\in\{0,\cdots, K\}$ and any $s\in  [\tau_{\ast}, \tau^{\ast}]$, if $| X_{\bot}(s)|\in I_i/I_{i-1}$ with $d | X_{\bot}(t)|^2/dt\big|_{t=s}\geq 0 $, then for any later time $l \in[s,\tau^{\ast}]$,   $|  X_{\bot}(l)|\in \cup_{i\leq j\leq K} I_j$.
\end{enumerate}

Based on the possible size of $i$, we  prove  \textbf{Fact (i)}  in  two substeps  as follows. The   \textbf{Fact (ii)} will be obtained as a byproduct. 

\medskip

  \textbf{Step 1A.} \quad  Given any $t_1, t_2\in  [\tau_{\ast}, \tau^{\ast}]$, if $\forall s\in [t_1, t_2], |  X_{\bot}(s)|\in I_0\cup I_1$.

\medskip

As a result of direct computation, we have 
  \be\label{oct22eqn1}
   \frac{d^2}{ds^2}  |  X_{\bot}(s)|^2 =2\big( \frac{| V_{\bot}(s)|^2 }{1+|V(s)|^2}  +C_0(X(s), V(s))\cdot K(s, X(s), V(s) )\big),
\ee
where
\be\label{2024nov4eqn1}
C_0(X(s), V(s))=\frac{\big(  X_{\bot}(s), 0\big)}{\sqrt{1+|V(s)|^2}}- \frac{ X_{\bot} (s)\cdot  V_{\bot}(s) }{  \big({1+|V(s)|^2}\big)^{3/2}} V(s).
\ee

From the rough estimate of the electromagnetic  field in   \eqref{maintheoremroughest} in Theorem  \ref{maintheorem1part1}, the following estimate holds for any $s\in [t_1, t_2],$
\be\label{nov28eqn78}
\begin{split}
\frac{  X(s )\cdot  V_{\bot}(s  )}{\sqrt{1+|V(s )|^2}}-  \frac{ X_{\bot}(t_1  )\cdot  V(t_1  )}{\sqrt{1+|V(t_1 )|^2}}&\geq \frac{4}{5}2^{2  \gamma_1 M_t-2\gamma_2 M_t} (s-t_1) \\
&\quad -3 \int_{t_1}^s 2^{\kappa_0  -\gamma_2  M_t} \big(\|E (\kappa, x)\|_{L^\infty_x}+\|B (\kappa, x)\|_{L^\infty_x}\big) d \kappa,\\
\Longrightarrow \quad  \frac{ X_{\bot}(s )\cdot   V(s  )}{\sqrt{1+|V(s )|^2}} & \geq \frac{4}{5}2^{ 2  \gamma_1 M_t-2\gamma_2 M_t} (s-t_1) \\
&\quad  -2^{\kappa_0  -\gamma_2 M_t+ (1+2\tilde{\alpha}_t +6\epsilon) M_t-100}  (s-t_1) -\frac{11}{10} 2^{\kappa_0  +     \gamma_1 M_t-\gamma_2 M_t}\\
&\geq \frac{3}{4}2^{ 2  \gamma_1 M_t-2\gamma_2 M_t } (s-t_1)-\frac{11}{10} 2^{\kappa_0  +   \gamma_1 M_t-\gamma_2 M_t},\\
 \Longrightarrow\quad  |X(s)|^2 -|X(t_1)|^2 & \geq \frac{3}{4}2^{2 \gamma_1 M_t-2\gamma_2 M_t} (s-t_1)^2\\
&\quad -\frac{11}{5} 2^{\kappa_0  +   \gamma_1 M_t-\gamma_2 M_t } (s-t_1). 
\end{split}
\ee
Since $|  X_{\bot}(t_2)|\leq 2^{\kappa_0+10}$, from the above estimate, we have 
\be\label{2024nov19eqn11}
2^{2\gamma_1 M_t-2\gamma_2 M_t} (t_2-t_1)^2\leq 2^{2\kappa_0 +20}, \quad \Longrightarrow\quad  |t_2-t_1|\leq 2^{\kappa_0+ \gamma_2 M_t-\gamma_1 M_t+10}.
\ee
The above estimates imply that if the elapsed time exceeds $2^{\kappa_0 + \gamma_2 M_t - \gamma_1 M_t + 10}$,  $| X_{\bot} (s)|$ will leave $I_0 \cup I_1$ and enter $\bigcup_{i=2}^K I_i \cup (I_2 \setminus I_1)$.

  As a result of direct computation, from  \eqref{backward},   we have
\be\label{nov17eqn21}
\begin{split}
  \big|| V_{\bot}(t_2)|  -  |  V_{\bot}(t_1)| \big|  & \lesssim \big|\int_{t_1}^{t_2}  \big(  V_{\bot} (s)/|V(s)|, 0\big) \cdot K(s, X(s), V(s) )  ds\big| , \\
     \big| |  V(t_2)| - |V(t_1)|  \big| & \lesssim  \big|  \int_{t_1}^{t_2} \tilde{V}(s)  \cdot K(s, X(s), V(s) ) d s  \big|. \\
\end{split}
\ee 
From the above estimates, the obtained estimate \eqref{2024nov19eqn11}, and the rough estimate of the electromagnetic  field   \eqref{maintheoremroughest} in Theorem  \ref{maintheorem1part1}, for any $s\in [t_1, t_2],$ we have
 \[
 \begin{split}
   \big| |  V_{\bot}(s)|  -| V_{\bot}(t_1)| \big|  +   \big| | V(s)|  -|  V(t_1)|  \big|   &  \lesssim   \int_{t_1}^{t_2} \big(\|E(s, \cdot)\|_{L^\infty} + \| B(s, \cdot)\|_{L^\infty} \big) ds\\ 
   & \lesssim  2^{\kappa_0+2\gamma_2  M_t      }  2^{  (1+2\tilde{\alpha}_t +6\epsilon)M_t} \lesssim  1. \\
  \end{split}
 \]

 \medskip

  \textbf{Step 1B.} \quad   Given any $t_1, t_2\in  [\tau_{\ast}, \tau^{\ast}]$, if  $\forall s\in [t_1, t_2], |  X_{\bot}(s)|\in I_i\cup I_{i+1} $ . 
 \medskip

  Recall again   \eqref{oct22eqn1}.   For any $s \in [t_1, t_2]$,  from the estimate \eqref{nov17eqn31} in  Proposition \ref{bootstraplemma1},  we have
\be\label{oct30eqn41}
\begin{split}
 \frac{  X_{\bot}(  s  )\cdot  V_{\bot}( s )}{\sqrt{1+|V(s )|^2}}- \frac{  X_{\bot}( t_1  )\cdot V_{\bot}( t_1 )}{\sqrt{1+|V(t_1 )|^2}}  &   \geq \frac{9}{10} 2^{ 2\gamma_1 M_t-2\gamma_2 M_t} (s -t_1) \\
 &\quad -   C \big[ \sum_{b\in \mathcal{T}+\mathcal{T}}   2^{(1-b)a_i  } 2^{ b( \gamma_1  -\gamma_2   )M_t+ (\gamma_1-\gamma_2- \epsilon) M_t  }\\
 &\quad\times (s-t_1) + 2^{a_i+3\alpha^{\star} M_t/4 -\gamma_2 M_t+50\epsilon M_t}\big],\\
\Longrightarrow \quad    \frac{ X_{\bot}(s   )\cdot  V_{\bot}(s   )}{\sqrt{1+|V(s  )|^2}} & \geq \frac{4}{5}  2^{  2\gamma_1 M_t-2\gamma_2 M_t}   (s -t_1) -  \frac{11}{10}|  X_{\bot}(t_1)|2^{  \gamma_1 M_t-\gamma_2 M_t}.
\end{split}
\ee
From the above estimate, we have
\be\label{oct30eqn21}
\begin{split}
 |   X_{\bot}(s )|^2  -  |  X_{\bot}(t_1)|^2 & \geq \int_{0}^{s-t_1} \frac{8}{5}  2^{   2\gamma_1 M_t-2\gamma_2 M_t}\kappa  - \frac{22}{10 } |  X_{\bot}(t_1)|2^{   \gamma_1 M_t-\gamma_2 M_t} d \kappa \\
& \geq  \frac{4}{5} 2^{   2\gamma_1 M_t-2\gamma_2 M_t} (s -t_1)^2  -  \frac{22}{10 }|   X_{\bot}(t_1)|2^{    \gamma_1 M_t-\gamma_2 M_t}  (s -t_1), \\
\Longrightarrow\quad  |   X_{\bot}( s  )| & \geq 2^{ \gamma_1 M_t-\gamma_2 M_t-5}  (s -t_1)  -2|  X_{\bot}(  t_1)|, \\
 \Longrightarrow \quad |t_2-t_1| &\lesssim 2^{a_i+\gamma_2 M_t-\gamma_1 M_t}. \\
\end{split}
\ee
Moreover, since $t_1, t_2\in [0, t]$ and $t\leq 2^{\epsilon M_t/100}$, we have
\[
 |t_2-t_1| \lesssim  \min\{  2^{a_i+\gamma_2 M_t-\gamma_1 M_t} , 2^{\epsilon M_t/50}\}. 
\]
Recall  \eqref{nov17eqn21}. From the above estimate of the duration of time and  the estimate  \eqref{nov17eqn31} in  Proposition \ref{bootstraplemma1},  for any $s \in [t_1, t_2]$, we have 
\[
\begin{split}
\big| |  V_{\bot}(s )|   -  |  V_{\bot}(t_1)| \big| & \lesssim     2^{  (3 {\alpha}^{\star}+20\epsilon)M_t/4} + \big(  \sum_{b\in \mathcal{T}+\mathcal{T}}   2^{-ba_i } 2^{ b (\gamma_1 -\gamma_2) M_t   + (\gamma_1 -2 \epsilon) M_t}   \big)(s-t_1)  \\
& \leq  2^{   \gamma_1 M_t-3\epsilon M_t/2 },\\
\big| |  V(s )|   -  |   V(t_1)|\big| &\lesssim  2^{(\gamma_2 -\gamma_1)M_t}\big[     2^{  (3 {\alpha}^{\star}+20\epsilon)M_t/4} \\
&\quad  +  \big(  \sum_{b\in \mathcal{T}+\mathcal{T}}   2^{-ba_i } 2^{ b (\gamma_1  -\gamma_2) M_t  +   (\gamma_1 -2 \epsilon) M_t }   \big)(s-t_1)  \big]\\
&\leq  2^{ \gamma_2 M_t -3\epsilon M_t/2  } .
\end{split}
\]

 Moreover, if $  X_{\bot} ( t_1  )\cdot   V_{\bot}( t_1 )\geq 0$, i.e., $d |X(t )|^2/dt\big|_{t=t_1}\geq 0$, and $| X_{\bot}( t_1  )|\in I_i/I_{i-1}=(2^{a_i+5}, 2^{a_i+10}]$, then we can improve the obtained estimate   \eqref{oct30eqn41}  as follows, 
 \be\label{oct30eqn51}
 \begin{split}
 \textup{If $  X_{\bot}( t_1  )\cdot  V_{\bot}( t_1 )\geq 0$}, \quad  \frac{  X_{\bot}(s   )\cdot  V_{\bot}(s   )}{\sqrt{1+|V(s  )|^2}} & \geq \frac{4}{5}  2^{ 2(\gamma_1 -\gamma_2) M_t}   (s -t_1) -   2^{-10+a_i } 2^{ (\gamma_1 -\gamma_2) M_t},\\
 \Longrightarrow\quad \forall s\in [t_1, t_2], \quad   |   X_{\bot}( s  )| & \geq 2^{ (\gamma_1 -\gamma_2) M_t  -1}  (s -t_1)  +2^{-1}|  X_{\bot}( t_1)|.
 \end{split}
 \ee

 Recall the definition of overlapping intervals in  \eqref{overlapintervals}.   From the above estimate, we  know that  $|   X_{\bot}( s  )|\in [2^{a_k+3}, 2^{a_k+8}]\subsetneq (2^{a_k}, 2^{a_{k}+10}]$. 

 To sum up,  from the preceding fact and estimate \eqref{oct30eqn21}, it follows that if $| {X}(t_1)|$ lies in $I_i/ I_{i-1}$ and the derivative $d|  X_{\bot}(s)|^2/ds$  is non-negative at $ s=t_1$, then $| {X}_{\bot}(s)|$ will leave $I_i$ and enter $I_{i+1}$ (rather than $I_{i-1}$) within a time interval of at most $2^{a_i + \gamma_2 M_t - \gamma_1 M_t}$. Therefore, our desired \textbf{Fact (i)} and \textbf{Fact (ii)} hold from the above discussion.

Now, we are ready to improve the bootstrap assumption  \eqref{bootstrap1}. 
Due to Fact (i), we need only demonstrate that the number of times the trajectory $| {X}_{\bot} (s)|$, $s \in [\tau_{\ast}, \tau^{\ast}]$, visits intervals of the form $I_i \cup I_{i+1}$ ($i = 0, \dots, K-1$) is at most $CM_t$, for some absolute constant $C$. This contributes only a logarithmic loss.

Since the $d| X_{\bot}(s)|/ds$ is bounded, the trajectory of $| X_{\bot}(s)|$ can only visit intervals $J_i:=I_i/I_{i-1}, i\in\{0, \cdots, K\}$ for finite times. We order the visited intervals within $[\tau, \tau^{\star}]$ with respect to time as follows $(J_{i_1},J_{i_2},\cdots, J_{i_X})$, $\forall p\in\{1,\cdots, X\}, i_{p}\in \{0,\cdots, K\}$. Due to the possible revisit scenario, e.g., $J_1\longrightarrow J_2 \longrightarrow J_1\longrightarrow \cdots$, the size of  finite number  $X$ is not    clear.

Due to the continuity of characteristics, we have $ i_{p}-i_{p+1}\in\{1,-1\}$. Let  $ {i_{\iota_0}}$ be the first local minimum of the ordered set $(i_{1}, i_2, \cdots, i_{X})$, then $\iota_0\leq K+1, i_{\iota+1}= i_{\iota}+1$. Let $\tau_{ \iota_0}$ be the time such that the   $|  X_{\bot}(\tau_{ \iota_0})|$leaves $I_{i_{ \iota_0}}$ and enters $I_{i_{ \iota_0 }+1}/I_{ \iota_0 }$. As $| X_{\bot}(s)|$ increases at time $\tau_{ \iota_0 }$, we know that  $d |  X_{\bot}(t)|^2/dt\big|_{t=\tau_{ \iota_0}}\geq 0 $. Therefore, from the obtained \textbf{Fact (ii)}, we know that $i_{\kappa}\geq i_{ \iota_0 }$  for any $\kappa \geq  \iota_0$. Let $\iota_1:=\inf\{\kappa: i_{\kappa}\geq i_{\iota_0}+2,  \kappa\geq \iota_0\}$. We know that $J_{i_{\iota_1}}=I_{i_{\iota_0+2}}/I_{i_{\iota_0+1}}$ and $\cup_{\kappa\in [ {\iota_0},  {\iota_1}-1]\cap \Z} J_{i_{\kappa}} \subset I_{i_{\iota_0+1}}\cup  I_{ \iota_0 }$.

Similarly,  let $\tau_1$ be the time such that   $|  X_{\bot}(\tau_{ \iota_1})|$leaves $I_{i_{ \iota_1}-1}=I_{i_{\iota_0+1}}$ and enters $J_{\iota_{\iota_1}}= I_{i_{ \iota_0 }+2}/I_{ i_{ \iota_0 }+1 }$, from the obtained \textbf{Fact (ii)}, we know that $i_{\kappa}\geq i_{ \iota_0 }+1$  for any $\kappa \geq  \iota_1$. Therefore, inductively, we can define a sequence of   $\{\iota_k\}_{k=0}^{m}$ such that the following property holds, 
\be\label{nov4eqn12}
 \begin{split}
 \{\kappa&: i_{\kappa}\geq i_{\iota_m}+1\} =\emptyset,\quad \forall k\in\{0,\cdots, m-1\}, \quad  i_{\iota_{k+1}}\geq i_{\iota_k}+1,\\
   X_{\iota_k}&:=\cup_{\kappa\in [\iota_k, \iota_{k+1}-1]\cap \Z} J_{i_{\kappa}}\subset I_{i_{\iota_k}+1}\cup I_{i_{\iota_k}},\\
  X_{\iota_m}&:=\cup_{\kappa\in [\iota_m, K]\cap \Z} J_{i_{\kappa}}\subset I_{i_{\iota_m}+1}\cup I_{i_{\iota_m}}.
\end{split}
\ee
If $i_{\iota_m}=K$, we use the convention that $I_{K+1}:=\emptyset.$

Since    $\{i_{\iota_k}\}_{k=0}^{m}$  is an increasing sequence with upper bound $K$, we know that $m\leq K+1$.  From the above discussion, we regroup the ordered set $(J_{i_1},J_{i_2},\cdots, J_{i_X})$, as follows, 
\be\label{nov4eqn11}
(J_{i_1}, J_{i_2},\cdots, J_{i_{\iota_0-1}}, X_{\iota_0},X_{\iota_1},\quad X_{\iota_m}),
\ee
in which the order is still consistent with the  order of time such that the characteristic  $| X_{\bot}(s)|, s\in [\tau_{\ast}, \tau^{\ast}]$  travels. 

Since the total number of sets in  \eqref{nov4eqn11}   is less than $2(K+1)\leq 100 M_t$, see  \eqref{overlapintervals},  which   only causes logarithmic loss.  From the relation in  \eqref{nov4eqn12} and  \textbf{Fact (i)}, our bootstrap assumption is improved.  Hence finishing the bootstrap argument, i.e., $\tau^{\ast}= \tau^{\star}$, which further implies that  $\tau^{\ast}= \tau^{\star}=\tau.$

\medskip

 \noindent  \textbf{Step 2.}  \qquad Estimating  the full velocity characteristics within the time interval $[0, \tau]$. 

\medskip

As a result of the  \textbf{Step} $\mathbf{1}$  and  \eqref{2021dec18eqn1}, we know that  $ \alpha_\tau M_\tau \leq \alpha_t(\tau) M_t \leq (\alpha^{\star}-2 \epsilon)M_t$.  Now, our goal is to show that $\tau = t  .$ Recall   \eqref{backward}. For any $(x,v)\in R_t(0)$,  $t_1, t_2\in [t^{\ast}, \tau]$, the following estimate holds from  the estimate \eqref{oct29eqn70} in Proposition \ref{bootstraplemma2}, 
\[
\begin{split}
\big||V(t_2)| - |V(t_1)|  \big|&\leq  \big|\int_{t_1}^{t_2} \tilde{V}(s)\cdot K(s, X(s), V(s)) d s\big|\\
& \lesssim  2^{   (\gamma-5\epsilon) M_t   }\leq 2^{  \gamma M_t-4\epsilon M_t}.\\
\end{split}
\]
 Therefore, our bootstrap assumption in \eqref{2021dec18eqn22} is improved.  Hence finishing the bootstrap argument, i.e., $\tau = t.$  Recall the definition of the majority set in  \eqref{may9en21}. Since $\tau = t$, after rerunning the argument in  \textbf{Step} $\mathbf{1}$,   we have 
 \[
\beta_t = \sup_{s\in[0, t]}\beta_t(s) \leq 1-2\epsilon, \quad \alpha_t = \sup_{s\in[0, t]}\alpha_t(s) \leq (2/3+\iota)-  \epsilon. 
 \]
 Hence finishing the proof of our desired estimate  \eqref{finalestimate}. 
 
The essential notations employed in this section are systematically detailed in Table \ref{tablesection4}.
\begin{table}[H]
\centering
\resizebox{\columnwidth}{!}{%
\begin{tabular}{ |c|c|c|c| } 
 \hline
 Notation & Definition & First appearance  & Remarks \\
 & &in this section  & \\ 
 \hline
 $\beta_t(s) $ & Definition \ref{tmajorityset}; \eqref{may9en21} &\eqref{2021dec18eqn21}  &   Measuring the maximum of velocity characteristics  \\
 & & &  starting from the majority set $R_t(0)$\\ 
\hline 
 $\tau_{\ast}, \tau^{\star}$& \eqref{may10eqn109} & \eqref{may10eqn109} & The start  time and the end time of\\
 & & & bootstrap argument for $|V_{\bot}(s)|$ \\
\hline
 $ \tau^{\ast}$ & \eqref{bootstrap1} & \eqref{bootstrap1}  &  Additional bootstrap argument for  $|V (s)|$ \\
 & & &  in time interval $[\tau_{\ast}, \tau^\star]$\\
\hline
$\gamma_1 $  & $\alpha^{\star}-2\epsilon$ & Above \eqref{may10eqn109} & Measuring $|V_{\bot}(\tau_{\ast})|$ \\
$\gamma_2$ & $ \eqref{gamma2definition} $ & \eqref{gamma2definition}  & Measuring $|V (\tau_{\ast})|$\\
\hline
$\kappa_0  $ & $-100M_t$ &  \eqref{overlapintervals}  & The threshold of inhomogeneous dyadic  \\
& & & decomposition for $|X_{\bot}(s)|$; local notation only\\
\hline
$K$ & $ \lceil (M_t/2+1-\kappa_0)/6\rceil$ &\eqref{overlapintervals}  & The number of dyadic interval for  \\
& & & $|X_{\bot}(s)|$; local notation only\\ 
\hline
\end{tabular}%
}
\caption{Essential notations in section \ref{bootstraparg}.}\label{tablesection4}
\end{table}

\section{ISS Part I: Proof of   Proposition \ref{bootstraplemma1}}\label{mainimprovedfull}

In this section, we rigorously prove that the desired estimate \eqref{nov17eqn31}, as presented in Proposition \ref{bootstraplemma1}, extends to a general class of coefficients. Assume that $C(\cdot, \cdot):\R^2 \times \R^3\rightarrow \R^3$ is a smooth function s.t., the following estimate holds for any $  s \in [t_1, t_2]$,
\be\label{oct23eqn51}
\begin{split}
&\big|\mathbf{P}_3\big(C(  x_{\bot}, V(s))\big) \big| +| X_{\bot} (s)|\big|\nabla_{  x_{\bot}} \mathbf{P}_3\big(C(  x_{\bot}, V(s))\big)  |_{  x_{\bot} =   X_{\bot}(s)} \big| \\
&+ \sum_{|\alpha|\leq 10} |V(s)|^{|\alpha|} \big|\nabla_{v}^{\alpha} \mathbf{P}_3\big(C( x_{\bot},  v)\big) |_{v = V(s)} \big|\lesssim 2^{(\gamma_1-\gamma_2)M_t } \mathcal{M}(C), \\ 
 &\big|C(  x_{\bot}, V(s))\big| +|  X_{\bot}(s)|\big|\nabla_{ x_{\bot}}C(  x_{\bot}, V(s))|_{ x_{\bot} =   X_{\bot}(s)} \big| \\
 &+ \sum_{|\alpha|\leq 10} |V(s)|^{|\alpha|} \big|\nabla_{v}^{\alpha}C( x_{\bot}, v)|_{v = V(s)} \big|\lesssim \mathcal{M}(C).
\end{split}
\ee
Under the above assumptions on the coefficients, we prove the following estimate:
\be\label{2025oct16eqn15eqn2}
\begin{split}
&\big|\int_{t_1}^{t_2} C(  X_{\bot}(s), V(s)) \cdot   K(s, X(s), V(s)) d s \big| \\
& \lesssim  \mathcal{M}(C) \big[\big(\sum_{b\in \mathcal{T}+\mathcal{T}}   2^{-ba_p  } 2^{ b(\gamma_1-\gamma_2)M_t+ (\gamma_1-2\epsilon )  M_t}\big)(t_2-t_1)+  2^{(3\alpha^{\star} +20\epsilon)M_t/4}\big]. 
\end{split}
\ee
Direct computation shows that the coefficients listed in \eqref{2025oct16eqn21} satisfy the assumptions in \eqref{oct23eqn51} with the bounds specified in \eqref{2025oct16eqn21}.
\subsection{The first reduction}
Recall the Duhamel's formula  in \eqref{march14eqn1}.  For any $t_1, t_2\in [0, t]\subset  [0, T) $, s.t., $   \forall s\in[t_1,t_2],$ $|  X_{\bot}(s)|\in I_p\subset [2^{a_p}, 2^{a_p+10}]$, where $a_{p}, p\in \{1,\cdots, K\}$, are defined in \eqref{overlapintervals}.  After localizing the size of frequency, the angle between the frequency variable, and the velocity characteristics $V(s)$, and using the decomposition in \eqref{sep17eqn32}, from   the estimate  \eqref{dec2eqn31} , we have
\be\label{oct2eqn41}
\begin{split}
&\big|\int_{t_1}^{t_2} C (  X_{\bot}(s), V(s)) \cdot   K(s, X(s), V(s)) d s \big| \\
&\lesssim \sum_{\begin{subarray}{c}
k\in \Z,n\in [-  M_t, 2]\cap \Z,\mu\in \{+,-\} \\ 
 i\in \{0,1,2,3,4\}, j\in [0, (1+2\epsilon)M_t]\cap \Z_+\\ 
\end{subarray}}  \big|\int_{t_1}^{t_2}     C (  X_{\bot}(s), V(s))  \cdot T_{k,j;n}^{\mu,i } (s, X(s), V(s)) d s \big| 
+ \mathcal{M}(C),\\ 
\end{split}
\ee
where 
\[
  T^{\mu,i}_{k,j;n} (s, X(s), V(s)):=  T_{k,j;n}^{\mu,i}(1, E)(s, X(s), V(s))+ \hat{V}(t)\times T_{k,j;n}^{\mu,i}(1, B)(s, X(s), V(s)),   
\] 
and the bilinear operator $T_{k,j;n}^{\mu,i}(\cdot, \cdot), i\in\{0,1,2,3,4\}, $ are defined in  \eqref{sep17eqn32}. Define
\be\label{indexsetsec4}
\begin{split}
\mathcal{E}_0=\mathcal{E}_1=\mathcal{E}_2&:=  \{(n, k):  k\in \Z_+,        n\in    [- \alpha^{\star} M_t/2 -30\epsilon M_t ,2]\cap \Z , 
\\ &\quad k+4n\geq  2(\gamma_1-\gamma_2)M_t/3-60\epsilon M_t, \\ 
   &\quad k+2n\geq 2 \alpha^{\star} M_t/3 + (\gamma_1-\gamma_2)M_t/3   -80\epsilon M_t  
 \}, \\ 
\mathcal{E}_3=\mathcal{E}_4&:=\{(n, k):   k\in  \Z_+     ,  n\in  [ (-  \alpha^{\star} +\gamma_1-\gamma_2) M_t/2 -80\epsilon M_t ,2]\cap \Z,\,\,   
  \\ 
 &\quad k+4n\geq  2(\gamma_1-\gamma_2)M_t/3-80\epsilon M_t, k+2n\geq  2 \alpha^{\star} M_t/3   -  \mathbf{1}_{n\in \mathcal{N}_t^2 }\alpha^{\star} M_t  /6\\
 &\quad  + \mathbf{1}_{n\in \mathcal{N}_t^1 }  (\gamma_1-\gamma_2)M_t/3 -80\epsilon M_t \} .
\end{split}
\ee

Recall the assumptions in Proposition \ref{bootstraplemma1}.  From the   estimates  \eqref{2024oct27eqn1} and  \eqref{2022feb24eqn1} in Theorem \ref{maintheoremellipitic} and the estimates \eqref{2024oct8eqn2},   \eqref{2024oct8eqn5}, and \eqref{2024Dec6eqn31}   in Theorem \ref{mainresultsfirstpart}, the following estimate holds  if $( n,k)\notin \mathcal{E}_i,i\in\{0,1,2,3,4\},$ or $a_p \leq -k-n +3\epsilon M_{t}/2$,
\be\label{oct22eqn2}
\begin{split}
&\sum_{j\in \Z_+} \big|\int_{t_1}^{t_2}     C ( X_{\bot}(s), V(s))  \cdot T_{k,j;n}^{\mu,i } (s, X(s), V(s)) d s \big| \\
& \lesssim    \big( \sum_{b\in\mathcal{T} + \mathcal{T}}  \mathcal{M}(C)   2^{-b a_p } 2^{ b(\gamma_1 -\gamma_2)M_t}  2^{(\gamma_1-3 \epsilon)M_t  }  \big) (t_2-t_1) .
\end{split}
\ee

Now, we let $i\in\{0,1,2,3,4\}, ( n,k)\in \mathcal{E}_i$ and $a_{p}\geq  -k-n +3\epsilon M_{t}/2$ be fixed.     Recall the decomposition of the localized acceleration force in \eqref{oct7eqn1}.   

We first consider the contribution from the hyperbolic part.  Note that, on the Fourier side, for the hyperbolic part, from \eqref{2022feb22eqn81} we have 
\be\label{oct25eqn2}
\begin{split}
 &  \int_{t_1}^{t_2}     C ( X_{\bot}(s), V(s))  \cdot  \mathfrak{H}_{k,j;n}^{\mu,i }(1) (s, X(s), V(s)) d s \\
 &  =  \int_{t_1}^{t_2} \int_{\R^3} e^{i X(s)\cdot \xi + i \mu s|\xi|}  C ( X_{\bot}(s), V(s ))\cdot  \mathcal{F}[  \mathfrak{H}_{k,j;n}^{\mu,i}](s, \xi, V(t))  d \xi  d s. \\
 \end{split}
\ee
Moreover, recall the equations satisfied by characteristics in  \eqref{backward}, we have 
\[
e^{i X(s )\cdot \xi + i \mu  s |\xi | } =  ({ i \hat{V}(s)\cdot \xi + i \mu |\xi| })^{-1} \p_s \big(e^{i X(s )\cdot \xi + i \mu  s |\xi | }\big). 
\]
Therefore, we  can use the above equality to exploit the smoothing effect by doing  integration by parts in ``$s$''  once in \eqref{oct25eqn2}.

 After combining the result of doing integration by parts in ``$s$'' with the equality  \eqref{oct7eqn1},  we have 
\be\label{oct25eqn41}
\begin{split}
& \big|\int_{t_1}^{t_2}     C ( X_{\bot}(s), V(s))  \cdot T_{k,j;n}^{\mu,i } (s, X(s), V(s)) d s \big|\\
&\lesssim 1 +   \big|  \mathfrak{E}^{\mu,i}_{k,j;n}(t_1, t_2)\big|  + \sum_{a \in\{0,1 \} }   \big|{}_a^{1}\mathfrak{H}^{\mu,i}_{k,j;n}(t_1, t_2)\big| +\big| {}_a^{1}Err^{\mu,i}_{k,j;n}(t_1, t_2)\big|  ,\\
\end{split}
\ee
where the elliptic contribution  $ \mathfrak{E}^{\mu, i}_{k,j;n} (t_1, t_2), i\in \{0,1,2,3,4\} $ are defined as follows, 
\be\label{oct25eqn1}
\begin{split}
  \forall i\in \{0,1,2,3\}, \quad \mathfrak{E}^{\mu,i}_{k,j;n}(t_1, t_2)&:=  \int_{t_1}^{t_2} C ( X_{\bot}(s), V(s))\cdot      \mathfrak{E}^{\mu, i}_{k,j;n}(1) (s, X(s), V(s))  d s, \\ 
     \mathfrak{E}^{\mu,4}_{k,j;n}(t_1, t_2) &= \int_{t_1}^{t_2}   \int_{\R^3}\int_{\R^3} e^{ix\cdot \xi } \hat{f}( s, \xi, v) \varphi_{n;-M_t} ( \tilde{\xi} + \mu \tilde{V}(s) )      \varphi_{j,n}^4(v,  V(s))    \\
  &\quad \times  C ( X_{\bot}(s), V(s))\cdot m_4(\xi, v , V(s))  \varphi_k(\xi)   d \xi d v d s, \\
   m_4(\xi, v , V(s)) & :=
   \frac{   4\pi\big( (\hat{v}-\hat{\zeta} )\times (\hat{v}\times \xi)+\xi(1-|\hat{v}|^2) }{|\xi|\big({  \hat{V}(s)\cdot \xi +   \mu |\xi| }\big)}, 
\end{split}
\ee
where  $ \mathfrak{E}^{\mu,i}_{k,j;n} ( \cdot)(s, X(s), V(s)), i\in\{0,1,2,3\},$ are defined in  \eqref{2022feb24eqn81}. 

Moreover, the error part  $ {}_a^{1}Err^{\mu,i}_{k,j;n}(t_1, t_2)  $ and the hyperbolic contribution  $ {}_a^{1}\mathfrak{H}^{\mu,i}_{k,j;n}(t_1, t_2),$ $a\in \{0,1\},$  in  \eqref{oct25eqn41}  are defined as follows, 
 \be\label{oct7eqn31}
 \begin{split}
{}_0^{1}Err^{\mu,i}_{k,j;n}(t_1, t_2) & :=  \sum_{i=1,2}(-1)^{i}  \int_{\R^3} e^{i X(t_i)\cdot \xi + i \mu t_i|\xi|} ({ i \hat{V}(t_i)\cdot \xi + i \mu |\xi| })^{-1} \\
&\quad \times  C ( X_{\bot}(t_i), V(t_i))\cdot   \mathcal{F}[ \mathfrak{H}_{k,j;n}^{\mu,i}](t_i, \xi, V(t_i)) d \xi,\\
{}_1^{1}Err^{\mu,i}_{k,j;n}(t_1, t_2)& :=    \int_{t_1}^{t_2} \int_{\R^3} e^{i X(s)\cdot \xi + i \mu  s|\xi|} 
 \big({  \hat{V}(s)\cdot \xi +   \mu |\xi|}\big)^{-1}  \\
 &\quad \times  { \big(\hat{ {V}}_{\bot}(s) \cdot \nabla_{  x_{\bot} } C (  X_{\bot}(s), V(s))\big)\cdot   \mathcal{F}[ \mathfrak{H}_{k,j;n}^{\mu,i}](s, \xi, V(s))}   d \xi  ds. \\
\end{split}
\ee
 
\be\label{nov12eqn61}
\begin{split}
{}_0^{1}\mathfrak{H}^{\mu,i}_{k,j;n}(t_1, t_2) &:=  \int_{t_1}^{t_2} \int_{\R^3} e^{i X(s)\cdot \xi + i \mu s |\xi|} \big[  
   K(s,X(s), V(s))\cdot \mathcal{F}[  {}^1\mathfrak{H}] (s, \xi, X_{\bot}(s), V(s))\big] d \xi d s,\\
{}_1^{1}\mathfrak{H}^{\mu,i}_{k,j;n}(t_1, t_2)&:= \int_{t_1}^{t_2} \int_{\R^3}\int_{\R^3} e^{i X(s)\cdot \xi   } 
    |\xi|^{-1} \big({ i \hat{V}(s)\cdot \xi + i \mu |\xi| }\big)^{-1} \mathcal{F}\big((E+\hat{v}\times B)f\big)(s, \xi, v)  \\
&\quad \cdot \nabla_v \big( C (  X_{\bot}(s), V(s))\cdot
    \tilde{\varphi}_{k,j,n}^i (v,\xi,  V(s))\big)
d\xi d v d s,\\
\end{split}
\ee
 where 
\be\label{oct7eqn90}
\begin{split}
\mathcal{F}[  {}^1 \mathfrak{H}] (s, \xi,   X_{\bot}(s), V(s))&=\nabla_{\zeta} \big[({ i \hat{\zeta} \cdot \xi + i \mu |\xi| })^{-1}   C( X_{\bot}(s),\zeta)\cdot     \mathcal{F}[ \mathfrak{H}_{k,j;n}^{\mu,i}](s, \xi, \zeta) \big]\big|_{\zeta=V(s)}, \\
   \tilde{\varphi}_{k,j,n}^a (v, \xi, V(s))&= \big(
 \frac{ (\hat{V}(s)-\hat{v})\times (\hat{v}\times \xi)   }{  i  (\mu |\xi|+ \hat{v}\cdot \xi  )  }
+  \hat{v} \big)\varphi_{j,n}^a (v,  {V}(s))\\
&\quad \times  \varphi_{n;-M_t}( \tilde{\xi} + \mu \tilde{V}(s) )   \varphi_k(\xi)   , \quad a\in\{0,1,2,3\},\\ 
  \tilde{\varphi}_{k,j,n}^4 (v, \xi, V(s))&= 
   \hat{v}  \varphi_{j,n}^4 (v,  {V}(s))   \varphi_{n;-M_t}( \tilde{\xi} + \mu \tilde{V}(s) )   \varphi_k(\xi)   .\\
\end{split}
\ee
 \begin{remark}

For brevity, the dependence of $\mathcal{F}[{}^1\mathfrak{H}]$ on the parameters $k, j, n$, etc., will be suppressed. This notational convention will be maintained throughout the subsequent discussion.
\end{remark}

Recall the estimate  \eqref{oct25eqn41}. The rest of this section is organized as follows. 
\begin{enumerate}
\item[$\bullet$] In section \ref{hyperbolicPartIest},  we estimate the hyperbolic parts ${}_a^{1}\mathfrak{H}^{\mu,i}_{k,j;n}(t_1, t_2)$, $a\in \{0,1\}$,  $i\in\{0,1,2,3,4\}$.
\item[$\bullet$] In section \ref{ellipticPartIest}, we estimate the elliptic parts  $ \mathfrak{E}^{\mu,i}_{k,j;n}(t_1, t_2)$,   $i\in\{0,1,2,3,4\}$,  in \eqref{oct25eqn41}. 
\item[$\bullet$]  

In section \ref{errortypesPartIest}, we derive estimates for the error terms ${}_a^{1}Err^{\mu,i}_{k,j;n}(t_1, t_2)$, with $a \in \{0, 1\}$ and $i \in \{0, 1, 2, 3, 4\}$. Furthermore, this section provides estimates for all error terms generated during the ISS estimation process for both the hyperbolic and elliptic parts.
\end{enumerate}

\subsection{Estimating the  hyperbolic parts}\label{hyperbolicPartIest}

In this section, we mainly control the hyperbolic parts ${}_a^{1}\mathfrak{H}^{\mu,i}_{k,j;n}(t_1, t_2), a\in \{0,1\},$  in \eqref{oct25eqn41}. 

\subsubsection{The estimate of ${}_0^{1}\mathfrak{H}^{\mu,i}_{k,j;n}(t_1, t_2) $ }

The estimate of ${}_0^{1}\mathfrak{H}^{\mu,i}_{k,j;n}(t_1, t_2) $ is summarized in the following Lemma. 
\begin{lemma}
Let  $i\in\{0,1,2,3,4\}, ( n,k)\in \mathcal{E}_i$ and $a_{p}\geq  -k-n +3\epsilon M_{t}/2$,  under the assumption of Proposition \ref{bootstraplemma1}, the following estimate holds, 
\be\label{2021dec22eqn60}
|{}_0^{1}\mathfrak{H}^{\mu,i}_{k,j;n}(t_1, t_2)|\lesssim   \mathcal{M}(C) \big[\big(\sum_{b\in \mathcal{T}+\mathcal{T}}   2^{-ba_p } 2^{ b(\gamma_1-\gamma_2)M_t+ (\gamma_1-3\epsilon) M_t}\big)(t_2-t_1)+  2^{(3 \alpha^{\star}+20\epsilon)M_t/4}\big]. 
\ee
\end{lemma}
\begin{proof}

Recall  \eqref{nov12eqn61},  and the decomposition of the localized acceleration force in \eqref{oct7eqn1}. We proceed in steps as follows. 

\medskip
\noindent \textbf{Step 1.}\quad Ruling out some trivial cases. 
\medskip

We first rule out the case $k$ is very large, e.g., $k\geq 50M_t$. Recall   \eqref{oct7eqn90}. From the estimates \eqref{2024oct8eqn5}   and \eqref{2024Dec6eqn31}  in Theorem \ref{mainresultsfirstpart},  the following estimate holds for any  $x  \in \R^3$, s.t., $|  x_{\bot}|   \in [2^{a_p  -5}, 2^{a_p   + 5}] $,
\be\label{2021dec24eqn2}
\begin{split}
\big|\int_{\R^3}  \mathcal{F}[{}^1\mathfrak{H} ](s, \xi,    X_{\bot}(s), V(s)) e^{ix\cdot \xi + i \mu   s |\xi|}   d \xi \big| &\lesssim \sum_{a\in \mathcal{T}}  2^{-a a_p + a(\gamma_1-\gamma_2)M_{t} -(k+2n)/2+4 0\epsilon M_t}    \mathcal{M}(C)\\
&\quad \times \big(  2^{2(\gamma_1-\gamma_2)M_t/3} \mathbf{1}_{n \in \mathcal{N}_t^1}  +2^{-\alpha^{\star}  M_t/4 }  \mathbf{1}_{n \in \mathcal{N}_t^2 }  \big). \\
\end{split}
\ee

From the above estimate and the rough estimate of the electromagnetic field   \eqref{maintheoremroughest}  in Theorem \ref{maintheorem1part1}, we  can rule out the case $k\geq 50M_t$ as follows, 
\be\label{2021dec24eqn1}
\sum_{k\in[50 M_t, \infty]\cap \Z_+} |{}_0^{1}\mathfrak{H}_{k,j;n}(t_1, t_2)|\lesssim \sum_{a\in \mathcal{T}}  2^{-a a_p + a(\gamma_1-\gamma_2)M_{t}}       \mathcal{M}(C)(t_2-t_1). 
\ee

Given that there are at most $(M_t)^2$ remaining cases, which introduce only a logarithmic loss, it is sufficient to consider a fixed pair $(n, k) \in \mathcal{E}_i$.

 Now,  we rule out the contribution comes from the elliptic parts of the  new introduced acceleration force in ${}_0^{1}\mathfrak{H}^{\mu,i}_{k,j;n}(t_1, t_2)$.   Recall  \eqref{2022feb24eqn81}, and the fact  that $\forall s\in [t_1, t_2], $ we have $|  X_{\bot}(s)|\geq 2^{-k-n+\epsilon M_t}$. From the first  estimate in \eqref{2022feb25eqn1} and the obtained estimate  \eqref{2021dec24eqn2}, for any $a_1\in \{0,1,2,3\},
 \mu_1\in\{+,-\} \cap\Z, $ we have
\be
\begin{split}
\sum_{\begin{subarray}{c}
k_1, j_1\in \Z_+\\
n_1\in [-M_t,2]\cap \Z\\
\end{subarray} } &\big|\int_{t_1}^{t_2} \int_{\R^3} e^{i X(s)\cdot \xi + i \mu   s |\xi|}    \mathfrak{E}^{\mu_1, a_1 }_{k_1,j_1;n_1}(s, X(s), V(s)) \cdot  \mathcal{F}[{}^1\mathfrak{H} ] (s, \xi,   X_{\bot}(s), V(s))  d \xi d s\big| \\
&  \lesssim  \sum_{b\in \mathcal{T}}  2^{-b a_p + (b+1/2)(\gamma_1-\gamma_2)M_{t}}   2^{  -(k+2n)/2 }  2^{-a_p/2 + (\alpha^{\star} +4\epsilon )M_t}  \mathcal{M}(C)(t_2-t_1)\\ 
 &\lesssim \sum_{d\in  \mathcal{T} +  \mathcal{T} }  2^{-d a_p + d(\gamma_1-\gamma_2)M_{t}} 2^{2\alpha^{\star}M_t/3+100\epsilon M_t} \mathcal{M}(C) (t_2-t_1).\\
\end{split}
\ee
 Therefore, recall again the decomposition  in \eqref{oct7eqn1},  from the above obtained estimate,   the estimate  \eqref{dec2eqn31},  we have 
\be\label{2024oct29eqn11}
\begin{split}
\big|{}_0^{1}\mathfrak{H}_{k,j;n}(t_1, t_2)\big|&\lesssim \sum_{\begin{subarray}{c}
k_1 \in \Z_+, j_1\in [0, (1+2\epsilon)M_t]\cap \Z,  \mu_1\in\{+,-\} \\ 
 i_1\in\{0,1,2,3,4\}, n_1\in [-M_t,2]\cap \Z, 
\end{subarray}}     |\mathfrak{H}_{ i, i_1}(t_1, t_2)|\\
&\quad +    \sum_{a \in  \mathcal{T} +  \mathcal{T} }  2^{-a a_p + a(\gamma_1-\gamma_2)M_{t}} 2^{(\alpha^{\star}-20\epsilon )M_t }(t_2-t_1) \mathcal{M}(C), \\
 \mathfrak{H}_{ i, i_1}(t_1, t_2)& =  \int_{t_1}^{t_2} \int_{\R^3} e^{i X(s)\cdot \xi + i \mu s|\xi|}  \mathfrak{H}_{k_1,j_1;n_1}^{\mu_1,i_1}(s, X(s), V(s) ) \cdot \mathcal{F}[{}^1\mathfrak{H} ] (s , \xi,   X_{\bot}(s), V(s))d \xi d s \\
\end{split}
\ee
where, for convenience, we suppress the dependence of $\mathfrak{H}_{ i, i_1}(t_1, t_2) $ with respect to $k,j,n,k_1,j_1,n_1,\mu,\mu_1$.

\medskip
\noindent \textbf{Step 2.}\quad The second iteration of smoothing.  

\medskip

 Let
\be\label{firstiterindex}
 \begin{split}
{}^{1}\mathcal{E}_{k,n}:=\{(k_1,n_1):  k_1 \in \Z_+, n_1\in [-M_t, 2]\cap \Z,    n_1&\geq (-\alpha^{\star}+\gamma_1-\gamma_2) M_t/2 -30\epsilon M_t  , \\ 
 (k_1+2n_1)- (k+2n)&\geq  \alpha^{\star} M_t/2 + \mathbf{1}_{n\in \mathcal{N}_t^2} \alpha^{\star} M_t/2  \\
 &\quad  - \mathbf{1}_{n\in \mathcal{N}_t^1} 4(\gamma_1-\gamma_2)M_t/3  -200\epsilon M_t \}. 
 \end{split}
 \ee
 
 We  can first rule out the case   $ (k_1,n_1)\notin {}^{1}\mathcal{E}_{k,n}$ easily. From   the obtained estimate  \eqref{2021dec24eqn2}  and  the estimates \eqref{2024oct8eqn5}   and \eqref{2024Dec6eqn31}  in Theorem  \ref{mainresultsfirstpart}, we have 
\be\label{2021dec23eqn32}
\begin{split}
 \big|  \mathfrak{H}_{ i, i_1}(t_1, t_2)\big|&\lesssim    \mathcal{M}(C) \big[ \sum_{b\in \mathcal{T}}  2^{-b a_p +  b (\gamma_1-\gamma_2)M_{t} -(k+2n)/2+40\epsilon M_t}   \\
 &\quad \times  \big(  2^{2(\gamma_1-\gamma_2)M_t/3} \mathbf{1}_{n \in \mathcal{N}_t^1}  +2^{-\alpha^{\star}  M_t/4 }  \mathbf{1}_{n \in \mathcal{N}_t^2 }  \big)      \big]  \\
&\quad \times\big[  \sum_{b\in\mathcal{T}} 2^{-b a_p} 2^{b (\gamma_1-\gamma_2)M_t}  2^{(\alpha^{\star}-10\epsilon)M_t -(\gamma_1-\gamma_2)M_t/4}   \\
&\quad  +  2^{-b a_p} 2^{b  (\gamma_1-\gamma_2) M_t } 2^{(k_1+2n_1)/2+3\alpha^\star M_t/4 +40\epsilon M_t } \big]\\
&\lesssim    \sum_{d \in  \mathcal{T} +  \mathcal{T} }  2^{- d a_p + d(\gamma_1-\gamma_2)M_{t}} 2^{(\alpha^{\star}-10\epsilon )M_t }(t_2-t_1) \mathcal{M}(C). \\
\end{split}
\ee
Now, it suffices to consider the case  $ (k_1,n_1)\in {}^{1}\mathcal{E}_{k,n}$.  Note that, on the Fourier side, we have
\[
\begin{split}
\mathfrak{H}_{ i, i_1}(t_1, t_2) &=  \int_{t_1}^{t_2} \int_{\R^3} \int_{\R^3} e^{i X(s)\cdot (\xi+\eta) + i \mu s|\xi| + i \mu_1 s|\eta|}  \\
  &\quad \times     \mathcal{F}[\mathfrak{H}_{k_1,j_1;n_1}^{\mu_1,i_1}](s, \eta , V(s) ) \cdot  \mathcal{F}[{}^1\mathfrak{H} ] (s, \xi,    X_{\bot}(s), V(s))   d \xi d\eta  d s .
  \end{split}
  \]

Thanks to  the fact that, $ \forall (k_1,n_1)\in {}^{1}\mathcal{E}_{k,n}$,  $  (k_1+2n_1)- (k+2n)\geq  \alpha^{\star} M_t/2-100\epsilon M_t  $, we have
\be
\big|\hat{V}( s)\cdot(\xi+\eta) +   \mu |\xi| +   \mu_1  |\eta|\big|\sim |\hat{V}(s)\cdot \eta    +   \mu_1  |\eta| | \sim 2^{k_1+2n_1}.
\ee
 
To take the advantage of high oscillation in time, we do integration by parts in  ``$s$'' one more time. As a result, after using the decomposition for the acceleration force in  \eqref{oct7eqn1},  the estimate \eqref{dec2eqn31}, we have  
\be\label{oct29eqn55}
\begin{split}
 \big| \mathfrak{H}_{ i, i_1}(t_1, t_2)\big| & \lesssim \mathcal{M}(C)+  \sum_{a=0,1,2  }  \big| Err^a_{ i, i_1}(t_1, t_2)\big|    \\
 &\quad +\sum_{\begin{subarray}{c}
  i_2 \in\{0,1,2,3,4\} \\
k_2 \in \Z_+, j_2\in [0, (1+2\epsilon)M_t]\cap \Z\\
 \mu_2\in\{+,-\},n_2 \in [-M_t,2]\cap \Z \\ 
\end{subarray}}  \big|  \mathfrak{H}_{ i, i_1,i_2}(t_1, t_2)\big|\\
\mathfrak{H}_{ i, i_1,i_2}(t_1, t_2)&:= \int_{t_1}^{t_2} \int_{\R^3}  \int_{\R^3} e^{i X(s)\cdot (\xi+\eta) + i \mu s|\xi| + i \mu_1 s|\eta|}  \\
  &\quad\times   \mathfrak{H}_{k_2,j_2;n_2}^{\mu_2,i_2}(s,X(s), V(s))\cdot \mathcal{F}[ {}^2\mathfrak{H}] (s, \xi,\eta,   X_{\bot}(s),V(s))  d \xi d\eta  d s, 
\end{split}
 \ee
where $\mathcal{F}[ {}^2\mathfrak{H}] (s, \xi,\eta,   X_{\bot}(s),V(s)) $ is defined as follows, 
\be\label{oct10eqn96}
\begin{split}
 \mathcal{F}[ {}^2\mathfrak{H}](s, \xi,\eta,   X_{\bot}(s),V(s)) 
&:= \nabla_\zeta \big[ \big( \Phi_2(\xi, \eta,  \zeta)\big)^{-1}  
 \mathcal{F}[\mathfrak{H}_{k_1,j_1;n_1}^{\mu_1,i_1}](s, \eta , \zeta) \\
 &\quad \cdot  \mathcal{F}[ {}^1\mathfrak{H}] (s, \xi,   X_{\bot}(s),\zeta)\big]\big|_{\zeta= V(s)},\\
 \Phi_2(\xi, \eta,  \zeta)&:=  \hat{\zeta} \cdot(\xi+ \eta   )     +   \mu  |\xi| +   \mu_1  |\eta|.
 \end{split}
\ee
and the error terms $Err^a_{ i, i_1}(t_1, t_2), a\in\{0,1,2\},$ are defined as follows, 
\be\label{oct8eqn1}
\begin{split}
 Err^0_{ i, i_1}(t_1, t_2)&:=  \sum_{i=1,2} (-1)^i   \int_{\R^3}\int_{\R^3} e^{i X(t_i)\cdot (\xi+\eta) + i \mu t_i|\xi| + i \mu_1 t_i|\eta|}\big( \Phi_2(\xi, \eta,  V(t_i))\big)^{-1}      \\
 &\quad\times     \mathcal{F}[\mathfrak{H}_{k_1,j_1;n_1}^{\mu_1,i_1}](t_i, \eta , V(t_i))  \cdot  \mathcal{F}[ {}^1\mathfrak{H}] (t_i, \xi,  X_{\bot}(t_i),V(t_i))   d \xi d\eta    \\
& \quad    +    \int_{t_1}^{t_2} \int_{\R^3} \int_{\R^3}e^{i X(s)\cdot (\xi+\eta) + i \mu s|\xi| + i \mu_1 s|\eta|}  \big( \Phi_2(\xi, \eta,  V(s))\big)^{-1} \\
 &\quad  \times            \mathcal{F}[\mathfrak{H}_{k_1,j_1;n_1}^{\mu_1,i_1}](s, \eta , V(s))  \cdot \big(  {\hat{V}}_{\bot}(s)\cdot \nabla_{ x_{\bot}}   \mathcal{F}[ {}^1\mathfrak{H}] (s , \xi,    X_{\bot}(s ),V(s ))  \big)       d \xi d\eta  d s, 
 \end{split}
 \ee
\be\label{oct10eqn77}
\begin{split}
Err^1_{ i, i_1}(t_1, t_2)&:=   \int_{t_1}^{t_2} \int_{\R^3}\int_{\R^3} e^{i X(s)\cdot (\xi+\eta) + i \mu s|\xi| + i \mu_1 s|\eta|}   \big( \Phi_2(\xi, \eta,  V(s))\big)^{-1}      \\
&\quad \times  \big( \p_s    \mathcal{F}[\mathfrak{H}_{k_1,j_1;n_1}^{\mu_1,i_1}](s, \eta , V(s))  \cdot  \mathcal{F}[ {}^1\mathfrak{H}] (s, \xi,    X_{\bot}(s),V(s))  \\
 &\quad+    \mathcal{F}[\mathfrak{H}_{k_1,j_1;n_1}^{\mu_1,i_1}](s, \eta , V(s))  \cdot  \p_s   \mathcal{F}[ {}^1\mathfrak{H}]  (s, \xi,   X_{\bot}(s),V(s)) \big)  d \xi d\eta  d s, 
 \end{split}
\ee
\be\label{oct10eqn93}
\begin{split}
 Err^2_{ i, i_1}(t_1, t_2):=&\sum_{\begin{subarray}{c}
k_2\in \Z_+, j_2\in [0, (1+2\epsilon)M_t]\cap \Z  \\
i_2\in\{0,1,2,3,4\}, \mu_2\in\{+,-\}, \\ 
 n_2 \in [-M_t,2]\cap \Z,  a_2\in\{0,1,2,3\}\\ 
\end{subarray}} \int_{t_1}^{t_2} \int_{\R^3}  \int_{\R^3} e^{i X(s)\cdot (\xi+\eta) + i \mu s|\xi| + i \mu_1 s|\eta|} \\
&\quad  \times \big(    \mathfrak{E}^{\mu_2, a_2 }_{k_2,j_2;n_2} (s, X(s), V(s))+ Ini_{k_2,j_2;n_2}^{\mu_2}(s, X(s), V(s))\big)\\
&\qquad \cdot \mathcal{F}[{}^2\mathfrak{H} ](s, \xi,\eta,    X_{\bot}(s),V(s))  d \xi d\eta  d s. \\
\end{split}
\ee

From  the estimate  \eqref{oct29eqn55}  and the estimate of error type terms  \eqref{2021dec22eqn32} in Lemma \ref{2021errhorstep1},    to estimate $   \mathfrak{H}_{ i, i_1}(t_1, t_2)$, it suffices to estimate $   \mathfrak{H}_{ i, i_1,i_2}(t_1, t_2)$.

\medskip
\noindent \textbf{Step 3.}\quad The third iteration of smoothing.  

\medskip

 Similar to the obtained estimate  \eqref{2021dec24eqn1}, we can first rule out the case $k_1\geq 50 M_t$. Now, it suffices to consider the fixed $(n_1, k_1)\in  {}^{1}\mathcal{E}_{k,n}$ s.t., $k_1\leq 50M_t$. 
Let
\be\label{seconditerindex}
 \begin{split}
{}^{2}\mathcal{E}_{k_1,n_1}:=\{( n_2,k_2):  k_2\in \Z_+,  &n_2\in [ (-\alpha^{\star}+\gamma_1-\gamma_2) M_t/2 -30\epsilon M_t  , 2]\cap \Z,\\
   & (k_2+2n_2)- (k_1+2n_1)\geq \alpha^{\star} M_t/6\\
    &\quad -6\iota M_t -560\epsilon M_t -(\gamma_1-\gamma_2)M_t \}  . 
 \end{split}
 \ee

Recall  \eqref{oct10eqn96}.    Similar to  the obtained  estimate \eqref{2021dec24eqn2},  from   the estimates in \eqref{2024oct8eqn1},  \eqref{2024oct8eqn5},  and \eqref{2024Dec6eqn31} in Theorem  \ref{mainresultsfirstpart},   for any  $ s\in [t_1,t_2], x,y \in \R^3$, s.t., $| x_{\bot}|, | y_{\bot}| \in [2^{a_p  -5}, 2^{a_p   + 5}] $, we have 
\be\label{oct11eqn20}
\begin{split}
&\big| \int_{\R^3} \int_{\R^3}  e^{i x\cdot \xi +  i y\cdot \eta + i \mu_1 s  |\eta| +  i \mu s |\xi| }   \mathcal{F}[{}^2\mathfrak{H} ]   (s, \xi,\eta,    X_{\bot}(s),V(s) )  d \xi d \eta \big|\\
& \lesssim  \big(\sum_{b\in \mathcal{T}}  2^{ -b a_p   +b(\gamma_1-\gamma_2)M_t  }   2^{40\epsilon M_t -(k+2n)/2}   (  2^{2(\gamma_1-\gamma_2)M_t/3} \mathbf{1}_{n \in \mathcal{N}_t^1}  +2^{-\alpha^{\star}  M_t/4 }  \mathbf{1}_{n \in \mathcal{N}_t^2 }  )  \big)\\
 &\quad  \times   \big( \big( 2^{(1-\epsilon)M_t}+ 2^{ 130\epsilon M_t} 
    2^{(k_1+2n_1)/2+ (\alpha^{\star}+3\iota) M_t  }\big) \big) 2^{-\gamma_2 M_t} 2^{-k_1-2n_1} (2^{-n}+2^{-n_1})\\
& \lesssim \sum_{b\in \mathcal{T} } 2^{- ba_p } 2^{(b+5/6) (\gamma_1-\gamma_2)M_t } 2^{ \alpha^{\star} M_t/6  + 220\epsilon M_{t} + 3\iota M_t   -(k_1+2n_1)/2}\mathcal{M}(C) .\\
\end{split}
\ee

From the above obtained estimate  \eqref{oct11eqn20}    and the estimates  in  \eqref{2024oct8eqn5} and \eqref{2024Dec6eqn31} in Theorem  \ref{mainresultsfirstpart},      we can rule out the case  $(k_2,n_2)\notin {}^{2}\mathcal{E}_{k_1,n_1} $ as follows, 
\be
\begin{split}
& \big| \int_{t_1}^{t_2} \int_{\R^3}  \int_{\R^3} e^{i X(s)\cdot (\xi+\eta) + i \mu s|\xi| + i \mu_1 s|\eta|}  \\
&\quad \times  \mathfrak{H}_{k_2,j_2;n_2}^{\mu_2,i_2}(s,X(s), V(s))\cdot  \mathcal{F}[{}^2\mathfrak{H} ] (s, \xi,\eta, X_{\bot}(s),V(s))  d \xi d\eta  d s\big|\\
& \lesssim \big( \sum_{b\in \mathcal{T}} 2^{-b a_p  }  2^{b (\gamma_1-\gamma_2) M_t }     ( 2^{  \alpha^{\star} M_t -(\gamma_1-\gamma_2)M_t/4 }+ 2^{40\epsilon M_t}   2^{(k_2+2n_2)/2+3  \alpha^{\star} M_t/4 }  )  \big)   \\
&\quad  \times \big(  \sum_{b\in \mathcal{T} } 2^{- ba_p } 2^{(b+5/6) (\gamma_1-\gamma_2)M_t } 2^{ \alpha^{\star} M_t/6  + 220\epsilon M_{t} + 3\iota M_t   -(k_1+2n_1)/2}\mathcal{M}(C)  \big)   \\
& \lesssim  \sum_{d\in \mathcal{T}+\mathcal{T} }   2^{- d a_p } 2^{d(\gamma_1-\gamma_2)M_t }   2^{ (\alpha^{\star} - 20\epsilon) M_t} \mathcal{M}(C).
\\
\end{split} 
\ee 

Now, it suffices to consider the case    the case  $(k_2,n_2)\in {}^{2}\mathcal{E}_{k_1,n_1}  $.  Note that, on the Fourier side, we have
\be\label{oct23eqn1}
\begin{split}
  \mathfrak{H}_{ i, i_1,i_2}(t_1, t_2)&=\int_{t_1}^{t_2} \int_{\R^3}  \int_{\R^3} e^{i X(s)\cdot (\xi+\eta + \sigma)  + i  \mu_2 s |\sigma| + i \mu s|\xi| + i \mu_1 s|\eta|} \\
  &\quad \times  \mathcal{F}[\mathfrak{H}_{k_2,j_2;n_2}^{\mu_2,i_2}](s,\sigma, V(s))\cdot    \mathcal{F}[{}^2\mathfrak{H} ]  (s, \xi,\eta,   X_{\bot}(s), V(s))  d \xi d\eta  d s.
  \end{split}
\ee 
Thanks to the fact that,  $(k_2+2n_2) - (k_1+2n_1) \geq \alpha^{\star}  M_t/12 $, we have
\[
\big|\hat{V}(t)\cdot(\xi+ \eta +\sigma) +  \mu_2   |\sigma| +   \mu  |\xi| +   \mu_1  |\eta| \big|\sim  |\hat{V}(t)\cdot \sigma  +  \mu_2   |\sigma|  |\sim 2^{k_2+2n_2}. 
\]

To exploit the high oscillation in time, we do integration by parts in ``$s$'' one more time. As a result,  after using the decomposition for the acceleration force in  \eqref{oct7eqn1}  for the new introduced acceleration force,    the estimate  \eqref{dec2eqn31},     we have   
\be\label{oct29eqn61}
\begin{split}
\big|    \mathfrak{H}_{ i, i_1,i_2}(t_1, t_2)\big|&\lesssim \mathcal{M}(C) +   \sum_{a=0,1,2  }  \big|Err^a_{ i, i_1,i_2}(t_1, t_2) \big| \\
&\quad +  \sum_{\begin{subarray}{c}
 i_3\in\{0,1,2,3,4\} \\ 
k_3 \in \Z_+, j_3\in [0, (1+2\epsilon)M_t]\cap \Z\\
 n_3 \in [-M_t,2]\cap \Z , \mu_3\in\{+,-\}\\ 
\end{subarray}}   \big|  \mathfrak{H}_{ i, i_1,i_2,i_3}(t_1, t_2)\big| , \\ 
  \mathfrak{H}_{ i, i_1,i_2,i_3}(t_1, t_2)&:=\int_{t_1}^{t_2} \int_{ \R^3 } \int_{ \R^3 } 
\int_{ \R^3 }   e^{i X(s)\cdot (\xi+\eta + \sigma)  + i  \mu_2 s |\sigma| + i \mu s |\xi| + i \mu_1 s |\eta|} \\ 
  &\quad \times  \mathfrak{H}_{k_3,j_3;n_3}^{\mu_3,i_3}(s,X(s), V(s))\cdot   \mathcal{F}[{}^3\mathfrak{H} ](s, \xi, \eta, \sigma,   X_{\bot}(s), V(s)) d\xi d \eta d \sigma ds, 
\end{split}
\ee
where    $ \mathcal{F}[{}^3\mathfrak{H} ](s, \xi, \eta, \sigma,    X_{\bot}(s), V(s))$ is defined as follows,  
\be\label{oct11eqn86}
\begin{split}
  \mathcal{F}[{}^3\mathfrak{H} ](s, \xi, \eta, \sigma,   X_{\bot}(s), V(s)) 
&:= \nabla_\zeta \big[\big(\Phi_3(\xi, \eta, \sigma ,\zeta) \big)^{-1}   \mathcal{F}[\mathfrak{H}_{k_2,j_2;n_2}^{\mu_2,i_2}](s,\sigma, \zeta)\\
&\quad \cdot   \mathcal{F}[{}^2\mathfrak{H} ] (s, \xi,\eta,    X_{\bot}(s),\zeta)   \big]\big|_{\zeta= V(s)}, \\
\Phi_3(\xi, \eta, \sigma ,\zeta)&:=  \hat{\zeta} \cdot(\xi+ \eta +\sigma )  +  \mu_2   |\sigma| +   \mu  |\xi| +   \mu_1  |\eta|.\\
\end{split}
\ee
and the error terms $Err^a_{ i, i_1,i_2}(t_1, t_2), a\in \{0,1,2,3\}$, are defined as follows, 
 \be\label{oct11eqn187}
 \begin{split}
  Err^0_{ i, i_1,i_2}(t_1, t_2) &:= \sum_{i=0,1}   \int_{\R^3}  \int_{\R^3} e^{i X(t_i)\cdot (\xi+\eta + \sigma)  + i  \mu_2 t_i |\sigma| + i \mu t_i|\xi| + i \mu_1 t_i |\eta|}\big(\Phi_3(\xi, \eta, \sigma ,V(t_i)) \big)^{-1} \\
  &\quad \times   \mathcal{F}[\mathfrak{H}_{k_2,j_2;n_2}^{\mu_2,i_2}] (t_i,\sigma, V(t_i))\cdot \mathcal{F}[{}^2\mathfrak{H} ] (t_i, \xi,\eta,    X_{\bot}(t_i),V(t_i))   d \xi d\eta \\
 &\quad   + 
  \int_{t_1}^{t_2} \int_{\R^3}  \int_{\R^3} e^{i X(s )\cdot (\xi+\eta + \sigma)  + i  \mu_2 s  |\sigma| + i \mu s |\xi| + i \mu_1 s  |\eta|} \big(\Phi_3(\xi, \eta, \sigma ,V(s)) \big)^{-1}\\
 &\quad \times     \mathcal{F}[\mathfrak{H}_{k_2,j_2;n_2}^{\mu_2,i_2}](s ,\sigma, V(s))\cdot \big(  { \hat{V}}_{\bot}(s)\cdot \nabla_{ x_{\bot}} \mathcal{F}[{}^2\mathfrak{H} ](s , \xi,\eta,    X_{\bot}(s ),V(s))\big)     d \xi d\eta  ds,\\
 \end{split}
\ee

 \be\label{oct11eqn42}
 \begin{split}
 Err^1_{ i, i_1,i_2}(t_1, t_2)&:= \int_{t_1}^{t_2} \int_{\R^3}  \int_{\R^3} e^{i X(s )\cdot (\xi+\eta + \sigma)  + i  \mu_2 s |\sigma| + i \mu s |\xi| + i \mu_1 s  |\eta|} \big(\Phi_3(\xi, \eta, \sigma ,V(s)) \big)^{-1} \\
 &\quad \times \big(    \mathcal{F}[\mathfrak{H}_{k_2,j_2;n_2}^{\mu_2,i_2}] (s,\sigma, V(s ))\cdot    \p_s \mathcal{F}[{}^2\mathfrak{H} ] (s, \xi,\eta,   X_{\bot}(s ),V(s)) \\
 &\quad +   \p_s   \mathcal{F}[\mathfrak{H}_{k_2,j_2;n_2}^{\mu_2,i_2}] (s ,\sigma, V(s))\cdot     \mathcal{F}[{}^2\mathfrak{H} ] (s , \xi,\eta,  X_{\bot}(s),V(s)) \big)  d \xi d\eta  ds, \\
 \end{split}
\ee

 \be\label{oct11eqn43}
 \begin{split}
  Err^2_{ i, i_1,i_2}(t_1, t_2) & :=    \sum_{\begin{subarray}{c}
k_3, j_3\in \Z_+,i_3\in\{0,1,2,3,4\}\\
 a_3\in\{0,1,2,3\} , \mu_3\in\{+,-\}\\ 
 n_3 \in [-M_t,2]\cap \Z  \\ 
\end{subarray}} \int_{t_1}^{t_2} \int_{\R^3}  \int_{\R^3}\int_{\R^3}  e^{i X(s)\cdot (\xi+\eta + \sigma)  + i  \mu_2 s |\sigma| + i \mu s|\xi| + i \mu_1 s|\eta|}  \\
 &\quad  \times  \big(   \mathfrak{E}^{\mu_3, a_3}_{k_3,j_3;n_3} (s, X(s), V(s))+ Ini_{k_3,j_3,n_3}^{\mu_3}(s, X(s), V(s)) \big) \\
 &\quad \cdot \mathcal{F}[{}^3\mathfrak{H} ](s, \xi, \eta, \sigma,  X_{\bot}(s), V(s)) d\xi d \eta d \sigma d s.
 \end{split}
\ee

From the estimate \eqref{oct29eqn61} and the estimate  \eqref{2021dec22eqn56}  in Lemma \ref{2021errhorstep2},   to estimate $  \mathfrak{H}_{ i, i_1,i_2}(t_1, t_2)$, it suffices to estimate $  \mathfrak{H}_{ i, i_1,i_2, i_3}(t_1, t_2)$.

\medskip
\noindent \textbf{Step 4.}\quad The fourth iteration of smoothing.  

\medskip

 As in the obtained estimate  \eqref{2021dec24eqn1}, we can rule out further the case $k_2\geq 50M_t$. It suffices to let $(k_2,n_2)\in {}^{2}\mathcal{E}_{k_1,n_1}  $ be fixed. Let
\be\label{thirditerindex}
 \begin{split}
{}^{3}\mathcal{E}_{k_2,n_2}:=\{( n_3, k_3): k_3 \in \Z_+,   &  n_3\in [-M_t, 2]\cap \Z,\, n_3 \geq (-\alpha^{\star}+\gamma_1-\gamma_2) M_t/2  -30\epsilon M_t, \\  
&  (k_3+2n_3)- (k_2+2n_2)\\
& \quad  \geq \alpha^{\star} M_t/3    -12\iota M_{t} -1100\epsilon M_t -2(\gamma_1-\gamma_2)M_t   \}.
 \end{split}
 \ee

Recall  \eqref{oct11eqn86}. Similar to the obtained estimate  \eqref{oct11eqn20},  from   the estimates in \eqref{2024oct8eqn1},  \eqref{2024oct8eqn5},  and \eqref{2024Dec6eqn31} in Theorem  \ref{mainresultsfirstpart},   the following estimate holds   for any   $
  s\in [t_1,t_2],$  $x,y,z\in \R^3$, s.t., $|  x_{\bot}|, | y_{\bot}|, |  z_{\bot}|\in [2^{a_p  -5}, 2^{a_p   + 5}] $,
\be\label{2021dec24eqn21}
\begin{split}
&\big| \int_{\R^3}  \int_{\R^3}  \int_{\R^3}   e^{i   (x\cdot\xi+y\cdot \eta + z\cdot \sigma)  + i  \mu_2 s |\sigma|  + i \mu_1 s |\eta|+ i \mu s |\xi|}    \mathcal{F}[{}^3\mathfrak{H} ] (s, \xi, \eta, \sigma,   X_{\bot}(s), V(s))  d \xi d \eta  d\sigma \big|\\
  &\lesssim \sum_{b\in \mathcal{T} }  \mathcal{M}(C)  2^{- ba_p } 2^{(b+11/6) (\gamma_1-\gamma_2)M_t } 2^{-\min\{n, n_1, n_2 \}} \\
  &\quad \times 2^{ \alpha^{\star} M_t/6  + 350\epsilon M_{t} + 6\iota M_{t}  -(k_1+2n_1)/2 - (k_2+2n_2)/2} \\
&\lesssim \sum_{b\in \mathcal{T} } \mathcal{M}(C) 2^{- ba_p } 2^{(b+4/3) (\gamma_1-\gamma_2)M_t + \alpha^{\star}M_{t^\star}/12+  490\epsilon M_{t} +6\iota M_{t}-(k_2+2n_2)/2} . \\
\end{split}
\ee

From the above estimate     and    the estimates in \eqref{2024oct8eqn5}  and \eqref{2024Dec6eqn31}   in Theorem  \ref{mainresultsfirstpart}, the following estimate holds if $(k_3,n_3)\notin {}^{3}\mathcal{E}_{k_2,n_2}, $
\be
\begin{split}
&|  \mathfrak{H}_{ i, i_1,i_2,i_3}(t_1, t_2)|\\
&\lesssim  \big[  \sum_{b\in \mathcal{T} } 2^{- ba_p } 2^{(b+4/3) (\gamma_1-\gamma_2)M_t + \alpha^{\star}M_{t^\star}/12+  490\epsilon M_{t} +6\iota M_{t} -(k_2+2n_2)/2} \mathcal{M}(C) \big] \\
&\quad \times \big( \sum_{b\in \mathcal{T}} 2^{-b a_p  }  2^{b (\gamma_1-\gamma_2) M_t }( 2^{  \alpha^{\star} M_t  -(\gamma_1-\gamma_2)M_{t^{\star}} /4 }+ 2^{40\epsilon M_t}     2^{(k_3+2n_3)/2+3  \alpha^{\star} M_t/4 } )  \big)\\
& \lesssim  \sum_{d\in \mathcal{T}+\mathcal{T} }   2^{- d a_p } 2^{d (\gamma_1-\gamma_2)M_t }   2^{ (\alpha^{\star} - 10\epsilon) M_t} \mathcal{M}(C) . 
\end{split}
\ee

It remains to consider the case   $(k_2,n_2)\in {}^{2}\mathcal{E}_{k_1,n_1}. $ Note that, on the Fourier side, we have
\[
\begin{split}
&\int_{t_1}^{t_2} \int_{\R^3} \int_{\R^3}  \int_{\R^3}  e^{i X(s)\cdot (\xi+\eta + \sigma)  + i s( \mu_2   |\sigma| +   \mu  |\xi| +  \mu_1 |\eta|) } \\
&\quad \times  \mathfrak{H}_{k_3,j_3;n_3}^{\mu_3,i_3}(s,X(s), V(s))\cdot  \mathcal{F}[{}^3\mathfrak{H} ] (s, \xi, \eta, \sigma,   X_{\bot}(s), V(s)) d\xi d \eta d \sigma d s\\
&=\int_{t_1}^{t_2} \int_{(\R^3)^4}  e^{i X(s)\cdot (\xi+\eta + \sigma+\kappa) +   i s ( \mu_3   |\kappa|  +  \mu_2   |\sigma| +   \mu  |\xi| +  \mu_1 |\eta|)}\\
&\quad \times  \mathcal{F}[\mathfrak{H}_{k_3,j_3;n_3}^{\mu_3,i_3}](s,\kappa, V(s)) \cdot  \mathcal{F}[{}^3\mathfrak{H} ] (s, \xi,\eta, \sigma,   X_{\bot}(s),V(s))  d \xi d\eta d \sigma d \kappa d s\\
\end{split}
\]
Thanks to the fact that,  $(k_3+2n_3)/2\gg (k_2+2n_2)/2\gg (k_1+2n_1)/2\gg (k +2n )/2  $, we have
\[
\big|\hat{V}(s)\cdot(\xi+ \eta +\sigma+\kappa) + \mu_3 |\kappa| +  \mu_2   |\sigma| +   \mu  |\xi| +   \mu_1  |\eta| \big|\sim  |\hat{V}(s)\cdot \kappa  +  \mu_3   |\kappa|  |\sim 2^{k_3+2n_3}. 
\]

To exploit the high oscillation in time, we do integration by parts in ``$s$'' one more time. As a result,    we have   
\be\label{oct12eqn76}
\begin{split}
  \big|\mathfrak{H}_{ i, i_1,i_2,i_3}(t_1, t_2)
&\big|\lesssim \mathcal{M}(C)+ \sum_{a=0,1,2  }    Err^a_{ i, i_1,i_2,i_3}(t_1, t_2) \\
&\quad  + \sum_{\begin{subarray}{c}
 i_4\in\{0,1,2,3,4\}\\
k_4 \in \Z_+, j_4\in   \in [0, (1+2\epsilon)M_t]\cap \Z \\ 
 n_4 \in [-M_t,2]\cap \Z , \mu_4\in\{+,-\}\\ 
\end{subarray}}    \mathfrak{H}_{ i, i_1,i_2,i_3,i_4}(t_1, t_2), \\
   \mathfrak{H}_{ i, i_1,i_2,i_3,i_4}(t_1, t_2)&:= \int_{t_1}^{t_2} \int_{(\R^3)^4}   e^{i X(s)\cdot (\xi+\eta + \sigma+\kappa) +i s( \mu_3   |\kappa|  +    \mu_2    |\sigma| +   \mu   |\xi| +  \mu_1   |\eta| )  } \\
&\quad \times  \mathfrak{H}_{k_4,j_4;n_4}^{\mu_4,i_4}(s,X(s), V(s)) \cdot \mathcal{F}[{}^4 \mathfrak{H}](s, \xi, \eta, \sigma,\kappa,    X_{\bot}(s), V(s)) d\xi d \eta d \sigma d \kappa d s, \\
\end{split}
\ee
where   $\mathcal{F}[{}^4 \mathfrak{H}](\cdots)$ is given as follows,
 \be\label{oct12eqn81}
 \begin{split}
   \mathcal{F}[{}^4 \mathfrak{H}] (s, \xi, \eta, \sigma,\kappa,    X_{\bot}(s), V(s))  
 & : =\nabla_{\zeta}\big[ \big( \Phi_4(\xi, \eta, \sigma, \kappa,\zeta)\big)^{-1}\\
 &\quad \times   \mathcal{F}[\mathfrak{H}_{k_3,j_3;n_3}^{\mu_3,i_3}](s,\kappa, \zeta) \cdot  \mathcal{F}[{}^3 \mathfrak{H}] (s, \xi,\eta, \sigma,  X_{\bot}(s),\zeta)  \big]\big|_{\zeta= V(s)},\\
 \Phi_4(\xi, \eta, \sigma, \kappa,\zeta)&:=  \hat{\zeta} \cdot(\xi+ \eta +\sigma+\kappa) + \mu_3 |\kappa| +  \mu_2   |\sigma| +   \mu  |\xi| +   \mu_1  |\eta|.
 \end{split}
\ee

Moreover, the error terms $Err^a_{ i, i_1,i_2,i_3}(t_1, t_2), a\in \{0,1,2\},$ are given as follows, 
\be\label{oct12eqn2}
\begin{split}
  Err^0_{ i, i_1,i_2,i_3}(t_1, t_2) 
  &:= \sum_{b=0,1}   \int_{(\R^3)^4}    e^{ i X(t_b)\cdot (\xi+\eta + \sigma+\kappa) + i t_b( \mu_3   |\kappa|  +    \mu_2    |\sigma| +   \mu   |\xi| +  \mu_1   |\eta| )} \big(  \Phi_4(\xi, \eta, \sigma, \kappa,V(t_b))\big)^{-1}\\
&   \times   \mathcal{F}[\mathfrak{H}_{k_3,j_3;n_3}^{\mu_3,i_3}](t_b,\kappa, V(t_b))\cdot \mathcal{F}[{}^3 \mathfrak{H}] (t_b, \xi,\eta,\sigma,   X_{\bot}(t_b),V(t_b))   d \xi d\eta d \eta d\kappa \\
& +  \int_{t_1}^{t_2}  \int_{(\R^3)^4}    e^{i X(s )\cdot (\xi+\eta + \sigma+\kappa)+    i s( \mu_3   |\kappa|  +    \mu_2    |\sigma| +   \mu   |\xi| +  \mu_1   |\eta| ) }   \big(  \Phi_4(\xi, \eta, \sigma, \kappa,V(s))\big)^{-1} \\
 &  \times    \mathcal{F}[\mathfrak{H}_{k_3,j_3;n_3}^{\mu_3,i_3}](s,\kappa, V(t_i))\cdot\big( \hat{  V}_{\bot}(s)\cdot  \nabla_{  x_{\bot}}  \mathcal{F}[{}^3 \mathfrak{H}] (s, \xi,\eta,\sigma,    X_{\bot}(s ),V(s)) \big)  d \xi d\eta d \eta d\kappa d s, \\
&\\
 Err^1_{ i, i_1,i_2,i_3}(t_1, t_2) 
 &:= \int_{t_1}^{t_2}  \int_{(\R^3)^4} e^{i X(s )\cdot (\xi+\eta + \sigma+\kappa)  +    i s( \mu_3   |\kappa|  +    \mu_2    |\sigma| +   \mu   |\xi| +  \mu_1   |\eta| ) } \big(  \Phi_4(\xi, \eta, \sigma, \kappa,V(s))\big)^{-1} \\
 &\quad \times  \big[\p_s   \mathcal{F}[\mathfrak{H}_{k_3,j_3;n_3}^{\mu_3,i_3}](s ,\kappa, V(s))\cdot \mathcal{F}[{}^3 \mathfrak{H}] (s, \xi,\eta,\sigma,     X_{\bot}(s),V(s)) \\
 & \quad +     \mathcal{F}[\mathfrak{H}_{k_3,j_3;n_3}^{\mu_3,i_3}](s,\kappa, V(s))\cdot \p_s\mathcal{F}[{}^3 \mathfrak{H}]  (s, \xi,\eta,\sigma,    X_{\bot}(s),V(s )) \big]\big]     d \xi d\eta d \eta d\kappa d s, \\
 &\\
 Err^2_{ i, i_1,i_2,i_3}(t_1, t_2)   
&:=  \sum_{\begin{subarray}{c}
k_4, j_4\in \Z_+,i_4\in\{0,1,2,3,4\}\\
 a_4\in\{0,1,2,3\} , \mu_4\in\{+,-\}\\ 
 n_ \in [-M_t,2]\cap \Z  \\ 
\end{subarray}} \int_{t_1}^{t_2} \int_{(\R^3)^4}   e^{i X(s)\cdot (\xi+\eta + \sigma+\kappa) +i s( \mu_3   |\kappa|  +    \mu_2    |\sigma| +   \mu   |\xi| +  \mu_1   |\eta| )  } \\
  & \quad  \times  \big(   \mathfrak{E}^{\mu_3, a_3}_{k_3,j_3;n_3} (s, X(s), V(s))+ Ini_{k_3,j_3,n_3}^{\mu_3}(s, X(s), V(s)) \big) \\
  & \quad \cdot \mathcal{F}[{}^4 \mathfrak{H}](s, \xi, \eta, \sigma,\kappa,   X_{\bot}(s), V(s)) d\xi d \eta d \sigma d \kappa d s. \\
\end{split}
\ee

The  estimate of error type terms  $Err^a_{ i, i_1,i_2,i_3}(t_1, t_2) $ are deferred to  Lemma \ref{2021errhorstep3}, see  \eqref{oct23eqn65}. Hence, from \eqref{oct12eqn76},   to estimate $ \mathfrak{H}_{ i, i_1,i_2,i_3}(t_1, t_2)$, it suffices to estimate $\mathfrak{H}_{ i, i_1,i_2,i_3,i_4}(t_1, t_2)$.

\medskip
\noindent \textbf{Step 5.}\quad The fifth (last) iteration of smoothing.  

\medskip

As in the obtained estimate  \eqref{2021dec24eqn1}, we can rule out further the case $k_3\geq 50M_t$. It suffices to let $(k_3,n_3)\in {}^{3}\mathcal{E}_{k_2,n_2}  $ be fixed.

 Recall  \eqref{oct12eqn76} and  \eqref{oct12eqn81}. Similar to  the obtained estimate  \eqref{2021dec24eqn21},   from   the estimates in \eqref{2024oct8eqn1},  \eqref{2024oct8eqn5},  and \eqref{2024Dec6eqn31} in Theorem  \ref{mainresultsfirstpart}, for any   $
  s\in [t_1,t_2],$  $x,y,z,w\in \R^3$, s.t., $| x_{\bot}|, | y_{\bot}|, | z_{\bot}|,|  w_{\bot}|\in [2^{a_p  -5}, 2^{a_p   + 5}] $,  we have 
\be\label{2022feb22eqn22}
\begin{split}
 &\big|\int_{\R^3} \int_{\R^3} \int_{\R^3}\int_{\R^3}   e^{i  x\cdot\xi + i y\cdot\eta+i z\cdot\sigma +  i w \cdot\kappa  + i s(\mu_3   |\kappa|  +   \mu_2   |\sigma| +  \mu |\xi| +   \mu_1  |\eta|) }\\
 &\quad \times   \mathcal{F}[{}^4 \mathfrak{H}](s, \xi, \eta, \sigma,\kappa,    X_{\bot}(s), V(s)) d\xi d \eta d \sigma d \kappa \big|\\
&\lesssim \sum_{b\in \mathcal{T} } 2^{- ba_p } 2^{(b+7/3) (\gamma_1-\gamma_2)M_t  + \alpha^{\star}M_{t^\star}/12+  540\epsilon M_{t} +9\iota M_{t} -(k_2+2n_2)/2} \\
&\quad \times  2^{-(k_3+2n_3)/2} 2^{-\min\{n, n_1,n_2,n_3\}}\mathcal{M}(C)\\
&\lesssim  \sum_{b\in \mathcal{T} } 2^{- ba_p } 2^{(b+2) (\gamma_1-\gamma_2)M_t   +12\iota M_{t} + 1600\epsilon M_{t} -   \alpha^{\star}M_{t^\star}/12 }2^{-(k_3+2n_3)/2}  \mathcal{M}(C). 
\end{split}
\ee

 Let
 \be\label{forthiterindex}
 \begin{split}
{}^{4}\mathcal{E}_{k_3,n_3}:=\big\{( n_4, k_4):  k_4 \in \Z_+,& n_4\in [-M_t, 2]\cap \Z,  n_4 \geq (-\alpha^{\star}+\gamma_1-\gamma_2) M_t/2   -30\epsilon M_t, \\  & (k_4+2n_4)- (k_3+2n_3)\\
&\geq 2\alpha^{\star} M_t/3 -24\iota M_{t} -4000\epsilon M_t -2(\gamma_1-\gamma_2)M_t  \big\}.
 \end{split}
 \ee

 From the   estimate  \eqref{2022feb22eqn22}   and   estimates in \eqref{2024oct8eqn5}   and \eqref{2024Dec6eqn31}  in Theorem  \ref{mainresultsfirstpart},  we can rule out the case  $(k_4,n_4)\notin {}^{4}\mathcal{E}_{k_3,n_3}, $ as follows, 
\be
\begin{split}
&\big|\mathfrak{H}_{ i, i_1,i_2,i_3,i_4}(t_1, t_2)\big|\\
&\lesssim \big(  \sum_{b\in \mathcal{T} } 2^{- ba_p } 2^{(b+2) (\gamma_1-\gamma_2)M_t   +12\iota  M_{t} + 1600\epsilon M_{t} -   \alpha^{\star}M_{t^\star}/12 }2^{-(k_3+2n_3)/2}  \mathcal{M}(C)\big) \\
&\quad\times \big( \sum_{b\in \mathcal{T}} 2^{-b a_p  }  2^{b (\gamma_1-\gamma_2) M_t }  ( 2^{  \alpha^{\star} M_t  -(\gamma_1-\gamma_2)M_{t^{\star}} /4 }+ 2^{40\epsilon M_t}     2^{(k_4+2n_4)/2+3  \alpha^{\star} M_t/4 } )\big) \\
& \lesssim   \sum_{d\in \mathcal{T}+\mathcal{T} }   2^{- d a_p } 2^{d (\gamma_1-\gamma_2)M_t }   2^{ (\alpha^{\star} - 10\epsilon) M_t} \mathcal{M}(C).\\
\end{split}
\ee

Therefore, it suffices to consider the case $(k_4,n_4)\in {}^{4}\mathcal{E}_{k_3,n_3} $. For this case, thanks to the high oscillation in time of the phase,  we do integration by parts in ``$s$'' one more time. As a result,    we have   
 \be\label{2022feb22eqn31}
 \begin{split}
 \mathfrak{H}_{ i, i_1,i_2,i_3,i_4}(t_1, t_2)
&= \sum_{a=0,1 }    Err^a_{ i, i_1,i_2,i_3,i_4}(t_1, t_2) + {}_{}^5\mathfrak{H}(t_1, t_2), \\
   {}_{}^5\mathfrak{H}(t_1, t_2)&:= \int_{t_1}^{t_2} \int_{(\R^3)^5 }   e^{i X(s)\cdot (\xi+\eta + \sigma+\kappa+\chi  )+ i s(  \mu_4 |\chi| + \mu_3   |\kappa|  +    \mu_2    |\sigma| +   \mu   |\xi| +  \mu_1   |\eta| )  } \\
&\times     K(s, X(s), V(s))\cdot  \mathcal{F}[{}^5 \mathfrak{H}](s, \xi, \eta, \sigma,\kappa, \chi ,    X_{\bot}(s), V(s)) d\chi d\xi d \eta d \sigma d \kappa d s,\\
\end{split}
\ee
 where, 
 \be 
 \begin{split}
 \mathcal{F}[{}^5 \mathfrak{H}](s, \xi, \eta, \sigma,\kappa,\chi ,    X_{\bot}(s), V(s))&: =\nabla_{\zeta}\big[ \big( \Phi_5 (\xi, \eta, \sigma, \kappa,\chi, \zeta)\big)^{-1} \mathcal{F}[\mathfrak{H}_{k_4,j_4;n_4}^{\mu_4,i_4}](s,\chi, \zeta)\\
 &\quad    \cdot  \mathcal{F}[{}^4 \mathfrak{H}](s, \xi,\eta, \sigma, \kappa,    X_{\bot}(s),\zeta)  \big]\big|_{\zeta= V(s)},\\
 \Phi_5(\xi, \eta, \sigma, \kappa,\chi, \zeta)&:=  \hat{\zeta} \cdot(\xi+ \eta +\sigma+\kappa+\chi)\\ 
 &\quad +\mu_4 |\chi|+ \mu_3 |\kappa| +  \mu_2   |\sigma|  +   \mu_1  |\eta| +   \mu  |\xi|.\\
 \end{split}
\ee
The error terms are given as follows, 
 \be\label{2022feb22eqn64}
\begin{split}
 &Err^0_{ i, i_1,i_2,i_3,i_4}(t_1, t_2)\\& := \sum_{b=0,1}   \int_{(\R^3)^5}  e^{i X(t_b)\cdot (\xi+\eta + \sigma+\kappa+\chi) +  i t_b( \mu_4 |\chi| +  \mu_3   |\kappa|  +    \mu_2    |\sigma| +   \mu   |\xi| +  \mu_1   |\eta| )}\big(  \Phi_5(\xi, \eta, \sigma, \kappa,\chi, V(t_b))\big)^{-1}\\
&\quad \times     \mathcal{F}[\mathfrak{H}_{k_4,j_4;n_4}^{\mu_4,i_4}](t_b,\kappa, V(t_b))\cdot  \mathcal{F}[{}^4 \mathfrak{H}](t_b, \xi,\eta,\sigma,\kappa,     X_{\bot}(t_b),V(t_b)) d \chi  d \xi d\eta d \kappa \\
  &\quad + \int_{t_1}^{t_2}  \int_{(\R^3)^5}   e^{i X(s )\cdot (\xi+\eta + \sigma+\kappa+\chi)}   e^{  i s( \mu_4 |\chi| + \mu_3   |\kappa|  +    \mu_2    |\sigma| +   \mu   |\xi| +  \mu_1   |\eta| ) }  \big(  \Phi_5(\xi, \eta, \sigma, \kappa,\chi, V(s))\big)^{-1}\\
& \quad \times     \mathcal{F}[\mathfrak{H}_{k_4,j_4;n_4}^{\mu_4,i_4}](s,\kappa, V(t_i))\cdot\big(  {\hat{V}}_{\bot}(s)\cdot  \nabla_{  x_{\bot} } \mathcal{F}[{}^4 \mathfrak{H}] (s, \xi,\eta,\sigma, \kappa,     X_{\bot}(s ),V(s)) \big) d\chi  d \xi d\eta d \eta d\kappa d s, \\
&\\
 &Err^1_{ i, i_1,i_2,i_3,i_4}(t_1, t_2) \\ & := \int_{t_1}^{t_2}  \int_{(\R^3)^5} e^{i X(s )\cdot (\xi+\eta + \sigma+\kappa+\chi) +i s( \mu_4 |\chi| + \mu_3   |\kappa|  +    \mu_2    |\sigma| +   \mu   |\xi| +  \mu_1   |\eta| )  } \big( \Phi_5(\xi, \eta, \sigma, \kappa,\chi, V(s))  \big)^{-1} \\
 &\quad \times   \big[\p_s    \mathcal{F}[\mathfrak{H}_{k_4,j_4;n_4}^{\mu_4,i_4}](s ,\kappa, V(s))\cdot \mathcal{F}[{}^4 \mathfrak{H}] (s, \xi,\eta,\sigma,\kappa,     X_{\bot}(s),V(s)) \\
  &\quad  +     \mathcal{F}[\mathfrak{H}_{k_4,j_4;n_4}^{\mu_4,i_4}](s,\kappa, V(s))\cdot \p_s \mathcal{F}[{}^4 \mathfrak{H}] (s, \xi,\eta,\sigma, \kappa,    X_{\bot}(s),V(s )) \big]\big]    d\chi d \xi d\eta d \eta d\kappa d s.
  \end{split}
\ee

The  estimate of error type terms  $Err^a_{ i, i_1,i_2,i_3}(t_1, t_2) $ are deferred to  Lemma \ref{2022errhorstep5}, see \eqref{2022feb22eqn71}. We  focus on the estimate of $  {}_{}^5\mathfrak{H}(t_1, t_2)$. Similar to the obtained estimate  \eqref{2022feb22eqn22},  from  the estimates in \eqref{2024oct8eqn1},  \eqref{2024oct8eqn5}  and \eqref{2024Dec6eqn31} in Theorem  \ref{mainresultsfirstpart},    $
\forall  s\in [t_1,t_2],$ we have 
\be
\begin{split}
\sum_{(k_4,n_4)\in {}^{4}\mathcal{E}_{k_3,n_3}}&\big|\int_{\R^3} \int_{\R^3} \int_{\R^3}\int_{\R^3}    e^{i X(s )\cdot (\xi+\eta + \sigma+\kappa+\chi)   +  i s( \mu_4 |\chi| + \mu_3   |\kappa|  +    \mu_2    |\sigma| +   \mu   |\xi| +  \mu_1   |\eta| )}  \\
&\quad \times   \mathcal{F}[{}^5 \mathfrak{H}] K(s, \xi, \eta, \sigma,\kappa,\chi,     X_{\bot}(s), V(s)) d\xi d \eta d \sigma d \kappa \big|\\
&\lesssim    \sum_{(k_4,n_4)\in {}^{4}\mathcal{E}_{k_3,n_3}}  \sum_{b\in \mathcal{T} } 2^{- ba_p } 2^{(b+2) (\gamma_1-\gamma_2)M_t   +15\iota M_{t} + 1600\epsilon M_{t} -   \alpha^{\star}M_{t^\star}/12 }\\
&\quad \times 2^{-(k_3+2n_3)/2}      2^{-(k_4+2n_4)/2} 2^{-\min\{n, n_1,n_2,n_3,n_4\}}\mathcal{M}(C)  \\
&\lesssim  \sum_{b\in \mathcal{T} } 2^{- ba_p }  2^{(b+2) (\gamma_1-\gamma_2)M_t  -3\alpha^{\star}M_t/2 }\mathcal{M}(C).\\
\end{split}
\ee
From the above estimate and the estimate of the electromagnetic field in \eqref{2024oct28eqn61} in  Theorem  \ref{maintheorem1part1}, we have 
\be\label{oct12eqn98}
\begin{split}
 \sum_{(k_4,n_4)\in {}^{4}\mathcal{E}_{k_3,n_3}} \big| {}_{}^5\mathfrak{H}(t_1, t_2)\big|&\lesssim  \mathcal{M}(C)(t_2-t_1)\big[  \sum_{b\in \mathcal{T} } 2^{- ba_p }  2^{(b+2) (\gamma_1-\gamma_2)M_t  -3\alpha^{\star}M_t/2 } \big]\\
 &\quad \times \big( 2^{-a_p/4}  2^{5M_t/4+  \alpha^{\star}  M_t/4+\epsilon M_t} +2^{2( \alpha^{\star}  + 4\epsilon)M_t} 2^{-a_p/2} \big)\\
& \lesssim  \sum_{d\in \mathcal{T}+\mathcal{T}} 2^{ -d a_p } 2^{ d (\gamma_1-\gamma_2)M_t   +5\alpha^{\star} M_t/8+300\epsilon M_{t} } \mathcal{M}(C)(t_2-t_1). 
  \end{split}
\ee
Hence, recall  \eqref{2022feb22eqn31}, our desired estimate  \eqref{2021dec22eqn60} holds after combining the above estimate and the estimate   \eqref{2022feb22eqn71}  in Lemma \ref{2022errhorstep5}. 
\end{proof}

 \subsubsection{The estimate of ${}_1^{1}\mathfrak{H}^{\mu,i}_{k,j;n}(t_1, t_2) $ }\label{hyperbolicTypeIIwithsymm}

The estimate of ${}_1^{1}\mathfrak{H}^{\mu,i}_{k,j;n}(t_1, t_2) $ is summarized in the following Lemma. 

\begin{lemma}
  Under the assumption of Proposition \ref{bootstraplemma1}, the following estimate holds for  any $i\in\{0,1,2,3,4\},$ $( n,k)\in \mathcal{E}_i$ s.t.,  $a_{p}\geq  -k-n +3\epsilon M_{t}/2$, 
\be\label{oct23eqn57}
\big|{}_1^{1}\mathfrak{H}^{\mu,i}_{k,j;n}(t_1, t_2)\big|  \lesssim   \sum_{b\in \mathcal{T}+\mathcal{T}} \mathcal{M}(C)  2^{-ba_p  } 2^{ b(\gamma_1-\gamma_2)M_t+\alpha^{\star} M_t-20\epsilon M_t}(t_2-t_1).
\ee  
\end{lemma}
\begin{proof}
Recall  \eqref{nov12eqn61}. Note that, in terms of kernel, we have 
 \be\label{oct7eqn69}
 \begin{split}
  {}_1^{1}\mathfrak{H}^{\mu,i}_{k,j;n}(t_1, t_2) &= \sum_{l\in [n+\epsilon M_t/2,2]\cap \Z}  \int_{t_1}^{t_2} \int_{\R^3} \int_{\R^3}f(s,X(s)-y, v)  \\
&\quad \times     \nabla_v  \mathcal{K}^{\mu,i}_{k,j,n}(y, v,   X_{\bot}(s), V(s))  \cdot \big(E(s,  X(s)-y)\\
&\quad + \hat{v}\times B(s,  X(s) -y)\big) d y d v d s,  \\
\end{split}
\ee
where
\[
\begin{split}
\mathcal{K}^{\mu,i}_{k,j;n,l}(y, v,  X_{\bot}(s), V(s)) & : =\int_{\R^3} e^{i y\cdot \xi }|\xi|^{-1} \big({ i \hat{V}(s)\cdot \xi + i \mu |\xi| }\big)^{-1} \varphi_{l; n+\epsilon M_t/2}(\tilde{v}-\tilde{V}(s)) \\
&\quad \times   C(  X_{\bot}(s), V(s))\cdot \tilde{\varphi}_{k,j,n}^i (v, \xi, V(s))   d\xi.\\
\end{split}
\]

Recall the assumption of the coefficient $C(\cdot, \cdot)$ in  \eqref{oct23eqn51}. Moreover, recall   \eqref{oct7eqn90}  and  \eqref{sep4eqn6}, note that   $\mathbf{P}\big(\tilde{\varphi}_{k,j,n}^i (v, \xi, V(t))\big)$ is smaller than $ \big(\tilde{\varphi}_{k,j,n}^i (v, \xi, V(t))\big)_3, i\in\{3,4\}$. Based on the preceding observations, the following estimate for the kernels is obtained by integrating by parts with respect to $\xi$, along the $\zeta$ direction, and in directions orthogonal to $\zeta$,
\be
\begin{split}
\sum_{i=0,1,2} &|\nabla_v \mathcal{K}^{\mu,i}_{k,j; n,l}(y, v,  X_{\bot}(s), V(s))| + \sum_{a=3,4} 2^{-\max\{n, (\gamma_1-\gamma_2) M_{t^\star}\} -\epsilon M_t} |\nabla_v \mathcal{K}^{\mu,a }_{k,j ; n,l}(y, v,  X_{\bot}(s), V(s))|\\
 & \lesssim  \mathcal{M}(C) 2^{k} 2^{-j-l + 2\epsilon M_t}(1+2^{k}|y\cdot \tilde{\zeta}|)^{-100}(1+2^{k+n}|y\times \tilde{\zeta}|)^{-100}. \\
 \end{split}
\ee

Recall the   definition of the index sets $\mathcal{E}_i$ in  \eqref{indexsetsec4}.   From the above estimate of kernels, the cylindrical symmetry of solution,   the Cauchy-Schwarz inequality,   the volume of support of $v$, and the conservation law \eqref{conservationlaw}, we have
\be\label{2022feb23eqn41}
\begin{split}
\sum_{i=0,1,2}  \big| {}_1^{1}\mathfrak{H}^{\mu,i}_{k,j;n}(t_1, t_2) \big|&\lesssim  \sum_{l\in [n+\epsilon M_t/2,2]\cap \Z}  2^{k -j-l + 2\epsilon M_t}\big( 2^{-a_p} 2^{-k-n+\epsilon M_t} 2^{3j+2l} \big)^{1/2} \\
&\quad \times  \big( 2^{-a_p} 2^{-k-n+\epsilon M_t} 2^{-j} \big)^{1/2} \mathcal{M}(C) (t_2-t_1)\\
&\lesssim 2^{-a_p-n +5\epsilon M_t} \mathcal{M}(C)   (t_2-t_1)\\
& \lesssim  2^{-a_p +\alpha^\star M_{t }/2+40\epsilon M_t} \mathcal{M}(C)   (t_2-t_1)\\
& \lesssim 2^{-a_p + (\gamma_1-\gamma_2)M_t + (\alpha^\star -  20\epsilon )M_t} \mathcal{M}(C)   (t_2-t_1).\\
\end{split}
\ee
Similarly, we have 
\be\label{2022feb23eqn42}
\begin{split}
\sum_{a=3,4} \big| {}_1^{1}\mathfrak{H}^{\mu,a}_{k,j;n}(t_1, t_2) \big|&\lesssim  2^{-a_p-n + \max\{n, (\gamma_1-\gamma_2) M_{t^\star}\} +6\epsilon M_t} \mathcal{M}(C)   (t_2-t_1)\\
&\lesssim 2^{-a_p + (\gamma_1-\gamma_2)M_t + (\alpha^\star -  20\epsilon )M_t} \mathcal{M}(C)   (t_2-t_1).\\
\end{split}
\ee
Hence the desired estimate  \eqref{oct23eqn57}  holds from the obtained estimates  \eqref{2022feb23eqn41}  and  \eqref{2022feb23eqn42}. 
\end{proof}

The proof of the preceding lemma provides, by a similar argument, estimates for $\partial_s \mathcal{F}[\mathfrak{H}_{k,j;n}^{\mu,a}](s, \xi, \zeta)$, for $a \in \{0, 1, 2, 3, 4\}$. More precisely, we have

\begin{lemma}\label{secondhypohori}
For any $s\in  [0,t],$   $  \zeta \in \R^3/\{0\}$,  we have  
\be\label{oct7eqn21}
\begin{split}
 \sum_{a\in\{0,1,2,3\}} & \big| \int_{\R^3} e^{ix\cdot \xi + i \mu  s |\xi|}    \p_s \mathcal{F}[ \mathfrak{H}_{k,j;n}^{\mu,a}](s, \xi, \zeta) d \xi  \big| \\
 & \lesssim  2^{5\epsilon M_t}  \min\{ |  x_{\bot}|^{-1} 2^{k+n },  |  x_{\bot}|^{-1/2}  2^{n/2} 2^{j+  \alpha_s M_s}, 2^{(k+2n)/2} 2^{j+ \alpha_s M_s}\},\\
 &\\
  &\big| \int_{\R^3} e^{ix\cdot \xi + i \mu s|\xi|}  \p_s \mathcal{F}[ \mathfrak{H}_{k,j;n}^{\mu,4}](s, \xi, \zeta) d \xi  \big| \\
&\lesssim  2^{5\epsilon M_t}   \min\big\{ |  x_{\bot}|^{-1/2}  2^{n/2} 2^{j+   \alpha_s M_s },   |  x_{\bot}|^{-1}  2^{k+n },  2^{(k+2n)/2}  \min\{ 2^{2M_s +n},  2^{(1+\alpha_s   ) M_s}\}   \big \}\\
&\quad +      2^{  2 n+3\epsilon M_t}   \min\big\{  2^{k/2+n/3 +2j/3+5  \alpha_s M_s/3},|  x_{\bot}|^{-1/2}2^{k+n/2+   \alpha_s M_s }, \\
&\quad   | x_{\bot}|^{-1/4} 2^{k/2+n/4+j/2+3 \alpha_s M_s/2}\big\}, \\
\end{split}
\ee 
Moreover, for the   projection of $ \p_s \mathcal{F}[ \mathfrak{H}_{k,j;n}^{\mu,a}](s, \xi, \zeta),$ onto the $xy$-plane, we have an improved estimate as follows, 
\be\label{2024oct27eqn61}
\begin{split}
 \sum_{a=0,1,2,3,4} & \big| \int_{\R^3} e^{ix\cdot \xi + i \mu s|\xi|}  \mathbf{P}\big(     \p_s \mathcal{F}[ \mathfrak{H}_{k,j;n}^{\mu,a}] (s, \xi, \zeta)\big) d \xi  \big|\\
 & \lesssim 2^{(k+2n)/2+ 4\epsilon M_t} \big(2^{2\alpha_s M_s} + (2^n+  |  \zeta_{\bot}| |\zeta|^{-1}) 2^{j+\alpha_s M_s}\big)\\
&\quad +   2^{(k+2n)/2+ 3\epsilon M_t} \big(2^{7 \alpha_s M_s /3}  + | \zeta_{\bot}| |\zeta|^{-1} ( 2^{2j/3 + 5\alpha_s M_s/3} + 2^{ j+    \alpha_s M_s  }) \big). 
\end{split}
\ee 
 
\end{lemma}
\begin{proof}
 Recall  \eqref{2022feb22eqn81}  and  \eqref{sep5eqn10}.   As a result of direct computation,  in terms of kernel, $\forall a \in \{0,1,2,3\},$ we have
\be\label{2021dec25eqn1}
\begin{split}
& \int_{\R^3} e^{i x\cdot \xi + i s \mu |\xi|} \p_s \mathcal{F}[ \mathfrak{H}_{k,j;n}^{\mu,a}](s, \xi,\zeta)  d \xi \\
& =  \sum_{l\in [ n + \epsilon M_t/4, 2]\cap \Z}  \int_{(\R^3)^2}   (E(s, x-y)+ \hat{v}\times B(s, x-y))\cdot \nabla_v K^{a;l,n}_{k,j }(y, v, \zeta) f(s,x-y, v) d y d v,
\end{split}  
\ee
where
\be\label{2021dec25eqn2}
 K^{a;l,n}_{k,j }(y, v, \zeta): =\int_{\R^3} e^{i y\cdot \xi }|\xi|^{-1}  m_{j,n}^a  (v , \xi , \zeta) 
  \varphi_{n;-M_t}\big(\tilde{\xi} + \mu \tilde{\zeta} \big)   \varphi_k(\xi)   \varphi_{l ; n + \epsilon M_t/4} (\tilde{\zeta}- \tilde{v}) d\xi. 
\ee

Meanwhile, for the case $a=4$, the formula is slightly different since we didn't do integration by parts in $s$ at the initial stage, see \eqref{oct7eqn1}. More precisely,  we have 
\be\label{oct7eqn29} 
\int_{\R^3} e^{i x\cdot \xi + i s\mu |\xi|} \p_s \mathcal{F}[ \mathfrak{H}_{k,j;n}^{\mu,4}](s, \xi,\zeta)  d \xi = \mathfrak{E}^4_{k,j,n}(s,x,\zeta)+Err^4_{k,j,n}(s,x,\zeta), 
\ee
where
\be\label{oct2eqn110}
\begin{split}
  \mathfrak{E}^4_{k,j,n}(s,x,\zeta)&:= \int_{\R^3}\int_{\R^3} e^{ix\cdot \xi } \hat{f}(s, \xi, v)   |\xi|^{-1}    i4\pi\big( {(\hat{v}-\hat{\zeta} )\times (\hat{v}\times \xi)+\xi(1-|\hat{v}|^2)} \big)   \\
  &\quad \times \varphi_k(\xi) \varphi_{n;-M_t}\big(\tilde{\xi} + \mu \tilde{\zeta} \big)   \varphi_{j,n}^4(v, \zeta)  d \xi d v, \\
  Err^4_{k,j,n}(s,x,\zeta)&:=\int_{\R^3}\int_{\R^3} e^{i x\cdot \xi   }   |\xi|^{-1} \mathcal{F}\big((E+\hat{v}\times B)f\big)(s, \xi, v) \\
  &\quad  \cdot \nabla_v \big(
   \hat{v}   \varphi_{j,n}^4 (v, \zeta)\big)
 \varphi_{n;-M_t}\big(\tilde{\xi} + \mu \tilde{\zeta} \big)  \varphi_k(\xi)  d\xi d v .
  \end{split}
\ee

 With the above preparation, now we proceed in steps based on the size of $a$ as follows. 

 \medskip

 \noindent \textbf{Step 1.}\quad  If $a\in\{0,1,2,3\}$.  

 \medskip

Recall  \eqref{2021dec25eqn2}. After doing integration by parts in $\xi$ along $\zeta$ directions and directions perpendicular to $\zeta$, we have
\be\label{2022feb22eqn90}
\begin{split}
  &  |  K^{a ;l,n}_{k,j }(y, v, \zeta)| +   (\max\{2^l, |  \zeta_{\bot}|/|\zeta|\})^{-1}   |  {\mathbf{P}} \big(  K^{a ;l,n}_{k,j }(y, v, \zeta)\big)| \\
 &  \lesssim  2^{2k+2n} 2^{-j-l}(1+2^{k}|y\cdot \tilde{\zeta}|)^{-100}(1+2^{k+n}|y\times \tilde{\zeta}|)^{-100}.
 \end{split}
\ee

Note that, after using the above estimate of kernel, the Cauchy-Schwarz inequality,   the volume of support of $v$, and the estimate  \eqref{nov24eqn41}  if $|v|\geq 2^{(\alpha_s + \epsilon)M_s}$,   we have
\be\label{oct7eqn11}
\begin{split}
&\big| \int_{(\R^3)^2}   (E(s, x-y)+ \hat{v}\times B(s, x-y))\cdot \nabla_v K^{a;l,n}_{k,j }(y, v, \zeta) f(s,x-y, v) d y d v\big|\\
&\lesssim 
    2^{2k+2n} 2^{-j-l + 2\epsilon M_t}   \big(    \min\{2^{3j+2l}, 2^{j+2\alpha_s M_s}\}  \big)^{1/2} \\
    &\quad \times  \big(\min\{  2^{-j},  2^{-3k-2n } \min\{2^{3j+2l}, 2^{j+2\alpha_s M_s}\} \} \big)^{1/2}\\
    &  \lesssim 2^{2k+2n} 2^{-j-l + 2\epsilon M_t} \min\big\{ (2^{3j+2l})^{1/2} (2^{-j})^{1/2},   (2^{j+2\alpha_s M_s})^{1/2} (2^{-3k-2n+3j+2l})^{1/2},  \\ 
    &\quad   (2^{j+2\alpha_s M_s})^{1/2}  (2^{-3k-2n+ j+2\alpha_s M_s})^{1/2}   \big\}  \\ 
& \lesssim 2^{4\epsilon M_t}\min\big\{  2^{2k+2n },     2^{(k+2n)/2} \min\{ 2^{j+\alpha_s M_s}, 2^{2\alpha_s M_s-l}\}\big\}. 
\end{split}
\ee
Recall  \eqref{2021dec25eqn2}. From the above estimate, we have 
\[
\big|\int_{\R^3} e^{i x\cdot \xi + i s \mu |\xi|}  \p_s \mathcal{F}[ \mathfrak{H}_{k,j;n}^{\mu,a}]s, \xi,\zeta)  d \xi\big|\lesssim 2^{5\epsilon M_t}\min\big\{  2^{2k+2n },     2^{(k+2n)/2}   2^{j+\alpha_s M_s}\}.
\]
Similarly, from the improved estimate of kernels $ {\mathbf{P}} \big(  K^{a ;l,n}_{k,j }(y, v, \zeta)\big)$ in  \eqref{2022feb22eqn90}, we have
\be\label{oct26eqn43}
|\int_{\R^3} e^{i x\cdot \xi + i s\mu |\xi|}  \mathbf{P}\big(   \p_s \mathcal{F}[ \mathfrak{H}_{k,j;n}^{\mu,a}](s, \xi,\zeta)\big)    d \xi| \lesssim  2^{(k+2n)/2+ 5\epsilon M_t} \big(2^{2\alpha_s M_s} + |  \zeta_{\bot}| |\zeta|^{-1} 2^{j+\alpha_s M_s}\big).
\ee

Moreover, if  $| x_{\bot} |\geq 2^{-k-n+\epsilon M_t/2},$ from cylindrical symmetry of the solution,  the volume of support of $v$, and the estimate  \eqref{nov24eqn41}  if $|v|\geq 2^{(\alpha_s + \epsilon)M_s}$,  we have 
\be\label{oct7eqn13}
\begin{split}
&|\int_{\R^3} e^{i x\cdot \xi + i s \mu |\xi|} \p_s \mathcal{F}[\mathfrak{H}_{k ,j ;n }^{\mu ,a}](s, \xi,\zeta)  d \xi|\\
& \lesssim   \sum_{l\in [ n + \epsilon M_t/4, 2]\cap \Z}  2^{2k+2n} 2^{-j-l + 2\epsilon M_t} \big(  \frac{2^{-k-n+\epsilon M_t}}{| x_{\bot} |}   \min\{2^{3j+2l}, 2^{j+2\alpha_s M_s}\}  \big)^{1/2}\\
&\quad \times \big(\min\{  \frac{2^{-k-n+\epsilon M_t-j}}{|  x_{\bot}|}, 2^{-3k-2n } \min\{2^{3j+2l}, 2^{j+2\alpha_s M_s}\} \} \big)^{1/2} \\
&\lesssim \sum_{l\in [ n + \epsilon M_t/4, 2]\cap \Z}  2^{2k+2n} 2^{-j-l + 2\epsilon M_t} \min\big\{ ( |  x_{\bot}|^{-1} 2^{-k-n+\epsilon M_t +3j+2l} )^{1/2}  \\ 
&\quad  \times ( |  x_{\bot}|^{-1} 2^{-k-n+\epsilon M_t -j} )^{1/2}, ( |  x_{\bot}|^{-1} 2^{-k-n+\epsilon M_t +3j+2l} )^{1/2} (2^{-3k-2n+j+2\alpha_s M_s})^{1/2} \big\} \\ 
&\lesssim 2^{ 4\epsilon M_t}\min\{ |  x_{\bot}|^{-1} 2^{k+n },  |   x_{\bot}|^{-1/2}  2^{n/2} 2^{j+\alpha_s M_s}\}. 
\end{split}
\ee
After combining the above two estimates   \eqref{oct7eqn11}  and  \eqref{oct7eqn13}, we have
\be\label{oct23eqn31}
\begin{split}
&|\int_{\R^3} e^{i x\cdot \xi + i s\mu |\xi|} \p_s  \mathcal{F}[\mathfrak{H}_{k ,j ;n }^{\mu ,a}] (s, \xi,\zeta)  d \xi|\\
&\lesssim 2^{4\epsilon M_t}  \min\{ | x_{\bot}|^{-1} 2^{k+n },  | x_{\bot}|^{-1/2}  2^{n/2} 2^{j+\alpha_s M_s}, 2^{(k+2n)/2} 2^{j+\alpha_s M_s}\}.
\end{split}
\ee
Hence finishing the proof of the desired estimate \eqref{oct7eqn21}.

 \medskip

 \noindent \textbf{Step 2.}\quad  If $a=4$. 

 \medskip
  
Recall   \eqref{oct7eqn29}  and the definition of cutoff function  in $\varphi^4_{j, n}(v, \zeta),  $ see  \eqref{sep4eqn6}.  We have $|\tilde{v}-\tilde{\zeta}|\in  [2^{ n - \epsilon M_t }, 2^{ n + \epsilon M_t }].$ As a result, if   $  2^{ n}\notin [     2^{-2\epsilon M_t}| \zeta_{\bot}|/|\zeta|, 2^{ 2\epsilon M_t}| \zeta_{\bot}|/|\zeta|] $, then we have $|  v_{\bot}|/|v|\sim   |  \zeta_{\bot}|/|\zeta|  $. From the estimate  \eqref{nov24eqn41}, it suffices to consider the case $ \max\{2^n, |  \zeta_{\bot}|/|\zeta|\} 2^j\leq 2^{(\alpha_s + 2\epsilon)M_s}$ if    $  2^{ n}\notin [  2^{-2\epsilon M_t}| \zeta_{\bot}|/|\zeta|, 2^{ 2\epsilon M_t}|  \zeta_{\bot}|/|\zeta|]. $

With minor modifications in the estimate of the case $a\in\{0,1,2,3\}$,  we have 
\be\label{oct7eqn44}
\begin{split}
|Err^4_{k,j,n}(s,x,\zeta)|& \lesssim  2^{4 \epsilon M_t}  \min\big\{ |  x_{\bot}|^{-1} 2^{k+n },  |  x_{\bot}|^{-1/2}  2^{n/2} 2^{j+\alpha_s M_s},  \\
&\quad  2^{(k+2n)/2} \min\{ 2^{2M_t + \vartheta^\star_0 }, 2^{(1+\alpha_s) M_s}\}  \big\}, \\ 
|\mathbf{P}\big(Err^4_{k,j,n}(s,x,\zeta)\big)|&\lesssim  2^{(k+2n)/2+ 4\epsilon M_t} 2^{2\alpha_s M_s}  .
\end{split}
\ee
Recall \eqref{oct2eqn110}. Note that, in terms of kernel, the following estimate holds, 
\[
\begin{split}
|   \mathfrak{E}^4_{k,j,n}(s,x,\zeta)|&\lesssim \sum_{l\in [-M_t,  n + \epsilon M_t]\cap \Z}  |  \mathfrak{E}^{4;l}_{k,j,n}(s,x,\zeta)|, \\
    \mathfrak{E}^{4;l}_{k,j,n}(s,x,\zeta)&:=2^{3k+4n+3\epsilon M_t } \int_{\R^3} \int_{\R^3} f(s, x-y, v)\varphi_{l;-M_t}(\tilde{v}-\tilde{\zeta}) \varphi_{j,n}^4(v, \zeta) \\
&\quad \times (1+2^{k}|y\cdot \tilde{\zeta}|)^{-N_0^3 }  (1+2^{k+n}|y\times \tilde{\zeta}|)^{- N_0^3 }  d y d v. \\
\end{split}
\]

From the above estimate, the cylindrical symmetry of the distribution function, and the conservation law  \eqref{conservationlaw}, we have 
\be\label{oct7eqn26} 
\begin{split}
\big|\mathfrak{E}^{4;l}_{k,j,n}(s,x,\zeta)\big| & \lesssim  2^{  2n +5\epsilon M_t} \min\{2^{3k+2n-j}, |  x_{\bot}|^{-1}2^{2k+ n-j},2^{j+2\alpha_s M_s}\}\\
& \lesssim  2^{  2n +5\epsilon M_t} \min\big\{ \big(|  x_{\bot} |^{-1}2^{2k+ n-j}\big)^{1/2}\big(2^{j+2\alpha_s M_s}\big)^{1/2}, \big( 2^{3k+2n-j}\big)^{1/6} \\
&\quad\times  \big( 2^{j+2\alpha_s M_s}\big)^{5/6}, \big( |  x_{\bot}|^{-1}2^{2k+ n-j}\big)^{1/4} \big( 2^{j+2\alpha_s M_s}\big)^{3/4}\big\}\\
& \lesssim    2^{  2n +5\epsilon M_t} \min\big\{ | x_{\bot}|^{-1/2}2^{k+n/2+ \alpha_s M_s },  2^{k/2+n/3 +2j/3+5 \alpha_s M_s/3},\\
&\quad | x_{\bot}|^{-1/4} 2^{k/2+n/4+j/2+3\alpha_s M_s/2}\big\}.\\
\end{split}
\ee
 Recall  \eqref{oct7eqn29}. After combining the above estimate, and the obtained estimate    \eqref{oct7eqn44}, our desired estimate  \eqref{oct7eqn21} holds.   Moreover, the desired estimate  \eqref{2024oct27eqn61}  holds from the obtained estimates  \eqref{oct26eqn43},  \eqref{oct7eqn44}, and  \eqref{oct7eqn26}. 
\end{proof}

\subsection{Estimating the   elliptic parts}\label{ellipticPartIest}

As summarized in the following Proposition, the goal of this subsection is devoted to the estimate of elliptic part  $\mathfrak{E}^{\mu,i}_{k,j;n}(t_1, t_2)$   in  \eqref{oct25eqn41}.

\begin{proposition}\label{elliptfinalpro}
Let  $i\in\{0,1,2,3,4\}, ( n,k)\in \mathcal{E}_i$ and $a_{p}\geq  -k-n +3\epsilon M_{t}/2$,  under the assumption of Proposition \ref{bootstraplemma1}, the following estimate holds, 
\be\label{oct25eqn21}
  \big|   \mathfrak{E}^{\mu, i }_{k,j;n}(t_1, t_2)\big|  \lesssim    \mathcal{M}(C)\big[\sum_{b\in \mathcal{T}+\mathcal{T}}   2^{-b a_p }2^{b(\gamma_1-\gamma_2)M_t  +( \alpha^{\star}-10\epsilon)M_t}  (t_2-t_1)  +2^{2\alpha^{\star}M_t/3 } \big].
\ee 

\end{proposition}

\subsubsection{The first reduction and the outline of proof}

This subsection first analyzes the characteristic time oscillations of the elliptic parts $\mathfrak{E}^{\mu,i}_{k,j;n}(t_1, t_2)$, $i \in \{0,1,2,3,4\}$, and   then decompose the elliptic parts into different pieces to exploit  the smoothing effect. The   strategy of estimating these elliptic parts is then outlined.

 Define $g(t,x,v)=f(t,x+t\hat{v}, v)$ to be  the profile of the distribution function $f(t,x,v)$, which is the pull-back of the nonlinear solution along the linear flow.  Recall  \eqref{2022feb24eqn81}  and  \eqref{oct25eqn1}.   In terms of the profile, for any $i \in \{0,1,2,3,4\},$ we have 
 \[
 \begin{split}
   \mathfrak{E}^{\mu, i }_{k,j;n} (t_1, t_2) & =    \int_{t_1}^{t_2}   \int_{\R^3}\int_{\R^3} e^{i X(s)\cdot \xi - i s \hat{v}\cdot \xi } \hat{g}(s, \xi, v) \psi_k(\xi)   \varphi_{j,n}^i(v, V(s)) 
\\
&\quad \times C (  X_{\bot}(s), V(s))\cdot     m_i(\xi, v,  V(s)) \varphi_{n;-M_t} ( \tilde{\xi} + \mu \tilde{V}(s) )      d\xi d v d s, \end{split}   
\]
where $m_i(\xi, v, V(s)),  i \in \{0,1,2,3,4\}, $ are defined in   \eqref{2022feb22eqn81}  and   \eqref{oct25eqn1}. 
 
Recall  \eqref{sep4eqn6}. For $v\in supp(\varphi_{j,n}^a(\cdot, V(s))\psi_{l}(\tilde{v}-\tilde{V}(s)), a\in\{0,1,2\}$, we have $\angle(v,  V(s))\sim 2^l \geq 2^{n+\epsilon M_t/2}$. For   $v\in supp(\varphi_{j,n}^3(\cdot, V(s))\psi_{l}( \tilde{v}-\tilde{V}(s) )$, we have $\angle(v,   V(s))\sim 2^l \leq  2^{n-\epsilon M_t/2}$. As a result, we have 
\begin{enumerate}
\item[(i)] For any $   a\in\{0,1,2\},  v\in supp(\varphi_{j,n}^a(\cdot, V(s))\psi_{l}(\tilde{v}-\tilde{V}(s))$, we have
\be\label{2024oc27eqn81}
|\hat{V}(s)\cdot \xi - \hat{v}\cdot \xi | = |\hat{V}(s)\cdot \xi +\mu |\xi|-\mu |\xi| - \hat{v}\cdot \xi|\sim 2^{k+2l}. 
\ee
\item[(ii)] For any $v\in supp(\varphi_{j,n}^3(\cdot, V(s))\psi_{l}(\tilde{v}-\tilde{V}(s))$, we have
\be\label{2024oc27eqn82}
|\hat{V}(s)\cdot \xi - \hat{v}\cdot \xi | = |\hat{V}(s)\cdot \xi   +\mu |\xi|-\mu |\xi| - \hat{v}\cdot \xi|\sim 2^{k+2n}.
\ee
\end{enumerate} 
 
 Unfortunately, the case $a = 4$ presents a significant departure from the others. The function  $\hat{V}(t) \cdot \xi - \hat{v} \cdot \xi$ may vanish within the support of $\varphi_{j,n}^4(v, V(t))$ (see \eqref{sep4eqn6}), which occurs when $\angle(v, \mu\xi)$ and $\angle(V(t), \mu\xi)$ are comparable.

 Based on the size of $\angle(v, V(t))$ and the size of $(\hat{v}-\hat{V}(t))\cdot \xi$, after choosing the threshold $ \bar{\kappa}:=(\gamma_1-\gamma_2)M_t-30\epsilon M_t-n,$  we decompose $   \mathfrak{E}^{\mu, 4 }_{k,j;n}(t_1, t_2)$ into different pieces as follows, 
\be\label{oct3eqn111}
\begin{split}
   \mathfrak{E}^{\mu, 4 }_{k,j;n}(t_1, t_2)&= \sum_{\kappa\in[\bar{\kappa},2]\cap \Z}       \mathfrak{E}^{\mu, 4;\kappa}_{k,j,n}(t_1, t_2), \\
  \mathfrak{E}^{\mu, 4;\kappa }_{k,j,n}(t_1, t_2)&=  \int_{t_1}^{t_2}  \int_{\R^3}  \int_{\R^3} e^{i X(s)\cdot \xi - i s\hat{v}\cdot \xi   }    \hat{g}(s, \xi, v)   \varphi_{n;-M_t} ( \tilde{\xi} + \mu \tilde{V}(s) )   \varphi_{j,n}^4(v, V(s))  \\
  &\quad \times   \varphi_{  \kappa;\bar{\kappa}}( \frac{\xi \cdot(\hat{v}-\hat{V}(s))}{|\hat{v}-\hat{V}(s)| |\xi|})    C ( X_{\bot}(s), V(s))\cdot     m_4(\xi, v,  V(s))   \varphi_k(\xi)  
   d \xi d v  d s.\\
\end{split}
\ee
 
 Roughly speaking, if we are away from the vanishing set of frequencies,   i.e., for any $\kappa\in (\bar{\kappa}, 2]\cap\Z$, $\xi\in  supp\big(\varphi_{  \kappa;\bar{\kappa}}(  {\xi \cdot(\hat{v}-\hat{V}(s))}\big/{|\hat{v}-\hat{V}(s)| |\xi|}) \varphi_{j,n}^4(v, V(s)) \big)$,   we have 
\be\label{2024oct29eqn65}
|\hat{V}(s)\cdot \xi - \hat{v}\cdot \xi |\in [ 2^{k+\kappa +  n-2\epsilon M_t },  2^{k+\kappa + n+2\epsilon M_t }].
\ee
 
Thanks to the above estimate and  the estimates \eqref{2024oc27eqn81}, \eqref{2024oc27eqn82}, we do integration by parts in ``$s$'' once  to exploit  the smoothing effect.  As a result, we have the following decomposition for the elliptic parts $\mathfrak{E}^{\mu,a}_{k,j,n}(t_1, t_2), a\in\{0,1,2,3\}$, and $ \mathfrak{E}^{\mu, 4;\kappa, l}_{k,j,n}(t_1, t_2).$
 
\begin{enumerate}
\item[(i)] If $ a\in \{0,1,2,3\}, $ we have
\be\label{2024oct23eqn11}
\mathfrak{E}^{\mu,a}_{k,j,n}(t_1, t_2)=  Err^{\mu,a}_{k,j;n}(t_1, t_2)  + \sum_{b=0,1}  {}_{b}^1\mathfrak{E}^{\mu,a}_{k,j;n}(t_1, t_2),
\ee
\item[(ii)] If $a=4,$ and any $\kappa\in (\bar{\kappa}, 2]\cap\Z$, we have
\be\label{2024oct23eqn12}
\begin{split}
 \mathfrak{E}^{\mu, 4;\kappa }_{k,j,n}(t_1, t_2)=  Err^{\mu,4;\kappa }_{k,j,n}(t_1, t_2)+ \sum_{b=0,1}  {}_{b}^1 \mathfrak{E}^{\mu,4;\kappa }_{k,j;n}(t_1, t_2)
\end{split}
\ee
\end{enumerate}
where  the error type terms $ Err^{\mu,a}_{k,j;n}(t_1, t_2), a\in \{0,1,2,3\}$, and ${}_{}^{} Err^{\mu,4;\kappa,l}_{k,j,n}(t_1, t_2) $ are defined as follows, 
\be\label{oct7eqn62}
\begin{split}
 Err^{\mu,a}_{k,j;n}(t_1, t_2) & = \sum_{ b=1,2} (-1)^b   \int_{\R^3}  \int_{\R^3} e^{i X(t_b   )\cdot \xi    }       \hat{f}(t_b, \xi, v) \mathcal{F}[ {}_{}^{1}\mathfrak{E}^a ](   X_{\bot}(t_b),   \xi, v, V(t_b))     d \xi d v \\
 &+  \int_{t_1}^{t_2}   \int_{\R^3}\int_{\R^3} e^{i X(s)\cdot \xi   } \hat{f}(s, \xi, v)      {\hat{V}}_{\bot}(s)\cdot \nabla_{  x_{\bot}}\big( \mathcal{F}[ {}_{}^{1}\mathfrak{E}^a ](   x_{\bot} ,   \xi, v, V(s)) \big)\big|_{  x_{\bot}= {X}_{\bot}(s)}  d\xi d v ds,\\
&\\
{}_{}^{} Err^{\mu,4;\kappa }_{k,j,n}(t_1, t_2)  &   =   \sum_{ b=1,2} (-1)^b   \int_{\R^3}   \int_{\R^3} e^{i X(t_b   )\cdot \xi    }       \hat{f}(t_b, \xi, v)    \mathcal{F}[ {}_{}^{1}\mathfrak{E}^{\kappa} ](  X_{\bot}(t_b),   \xi, v, V(t_b)) d \xi d v \\
&   +   \int_{t_1}^{t_2} \int_{\R^3}   \int_{\R^3} e^{i X(s   )\cdot \xi    }       \hat{f}(s, \xi, v)   {\hat{{V}}}_{\bot}(s)\cdot \nabla_{  x_{\bot}}\big(  \mathcal{F}[ {}_{}^{1}\mathfrak{E}^{\kappa} ](  x_{\bot},   \xi, v, V(s )) \big)\big|_{ x_{\bot}=   X_{\bot}(s)} d \xi d v d s, \\
 \end{split}
\ee
where
\be\label{2024oct28eqn11}
\begin{split}
\mathcal{F}[ {}_{}^{1}\mathfrak{E}^{a } ](   x_{\bot},   \xi, v, \zeta)&:= C(  x_{\bot}, \zeta)\cdot m_a(\xi, v, \zeta)   (i(\hat{\zeta} -\hat{v})\cdot \xi)^{-1}  
 \varphi_k(\xi)    \varphi_{n;-M_t} ( \tilde{\xi} + \mu \tilde{\zeta}  )      \varphi_{j,n}^a(v, \zeta)  , \\
\mathcal{F}[ {}_{}^{1}\mathfrak{E}^{\kappa} ](   x_{\bot},   \xi, v, \zeta)&:=  C(  x_{\bot}, \zeta)\cdot m_4(\xi, v, \zeta) (i(\hat{\zeta} -\hat{v})\cdot \xi)^{-1}   \varphi_k(\xi)    \varphi_{n;-M_t} ( \tilde{\xi} + \mu \tilde{\zeta} ) \\ 
&\quad \times \varphi_{j,n}^4(v, \zeta)       \varphi_{  \kappa; \bar{\kappa} }( \frac{\xi \cdot(\hat{v}- \tilde{\zeta} )}{|\hat{v}-\hat{\zeta}| |\xi|})       .
\end{split}
\ee

The first iteration terms ${}_{b}^1\mathfrak{E}^{\mu,a}_{k,j;n}(t_1, t_2)$ and ${}_{b}^1 \mathfrak{E}^{\mu,i;\kappa, l}_{k,j;n}(t_1, t_2), b\in\{0,1\},$  are defined as follows, 
\be\label{oct29eqn82}
\begin{split}
 {}_{0}^1\mathfrak{E}^{\mu,a}_{k,j;n}(t_1, t_2)&:= \int_{t_1}^{t_2}    \int_{\R^3}  \int_{\R^3} e^{i X(s  )\cdot \xi    }   \mathcal{F}\big[(E+\hat{v}\times B)f \big](s,\xi, v)\\
 &\quad \cdot   \nabla_v \big[  \mathcal{F}[ {}_{}^{1}\mathfrak{E}^{a} ](   X_{\bot}(s),   \xi, v, \zeta) \big]  d \xi d v  d s, \\
  {}_{1}^1\mathfrak{E}^{\mu,a}_{k,j;n}(t_1, t_2)&:=  \int_{t_1}^{t_2}   \int_{\R^3}\int_{\R^3} e^{i X(s)\cdot \xi   }  \hat{f}(s, \xi, v)    \\
 &\quad  \times   K(s, X(s),V(s)) \cdot \nabla_{\zeta}\big( \mathcal{F}[ {}_{}^{1}\mathfrak{E}^{a} ](  X_{\bot}(s),   \xi, v, \zeta) \big)\big|_{\zeta=V(s)}  d\xi d v d s,\\
 \end{split}
\ee
\be\label{2024oct28eqn1}
\begin{split}
{}_{0}^1\mathfrak{E}^{\mu,4;\kappa  }_{k,j;n}(t_1, t_2) &=  \int_{t_1}^{t_2}   \int_{\R^3}\int_{\R^3} e^{i X(s)\cdot \xi   }         \mathcal{F}\big[(E+\hat{v}\times B)f \big](s,\xi, v) \\
&\quad \cdot   \nabla_v\big(  \mathcal{F}[ {}_{}^{1}\mathfrak{E}^{\kappa} ](  X_{\bot}(s),   \xi, v, \zeta) \big)\big|_{\zeta=V(s)}  d\xi d v d s,\\
  {}_{1}^1\mathfrak{E}^{\mu,4;\kappa  }_{k,j;n}(t_1, t_2) &=  \int_{t_1}^{t_2}   \int_{\R^3}\int_{\R^3} e^{i X(s)\cdot \xi   }  \hat{f}(s, \xi, v)      \\
 &\quad\times  K(s, X(s),V(s))\cdot \nabla_{\zeta}\big(  \mathcal{F}[ {}_{}^{1}\mathfrak{E}^{\kappa} ](  X_{\bot}(s),   \xi, v, \zeta) \big)\big|_{\zeta=V(s)}  d\xi d v d s,\\
\end{split}
\ee

To obtain  \eqref{2024oct23eqn11} and \eqref{2024oct23eqn12}, we used the fact that $\p_t g(t,x-t\hat{v}, v)= (\p_t + \hat{v}\cdot \nabla_x)f(t,x , v) = -\nabla_v\cdot\big((E+\hat{v}\times B)f\big)(t,x , v) $ and we did integration by parts in $v$.

The rest of this subsection is organized as follows. 
\begin{enumerate}
\item[$\bullet$] In section \ref{trivialcasespartIISS}, we first rule out some trivial cases. Moreover, we estimate the threshold case $\kappa=\bar{\kappa}$ for the case $i=4.$ As a result, it allows us to focus on certain range of parameters, for which we exploit the smoothing effect, i.e., using the decompositions in \eqref{2024oct23eqn11} and \eqref{2024oct23eqn12}. 

\item[$\bullet$] In section \ref{mainpartsISSpart1firs}, we estimate $ {}_{0}^1\mathfrak{E}^{\mu,a}_{k,j;n}(t_1, t_2)$ in \eqref{2024oct23eqn11} for the case $a\in\{0,1,2,3\}$ and estimate $ {}_{0}^1 \mathfrak{E}^{\mu,4;\kappa }_{k,j;n}(t_1, t_2)$  in \eqref{2024oct23eqn12} for the case $a=4.$

\item[$\bullet$] In section \ref{mainpartsISSpart1seco}, we estimate $ {}_{1}^1\mathfrak{E}^{\mu,a}_{k,j;n}(t_1, t_2)$ in \eqref{2024oct23eqn11} for the case $a\in\{0,1,2,3\}$ and estimate  $ {}_{1}^1 \mathfrak{E}^{\mu,4;\kappa }_{k,j;n}(t_1, t_2)$ in \eqref{2024oct23eqn12} for the case $a=4.$
\item[$\bullet$] In section \ref{mainpartsISSpart1error}, we estimate  the error type terms created in the ISS process of estimating the elliptic parts. 
\end{enumerate}

\subsubsection{Ruling out some trivial cases. }\label{trivialcasespartIISS}

In the following Lemma, we first rule out some range of parameters of $k,n, j$ etc. More precisely,

\begin{lemma}\label{trivialcaseell}
Let  $i\in\{0,1,2,3,4\}, ( n,k)\in \mathcal{E}_i$ and $a_{p}\geq  -k-n +3\epsilon M_{t}/2$,  under the assumption of Proposition \ref{bootstraplemma1},  the following estimate holds if either   $ a_p  \leq -k+\alpha^{\star } M_t/2$, or $n\leq (\gamma_1-\gamma_2)M_t/3-30\epsilon M_t$, or $j\leq  \alpha^{\star } M_t +(\gamma_1-\gamma_2)M_t/2-15\epsilon M_t$, or $k+2n \leq  2 \alpha^{\star } M_t+3(\gamma_1 -\gamma_2) M_t/2-30\epsilon M_t,$
\be\label{oct3eqn52}
 \big| \mathfrak{E}_{k,j;n}^{\mu, i}(t_1, t_2)\big|\lesssim  \sum_{b\in \mathcal{T}+\mathcal{T}} 2^{-b a_p} 2^{b(\gamma_1-\gamma_2)M_t +(\alpha^{\star}-10\epsilon)M_t}\mathcal{M}(C) (t_2-t_1). 
\ee
\end{lemma}
 
 \begin{proof}

 Recall  \eqref{oct25eqn1}. If $l\geq (\gamma_1-\gamma_2+\epsilon)M_t$, then we have $|  v_{\bot}|\sim 2^{j+l}$. From the estimate  \eqref{nov24eqn41}, it suffices to consider the case $j+l\leq (\alpha_t+2\epsilon)M_t$  if $l\geq (\gamma_1-\gamma_2+\epsilon)M_t$. 

 From  the assumption of the coefficient $C$  in  \eqref{oct23eqn51}, after representing $ \mathfrak{E}_{k,j;n}^{\mu, i} (t_1, t_2), i\in \{0,1,2,3,4\},$  in terms of kernel, and doing integration by parts in  $\xi$ along $V(s)$ direction and directions perpendicular to $V(s)$, we have 
\[
\begin{split}
\big| \mathfrak{E}_{k,j;n}^{\mu, i} (t_1, t_2)\big|&\lesssim  \mathcal{M}(C) \int_{t_1}^{t_2} \int_{\R^3} 2^{2k+2n+\epsilon M_t}(2^{\max\{l,n\}}+2^{(\gamma_1-\gamma_2)M_t}) f(s , X(s)-y , v )  \\
&\quad \times    (1+2^k|y\cdot \tilde{V}(s )|)^{-100}(1+2^{k+n}|y\times  \tilde{V}(s )|)^{-100}    \varphi_{[j-2,j+2]}(v)dy d v d s. \\
\end{split}
\]
From the cylindrical symmetry, the conservation law \eqref{conservationlaw},   the volume of support of $v$, and the estimate \eqref{nov24eqn41} if $|  v_{\bot} |\geq 2^{(\alpha_s + \epsilon)M_s}$,   we have 
\be\label{oct13eqn61}
\begin{split}
  &\big|  \mathfrak{E}_{k,j;n}^{\mu, i}  (t_1, t_2)\big| \\
  &\lesssim  \int_{t_1}^{t_2} \mathcal{M}(C) 2^{\max\{l,n, (\gamma_1-\gamma_2)M_t\}+3\epsilon M_t}\min\big\{\frac{2^{k+n+\epsilon M_t/2-j}}{|  X_{\bot}(t)|}, 2^{-k}\min\{2^{3j+2n},  2^{ j+2 \alpha_s M_s}\}  \big\} d s \\
  &\lesssim  2^{\max\{l,n, (\gamma_1-\gamma_2)M_t\}+3\epsilon M_t} \min\big\{ ( {2^{-a_p+k+n+\epsilon M_t/2-j}}  )^{1/2}  (2^{-k}\min\{2^{3j+2n},  2^{ j+2 \alpha_s M_s}\} )^{1/2},\\
  &\quad    ( {2^{-a_p+k+n+\epsilon M_t/2-j}}  )^{2/3}   (2^{-k}\min\{2^{3j+2n},  2^{ j+2 \alpha_s M_s}\} )^{1/3}\big\} \mathcal{M}(C)(t_2-t_1) \\
  &\lesssim  2^{\max\{l, n, (\gamma_1-\gamma_2)M_t\}+3\epsilon M_t}  \mathcal{M}(C)(t_2-t_1)  \\
  &\quad \times  \min\big\{2^{-a_p /2} \min\{2^{j+3n/2}, 2^{n/2+ \alpha^{\star} M_t}\}, 2^{-2a_p /3}2^{k/3+2n/3} 2^{\alpha^{\star} M_t/3+n/3} \big\}. 
  \end{split}
\ee

From the above estimate, the following estimate always holds if either $ a_p \leq -k+ \alpha^{\star} M_t/2$, or $n\leq  (\gamma_1-\gamma_2)M_t/3-30\epsilon M_t,$ or $j\leq  (\gamma_1-\gamma_2)M_t/2+(\alpha^{\star}-15\epsilon) M_t$,
\[
 \big|  \mathfrak{E}_{k,j;n}^{\mu, i} (t_1, t_2)\big|\lesssim   \sum_{b\in \mathcal{T}+\mathcal{T}} 2^{-b a_p} 2^{b(\gamma_1-\gamma_2)M_t +(\alpha^{\star}-11\epsilon)M_t}\mathcal{M}(C) (t_2-t_1).
\]

Moreover, from the obtained estimate \eqref{oct13eqn61}, the following estimate holds if $k+2n\leq 2(\alpha^{\star}-15\epsilon) M_t+3(\gamma_1-\gamma_2) M_t/2,$
  \be 
  \begin{split}
\big|  \mathfrak{E}_{k,j;n}^{\mu, i} (t_1, t_2)\big|&\lesssim   2^{-3a_p/4 + 3\epsilon M_t} 2^{(k+2n)/2}\mathcal{M}(C) (t_2-t_1) \\
& \lesssim  2^{-3 a_p/4} 2^{3(\gamma_1-\gamma_2)M_t/4 +(\alpha^{\star}-11\epsilon)M_t}\mathcal{M}(C) (t_2-t_1). 
\end{split}
\ee
Hence finishing the proof of the desired estimate \eqref{oct3eqn52}. 
\end{proof}

In view of the obtained estimate  \eqref{oct3eqn52}   in the above Lemma,  now,   it suffices to   focus on the case    $ a_p  \geq -k+\alpha^{\star } M_t/2$,  $n\geq (\gamma_1-\gamma_2)M_t/3-30\epsilon M_t$,   $j\geq  \alpha^{\star } M_t +(\gamma_1-\gamma_2)M_t/2-15\epsilon M_t$,   $k+2n \geq    2 \alpha^{\star } M_t+3(\gamma_1 -\gamma_2) M_t/2-30\epsilon M_t,$   for the estimate of elliptic parts $   \mathfrak{E}_{k,j;n}^{\mu, i}(t_1, t_2), i \in \{0,1,2, 3,4\}.  $

In the following lemma, we estimate the threshold case $\kappa=\bar{\kappa}$ for the case $i=4.$

\begin{lemma}\label{smallangest}
Let  $ \bar{\kappa}:=(\gamma_1-\gamma_2)M_t-100\epsilon M_t-n$. Assume that $ \forall s\in[t_1,t_2],|  X_{\bot}(s)|\in I_p, (n,k)\in \mathcal{E}_4$,   the following estimate holds if $ a_p   \geq   -k+ \alpha^{\star} M_t/2$,  $n\geq  (\gamma_1-\gamma_2)M_t/3-30\epsilon M_t$,   $j\geq  (\gamma_1-\gamma_2)M_t/2+(\alpha^{\star}-15\epsilon) M_t $, and $k+2n \geq  2(\alpha^{\star}-15\epsilon) M_t+3(\gamma_1-\gamma_2) M_t/2,$  
\be\label{oct3eqn115}
      \mathfrak{E}^{\mu, 4;\bar{\kappa} }_{k,j,n}(t_1, t_2)  \lesssim     2^{-a_p/2}2^{(\gamma_1-\gamma_2)M_t/2 +( \alpha^{\star}-10\epsilon)M_t}  \mathcal{M}(C)(t_2-t_1).
\ee
 \end{lemma}

\begin{proof}

Recall  \eqref{oct3eqn111}  and the definition of $\varphi_{j,n}^4(v, \zeta)$ in  \eqref{sep4eqn6}. For simplicity in notation, we  denote $\tilde{\theta}(v,s):={(\hat{v}-\hat{V}(s))}/{|\hat{v}-\hat{V}(s)|}$. Let    $\{ \tilde{\theta}(v,s), \theta_1(v, s) \} \subset span\{v, V(s)\}$ be an orthonormal frame of the linear space $ span\{v, V(s)\}$ and  $\theta_2(v, s) $ be an unit vector in the linear space $\big(span\{v, V(s)\}\big)^{\bot}$. Hence $\{  \tilde{\theta}(v,s) , \theta_1(v,s), $ $ \theta_2(v, s)\}$ is an orthonormal frame of $\R^3.$

 From  the assumption of the coefficient $C$  in  \eqref{oct23eqn51}, after writing $  \mathfrak{E}^{\mu, 4;\bar{\kappa} }_{k,j,n}(t_1, t_2)$  in terms of kernel, and doing integration by parts in  $\xi$  with respect to the orthonormal frame $\{(\hat{v}-\hat{V}(s))/|\hat{v}-\hat{V}(s), \theta_1(v,s),\theta_2(v,s)\}$, we have 
\[
\begin{split}
&\big|  \mathfrak{E}^{\mu,4;\bar{\kappa} }_{k,j,n}(t_1, t_2)\big| \\
& \lesssim \int_{t_1}^{t_2}   \int_{\R^3}\int_{\R^3} \mathcal{M}(C) 2^{2k+n+\bar{\kappa} +\max\{n, (\gamma_1-\gamma_2)M_t\}+\epsilon M_t}  f(s, X(s)-y , v )  \varphi_{[j-2,j+2]}(v)    \\
&\quad \times  \psi_{ [n-\epsilon M_t,  n +\epsilon M_t]} (\tilde{v}-\tilde{V}(s)  )   (1+2^{k+\kappa}(y\cdot (\hat{v}-\hat{V}(s)/|\hat{v}-\hat{V}(s)|) ))^{-100}\\
&\quad \times ( 1+ 2^{k+n}|y\cdot \theta_1(v,s)|)^{-100}  (1+2^k|y\cdot \theta_2(v, s)|)^{-100}  d y d v ds. 
\end{split}
\]

To fully exploit the cylindrical symmetry of the solution and quantify the resulting gains, we need to estimate the Jacobian of the coordinate transformation. Let
\[
\begin{split}
& \tilde{y}:=y\cdot \tilde{\theta} , \quad \tilde{y}_1:= y\cdot \theta_1(v, s), \quad \tilde{y}_2:=y\cdot \theta_{2}(v,s), \quad y = \tilde{y} \tilde{\theta} + \tilde{y}_1 \theta_1(v,s) + \tilde{y}_2\theta_2(v,s)\\
&\forall i \in \{1,2,3\},\quad y_i=y\cdot {e}_i, \quad r:=\sqrt{|X_1(s)-y_1|^2 + |X_2(s)-y_2|^2}.
\end{split}
\]
Note that, we have the following two options of changing coordinates $(\tilde{y}_1, \tilde{y})\rightarrow(r,y_3)$ or  $(\tilde{y}_2, \tilde{y})\rightarrow(r,y_3)$. 
Recall that $|  X_{\bot}(s)|\sim 2^{a_p }\geq 2^{-k-n- \bar{\kappa}  +\epsilon M_t}$.  For any $y\in \R^3$, s.t., $|y|\leq 2^{-k-\bar{\kappa}+\epsilon M_t/2}$, we have 
\[
\begin{split}
J_1&=det\big(\begin{bmatrix}
\p_{\tilde{y}_1} r & \p_{\tilde{y}} r\\
\p_{\tilde{y}_1} y_3 & \p_{\tilde{y}} y_3\\
\end{bmatrix}\big)= \p_{\tilde{y}_1} r\p_{\tilde{y}} y_3- \p_{\tilde{y}} r \p_{\tilde{y}_1} y_3\\ 
&= \frac{-1}{r}\big[\theta_1\cdot  (X_1(s)-y_1),(X_2(s)-y_2), 0  \big)\tilde{\theta}\cdot e_3 -  \tilde{\theta}\cdot  (X_1(s)-y_1),(X_2(s)-y_2), 0  \big){\theta}_1\cdot e_3  \big]\\
&= \frac{1}{r}\big( -(X_2(s)-y_2),(X_1(s)-y_1), 0  \big)\cdot \big(\theta_1\times \tilde{\theta} \big),\\
J_2&=det\big(\begin{bmatrix}
\p_{\tilde{y}_2} r & \p_{\tilde{y}} r\\
\p_{\tilde{y}_2} y_3 & \p_{\tilde{y}} y_3\\
\end{bmatrix}\big) = \frac{1}{r}\big( -(X_2(s)-y_2),(X_1(s)-y_1), 0  \big)\cdot \big(\theta_2\times \tilde{\theta} \big)\\
\end{split}
\]
From the above results of computation, we conclude that
\be\label{jacobianthreshold}
\begin{split}
 |J_1|+|J_2|&\gtrsim \big| \tilde{\theta}(v,s)\times  \frac{\big( -(X_2(s)-y_2), (X_1(s)-y_1), 0  \big)}{\sqrt{|X_1(s)-y_1|^2 + |X_2(s)-y_2|^2}}   \big|\\
&=\big|   \tilde{\theta}(v,s)\times \frac{\big(  (-X_2(s) ,X_1(s), 0  \big)}{|  X_{\bot}(s)|} \big|+\mathcal{O}(|  X_{\bot}(s)|^{-1} 2^{-k-\bar{\kappa}+\epsilon M_t/2}) .\\
\end{split}
\ee

Moreover, note that, for any fixed  $a, b\in \mathbb{S}^2,  $ $0< \epsilon \ll 1$, we have $|\{c: c\in\mathbb{S}^2,  |(a-c)\times b|\leq \epsilon\}|\lesssim \epsilon^{3/2}$.   This  claim follows straightforwardly if $a=b$. It suffices to consider the case $a\neq b.$   Let 
 $
  \mathbf{e}_1:=b,\quad \mathbf{e}_2:= ({a - (b\cdot a) b})/{|a-(b\cdot a) b |},  \quad \mathbf{e}_3= \mathbf{e}_1\times \mathbf{e}_2.$ Then $\{\mathbf{e}_1, \mathbf{e}_2, \mathbf{e}_3\}$ is a orthonormal frame. In terms of the orthonormal frame,  we have
 \[
 \begin{split}
  &a= \cos\phi_0 \mathbf{e}_1 + \sin\phi_0  \mathbf{e}_2, \quad c=\cos\theta\cos \phi \mathbf{e}_1+ \cos\theta \sin  \phi \mathbf{e}_2+\sin\theta \mathbf{e}_3,\\ 
 &(a-c)\times b  = -( \sin\phi_0 -\cos\theta \sin  \phi) \mathbf{e}_3 + \sin\theta \mathbf{e}_2,  \\
 &|(a-c)\times b|\leq \epsilon \quad \Longrightarrow |\sin \theta  |\leq \epsilon, \quad | \sin\phi_0 -\sin \phi|\lesssim \epsilon,\\
  \end{split}
\] 
From the above  estimate of $\theta, \phi$, we conclude that 
\be
\begin{split}
    &|\{c: c\in\mathbb{S}^2,  |(a-c)\times b|\leq \epsilon\}|\\
 &\lesssim |\{(\theta, \phi): \theta, \phi\in[0,2\pi], |\sin \theta  |\leq \epsilon, \, | \sin\phi_0 -\sin \phi|\lesssim \epsilon \}|\lesssim \epsilon^{3/2}.\\
 \end{split}
\ee 
From the above estimate and the fact that $|\hat{v}-\hat{V}(s)|\lesssim 2^{n+\epsilon M_t}$, we have 
 \be\label{oct3eqn100}
 \begin{split}
 & \big|\{v: v\in  \varphi_{j,n}^i(v, \zeta), | X_{\bot}(s)|^{-1} \big|  \tilde{\theta}(v,s)\times  {(-X_2(s), X_1(s),0)} \big| \lesssim  2^{a} \} \big|\\
 &\lesssim 2^{3j+3a/2+3n/2+2\epsilon M_t}. \\ 
\end{split}
\ee
Moreover, we have the following facts:
\begin{enumerate}
\item[(i)]  If $ n\geq (\gamma_1-\gamma_2+2\epsilon)M_t, $ we have
\be\label{2024feb13eqn1}
|  X_{\bot}(s)|^{-1} \big|  \tilde{\theta}(v,s)\times  {(-X_2(s), X_1(s),0)} \big| \gtrsim 2^{n }\geq  2^{-\alpha^{\star} M_t/2-\bar{\kappa}+10\epsilon M_t}.
\ee
\item[(ii)] If $ |  X_{\bot}(s)|^{-1} \big|  \tilde{\theta}(v,s)\times  {(-X_2(s), X_1(s),0)} \big|\lesssim  2^{(\gamma_1-\gamma_2-2\epsilon)M_t}$, we have
\be
 \frac{| v_{\bot}|}{|v|}\sim 2^{(\gamma_1-\gamma_2)M_t}.
\ee
\end{enumerate}

  With the above preparation, we first  do  inhomogeneous dyadic decomposition for the size of  $\tilde{\theta}(v,s)\times  \big({(-X_2(s), X_1(s),0)}/{|  X_{\bot}(s)|}\big) $ with the threshold $\bar{a}:= -\alpha^{\star} M_t/2-\bar{\kappa}+\epsilon M_t$. After using the cylindrical symmetry of the distribution function, the volume of support of $v$,   the estimate  \eqref{nov24eqn41}  if $|  v_{\bot}|\geq 2^{(\alpha_s+\epsilon)M_s}$, and  the fact that $|  v_{\bot}|/|v|\sim 2^{(\gamma_1-\gamma_2)M_t}$ if $a \leq (\gamma_1-\gamma_2)M_t-2\epsilon M_t $, we have
\be\label{oct3eqn131}
\begin{split}
&\big|  \mathfrak{E}^{\mu,4;\bar{\kappa} }_{k,j,n}(t_1, t_2)\big|\\
&\lesssim \sum_{a\in  [\bar{a},2]}  2^{2k+n+\bar{\kappa} +\max\{n, (\gamma_1-\gamma_2)M_t\}+3\epsilon M_t} \mathcal{M}(C) \Big[\int_{t_1}^{t_2}  \min\big\{ \frac{2^{-k-n-a-j }\mathbf{1}_{a> \bar{a}} + 2^{-k-\bar{\kappa}-j}\mathbf{1}_{a=\bar{a}}}{| X_{\bot}(s)|},  \\
&\quad  \frac{2^{-k-\bar{\kappa}-j }}{|  X_{\bot}(s)|},  2^{-3k-n-\bar{\kappa}}  \min\{2^{j+2 \alpha^{\star}  M_t},   2^{3j+3a/2+3n/2+2\epsilon M_t} \}  \big\}  d s\Big]\\
 &\lesssim \sum_{a\in [ \bar{a},2]}   2^{-a_p/2+\max\{n, (\gamma_1-\gamma_2)M_t\}+5\epsilon M_t}  \mathcal{M}(C)(t_2-t_1) \big[2^{\alpha^{\star}  M_t+\bar{\kappa}/2-a/2}\mathbf{1}_{a> (\gamma_1-\gamma_2-4\epsilon)M_t } \\
 &\quad +\mathbf{1}_{n\leq (\gamma_1-\gamma_2+2\epsilon)M_t}2^{\alpha^{\star}  M_t-(\gamma_1-\gamma_2)M_t}\big(   2^{3a/4+5n/4 }\mathbf{1}_{a=\bar{a}}   +   2^{\bar{\kappa}/2+a/4+3n/4} \mathbf{1}_{\bar{a}<a< (\gamma_1-\gamma_2-4\epsilon)M_t} \big) \big] \\
 &\lesssim   2^{-a_p/2}2^{(\gamma_1-\gamma_2)M_t/2 +( \alpha^{\star}-10\epsilon)M_t}  \mathcal{M}(C)(t_2-t_1).
\end{split}
\ee
Hence finishing the proof of the desired estimate  \eqref{oct3eqn115}. 
\end{proof}

\subsubsection{The estimate of $ {}_{0}^1\mathfrak{E}^{\mu,a}_{k,j;n}(t_1, t_2), a\in\{0,1,2,3\}$ and $  {}_{0}^1 \mathfrak{E}^{\mu,4;\kappa }_{k,j;n}(t_1, t_2)$ }\label{mainpartsISSpart1firs}

 In this section, we mainly estimate   $ {}_{0}^1\mathfrak{E}^{\mu,a}_{k,j;n}(t_1, t_2), a\in\{0,1,2,3\}$ in \eqref{2024oct23eqn11} and $  {}_{0}^1 \mathfrak{E}^{\mu,4;\kappa }_{k,j;n}(t_1, t_2)$ in \eqref{2024oct23eqn12}. More precisely, we have 

\begin{lemma}\label{ellerrtyp2}
Let  $ \bar{\kappa}:=(\gamma_1-\gamma_2)M_t-100\epsilon M_t-n$, $i\in\{0,1,2,3,4\}, ( n,k)\in \mathcal{E}_i$ and $a_{p}\geq  -k-n +3\epsilon M_{t}/2$,  under the assumption of Proposition \ref{bootstraplemma1},  the following estimate holds  if $ a_p  \geq   -k+ \alpha^{\star} M_t/2$,  $n\geq  (\gamma_1-\gamma_2)M_t/3-30\epsilon M_t$,   $j\geq  (\gamma_1-\gamma_2)M_t/2+(\alpha^{\star}-15\epsilon) M_t $, and $k+2n \geq  2(\alpha^{\star}-15\epsilon) M_t+3(\gamma_1-\gamma_2) M_t/2,$    
\be\label{oct7eqn65}
\begin{split}
  \sum_{a=0,1,2,3}&|  {}_{0}^1\mathfrak{E}^{\mu,a}_{k,j;n}(t_1, t_2) | +\sum_{\kappa\in (\bar{\kappa},2]\cap \Z } |  {}_{0}^1 \mathfrak{E}^{\mu,4;\kappa }_{k,j;n}(t_1, t_2) | \\
  & \lesssim  2^{-a_p+(\gamma_1-\gamma_2) M_t + (\alpha^{\star}-15\epsilon)M_t}\mathcal{M}(C)(t_2-t_1) .
  \end{split}
\ee
\end{lemma}
\begin{proof}

 Based on the size of $a$, we proceed in steps as follows. 

\medskip

\noindent \textbf{Step 1.} \quad The estimate of $  {}_{0}^1\mathfrak{E}^{\mu,a}_{k,j;n}(t_1, t_2), a\in \{0,1,2,3\}.$

\medskip

Recall  \eqref{oct29eqn82}. From the estimates \eqref{2024oc27eqn81} and \eqref{2024oc27eqn82},  after doing integration by parts in $\xi$ along $V(s)$ direction and directions perpendicular to $V(s)$ many times  for the kernel, we have
\[
\begin{split}
|{}_{0}^1\mathfrak{E}^{\mu,a}_{k,j;n}(t_1, t_2) |&\lesssim  \int_{t_1}^{t_2} \int_{\R^3} \int_{\R^3}    2^{k+2n-2\max\{l,n\}-j-\min\{l,n\} +2\epsilon M_t}\mathcal{M}(C)  \psi_{[l-2, l+2]} (\tilde{v}-\tilde{V}(s)) \\
&\quad \times   f(s, X(s)-y,v )  \big(|E(s, X(s )-y)|+ |B(s, X(s)-y)|\big)   \\
 &\quad  \times  (1+2^k|y\cdot \tilde{V}(s)|)^{-100}(1+2^{k+n}|y\times  \tilde{V}(s)|)^{-100} d y dv  ds.
 \end{split}
\]

Recall that $\forall s\in [t_1, t_2], |  X_{\bot}(s)|\geq 2^{-k-n+\epsilon M_t}$.  From the above estimate, the Cauchy-Schwarz inequality, the conservation law \eqref{conservationlaw}, and the cylindrical symmetry of the distribution function and the electromagnetic field, we have 
\be\label{2021dec27eqn25}
\begin{split}
| {}_{0}^1\mathfrak{E}^{\mu,a}_{k,j;n}(t_1, t_2)  |&\lesssim   2^{k+ n- l-j-\max\{l,n\}+4\epsilon M_t} \mathcal{M}(C)  (t_2-t_1) \\
&\quad  \times \big( 2^{-k-n+\epsilon M_t} 2^{-a_p} 2^{3j+2l} \big)^{1/2}  \big( 2^{-k-n+\epsilon M_t} 2^{-a_p} 2^{-j} \big)^{1/2}  \\
&\lesssim 2^{-a_p-(\gamma_1-\gamma_2)M_t/3+ 10\epsilon M_t} \mathcal{M}(C)(t_2-t_1).\\
\end{split} 
\ee

\medskip

\noindent \textbf{Step 2.} \quad The estimate of $   {}_{0}^1 \mathfrak{E}^{\mu,4;\kappa }_{k,j;n}(t_1, t_2), \kappa\in (\bar{\kappa},2]\cap \Z,  \bar{\kappa}= (\gamma_1-\gamma_2)M_t-n-100\epsilon M_t.$

\medskip

Recall  \eqref{2024oct28eqn1}. Due to the cutoff function $\varphi_{j,n}^4(v, \zeta)$ (see   \eqref{sep4eqn6}), for this case we have $|\tilde{v}-\tilde{\zeta}|\in [ 2^{n-3\epsilon M_t/2}, 2^{n+3\epsilon M_t/2}]$. 
  After doing 
integration by parts in $\xi$ with respect to the orthonormal frame  $\{(\hat{v}-\hat{V}(s))/|\hat{v}-\hat{V}(s), \theta_1(v, s),\theta_2(v, s)\}$ (defined in  Lemma \ref{smallangest}), many times  for the  kernel, 
 we have 
\[
\begin{split}
  \big| {}_{0}^1 \mathfrak{E}^{\mu,4;\kappa }_{k,j;n}(t_1, t_2)\big| &\lesssim   
 \int_{t_1}^{t_2} \int_{\R^3} \int_{\R^3}   2^{k-j-\min\{n, \kappa\}  +4\epsilon M_t }\mathcal{M}(C)    \varphi_{[j-2,j+2]}(v)\\
 &\quad \times f(s, X(s)-y,v ) \big(|E(s, X(s )-y)|+ |B(s, X(s)-y)|\big)     \\
 & \quad \times(1+2^{k+  {\kappa} }(y\cdot (\hat{v}-\hat{V}(s )/|\hat{v}-\hat{V}(s)|) )) ( 1+ 2^{k+n}|y\cdot \theta_1(v, s )|)^{-100}   \\
 &\quad \times (1+2^k|y\cdot \theta_2(v, s)|)^{-100}   \psi_{[n-2\epsilon M_t, n+2\epsilon M_t]} (\tilde{v}-\tilde{V}(s))  dy d v  d s. 
 \end{split}
\]

Note that $\forall s\in [t_1, t_2], |  X_{\bot}(s)|\geq 2^{ -k+\alpha^{\star } M_t/2}\geq 2^{-k-n-\kappa+\epsilon M_t}$.   From the above estimate,   the first fact in \eqref{2024feb13eqn1},   the Cauchy-Schwarz inequality, the conservation law  \eqref{conservationlaw}, and the cylindrical symmetry of the distribution function and the electromagnetic field, we have 
\be\label{oct3eqn132}
\begin{split}
&|  {}_{0}^1 \mathfrak{E}^{\mu,4;\kappa }_{k,j;n}(t_1, t_2) |\\& \lesssim      2^{k-j-\min\{n, \kappa\} +10\epsilon M_t }\mathcal{M}(C) (t_2-t_1)   \\
&\quad  \times    \big( (2^{-k-2n }\mathbf{1}_{n\geq (\gamma_1-\gamma_2+2\epsilon)M_t} + 2^{-k-\min\{n,\kappa\}} \mathbf{1}_{n\leq (\gamma_1-\gamma_2+2\epsilon)M_t}) 2^{-a_p-j} \big)^{1/2} \\
&\quad \times  \big( (2^{-k-2n }\mathbf{1}_{n\geq (\gamma_1-\gamma_2+2\epsilon)M_t} + 2^{-k-\min\{n,\kappa\}} \mathbf{1}_{n\leq (\gamma_1-\gamma_2+2\epsilon)M_t}) 2^{-a_p} 2^{3j+2n} \big)^{1/2} \\
   & \lesssim 2^{-a_p-(\gamma_1-\gamma_2)M_t +200\epsilon M_t}\mathcal{M}(C)(t_2-t_1)\\
   &\lesssim 2^{-a_p+(\gamma_1-\gamma_2) M_t + (\alpha^{\star}-20\epsilon)M_t}\mathcal{M}(C)(t_2-t_1) . \\
    \end{split}
\ee
In the above estimate, we used the fact that $(\gamma_1-\gamma_2)M_t\geq -50\epsilon M_t$ if $n\leq  (\gamma_1-\gamma_2+2\epsilon)M_t $ because  $n\geq  (\gamma_1-\gamma_2)M_t/3-30\epsilon M_t$. 

To sum up, our desired estimate  \eqref{oct7eqn65}  holds from the obtained estimate \eqref{2021dec27eqn25}  and \eqref{oct3eqn132}.
\end{proof}

\subsubsection{The estimate of $ {}_{1}^1\mathfrak{EH}^{\mu,a}_{k,j;n}(t_1, t_2), a\in\{0,1,2,3\}$ and $  {}_{1}^1 \mathfrak{E}^{\mu,4;\kappa}_{k,j;n}(t_1, t_2)$ }\label{mainpartsISSpart1seco}

This section is devoted to prove the following Lemma.

 \begin{lemma}\label{mainell3part1}
Let  $ \bar{\kappa}:=(\gamma_1-\gamma_2)M_t-100\epsilon M_t-n$, $i\in\{0,1,2,3,4\}, ( n,k)\in \mathcal{E}_i$ and $a_{p}\geq  -k-n +3\epsilon M_{t}/2$,  under the assumption of Proposition \ref{bootstraplemma1},  the following estimate holds if $ a_p  \geq   -k+ \alpha^{\star} M_t/2$,  $n\geq  (\gamma_1-\gamma_2)M_t/3-30\epsilon M_t$,   $j\geq  (\gamma_1-\gamma_2)M_t/2+(\alpha^{\star}-15\epsilon) M_t $, and $k+2n \geq  2(\alpha^{\star}-15\epsilon) M_t+3(\gamma_1-\gamma_2) M_t/2,$   
\be\label{oct14eqn61}
\begin{split}
 \sum_{a=0,1,2,3}&| {}_{1}^1\mathfrak{EH}^{\mu,a}_{k,j;n}(t_1, t_2)| + \sum_{\kappa\in (\bar{\kappa},2]\cap \Z }    | {}_{1}^1 \mathfrak{E}^{\mu,4;\kappa }_{k,j;n}(t_1, t_2) |\\
&\lesssim  \mathcal{M}(C)\big[   \sum_{b\in \mathcal{T}+\mathcal{T}} 2^{-b a_p+b(\gamma_1-\gamma_2)M_t + (\alpha^{\star}-10\epsilon)M_t}(t_2-t_1) + 2^{\alpha^{\star}M_t/2}\big].  
\end{split}
\ee
\end{lemma}
\begin{proof}
   
Recall  \eqref{oct29eqn82}  and  \eqref{2024oct28eqn1}. We first dyadically  decompose the new introduced electromagnetic field $K(s,X(s),V(s))$ and then decompose the localized  electromagnetic field into elliptic parts and the hyperbolic parts  by using    the decomposition of the localized  acceleration force in  \eqref{oct7eqn1}. As a result, from the estimate \eqref{dec2eqn31}, we have 
 \[
 \begin{split}
 \sum_{a=0,1,2,3}&| {}_{1}^1\mathfrak{EH}^{\mu,a}_{k,j;n}(t_1, t_2)| + \sum_{\kappa\in (\bar{\kappa},2]\cap \Z }   | {}_{1}^1 \mathfrak{E}^{\mu,4;\kappa }_{k,j;n}(t_1, t_2) |\\
  & \lesssim  \sum_{\begin{subarray}{c}
  k_1\in \Z_{+}, \mu_1\in \{+, -\}\\ 
   j_1\in[0, (1+2\epsilon)M_t]\cap \Z\\
    n_1\in [-M_t, 2]\cap \Z\\
    i, i_1\in\{0,1,2,3,4\}
  \end{subarray}}\big[   \mathfrak{HE}_{k_1,j_1;n_1}^{i,i_1, \mu_1}(t_1,t_2) 
 +\mathfrak{EE}_{k_1,j_1;n_1}^{i,i_1, \mu_1}(t_1, t_2) \big] + \mathcal{M}(C),
 \end{split} 
 \]
 where, $\forall a \in \{0,1,2,3\}$, 
 \be\label{2022feb25eqn61}
 \begin{split}
\mathfrak{HE}_{k_1,j_1;n_1}^{a,i_1, \mu_1}(t_1, t_2)&:=\int_{t_1}^{t_2}   \int_{\R^3}\int_{\R^3} e^{i X(s)\cdot \xi   }  \hat{f}(s, \xi, v)    \\
&\quad \times   \mathfrak{H}_{k_1,j_1;n_1}^{\mu_1,i_1}(s, X(s), V(s) ) \cdot \nabla_{\zeta}\big(  \mathcal{F}[ {}_{}^{1}\mathfrak{E}^{a } ](   X_{\bot}(s),   \xi, v, \zeta) \big)\big|_{\zeta=V(s)}  d\xi d v d s, \\ 
&\\
\mathfrak{HE}_{k_1,j_1;n_1}^{4,i_1, \mu_1}(t_1, t_2)&:= \sum_{\kappa\in (\bar{\kappa},2]\cap \Z    } \int_{t_1}^{t_2}   \int_{\R^3}\int_{\R^3} e^{i X(s)\cdot \xi   }  \hat{f}(s, \xi, v)     \\ 
&\quad  \times \mathfrak{H}_{k_1,j_1;n_1}^{\mu_1,i_1}(s, X(s), V(s) )  \cdot \nabla_{\zeta}\big(     \mathcal{F}[ {}_{}^{1}\mathfrak{E}^{\kappa} ](   X_{\bot}(s),   \xi, v, \zeta)\big)\big|_{\zeta=V(s)}
  d\xi d v d s,
\end{split}
\ee

\be\label{2026may1eqn1}
\begin{split}
\mathfrak{EE}_{k_1,j_1;n_1}^{a,i_1, \mu_1}(t_1, t_2)&:=  \int_{t_1}^{t_2}   \int_{\R^3}\int_{\R^3} e^{i X(s)\cdot \xi   } \hat{f}(s, \xi, v)   \mathbf{1}_{i_1\in\{0,1,2,3\} }  \\
&\quad \times   \mathfrak{E}^{\mu_1, i_1 }_{k_1,j_1;n_1} (s,x, V(s))  \cdot \nabla_{\zeta}\big(  \mathcal{F}[ {}_{}^{1}\mathfrak{E}^{ a} ](   X_{\bot}(s),   \xi, v, \zeta) \big)\big|_{\zeta=V(s)}    d\xi d v d s, \\ 
&\\
\mathfrak{EE}_{k_1,j_1;n_1}^{4,i_1, \mu_1}(t_1, t_2)&:=   \sum_{\kappa\in (\bar{\kappa},2]\cap \Z    }  \int_{t_1}^{t_2}   \int_{\R^3}\int_{\R^3} e^{i X(s)\cdot \xi   }  \hat{f}(s, \xi, v)   \mathbf{1}_{i_1\in\{0,1,2,3\} } \\ 
&\quad \times     \mathfrak{E}^{\mu_1, i_1;l_1}_{k_1,j_1;n_1} (s,x, V(s))  \cdot \nabla_{\zeta}\big(    \mathcal{F}[ {}_{}^{1}\mathfrak{E}^{\kappa} ](   X_{\bot}(s),   \xi, v, \zeta) \big)\big|_{\zeta=V(s)}  d\xi d v d s.
\end{split}
\ee

With the above preparation, now we proceed in steps as follows. 

\medskip 

\noindent \textbf{Step 1.}\quad The estimate of $\mathfrak{EE}_{k_1,j_1;n_1}^{i,i_1, \mu_1}(t_1, t_2) , i,i_1\in \{0,1,2,3,4\}.$

\medskip

Recall  \eqref{2024oct28eqn11}. By doing integration by parts in $\xi$ along $V(s)$ and directions perpendicular to $V(s)$ for the kernel and using the cylindrical symmetry, the volume of support of $v$, the estimate  \eqref{nov24eqn41}  if $|  v_{\bot}|\geq 2^{(\alpha_s +\epsilon)M_s}$, and the conservation law  \eqref{conservationlaw},   for any   $s\in [t_1, t_2], x\in \R^3$, s.t., $|  x_{\bot}|\in [2^{a_p  -5}, 2^{a_p  +5}]$, we have
 \be\label{oct13eqn52}
\begin{split}
& \big| \int_{\R^3} \int_{\R^3} e^{i x\cdot \xi} \hat{f}(s, \xi, v)  \nabla_\zeta\big(   \mathcal{F}[ {}_{}^{1}\mathfrak{E}^{ a} ](   X_{\bot}(s),   \xi, v, \zeta)\big)\big|_{\zeta=V(s)}  d\xi d v \big| \\
& \lesssim \mathcal{M}(C) 2^{  - \gamma_2 M_t +k - n }\int_{\R^3}  \int_{\R^3}  f(s, X(s)-y, v) \\
&\quad (1+2^k|y\cdot \tilde{V}(s)|)^{-100}(1+2^{k+n}|y\times  \tilde{V}(s)|)^{-100}  dy d v \\
&\lesssim \mathcal{M}(C) 2^{ - \gamma_2 M_t +3\epsilon M_t }  \min\big\{ 2^{-a_p  - 2n -j}, 2^{-2k-3n} \min\{2^{j+2\alpha^{\star} M_t}, 2^{3j}\}, 2^{k-n-j}\big\}\\
&\lesssim\mathcal{M}(C) 2^{ - \gamma_2 M_t +3\epsilon M_t }\min\big\{ (2^{-a_p  - 2n -j})^{1/2} (2^{-2k-3n+j+2\alpha^{\star} M_t})^{1/2},\\
&\quad  (2^{k-n-j})^{2/3} (2^{-2k-3n+2j+\alpha^{\star} M_t})^{1/3}   \big\}  \\
&\lesssim \mathcal{M}(C)2^{20\epsilon M_t} \min\{ 2^{-a_p/2 -k-5n/2+(\gamma_1-\gamma_2)M_t}, 2^{-5n/3-\gamma_2M_t+\alpha^{\star}M_t/3} \}   \\
&\lesssim \mathcal{M}(C) 2^{ 100\epsilon M_t}\min\{  2^{-a_p/2 +(\gamma_1-\gamma_2)M_t/2-k-n}, 2^{-2\alpha^{\star} M_t/3}\}  .  
\end{split}
\ee

 Similarly, after  doing integration by parts in $\xi$ with respect to the orthonormal frame  $\{(\hat{v}-\hat{V}(s))/|\hat{v}-\hat{V}(s), \theta_1(v,s), $ $\theta_2(v,s)\}$ (defined in  Lemma \ref{smallangest}) for the kernel and using the   first estimate in   \eqref{2024feb13eqn1}, the cylindrical symmetry, the volume of support of $v$, the estimate \eqref{nov24eqn41} if $|  v_{\bot}|\geq 2^{(\alpha_s +\epsilon)M_s}$, and the conservation law  \eqref{conservationlaw}, the following estimate holds for any $s\in [t_1, t_2], x\in \R^3$, s.t., $|  x_{\bot}|\in [2^{a_p -5}, 2^{a_p  +5}]$, 
\be\label{oct13eqn53} 
\begin{split}
&\big| \int_{\R^3} \int_{\R^3} e^{i x\cdot \xi} \hat{f}(s, \xi, v)  \nabla_\zeta\big(   \mathcal{F}[ {}_{}^{1}\mathfrak{E}^{\kappa } ](   X_{\bot}(s),   \xi, v, \zeta)\big)\big|_{\zeta=V(s)}   d\xi d v \big|\\
&\lesssim \mathcal{M}(C)  2^{ - \gamma_2 M_t +k   - \min\{n,  {\kappa}\}  }\int_{\R^3}  \int_{\R^3}   f(s, X(s)-y, v) ( 1+ 2^{k+n}|y\cdot \theta_1(v,s )|)^{-100}  \\
&\quad \times     (1+2^k|y\cdot \theta_2(v, s )|)^{-100}   (1+2^{k+    {\kappa} }(y\cdot (\hat{v}-\hat{V}(s )/|\hat{v}-\hat{V}(s )|) ))^{-100}    dy d v \\
& \lesssim  \mathcal{M}(C)   2^{ - \gamma_2 M_t +3\epsilon M_t   - \min\{n,   {\kappa}\}  }    \min\big\{  2^{-2k-n-   {\kappa} }\min\{2^{j+2\alpha^{\star}M_t}, 2^{3j+2n}\}, 2^{k-j}, 2^{-a_p -j}  \\
&  \quad   \times   (2^{-2n }\mathbf{1}_{n\geq (\gamma_1-\gamma_2+2\epsilon)M_t} + 2^{-\min\{n,\kappa\}} \mathbf{1}_{n\leq (\gamma_1-\gamma_2+2\epsilon)M_t}) \big\} \\
 &\lesssim  \mathcal{M}(C)  2^{200\epsilon M_t}\min\{ 2^{-a_p/2 +(\gamma_1-\gamma_2)M_t/2-(k+n+ \min\{n,   {\kappa}\}) }, 2^{-\alpha^{\star}M_t/2  } \}. 
 \end{split}
\ee

   From the  first estimate in \eqref{2022feb24eqn1} in Theorem \ref{maintheoremellipitic}  and the obtained estimates  \eqref{oct13eqn52}   and \eqref{oct13eqn53}, the following estimate holds for any $i,i_1\in\{0,1\cdots,4\},\mu_1\in \{+,-\}, $
\be
\begin{split}
 | \mathfrak{EE}_{k_1,j_1;n_1}^{i,i_1, \mu_1}(t_1, t_2) | & \lesssim  2^{-a_p/2 +(\gamma_1-\gamma_2)M_t/2-(k+n+ \min\{n,   {\kappa}\}) +200\epsilon M_t}\\
 &\quad \times 2^{-a_p/2 + (\alpha^{\star} +4\epsilon )M_t} \mathcal{M}(C)(t_2-t_1)\\
 &\lesssim  2^{-a_p  +(\gamma_1-\gamma_2)M_t+  \alpha^{\star} M_t/2}\mathcal{M}(C)(t_2-t_1).\\
 \end{split}
\ee

\medskip 

\noindent \textbf{Step 2.}\quad The estimate of $\mathfrak{HE}_{k_1,j_1;n_1}^{i,i_1, \mu_1}(t_1, t_2) , i,i_1\in \{0,1,2,3,4\}.$

\medskip

We first rule out some trivial cases. More precisely, if    $(k_1+2n_1)/2\leq    (k-n+\kappa)/2  + 5\alpha^{\star} M_t/8-300\epsilon M_t $ or $n_1< -\alpha^{\star} M_t/2+  (\gamma_1-\gamma_2)M_t/2-30\epsilon M_t$, from the obtained estimates  \eqref{oct13eqn52}  and  \eqref{oct13eqn53}  and  the estimates in \eqref{2024oct8eqn5}  and \eqref{2024Dec6eqn31} in Theorem \ref{mainresultsfirstpart},    for any $i,i_1\in\{0,1\cdots,4\},\mu_1\in \{+,-\}, $ we have 
 \[
 \begin{split}
 \big|\mathfrak{HE}_{k_1,j_1;n_1}^{i,i_1, \mu_1}(t_1, t_2) \big|& \lesssim   \mathcal{M}(C) (t_2-t_1)  2^{-a_p/2 +(\gamma_1-\gamma_2)M_t/2-(k+n+ \min\{n,   {\kappa}\}) +200\epsilon M_t}  \\
  &\quad  \times\big(\sum_{a\in \mathcal{T}} 2^{-a a_p + a(\gamma_1-\gamma_2)M_t}  (2^{(\alpha^{\star} -9\epsilon) M_t-(\gamma_1-\gamma_2)M_t/4}\\
  &\quad + 2^{(k_1+2n_1)/2+3\alpha^{\star}M_t/4 } \mathbf{1}_{n_1\geq    (-\alpha^{\star}+\gamma_1-\gamma_2)M_t/2-30\epsilon M_t}  ) \big)\\
  &\lesssim \sum_{b\in \mathcal{T}+\mathcal{T} } 2^{-b a_p + b(\gamma_1-\gamma_2)M_t+(\alpha^{\star}-20\epsilon)M_t} \mathcal{M}(C)(t_2-t_1).
  \end{split}
\]

Now, it remains to consider the case  $(k_1+2n_1)/2\geq    (k-n+\kappa)/2  + 5\alpha^{\star} M_t/8-300\epsilon M_t $ and $n_1\geq -\alpha^{\star} M_t/2+  (\gamma_1-\gamma_2)M_t/2-30\epsilon M_t$. Recall  \eqref{2022feb25eqn61}.  On the Fourier side, we have
\be 
\begin{split}
\mathfrak{HE}_{k_1,j_1;n_1}^{i,i_1, \mu_1}(t_1, t_2) &=  \int_{t_1}^{t_2}   \int_{(\R^3)^3}   e^{i X(s)\cdot (\xi+ \eta) +i s \mu_1 |\eta| - is\hat{v}\cdot \xi }  \\
  &\quad \times  \mathcal{F}[ {}_{}^{2} \mathfrak{E}_{i,i_1} ](s,  X_{\bot}(s), \xi, \eta, v, \zeta)\big|_{\zeta= V(s)}   d\xi d\eta  d v d s,\\
\end{split}
\ee
where 
\be
\begin{split}\label{2021dec27eqn31}
\mathcal{F}[ {}_{}^{2}\mathfrak{E}_{a,i_1} ](s,  x_{\bot}, \xi, \eta, v, \zeta) &:=\mathcal{F}[ \mathfrak{H}_{k_1,j_1;n_1}^{\mu_1,i_1}](s, \eta , \zeta)     \cdot \nabla_{\zeta}\big( \mathcal{F}[ {}_{}^{1}\mathfrak{E}^{a } ](    x_{\bot},   \xi, v, \zeta) \big)  \hat{g}(s, \xi, v) ,  a\in\{0,1,2,3\},\\ 
\mathcal{F}[ {}_{}^{2}\mathfrak{E}_{4,i_1} ](s,  x_{\bot}, \xi, \eta, v, \zeta) &:=\sum_{\kappa\in (\bar{\kappa},2]\cap \Z    }  \mathcal{F}[ \mathfrak{H}_{k_1,j_1;n_1}^{\mu_1,i_1}](s, \eta , \zeta)     \cdot \nabla_{\zeta}\big( \mathcal{F}[ {}_{}^{1}\mathfrak{E}^{\kappa} ](  x_{\bot},   \xi, v, \zeta) \big)        \hat{g}(s, \xi, v) , \\  
\end{split}
\ee

Note that, due to the fact that  $(k_1+2n_1)/2\geq    (k-n+\kappa)/2  + 5\alpha^{\star} M_t/8- 300\epsilon M_t $,  we have
\[
\big|\hat{V}(t)\cdot (\xi+ \eta) +   \mu_1 |\eta| -  \hat{v}\cdot \xi\big| \sim |\hat{V}(t)\cdot \eta+  \mu_1 |\eta| | \sim 2^{k_1+2n_1}.
\]
To exploit high oscillation of phase  in $s$, we do integration by parts in  ``$s$'' once. As a result, for any $i,i_1\in\{0,1,2,3,4\}, \mu_1\in \{+, -\},$   we have
\[
  | \mathfrak{HE}_{k_1,j_1;n_1}^{i,i_1, \mu_1}(t_1, t_2)| \leq  \sum_{a=0,1,2}  | {}_{a}^2 \mathfrak{E}_{k_1,j_1;n_1}^{i,i_1, \mu_1}(t_1, t_2)|, 
\]
where
\be\label{oct13eqn80}
\begin{split}
  {}_{0}^2 \mathfrak{E}_{k_1,j_1;n_1}^{i,i_1, \mu_1}(t_1, t_2) 
 &= \sum_{a=1,2} \big| \int_{\R^3} \int_{\R^3}\int_{\R^3} e^{i X(t_a)\cdot (\xi+ \eta) +i t_a \mu_1 |\eta| - it_a \hat{v}\cdot \xi }   ( \hat{V}(t_a)\cdot (\xi+ \eta) +   \mu_1 |\eta| -  \hat{v}\cdot \xi)^{-1}\\
&\quad \times    \mathcal{F}[ {}_{}^{2}\mathfrak{E}_{i,i_1} ](t_a,  X_{\bot}(t_a), \xi, \eta, v, \zeta)\big|_{\zeta= V(t_a)}  d\xi d\eta  d v\big| \\
&\quad  + \big| \int_{t_1}^{t_2}    \int_{\R^3} \int_{\R^3}\int_{\R^3}  e^{i X(s)\cdot (\xi+ \eta) +i s  \mu_1 |\eta| - is \hat{v}\cdot \xi }  ( \hat{V}(s)\cdot (\xi+ \eta) +   \mu_1 |\eta| -  \hat{v}\cdot \xi)^{-1}\\
&\quad \times {\hat{V}}_{\bot}(s)\cdot \nabla_{  x_{\bot} } \mathcal{F}[ {}_{}^{2}\mathfrak{E}_{i,i_1} ](s,  x_{\bot}, \xi, \eta, v, \zeta)\big|_{  x_{\bot}=  X_{\bot}(s), \zeta= V(s)}   d\xi d\eta  d v  ds\big|,\\
  {}_{1}^2 \mathfrak{E}_{k_1,j_1;n_1}^{i,i_1, \mu_1}(t_1, t_2) 
&=  \int_{t_1}^{t_2}   \int_{\R^3} \int_{\R^3}\int_{\R^3} e^{i X(s)\cdot (\xi+ \eta) +i s \mu_1 |\eta| - is  \hat{v}\cdot \xi } ( \hat{V}(s)\cdot (\xi+ \eta) +   \mu_1 |\eta| -  \hat{v}\cdot \xi)^{-1}\\
&\quad \times    \p_s \mathcal{F}[ {}_{}^{2}\mathfrak{E}_{i,i_1} ](s, x_{\bot}, \xi, \eta, v, \zeta)\big|_{  x_{\bot} =  X_{\bot}(s), \zeta=V(s)}      d\xi d\eta  d v ds, \\
  {}_{2}^2 \mathfrak{E}_{k_1,j_1;n_1}^{i,i_1, \mu_1}(t_1, t_2) 
&=  \int_{t_1}^{t_2}  \int_{\R^3} \int_{\R^3}\int_{\R^3} e^{i X(s)\cdot (\xi+ \eta) +i s \mu_1 |\eta| - i s \hat{v}\cdot \xi } K(s, X(s), V(s))\\
&\quad \cdot  \nabla_\zeta \big[ i ( \hat{\zeta} \cdot (\xi+ \eta) +   \mu_1 |\eta| -  \hat{v}\cdot \xi)^{-1} \mathcal{F}[ {}_{}^{2}\mathfrak{E}_{i,i_1} ](s,  X_{\bot}(s), \xi, \eta, v, \zeta)  \big]\big|_{\zeta= V(s)} d\xi d\eta  d v d s   .\\
\end{split}
\ee

To estimate the above terms, we proceed in three sub-steps as follows. 

\medskip

  \textbf{Step 2A.}\quad The estimate of $  {}_{0}^2 \mathfrak{E}_{k_1,j_1;n_1}^{i,i_1, \mu_1}(t_1, t_2)$.

\medskip

Recall  \eqref{oct13eqn80}, \eqref{2021dec27eqn31} and the assumption of the coefficient ``$C$'' in Proposition \ref{bootstraplemma1}. From  the obtained estimates \eqref{oct13eqn52}  and   \eqref{oct13eqn53},  and the estimates in \eqref{2024oct8eqn1}, \eqref{2024oct8eqn5},  and \eqref{2024Dec6eqn31} in Theorem \ref{mainresultsfirstpart}, after writing  $   {}_{0}^2 \mathfrak{E}_{k_1,j_1;n_1}^{i,i_1, \mu_1}(t_1, t_2)$ on the physical space, we have 
\be
\begin{split}
 \big|    {}_{0}^2 \mathfrak{E}_{k_1,j_1;n_1}^{i,i_1, \mu_1}(t_1, t_2) \big|&\lesssim \mathcal{M}(C)\big[2^{-a_p+(\gamma_1-\gamma_2)M_t}(t_2-t_1) + 1 \big]2^{-\alpha^{\star} M_t/2+100\epsilon M_t}\\
&\quad \times(2^{(1-\epsilon)M_t} + 2^{(k_1+2n_1)/2+(\alpha^{\star} +3\iota + 40\epsilon) M_t  }) 2^{-(k_1+2n_1)} \\
& \lesssim  \mathcal{M}(C)\big[2^{-a_p+(\gamma_1-\gamma_2)M_t}(t_2-t_1) + 1 \big]2^{\alpha^{\star}M_t/2}.
\end{split} 
\ee

\medskip

  \textbf{Step 2B.}\quad  The estimate of  $  {}_{1}^2 \mathfrak{E}_{k_1,j_1;n_1}^{i,i_1, \mu_1}(t_1, t_2)$.

\medskip

Recall  \eqref{oct13eqn80}  and  \eqref{2021dec27eqn31}. There are two cases in computing   $\p_s \mathcal{F}[ {}_{}^{2}\mathfrak{E}_{i,i_1} ](s,  x_{\bot}, \xi, \eta, v, \zeta)$. The first case happens  when $\p_s$ hits $\mathcal{F}[ \mathfrak{H}_{k_1,j_1;n_1}^{\mu_1,i_1}](s, \eta , \zeta)$. For this case, from the estimates in  \eqref{oct7eqn21}  and  \eqref{2024oct27eqn61}  in Lemma \ref{secondhypohori} and the obtained estimates   \eqref{oct13eqn52}  and  \eqref{oct13eqn53},  we have
\be\label{oct14eqn11}
\begin{split}
&\big| \int_{t_1}^{t_2}   \int_{(\R^3)^3}   e^{i X(s)\cdot (\xi+ \eta) +i s \mu_1 |\eta| - i s  \hat{v}\cdot \xi }  \big(\hat{V}(s)\cdot (\xi+ \eta) +   \mu_1 |\eta| -  \hat{v}\cdot \xi \big)^{-1}  \\
&\quad \times \p_s  \mathcal{F}[ \mathfrak{H}_{k_1,j_1;n_1}^{\mu_1,i_1}](s, \eta , V(s) )    \cdot \nabla_{\zeta}\big(  \mathcal{F}[ {}_{}^{1}\mathfrak{E}^{ a} ](    x_{\bot},   \xi, v, \zeta) \big) \hat{g}(s, \xi, v)   d\xi d\eta  d v ds \big|\\
&+\big| \int_{t_1}^{t_2}   \int_{(\R^3)^3}   e^{i X(s)\cdot (\xi+ \eta) +i s \mu_1 |\eta| - i s  \hat{v}\cdot \xi }\big( \hat{V}(s)\cdot (\xi+ \eta) +   \mu_1 |\eta| -  \hat{v}\cdot \xi\big)^{-1}  \\
&\quad \times    {  \p_s \mathcal{F}[ \mathfrak{H}_{k_1,j_1;n_1}^{\mu_1,i_1}](s, \eta , V(s) )      \cdot \nabla_{\zeta}\big(  \mathcal{F}[ {}_{}^{1}\mathfrak{E}^{\kappa} ](   x_{\bot},   \xi, v, \zeta) \big)  \hat{g}(s, \xi, v) }{ }     d\xi d\eta  d v ds \big|\\
&\lesssim   2^{ -a_p /2  +10\epsilon M_t } 2^{-a_p/2 +(\gamma_1-\gamma_2)M_t/2-(k+n+ \min\{n,   {\kappa}\}) }\mathcal{M}(C) (t_2-t_1) \\
&\quad \times \big( 2^{(1+\alpha^{\star})M_t/2-k_1/4-n_1} + 2^{ \alpha^{\star} M_t +n_1/2}  \big)\\
&\lesssim   2^{-a_p + (\gamma_1-\gamma_2)M_t + 2\alpha^{\star} M_t/3}\mathcal{M}(C) (t_2-t_1). \\
\end{split}
\ee

The other  case happens  when   $\p_s$ hits $\widehat{g}(s,\cdot,v)$.  After doing integration by parts in $v$ for the term $\p_s \hat{g}(s, \xi, v)$, we have 
\be\label{2024oct23eqn31}
\begin{split}
 & \big|\int_{\R^3} \int_{\R^3} e^{i x\cdot \xi} e^{-is\hat{v}\cdot \xi }  \nabla_{\zeta}\big( \mathcal{F}[ {}_{}^{1}\mathfrak{E}^{a} ](    x_{\bot},   \xi, v, \zeta) \big) \p_s \hat{g}(s, \xi, v) d v  d \xi  \big|\\
  &+  \big|\int_{\R^3} \int_{\R^3} e^{i x\cdot \xi} e^{-is\hat{v}\cdot \xi }   \nabla_{\zeta}\big( \mathcal{F}[ {}_{}^{1}\mathfrak{E}^{\kappa} ](   x_{\bot},   \xi, v, \zeta) \big) \p_s \hat{g}(s, \xi, v) d v  d \xi  \big| \\
  &\lesssim \big| \int_{\R^3} \int_{\R^3} e^{i x\cdot \xi} \mathcal{F}\big[(E+\hat{v}\times B)f](s, \xi, v)\cdot \nabla_v\big[ \psi_k(\xi)  \nabla_{\zeta}\big( \mathcal{F}[ {}_{}^{1}\mathfrak{E}^{a} ] (   x_{\bot},   \xi, v, \zeta) \big|_{\zeta=V(s)} \big] d \xi d v\big| \\
  &\quad +\big| \int_{\R^3} \int_{\R^3} e^{i x\cdot \xi} \mathcal{F}\big[(E+\hat{v}\times B)f](s, \xi, v)\cdot \nabla_v\big[ \psi_k(\xi)  \nabla_{\zeta}\big( \mathcal{F}[ {}_{}^{1}\mathfrak{E}^{\kappa} ](  x_{\bot},   \xi, v, \zeta) \big)  \big|_{\zeta=V(s)} \big] d \xi d v\big|
  \end{split}
 \ee
 From the above estimate, after   using the estimate of kernel, which is obtained by doing integration by parts in $\xi$ along $V(s)$ direction and directions perpendicular to $V(s)$,   for any $a\in \{0,1,2,3\},\kappa\in (\bar{\kappa},2]\cap \Z, $ $x\in \R^3$, s.t., $| x_{\bot}| \in [2^{a_p  -5}, 2^{a_p   +5}]$, we have 
\be\label{2021dec27eqn41} 
\begin{split}
\eqref{2024oct23eqn31}&\lesssim \sum_{U\in \{E, B\}} \int_{\R^3} \int_{\R^3} 2^{  - \gamma_2 M_t +k-j  +3\epsilon M_t} \mathcal{M}(C)\big|U(s, x-y)\big| f(s, x-y, v)\\
&\quad \times  \big[2^{-2n}(1+2^k|y\cdot \tilde{V}(s)|)^{-100}(1+2^{k+n}|y\times  \tilde{V}(s)|)^{-100}\\
&\quad  +2^{   -2\min\{n,\kappa\} }    ( 1+ 2^{k+n}|y\cdot \theta_1(v,s )|)^{-100}(1+2^k|y\cdot \theta_2(v, s )|)^{-100} \\
&\quad   \times   (1+2^{k+    {\kappa} }(y\cdot (\hat{v}-\hat{V}(s )/|\hat{v}-\hat{V}(s )|) ))^{-100} \big]   dy d v .\\
\end{split}
\ee   
After  using the localizing of kernel, the Cauchy-Schwarz inequality,  the cylindrical symmetry of the solution, from the first fact in  estimate   \eqref{2024feb13eqn1},    we have 
\be\label{2021dec27eqn53}
\begin{split}
 \eqref{2021dec27eqn41}  & \lesssim   \mathcal{M}(C) 2^{  - \gamma_2 M_t +k-j-2\min\{n,\kappa\}  +3\epsilon M_t}\\ 
&\quad \times \big(2^{-k +\epsilon M_t}(2^{-2n }\mathbf{1}_{n\geq (\gamma_1-\gamma_2+2\epsilon)M_t} + 2^{-\min\{n,\kappa\}} \mathbf{1}_{n\leq (\gamma_1-\gamma_2+2\epsilon)M_t})  2^{-a_p} 2^{3j+2n} \big)^{1/2} \\
&\quad \times \big(2^{-k+\epsilon M_t}(2^{-2n }\mathbf{1}_{n\geq (\gamma_1-\gamma_2+2\epsilon)M_t} + 2^{-\min\{n,\kappa\}} \mathbf{1}_{n\leq (\gamma_1-\gamma_2+2\epsilon)M_t}) 2^{-a_p} 2^{-j} \big)^{1/2} \\ 
& \lesssim 2^{-a_p -n-2\min\{n,\kappa\}+200\epsilon M_t-\gamma_2 M_t}\mathcal{M}(C) . 
\end{split}
\ee
From the above estimate  \eqref{2021dec27eqn53},    the estimate  \eqref{oct14eqn11},   and the estimates in \eqref{2024oct8eqn1}, \eqref{2024oct8eqn5},  and \eqref{2024Dec6eqn31} in Theorem \ref{mainresultsfirstpart}, we have 
\be\label{oct14eqn31}
\begin{split}
 \big|  {}_{1}^2 \mathfrak{E}_{k_1,j_1;n_1}^{i,i_1, \mu_1}(t_1, t_2) \big|&\lesssim\mathcal{M}(C) (t_2-t_1)\big[ 2^{-a_p + (\gamma_1-\gamma_2)M_t +2 \alpha^{\star} M_t/3}\\
 &\quad + 2^{-a_p -n-2\min\{n,\bar{\kappa} \}+200\epsilon M_t-\gamma_2 M_t}\\
&\quad \times (2^{(1-\epsilon)M_t}+2^{(k_1+2n_1)/2+\alpha^{\star}M_t +5\iota M_t} )2^{-(k_1+2n_1)} \big]\\
& \lesssim \mathcal{M}(C) (t_2-t_1)  2^{-a_p + (\gamma_1-\gamma_2)M_t + 2\alpha^{\star} M_t/3}. \\
\end{split}
\ee
 
\medskip

  \textbf{Step 2C.}\quad The estimate of ${}_{2}^2 \mathfrak{E}_{k_1,j_1;n_1}^{i,i_1, \mu_1}(t_1, t_2)$.

\medskip

Recall \eqref{oct13eqn80}.  With minor modifications in obtaining the estimates  \eqref{oct13eqn52}  and  \eqref{oct13eqn53}, from   the estimates in \eqref{2024oct8eqn1}, \eqref{2024oct8eqn5}, and \eqref{2024Dec6eqn31} in Theorem \ref{mainresultsfirstpart},    for any $x\in \R^3$, s.t., $| x_{\bot}|\in[2^{a_p-5}, 2^{a_p+5}]$, we have 
\be
\begin{split}
 &\big|   \int_{ \R^3}  \int_{ \R^3}   \int_{ \R^3}  e^{i x\cdot (\xi+ \eta) +i s \mu_1 |\eta| - i s \hat{v}\cdot \xi } \nabla_\zeta  \big[    \mathcal{F}[ {}_{}^{2}\mathfrak{E}_{i,i_1} ](s, X_{\bot}(s), \xi, \eta, v, \zeta)  \\
 &\quad \times  ( \hat{\zeta} \cdot (\xi+ \eta) +   \mu_1 |\eta| -  \hat{v}\cdot \xi)^{-1} \big]\big|_{\zeta= V(s)}d\xi d\eta  d v \big|\\
 &\lesssim   2^{-a_p/2 +(\gamma_1-\gamma_2)M_t/2-(k+n+ \min\{n,   \bar{\kappa} \}) +50\epsilon M_t }  \mathcal{M}(C)\\
 &\quad \times2^{-(k_1+2n_1)-\gamma_2 M_t-\min\{n,\kappa, n_1\}} (2^{(1-\epsilon)M_t}  +2^{(k_1+2n_1)/2+(\alpha^{\star}+3\iota+40\epsilon) M_t})\\
&\lesssim  2^{-a_p/2 -3(\gamma_1-\gamma_2)M_t/2-5\alpha^{\star}M_t/2+3\iota M_t} \mathcal{M}(C). \\
\end{split}
\ee
From the above estimate and the estimate \eqref{2024oct28eqn61}  in Theorem \ref{maintheorem1part1},   for any $i, i_1\in\{0,1,2\},\mu_1\in \{+, -\},$ we have 
 \be\label{2021dec27eqn71}
 \begin{split}\big|  {}_{2}^2 \mathfrak{E}_{k_1,j_1;n_1}^{i,i_1, \mu_1}(t_1, t_2)\big|&\lesssim  2^{-a_p/2 -3(\gamma_1-\gamma_2)M_t/2-5\alpha^{\star}M_t/2 +3\iota M_t }  \mathcal{M}(C)(t_2-t_1)\\
 &\quad \times  \big(2^{-a_p 
 /2 + 2\alpha^{\star} M_t}  + 2^{-a_p 
 /4 + 5M_t/4+\alpha^{\star} M_t/4}  \big)\\
 &\lesssim  \sum_{b\in \mathcal{T}+\mathcal{T}} 2^{-b a_p+b(\gamma_1-\gamma_2)M_t + 7\alpha^{\star}M_t/8+3\iota M_t}\mathcal{M}(C)(t_2-t_1).\\
 \end{split}
\ee
Hence finishing the proof of our desired estimate  \eqref{oct14eqn61}.
\end{proof}

 \subsection{Estimating   error type terms}\label{errortypesPartIest}

This section addresses the estimation of all error terms arising from the iterative smoothing scheme presented in the preceding two sections. This estimation is naturally divided into two subsections, one for the hyperbolic part and one for the elliptic part.  
 \subsubsection{Error terms in estimating   the hyperbolic part}

Firstly, we estimate the error type terms in  \eqref{oct25eqn41}. 
\begin{lemma}\label{erroesttyp1}
 Under the assumption of Proposition \ref{bootstraplemma1}, the following estimate holds for  $i\in\{0,1,2,3,4\}$, $( n,k)\in \mathcal{E}_i$ s.t., $a_{p}\geq  -k-n +3\epsilon M_{t}/2$,
\be\label{2021dec22eqn11}
\begin{split}
&\sum_{a\in\{0,1\}}   \big|{ }_a^{1}Err^{\mu,i}_{k,j;n}(t_1, t_2)\big|  \\
&\lesssim  \mathcal{M}(C) \big[  \sum_{b\in \mathcal{T}+\mathcal{T}}  2^{-ba_p } 2^{ b(\gamma_1-\gamma_2)M_t+(\alpha^{\star}-20\epsilon)  M_t} (t_2-t_1) +   2^{3\alpha^{\star} M_t/4+100\epsilon M_t} \big]. 
\end{split}
\ee
\end{lemma}
\begin{proof}

Recall   \eqref{oct7eqn31}  and the definition of the index sets $\mathcal{E}_i$ in  \eqref{indexsetsec4}.   From the estimate \eqref{2024oct8eqn1} in Theorem \ref{mainresultsfirstpart}, $\forall i \in \{0,1,2,3,4\},$ we have 
\be\label{oct2eqn81}
\begin{split}
\big| { }_0^{1}Err^{\mu,i}_{k,j;n}(t_1, t_2)\big|&\lesssim  2^{-(k+2n) }\big( 2^{(1-\epsilon)M_t} + 2^{(k+2n)/2+\alpha^{\star} M_t + 30\epsilon M_t}\big)\mathcal{M}(C)\\
& \lesssim 2^{3\alpha^{\star} M_t/4+ 100\epsilon M_t} \mathcal{M}(C).
\end{split}
\ee
Recall    the assumption of the coefficient $C(\cdot, \cdot)$ in  \eqref{oct23eqn51}.  From the estimate \eqref{2024oct8eqn1} in Theorem \ref{mainresultsfirstpart}, we have
 \be\label{oct3eqn2}
 \begin{split}
 \big| { }_1^{1}Err^i_{k,j;n}(t_1, t_2)\big| &\lesssim  2^{-(k+2n)+(\gamma_1- \gamma_2)M_t}  2^{-a_p}  \mathcal{M}(C) (t_2-t_1)\\
 &\quad \times \big[2^{(1-\epsilon)M_t} +  2^{ 40\epsilon M_t}\min\{2^{(k+2n)/2+ \alpha^{\star}  M_t } \big]\\
 &\lesssim  2^{-a_p + (\gamma_1-\gamma_2)M_t+ 3\alpha^{\star} M_t/4+200\epsilon M_t}\mathcal{M}(C) (t_2-t_1). 
 \end{split}
 \ee
Hence the desired estimate  \eqref{2021dec22eqn11}  holds from the obtained estimate  \eqref{oct2eqn81}  and  \eqref{oct3eqn2}.

\end{proof}

Now, we estimate the error terms produced in the second iteration process. For the error terms in  \eqref{oct29eqn55}, we have

\begin{lemma}\label{2021errhorstep1}
 Under the assumption of Proposition \ref{bootstraplemma1}, the following estimate holds for any  $i,i_1\in\{0,1,2,3,4\}, ( n,k)\in \mathcal{E}_i$, $(n_1, k_1)\in {}_{}^1\mathcal{E}_{k,n}$,  and  $a_{p}\geq  -k-n +3\epsilon M_{t}/2$,   
\be\label{2021dec22eqn32}
\begin{split}
\sum_{a\in\{0,1,2 \}} \big| Err^a_{ i, i_1}(t_1, t_2)\big|  &\lesssim  \big[  \sum_{b\in \mathcal{T}+\mathcal{T}}  2^{-ba_p } 2^{ b(\gamma_1-\gamma_2)M_t} 2^{ (\alpha^{\star}-20\epsilon) M_t} (t_2-t_1) + 2^{3\alpha^{\star}M_t/4}\big] \mathcal{M}(C) .
\end{split}
\ee
\end{lemma}
\begin{proof}

Recall the definition of $\mathcal{E}_i,i\in\{0,1,2,3,4\}$ in \eqref{indexsetsec4} and the definition of ${}_{}^1 \mathcal{E}_{k,n}$ in \eqref{firstiterindex}. Based on the size of $a\in \{0,1,2\}$, we proceed in three steps as follows. 

\medskip

\noindent \textbf{Step 1.}\quad  The estimate of  $Err^0_{ i, i_1}(t_1, t_2)$.

\medskip

Recall  \eqref{oct8eqn1},  \eqref{oct7eqn90}, and the assumption  of the coefficient  ``$C$'' in Proposition \ref{bootstraplemma1}.    From       the estimates in  \eqref{2024oct8eqn1} in Theorem  \ref{mainresultsfirstpart},  we have 
\be\label{2021dec25eqn41}
\begin{split}
|  Err^0_{ i, i_1}(t_1, t_2)|&\lesssim   \mathcal{M}(C) \big(2^{-a_p +(\gamma_1-\gamma_2) M_t} +1 \big)  2^{ (\alpha^{\star}M_t+3\iota M_t+130\epsilon M_t - 4\vartheta^\star_0 )-  \gamma_2 M_t }2^{-k-2n} \\
&\quad \times 2^{- (k_1+2n_1)/2 -  n } \big(2^{(1-\epsilon)M_t}+2^{(k+2n)/2+ (\alpha^{\star}M_t+3\iota M_t+130\epsilon M_t - 4\vartheta^\star_0 )} \big)  \\
&\lesssim\mathcal{M}(C) \big(2^{-a_p +(\gamma_1-\gamma_2) M_t} +1 \big) 2^{4\alpha^{\star}M_t/5  }. \\
\end{split}
\ee

\medskip

\noindent \textbf{Step 2.}\quad  The estimate of  $  Err^1_{ i, i_1}(t_1, t_2)$.

\medskip

Recall  \eqref{oct10eqn77}  and  \eqref{oct7eqn90}. With minor modifications in obtaining estimates in  \eqref{oct7eqn21}  and  \eqref{2024oct27eqn61}  in Lemma \ref{secondhypohori}, we have 
\be\label{oct11eqn50}
\begin{split}
 \sup_{\zeta \in \R^3}&\big| \int_{\R^3} e^{ix\cdot \xi + i \mu s |\xi|}    \p_s\mathcal{F}[ {}^1 \mathfrak{H}] (s, \xi,   X_{\bot}(s),V(s))  d \xi  \big| \\
 &\lesssim  \big[   \min\{ |  x_{\bot}|^{-1} 2^{k+n },    2^{(k+2n)/2} 2^{j+\alpha_s M_s}\}+ 2^{2n }   | x_{\bot}|^{-1/2}2^{k+n/2+ \alpha_s M_s } \big]\\
&\quad \times 2^{ -\gamma_2 M_t  -k-3n+ 5\epsilon M_t} \mathcal{M}(C) \\
& \lesssim \big[   \big(|  x_{\bot}|^{-1} 2^{k+n }\big)^{1/2}  \big(2^{(k+2n)/2} 2^{j+\alpha_s M_s}\}\big)^{1/2}+ 2^{2n }   | x_{\bot}|^{-1/2}2^{k+n/2+ \alpha_s M_s } \big] \\
& \quad \times 2^{ -\gamma_2 M_t  -k-3n+ 5\epsilon M_t} \mathcal{M}(C) \\
& \lesssim  2^{ -a_p /2-  \gamma_2 M_t +10\epsilon M_t }\big( 2^{(j+\alpha_s)M_s/2-k/4-2n} + 2^{ \alpha_s M_s -n/2}  \big) \mathcal{M}(C). \\
\end{split}
\ee
  From the above estimate,  the estimates in  \eqref{oct7eqn21}  and  \eqref{2024oct27eqn61}  in Lemma \ref{secondhypohori},   we have 
\be\label{oct10eqn85}
\begin{split}
&\big|   Err^1_{ i, i_1}(t_1, t_2)\big|\\
&\lesssim \big[\sum_{b\in \mathcal{T}}  2^{-b a_p + b (\gamma_1-\gamma_2)M_{t} -(k+2n)/2+20\epsilon M_t}      (  2^{2(\gamma_1-\gamma_2)M_t/3} \mathbf{1}_{n \in \mathcal{N}_t^1}  +2^{-\alpha^{\star}  M_t/4 }  \mathbf{1}_{n \in \mathcal{N}_t^2 } ) \big]\\
&\quad \times  \big[2^{ -a_p /2  +10\epsilon M_t } ( 2^{(j+\alpha^{\star})M_t/2-k_1/4-n_1} + 2^{ \alpha_s M_s +n_1/2}   ) \big] \mathcal{M}(C)(t_2-t_1)\\
& \quad +  \big[2^{ -a_p /2-  \gamma_2 M_t +10\epsilon M_t } ( 2^{(j+\alpha_s)M_s/2-k/4-2n} + 2^{ \alpha_s M_s -n/2}   )\big]   \\
&\quad \times   \big(\sum_{a\in \mathcal{T}}  2^{-a a_p + a(\gamma_1-\gamma_2)M_{t}}  2^{-(k_1+2n_1)/2+3\alpha^{\star}M_t/4} \big) \mathcal{M}(C) (t_2-t_1)\\
&\lesssim \sum_{b\in \mathcal{T}+\mathcal{T}}   2^{-b a_p + (b+1)(\gamma_1-\gamma_2)M_{t} + 5\alpha^{\star}M_t/6+200\epsilon M_t}    \mathcal{M}(C)(t_2-t_1).\\
\end{split}
\ee

\medskip

\noindent \textbf{Step 3.}\quad    The estimate of  $  Err^2_{ i, i_1}(t_1, t_2)$.

\medskip

Recall  \eqref{oct10eqn93}.  From the estimate  \eqref{oct11eqn20}   and the estimate of elliptic 
parts 
in \eqref{2022feb24eqn1}  in Theorem \ref{maintheoremellipitic},  we have
\be
\begin{split} 
\big| Err^2_{ i, i_1}(t_1, t_2)\big|& \lesssim   \sum_{b\in \mathcal{T} }  2^{- ba_p } 2^{(b+5/6) (\gamma_1-\gamma_2)M_t }  2^{-a_p/2 + (\alpha^{\star} +4\epsilon )M_t}\\
&\quad \times  2^{ \alpha^{\star} M_t/6  + 220\epsilon M_{t} + 3\iota M_t   -(k_1+2n_1)/2} \mathcal{M}(C)  (t_2-t_1)  \\
&\lesssim  \sum_{b\in \mathcal{T}+\mathcal{T}}   2^{-b a_p + (b+1)(\gamma_1-\gamma_2)M_{t} +  2\alpha^{\star}M_t/3 }    \mathcal{M}(C)(t_2-t_1). \\
\end{split}
\ee
Therefore, the desired estimate  \eqref{2021dec22eqn32}  holds from the above estimate and the obtained estimates \eqref{2021dec25eqn41},   and  \eqref{oct10eqn85}. 
 
\end{proof}
 
 For the error terms produced in the third iteration process  in \eqref{oct29eqn61}, we have
\begin{lemma}\label{2021errhorstep2}
 Under the assumption of Proposition \ref{bootstraplemma1}, the following estimate holds for any  $i,i_1, i_2\in\{0,1,2,3,4\}, ( n,k)\in \mathcal{E}_i$, $(n_1, k_1)\in {}_{}^1\mathcal{E}_{k,n}$,  $(n_2, k_2)\in {}_{}^2\mathcal{E}_{k_1,n_1}$, and  $a_{p}\geq  -k-n +3\epsilon M_{t}/2$,    
\be\label{2021dec22eqn56}
\sum_{a=0,1,2  }  \big| Err^a_{ i, i_1,i_2}(t_1, t_2)\big| \lesssim  \big[ \sum_{b\in \mathcal{T}+\mathcal{T}}   2^{-ba_p  } 2^{ b(\gamma_1-\gamma_2)M_t+\gamma_1 M_t-20\epsilon M_t}(t_2-t_1) + 2^{2\alpha^{\star}M_t/3}\big] \mathcal{M}(C).
\ee
\end{lemma}
\begin{proof}
Recall the definition of $\mathcal{E}_i,i\in\{0,1,2,3,4\}$ in  \eqref{indexsetsec4}, the definition of ${}_{}^1 \mathcal{E}_{k,n}$ in  \eqref{firstiterindex}, and  the definition of ${}_{}^2 \mathcal{E}_{k_1,n_1}$ in \eqref{seconditerindex}.  Based on the size of $a\in \{0,1,2\}$, we proceed in three steps as follows.

\medskip

\noindent \textbf{Step 1.}\quad The estimate of $   Err^0_{ i, i_1,i_2}(t_1, t_2) $. 

\medskip

Recall  \eqref{oct11eqn187},  \eqref{oct10eqn96}, the assumption of the coefficient ``$C$'' in Proposition \ref{bootstraplemma1}.  From       the estimates in  \eqref{2024oct8eqn1} in Theorem  \ref{mainresultsfirstpart}, we have 
\be\label{2021dec25eqn61}
\begin{split}
\big|   Err^0_{ i, i_1,i_2}(t_1, t_2)\big| &\lesssim\big(2^{-a_p +(\gamma_1-\gamma_2) M_t} +1 \big)   2^{2(\alpha^{\star}M_t+3\iota M_t+130\epsilon M_t - 4\vartheta^\star_0 )- 2\gamma_2 M_t }\\
&\quad \times \big(2^{(1-\epsilon)M_t}+2^{(k+2n)/2+ (\alpha^{\star}M_t+3\iota M_t+130\epsilon M_t - 4\vartheta^\star_0 )}\big)\\
&\quad \times 2^{-k-2n} 2^{- (k_1+2n_1)/2-  (k_2+2n_2)/2- 2\min\{n, n_1\} } \mathcal{M}(C)\\
&\lesssim \big(2^{-a_p +(\gamma_1-\gamma_2) M_t} +1 \big) 2^{2\alpha^{\star}M_t/3  }\mathcal{M}(C). \\
\end{split}
\ee
 
\medskip

\noindent \textbf{Step 2.}\quad  The estimate of $  Err^1_{ i, i_1,i_2}(t_1, t_2)$. 

\medskip

Recall  \eqref{oct11eqn42}. From the obtained estimate  \eqref{oct11eqn20}  and  the    estimates  \eqref{oct7eqn21}  and  \eqref{2024oct27eqn61}  in Lemma \ref{secondhypohori}, we have
\be\label{oct11eqn60}
\begin{split}
&\Big|\int_{t_1}^{t_2} \int_{\R^3}  \int_{\R^3} e^{i X(s )\cdot (\xi+\eta + \sigma)  + i  \mu_2 s |\sigma| + i \mu s |\xi| + i \mu_1 s  |\eta|}  \\
&\quad \times   \big( \hat{V}(s ) \cdot(\xi+ \eta +\sigma) +  \mu_2   |\sigma| +   \mu  |\xi| +   \mu_1  |\eta| \big)^{-1} \\
&  \quad \times \p_s  \mathcal{F}[\mathfrak{H}_{k_2,j_2;n_2}^{\mu_2,i_2}]  (s ,\sigma, V(s ))\cdot    \mathcal{F}[{}_{}^{2}\mathfrak{H}]  (s , \xi,\eta,    X_{\bot}(s ),V(s ))     d \xi d\eta  d s \Big|\\
& \lesssim    \sum_{b\in \mathcal{T} } 2^{- ba_p } 2^{(b+5/6) (\gamma_1-\gamma_2)M_t } 2^{ \alpha^{\star} M_t/6  + 220\epsilon M_{t} + 3\iota M_t   -(k_1+2n_1)/2}  \mathcal{M}(C)(t_2-t_1)\\
 &\quad \times  2^{ -a_p /2  +10\epsilon M_t }\big( 2^{(j+\alpha^{\star})M_t/2-k_2/4-n_2} + 2^{ \alpha_s M_s +n_2/2}  \big)   \\
 &\lesssim \sum_{b\in \mathcal{T} +  \mathcal{T}} 2^{- ba_p } 2^{(b+1/2) (\gamma_1-\gamma_2)M_t }2^{3\alpha^{\star} M_t/4}\mathcal{M}(C) (t_2-t_1). \\
 \end{split}
 \ee

Moreover, from the obtained estimates   \eqref{2021dec24eqn2},  \eqref{oct11eqn50}, the estimates in \eqref{2024oct8eqn5}  and \eqref{2024Dec6eqn31} in Theorem \ref{mainresultsfirstpart},  the    estimates \eqref{oct7eqn21} and  \eqref{2024oct27eqn61} in Lemma \ref{secondhypohori},   for any $s\in [t_1, t_2],$ $x, y \in \R^3$, s.t., $|  x_{\bot}|, |  y_{\bot} |\in [2^{a_p  -5}, 2^{a_p  +5}]$, we have 
\be\label{oct12eqn41}
\begin{split}
&\big| \int_{\R^3} \int_{\R^3} e^{i x\cdot \xi}   e^{ i y\cdot \eta} e^{  i \mu s |\xi| + i \mu_1 s  |\eta|}\, \p_s   \mathcal{F}[{}_{}^{2}\mathfrak{H}] (s, \xi,\eta,    X_{\bot}(s),V(s))\big|\\
&\lesssim  \sum_{d\in \mathcal{T} + \mathcal{T}} \mathcal{M}(C)   2^{-d a_p + a(\gamma_1-\gamma_2)M_{t} - \alpha^\star M_t -\min\{n,n_1\}} \\
&\quad \times\big[\big( 2^{(j+\alpha_s)M_s/2-k/4-2n} + 2^{ \alpha_s M_s -n/2}  \big)2^{-\alpha^\star M_t/4- (k_1+2n_1)/2  } \\
&\quad + \big(2^{ (j+a_sM_s)/2-k_1/4-n_1 } + 2^{\alpha_s M_s +  n_1/2}\big)2^{-(k+2n)/2}
  \big] \\
    &\lesssim \sum_{d\in \mathcal{T}+\mathcal{T}}  2^{-d a_{p} } 2^{d  (\gamma_1-\gamma_2)M_t  }  2^{ 5\alpha^{\star} M_t/12  } \mathcal{M}(C).
\end{split}
\ee
From the above estimate, and   the estimates in  \eqref{2024oct8eqn1} in Theorem  \ref{mainresultsfirstpart}, we have
\be
\begin{split}
&\Big|\int_{t_1}^{t_2} \int_{\R^3}  \int_{\R^3} e^{i X(s )\cdot (\xi+\eta + \sigma)  + i  \mu_2 s  |\sigma| + i \mu s |\xi| + i \mu_1 s  |\eta|} \\
&\quad \times   \big( \hat{V}(s) \cdot(\xi+ \eta +\sigma) +  \mu_2   |\sigma| 
+   \mu  |\xi| +   \mu_1  |\eta| \big)^{-1} \\
&\quad \times   \mathcal{F}[\mathfrak{H}_{k_2,j_2;n_2}^{\mu_2,i_2}] (s ,\sigma, V(s ))\cdot    \p_s  \mathcal{F}[{}_{}^{2}\mathfrak{H}]  (s , \xi,\eta,     X_{\bot}(s), V(s ))      d \xi d\eta  d s \Big|\\
&\lesssim   \big(\sum_{d\in \mathcal{T}+\mathcal{T}}  2^{-d a_{p} } 2^{(d+1) (\gamma_1-\gamma_2)M_t }  2^{  5\alpha^{\star} M_t/12 } \big) 2^{\alpha^{\star} M_t-(k_2+2n_2)/2} \mathcal{M}(C)(t_2-t_1).
\end{split}
\ee
After combining the above estimate and the obtained estimate  \eqref{oct11eqn60}, we have
\be\label{oct11eqn70}
\big| Err^1_{ i, i_1,i_2}(t_1, t_2) \big|\lesssim  \sum_{b\in \mathcal{T} +  \mathcal{T}} 2^{- ba_p } 2^{(b+1/2) (\gamma_1-\gamma_2)M_t }2^{4\alpha^{\star} M_t/5}\mathcal{M}(C) (t_2-t_1). 
\ee

\medskip

\noindent \textbf{Step 3.}\quad  The estimate of $   Err^2_{ i, i_1,i_2}(t_1, t_2) $. 

\medskip

Recall  \eqref{oct11eqn43}. From the obtained  estimate      \eqref{2021dec24eqn21},  and the estimate of elliptic 
parts 
in \eqref{2022feb24eqn1}  in Theorem \ref{maintheoremellipitic}, we have 
\be\label{oct11eqn95}
\begin{split}
\big|   Err^2_{ i, i_1,i_2}(t_1, t_2)\big|&\lesssim    \sum_{b\in \mathcal{T} }   2^{-a_p/2 + (\alpha^{\star} +4\epsilon )M_t}  (t_2-t_1) \mathcal{M}(C)  \\
&\quad \times  2^{- ba_p } 2^{(b+4/3) (\gamma_1-\gamma_2)M_t + \alpha^{\star}M_{t^\star}/12+  490\epsilon M_{t} +6\iota -(k_2+2n_2)/2} \\
 &  \lesssim  \sum_{d\in \mathcal{T}+\mathcal{T} }2^{- d a_p } 2^{d (\gamma_1-\gamma_2)M_t + \alpha^{\star}M_t/2  }\mathcal{M}(C) (t_2-t_1). \\
\end{split}
\ee
To sum up, the desired estimate \eqref{2021dec22eqn56} holds after combining the obtained estimates \eqref{2021dec25eqn61},  \eqref{oct11eqn70}, and  \eqref{oct11eqn95}. 
\end{proof}

For the error terms produced in the fourth iteration process  in \eqref{oct12eqn76}, we have
\begin{lemma}\label{2021errhorstep3}
 Under the assumption of Proposition \ref{bootstraplemma1}, the following estimate holds for any  $i,i_1, i_2, i_3\in\{0,1,2,3,4\}, ( n,k)\in \mathcal{E}_i$, $(n_1, k_1)\in {}_{}^1\mathcal{E}_{k,n}$,  $(n_2, k_2)\in {}_{}^2\mathcal{E}_{k_1,n_1}$, $(n_3, k_3)\in {}_{}^3\mathcal{E}_{k_2,n_2}$, and  $a_{p}\geq  -k-n +3\epsilon M_{t}/2$,
\be\label{oct23eqn65}
\sum_{a=0,1 }  \big| Err^a_{ i, i_1,i_2,i_3}(t_1, t_2)\big| \lesssim   \mathcal{M}(C)\big[\sum_{b\in \mathcal{T}+\mathcal{T}} 2^{- ba_p } 2^{(b+1) (\gamma_1-\gamma_2)M_t }  2^{5\alpha^{\star} M_t/8}(t_2-t_1) +  2^{\alpha^{\star} M_t/2} \big]. 
\ee
\end{lemma}
\begin{proof}

Recall the definition of $\mathcal{E}_i,i\in\{0,1,2,3,4\}$ in  \eqref{indexsetsec4}, the definition of ${}_{}^1 \mathcal{E}_{k,n}$ in  \eqref{firstiterindex},    the definition of ${}_{}^2 \mathcal{E}_{k_1,n_1}$ in  \eqref{seconditerindex}, and the definition of ${}_{}^3 \mathcal{E}_{k_2,n_2}$ in  \eqref{thirditerindex}.  Based on the size of $a\in \{0,1,2\}$, we proceed in two steps as follows.

\medskip

\noindent \textbf{Step 1.}\quad The estimate of $  Err^0_{ i, i_1,i_2,i_3}(t_1, t_2) $.

\medskip

Recall  \eqref{oct12eqn2},   \eqref{oct11eqn86},  and the assumption of the coefficient ``$C$'' in Proposition \ref{bootstraplemma1}.   From the obtained estimate  \eqref{2021dec24eqn21},   the estimates in  \eqref{2024oct8eqn1} in Theorem  \ref{mainresultsfirstpart}, we have 
\be\label{oct12eqn95}
\begin{split}
\big| Err^0_{ i, i_1,i_2,i_3}(t_1, t_2)  \big| & \lesssim  \big( 2^{-a_p}2^{(\gamma_1-\gamma_2)M_t}(t_2-t_1) +1\big)   2^{-k-2n} 2^{3\alpha^{\star}M_t- 3\gamma_2 M_t}\\
&\quad \times  2^{ - (k_1+2n_1)/2-  (k_2+2n_2)/2-(k_3+2n_3)/2-3\min\{n, n_1\}+200\epsilon M_t} \\
&\quad \times (2^{(1-\epsilon)M_t}+2^{(k+2n)/2+\alpha^{\star}M_t+30\epsilon M_t})\mathcal{M}(C)\\
& \lesssim  \big(  2^{-a_p} 2^{(\gamma_1-\gamma_2)M_t}(t_2-t_1) +1\big)  2^{\alpha^{\star} M_t/2} \mathcal{M}(C). 
\end{split}
\ee

\medskip

\noindent \textbf{Step 2.}\quad The estimate of $   Err^1_{ i, i_1,i_2,i_3}(t_1, t_2) $.

\medskip

Recall  \eqref{oct11eqn42} and   \eqref{oct11eqn86}.  From the obtained estimates  \eqref{oct12eqn41}  and   the estimates in  \eqref{2024oct8eqn1} in Theorem  \ref{mainresultsfirstpart},  for any $x,y,z\in \R^3$, s.t., $|  x_{\bot}|, | y_{\bot}|, | z_{\bot}|\in [2^{a_p  -5}, 2^{a_p   + 5}] $, we have 
\be
\begin{split}
   &\big| \int_{\R^3}\int_{\R^3} \int_{\R^3}  e^{i x\cdot \xi +i y\cdot \eta + i z \cdot \sigma    + i\mu_3 s |\kappa|  + i  \mu_2 s  |\sigma| + i \mu s |\xi| + i \mu_1 s |\eta|}  \\
  &\quad\times  \nabla_\zeta \big[  \big( \hat{\zeta}\cdot(\xi+ \eta +\sigma) +  \mu_2   |\sigma| +   \mu  |\xi| +   \mu_1  |\eta| \big)^{-1} \\
&\quad \times \mathcal{F}[ \mathfrak{H}_{k_2,j_2;n_2}^{\mu_2,i_2}](s,\sigma, \zeta)\cdot \p_s \mathcal{F}[ {}^2\mathfrak{H} ](s, \xi,\eta,  X_{\bot}(s),\zeta) \big]\big|_{\zeta= V(s)} d\xi d \eta d \sigma  \big| \\
&  \lesssim   \sum_{d\in \mathcal{T}+\mathcal{T}}  2^{-d a_{p} } 2^{d  (\gamma_1-\gamma_2)M_t  }  2^{ 5\alpha^{\star} M_t/12  }  2^{-\gamma_2 M_t+100\epsilon M_t} 2^{\alpha^{\star}M_t - (k_2+2n_2)/2} 2^{-\min\{n, n_1, n_2\}} \mathcal{M}(C)  \\
 &\lesssim  \sum_{b\in \mathcal{T}+\mathcal{T}}  2^{-b a_{p} } 2^{(b+1) (\gamma_1-\gamma_2)M_t+\alpha^{\star} M_t/4 +1000\epsilon M_t} \mathcal{M}(C)   .
 \end{split}
\ee
Moreover, from the obtained estimate  \eqref{oct11eqn20}, estimates \eqref{oct7eqn21} and \eqref{2024oct27eqn61} in Lemma \ref{secondhypohori}, the following estimate holds   if $x,y,z\in \R^3$, s.t., $|  x_{\bot}|, | y_{\bot}|, | z_{\bot}|\in [2^{a_p  -5}, 2^{a_p   + 5}] $, 
\be
\begin{split}
  &\big| \int_{\R^3}\int_{\R^3} \int_{\R^3}  e^{i x\cdot \xi +i y\cdot \eta + i z \cdot \sigma    + i\mu_3 s |\kappa|  + i  \mu_2 s  |\sigma| + i \mu s |\xi| + i \mu_1 s |\eta|}  \\
  &\quad \times  \nabla_\zeta \big[  \big( \hat{\zeta}\cdot(\xi+ \eta +\sigma) +  \mu_2   |\sigma| +   \mu  |\xi| +   \mu_1  |\eta| \big)^{-1} \\
&\quad\times \p_s \mathcal{F}[ \mathfrak{H}_{k_2,j_2;n_2}^{\mu_2,i_2}] (s,\sigma, \zeta)\cdot  \mathcal{F}[ {}^2\mathfrak{H} ](s, \xi,\eta,    X_{\bot}(s),\zeta) \big]\big|_{\zeta= V(s)} d\xi d \eta d \sigma  \big| \\
& \lesssim  \sum_{b\in \mathcal{T} } 2^{- ba_p } 2^{(b+5/6) (\gamma_1-\gamma_2)M_t }    2^{ -a_p /2 -\gamma_2 M_t -\min\{n,n_1, n_2\} }\\
 &\quad \times  2^{ \alpha^{\star} M_t/6  + 220\epsilon M_{t} + 3\iota M_t   -(k_1+2n_1)/2}\big( 2^{(1+\alpha_s)M_s/2-k_2/4-  n_2} + 2^{ \alpha_s M_s +n_2/2}  \big)    \mathcal{M}(C)\\
 &\lesssim   \sum_{d\in \mathcal{T}+\mathcal{T}}  2^{-d a_{p} } 2^{(d+1) (\gamma_1-\gamma_2)M_t+\alpha^{\star} M_t/4 }\mathcal{M}(C). \\
 \end{split}
\ee

After combining the above two estimates,     for any $x,y,z\in \R^3$, s.t., $|  x_{\bot}|, |  y_{\bot}|, |  z_{\bot}|\in [2^{a_p  -5}, 2^{a_p   + 5}] $, we have 
\be\label{oct11eqn72}
\begin{split}
 & \big| \int_{\R^3}\int_{\R^3} \int_{\R^3}  e^{i x\cdot \xi +i y\cdot \eta + i z \cdot \sigma      + i  \mu_2 s  |\sigma| + i \mu s |\xi| + i \mu_1 s |\eta|}  \p_s\mathcal{F}[ {}^3\mathfrak{H} ] (s , \xi,\eta,\sigma,     X_{\bot}(s ),V(s))d \xi d \eta d \sigma \big|\\
 &\lesssim  \sum_{b\in \mathcal{T}+\mathcal{T}}  2^{-b a_{p} } 2^{(b+1) (\gamma_1-\gamma_2)M_t+\alpha^{\star} M_t/4 +1000\epsilon M_t}\mathcal{M}(C) . \\
\end{split}
\ee
From the above estimate \eqref{oct11eqn72}, the estimate of $  \mathcal{F}[ {}^3\mathfrak{H} ] (t , \xi,\eta,\sigma,      X_{\bot}(t ),V(t ))$ in  \eqref{2021dec24eqn21},      the estimates in  \eqref{2024oct8eqn1} in Theorem  \ref{mainresultsfirstpart},  and the estimates  \eqref{oct7eqn21}    and  \eqref{2024oct27eqn61}  in Lemma \ref{secondhypohori},   we have
\be\label{oct12eqn97}
\begin{split}
&\big| Err^1_{ i, i_1,i_2,i_3}(t_1, t_2)\big| \\
&\lesssim     \mathcal{M}(C)(t_2-t_1)\big[\sum_{b\in \mathcal{T}+\mathcal{T}}  2^{-b a_{p} } 2^{(b+1) (\gamma_1-\gamma_2)M_t+\alpha^{\star} M_t/4 +1000\epsilon M_t}  2^{\alpha^{\star} M_t - (k_3+2n_3)/2}\\
&\quad +  \sum_{b\in \mathcal{T} } 2^{- ba_p } 2^{(b+4/3) (\gamma_1-\gamma_2)M_t + \alpha^{\star}M_{t^\star}/12+  490\epsilon M_{t} +6\iota -(k_2+2n_2)/2}    2^{ -a_p /2    }\\
&\quad \times \big( 2^{(1+\alpha_s)M_s/2-k_3/4-  n_3} + 2^{ \alpha_s M_s +n_3/2}  \big)   \big]\\
&\lesssim \sum_{d\in \mathcal{T}+\mathcal{T}}  2^{-d a_{p} } 2^{(d+1) (\gamma_1-\gamma_2)M_t+ \alpha^{\star} M_t/2 }\mathcal{M}(C)(t_2-t_1).\\
\end{split}
\ee
To sum up, after combining the obtained estimates \eqref{oct12eqn95} and  \eqref{oct12eqn97}, our desired estimate \eqref{oct23eqn65} holds. 
\end{proof}
 
For the error terms produced in the last iteration process  in \eqref{2022feb22eqn31}, we have
 \begin{lemma}\label{2022errhorstep5}
 Under the assumption of Proposition \ref{bootstraplemma1}, the following estimate holds for any  $i,i_1, i_2, i_3\in\{0,1,2,3,4\}, ( n,k)\in \mathcal{E}_i$, $(n_1, k_1)\in {}_{}^1\mathcal{E}_{k,n}$,  $(n_2, k_2)\in {}_{}^2\mathcal{E}_{k_1,n_1}$, $(n_3, k_3)\in {}_{}^3\mathcal{E}_{k_2,n_2}$, $(n_4, k_4)\in {}_{}^3\mathcal{E}_{k_3,n_3}$, and  $a_{p}\geq  -k-n +3\epsilon M_{t}/2$,
\be\label{2022feb22eqn71}
\sum_{a=0,1   }  \big| Err^a_{ i, i_1,i_2,i_3,i_4}(t_1, t_2)\big| \lesssim    \mathcal{M}(C)\big[ \sum_{b\in \mathcal{T}+\mathcal{T}} 2^{- ba_p } 2^{(b+1) (\gamma_1-\gamma_2)M_t }  2^{ \alpha^{\star} M_t/3}(t_2-t_1)  + 1\big].
\ee
\end{lemma}
 \begin{proof}
 Recall the definition of $\mathcal{E}_i,i\in\{0,1,2,3,4\}$ in \eqref{indexsetsec4}, the definition of ${}_{}^1 \mathcal{E}_{k,n}$ in \eqref{firstiterindex},    the definition of ${}_{}^2 \mathcal{E}_{k_1,n_1}$ in  \eqref{seconditerindex},   the definition of ${}_{}^3 \mathcal{E}_{k_2,n_2}$ in  \eqref{thirditerindex}, and the definition of ${}_{}^4 \mathcal{E}_{k_3,n_3 }$ in \eqref{forthiterindex}.  Moreover, recall  \eqref{2022feb22eqn64}  and the assumption of the coefficient ``$C$'' in Proposition \ref{bootstraplemma1}, from the        estimates in  \eqref{2024oct8eqn1} in Theorem  \ref{mainresultsfirstpart},    we have 
\be\label{2022feb22eqn61}
\begin{split}
&\big| Err^0_{ i, i_1,i_2,i_3,i_4}(t_1, t_2)  \big| \\
&\lesssim  \mathcal{M}(C) \big( 2^{-a_p}2^{(\gamma_1-\gamma_2)M_t}(t_2-t_1) +1\big)   2^{4(\alpha^{\star} +3\iota  + 130\epsilon  )M_t - 4\gamma_2 M_t}\\
&\quad \times (2^{(1-\epsilon)M_t}+2^{(k+2n)/2+ (\alpha^{\star} +3\iota  + 130\epsilon  ) }) 2^{-k-2n}  \\
&\quad \times  2^{ - (k_1+2n_1)/2-  (k_2+2n_2)/2-(k_3+2n_3)/2-(k_4+2n_4)/2-4\min\{n, n_1,n_2,n_3,n_4\}+200\epsilon M_t}   \\
&\lesssim   \big( 2^{-a_p} 2^{(\gamma_1-\gamma_2)M_t}(t_2-t_1) +1\big) \mathcal{M}(C). \\
\end{split}
\ee

Lastly, we estimate $Err^1_{ i, i_1,i_2,i_3,i_4}(t_1, t_2)$. Recall  \eqref{oct12eqn81}. Similar to the obtained estimate  \eqref{oct11eqn72}, after using the estimates   \eqref{oct11eqn72} and  \eqref{2021dec24eqn21} ,  the        estimates in  \eqref{2024oct8eqn1} in Theorem  \ref{mainresultsfirstpart},  the  estimates  \eqref{oct7eqn21}  and   \eqref{2024oct27eqn61}  in Lemma \ref{secondhypohori},   for any $x,y,z,w\in \R^3$, s.t., $|  x_{\bot}|, |  y_{\bot}|,$ $ |  z_{\bot}|, |  w_{\bot}|\in [2^{a_p  -5}, 2^{a_p   + 5}] $, we have 
\[
\begin{split}
 &\big| \int_{\R^3} \int_{\R^3}\int_{\R^3} \int_{\R^3}  e^{i x\cdot \xi +i y\cdot \eta + i z \cdot \sigma  + i w\cdot \kappa  + i\mu_3 s |\kappa|  + i  \mu_2 s  |\sigma|+ i \mu_1 s |\eta|  + i \mu s |\xi|  }  \\
 &\quad \times   \p_s  \mathcal{F}[ {}^4\mathfrak{H} ] (s, \xi, \eta, \sigma,\kappa,    X_{\bot}(s), V(s)) d \xi d \eta d \sigma d \kappa \big|\\
&\lesssim  \sum_{b\in \mathcal{T}+\mathcal{T}}  \mathcal{M}(C)  2^{-b a_{p} } 2^{(b+1) (\gamma_1-\gamma_2)M_t-\gamma_2 M_t+100\epsilon M_t-\min\{n, n_1, n_2,n_3\} } \\
&\quad \times  \big[ 2^{\alpha^{\star} M_t/4 +1000\epsilon M_t}    2^{\alpha^{\star}M_t - (k_3+2n_3)/2} +  2^{  \alpha^{\star}M_{t^\star}/12+  490\epsilon M_{t} +6\iota -(k_2+2n_2)/2} \\
& \quad \times \big( 2^{(1+\alpha_s)M_s/2-k_3/4-  n_3} + 2^{ \alpha_s M_s +n_3/2}  \big)\big]   \\
&\lesssim   \sum_{b\in \mathcal{T}+\mathcal{T}}  \mathcal{M}(C)  2^{-b a_{p} } 2^{(b+1) (\gamma_1-\gamma_2)M_t }.
\end{split}
\]

Recall  \eqref{2022feb22eqn64}. From the above estimate, the estimate \eqref{2022feb22eqn22},  the        estimates in  \eqref{2024oct8eqn1} in Theorem  \ref{mainresultsfirstpart}, the  estimates  \eqref{oct7eqn21}  and  \eqref{2024oct27eqn61}  in Lemma \ref{secondhypohori}, we have
\[
\begin{split}
\big| Err^1_{ i, i_1,i_2,i_3,i_4}(t_1, t_2)  \big|&\lesssim   \sum_{b\in \mathcal{T}+\mathcal{T}}  \mathcal{M}(C) (t_2-t_1)  2^{-b a_{p} } 2^{(b+1) (\gamma_1-\gamma_2)M_t }\\
&\quad \times \big[2^{\alpha^\star M_t-(k_4+2n_4)/2 }+ \big( 2^{(1+\alpha_s)M_s/2-k_4/4-  n_4} + 2^{ \alpha_s M_s +n_4/2}  \big) \\
&\quad \times  2^{  12\iota  + 1600\epsilon M_{t} -   \alpha^{\star}M_{t^\star}/12 -(k_3+2n_3)/2}    \big] \\
& \lesssim  \sum_{b\in \mathcal{T}+\mathcal{T}}  \mathcal{M}(C) (t_2-t_1)  2^{-b a_{p} } 2^{(b+1) (\gamma_1-\gamma_2)M_t +\alpha^{\star}M_{t^\star}/3  }.\\
\end{split}
\]
 Therefore, the desired estimate  \eqref{2022feb22eqn71}  holds from the above estimate and the estimate  \eqref{2022feb22eqn61}. 
 \end{proof}

 \subsubsection{Error terms in estimating   the elliptic parts}\label{mainpartsISSpart1error}
 
In estimating the elliptic parts, error terms  only produced in  the first interation, see \eqref{2024oct23eqn11} and \eqref{2024oct23eqn12}. For these error terms, we have 

\begin{lemma}\label{error3part1}
Let $ \bar{\kappa}:=(\gamma_1-\gamma_2)M_t-100\epsilon M_t-n$,  $i\in\{0,1,2,3,4\}, ( n,k)\in \mathcal{E}_i$ and $a_{p}\geq  -k-n +3\epsilon M_{t}/2$,  under the assumption of Proposition \ref{bootstraplemma1},  the following estimate holds if $ a_p  \geq   -k+ \alpha^{\star} M_t/2$,  $n\geq  (\gamma_1-\gamma_2)M_t/3-30\epsilon M_t$,   $j\geq  (\gamma_1-\gamma_2)M_t/2+(\alpha^{\star}-15\epsilon) M_t $, and $k+2n \geq  2(\alpha^{\star}-15\epsilon) M_t+3(\gamma_1-\gamma_2) M_t/2,$   
\be\label{oct3eqn31}
\begin{split}
\sum_{a=0,1,2,3}&\big|   Err^{\mu,a }_{k,j;n}(t_1, t_2)  \big| + \sum_{\kappa\in (\bar{\kappa}, 2]\cap \Z}  |  Err^{\mu,4;\kappa}_{k,j;n}(t_1, t_2) | \\
& \lesssim   \mathcal{M}(C)\big(2^{-a_p+(\gamma_1-\gamma_2)M_t}(t_2-t_1) +1 \big) 2^{2\alpha^{\star}M_t/3    }. 
\end{split}
\ee
\end{lemma}
\begin{proof}

Based on the size of $a\in \{0,1,2,3,4\}$, we proceed in two steps as follows. 

\medskip

\noindent \textbf{Step 1.}\quad  If $a\in\{0,1,2,3\}. $

\medskip

 Recall  \eqref{oct7eqn62}. Note that, in terms of kernel,  we have 
\[
\begin{split}
\big|Err^{\mu,a }_{k,j;n}(t_1, t_2)\big|&\lesssim \sum_{i=1,2}\int_{\R^3} \int_{\R^3} f(t_i, X(t_i)-y, v) \big|K^a_{k,n}(  X_{\bot}(t_i), y, v, V(t_i))\big| dyd v \\
 &\quad + \int_{t_1}^{t_2} \int_{\R^3}  \int_{\R^3}  f(s, X(s)-y, v) \big|\widetilde{K}^a_{k,n}(  X_{\bot}(s),y, v, V(s))\big| dy d v d s, \\
 \end{split}
\]
where,  $\forall s\in [t_1, t_2],  $  the kernels are defined as follows,
\[
\begin{split}
 K^{a }_{k,j,n}(  X_{\bot}(s), y, v, V(s)) &= \int_{\R^3} e^{i y\cdot \xi }  \mathcal{F}[ {}_{}^{1}\mathfrak{E}^{a } ](   X_{\bot}(s),   \xi, v, V(s))   d \xi, \\ 
 \widetilde{K}^{a }_{k,j,n}(  X_{\bot}(s), y, v, V(s)) & = \int_{\R^3} e^{i y\cdot \xi }  {\hat{V}}_{\bot}(s)\cdot \nabla_{  x_{\bot}}\big(    \mathcal{F}[ {}_{}^{1}\mathfrak{E}^{ a} ]( x_{\bot} ,   \xi, v, V(s))\big)\big|_{  x_{\bot} =   X_{\bot}(s)} d \xi .
\end{split}
 \]
  Recall \eqref{2024oct28eqn11}. After doing integration by parts in $\xi$ along $V(s)$ direction and directions perpendicular to $V(s)$,  the following estimate holds for the above defined kernels, 
\be 
\begin{split}
&| K^{a }_{k,j,n}(  X_{\bot}(s),y, v, V(s)) | 
+ 2^{a_p-(\gamma_1-\gamma_2)M_t} | \widetilde{K}^{a }_{k,j,n}(  X_{\bot}(s), y, v, V(s)) | \\&
\lesssim \mathcal{M}(C)    2^{k+2n-2\max\{l,n\} + 2\epsilon M_t} (1+2^k|y\cdot \tilde{V}(s)|)^{-100}(1+2^{k+n}|y\times  \tilde{V}(s)|)^{-100} .\\
\end{split}
\ee
From the above estimate of kernels, the volume of support of $v$, and the conservation law  \eqref{conservationlaw},    we have 
\be\label{oct7eqn81}
\begin{split}
\big|Err^{\mu,a}_{k,j;n}(t_1, t_2)\big|&\lesssim \mathcal{M}(C) 2^{k+  2n-2\max\{l,n\} +3\epsilon M_t} \min\{  2^{-j},  2^{-3k-2n} 2^{3j+2l} \}\\
&\quad \times \big(2^{-a_p+(\gamma_1-\gamma_2)M_t}(t_2-t_1) +1 \big)\\
& \lesssim  \mathcal{M}(C)\big(2^{-a_p+(\gamma_1-\gamma_2)M_t }(t_2-t_1) +1 \big) 2^{M_t/3 + 5\epsilon M_t}\\
& \lesssim  \mathcal{M}(C)\big(2^{-a_p+(\gamma_1-\gamma_2)M_t}(t_2-t_1) +1 \big) 2^{\alpha^{\star}M_t/2  }. 
\end{split}
\ee
 
\medskip

\noindent \textbf{Step 2.}\quad   If $a=4$. 

\medskip

 Note that, due to the cutoff function $\varphi_{j,n}^4(v, \zeta)$ (see \eqref{sep4eqn6}), for this case, we have $|
 \tilde{v}-\tilde{V}(s)|\in [ 2^{n-3\epsilon M_t/2}, 2^{n+3\epsilon M_t/2}]$. Recall \eqref{oct7eqn62}. In terms of kernels, 
 we have
 \be\label{2021dec27eqn21}
 \begin{split}
   | Err^{\mu,4;\kappa}_{k,j;n}(t_1, t_2) | &\lesssim     \sum_{  a=1,2}     \int_{ \R^3  }  \int_{ \R^3  } f(t_a, X(t_a)-y , v ) \mathcal{K}_{k,j,n}^{\kappa}(  X_{\bot}(t_a), y, v, V(t_a)) dy d v \\
&\quad +  \int_{t_1}^{t_2}   \int_{ \R^3  }  \int_{ \R^3  }  f(s, X(s)-y , v ) \widetilde{\mathcal{K}}_{k,j,n}^{\kappa}( X_{\bot}(s),y, v, V(s)) dy d v ds, \\
\end{split}
 \ee
where $\forall s\in [t_1, t_2],  $  the kernels are defined as follows,
\[
\begin{split}
  \mathcal{K}_{k,j,n}^{\kappa}(  X_{\bot}(s), y, v, V(s)) & = \int_{\R^3} e^{i y\cdot \xi }   \mathcal{F}[ {}_{}^{1}\mathfrak{E}^{\kappa} ](   X_{\bot}(s),   \xi, v, V(s))   d \xi, \\ 
 \widetilde{\mathcal{K}}_{k,j,n}^{\kappa}( X_{\bot}(s), y, v, V(s)) &= \int_{\R^3} e^{i y\cdot \xi }  {\hat{V}}_{\bot}(s)\cdot \nabla_{ x_{\bot}}\big(   \mathcal{F}[ {}_{}^{1}\mathfrak{E}^{\kappa} ](   x_{\bot},   \xi, v, V(s))\big)\big|_{  x_{\bot} =  X_{\bot}(s)}     d \xi .
\end{split}
\]

 As in the proof of Lemma \ref{smallangest}, after doing 
integration by parts in $\xi$ with respect to the orthonormal frame  $\{(\hat{v}-\hat{V}(s))/|\hat{v}-\hat{V}(s), \theta_1(v, s),\theta_2(v, s)\}$ for the  kernel, for  $\forall s\in [t_1, t_2]$, we have
\be 
\begin{split}
&|   \mathcal{K}_{k,j,n}^{ \kappa}(  X_{\bot}(s),y, v, V(s))  | 
+ 2^{a_p-(\gamma_1-\gamma_2)M_t} |  \widetilde{\mathcal{K}}_{k,j,n}^{ \kappa}(  X_{\bot}(s), y, v, V(s))  |\\
&\lesssim   2^{k }\mathcal{M}(C)     (1+2^{k+  {\kappa} }(y\cdot (\hat{v}-\hat{V}(s)/|\hat{v}-\hat{V}(s)|) )) \\
&\quad \times ( 1+ 2^{k+n}|y\cdot \theta_1(v, s)|)^{-100}  (1+2^k|y\cdot \theta_2(v,s)|)^{-100}  .\\
\end{split}
\ee

Recall that $\bar{\kappa}:=(\gamma_1-\gamma_2)M_t-n-30\epsilon M_t.$  From the above estimate of kernels, the obtained estimate  \eqref{2021dec27eqn21}, the volume of support of $v$,  and the conservation law  \eqref{conservationlaw}, we have
\be
\begin{split}
    | Err^{\mu,4;\kappa}_{k,j;n}(t_1, t_2) | & \lesssim    \mathcal{M}(C) \big(2^{-a_p+(\gamma_1-\gamma_2) M_t}(t_2-t_1) +1 \big) 2^{k+  \epsilon M_t} \\
 &\quad \times \min\{  2^{-j},  2^{-3k-n-\kappa} 2^{3j+2n+2\epsilon M_t} \}\\
& \lesssim   \mathcal{M}(C)\big(2^{-a_p+(\gamma_1-\gamma_2)M_t}(t_2-t_1) +1 \big) 2^{\alpha^{\star}M_t/2 -(\gamma_1-\gamma_2) M_t/3+ 30\epsilon M_t  }\\
& \lesssim   \mathcal{M}(C)\big(2^{-a_p+(\gamma_1-\gamma_2)M_t}(t_2-t_1) +1 \big) 2^{2\alpha^{\star}M_t/3    }.\\
\end{split}
\ee
Therefore, the desired estimate  \eqref{oct3eqn31}  holds after combining the above estimate and the obtained estimate  \eqref{oct7eqn81}. 
\end{proof}

\subsubsection{Proof of Proposition \ref{elliptfinalpro}}\label{mainproposition2proof}
 
The desired estimate \eqref{oct25eqn21} holds  from the estimate \eqref{oct3eqn52} in Lemma \ref{trivialcaseell} for the trivial cases, the   estimate \eqref{oct3eqn115} in Lemma \ref{smallangest} for the threshold case, the estimate  \eqref{oct7eqn65}  in Lemma \ref{ellerrtyp2}, and the estimate  \eqref{oct14eqn61}  in Lemma 
\ref{mainell3part1} for the main parts in the decompositions in \eqref{2024oct23eqn11} and \eqref{2024oct23eqn12}, and  the estimate  \eqref{oct3eqn31}  in Lemma \ref{error3part1} for the error parts in the decompositions in \eqref{2024oct23eqn11} and \eqref{2024oct23eqn12}.

The essential notations employed in this section are systematically detailed in Table \ref{tablesection5}.
\begin{table}[H] 
\centering
\resizebox{\columnwidth}{!}{%
\begin{tabular}{ |c|c|c|c| } 
 \hline
 Notation & Definition &  Remarks \\
 \hline
 $\mathcal{E}_i$ & \eqref{indexsetsec4}  & The index set of $(k,n)$; Providing  lower bounds for $k+2n$ and $n$. \\
 \hline
 $\mathfrak{E}^{\mu,i}_{k,j;n}(t_1, t_2)$ & \eqref{oct25eqn1} & Accumulated effect of the elliptic parts over the time interval $[t_1, t_2]$\\
 \hline
 ${}_a^{1}\mathfrak{H}^{\mu,i}_{k,j;n}(t_1, t_2)$, $a\in\{0,1\}$ & \eqref{nov12eqn61} &  Accumulated effect of the    Type-I and Type-II hyperbolic part \\
 \hline
 ${}_a^{1}Err^{\mu,i}_{k,j;n}(t_1, t_2)$ & \eqref{oct7eqn31} & The error terms in the first reduction; estimated in Lemma \ref{erroesttyp1} \\ 
 \hline
 ${}^{1}\mathcal{E}_{k,n}$ & \eqref{firstiterindex} & The index set of $(k_1,n_1)$ for the second iteration\\
 \hline
 $Err^a_{ i, i_1}(t_1, t_2)$ & \eqref{oct8eqn1}--\eqref{oct10eqn93} & Error type terms in  the second iteration; estimated in Lemma \ref{2021errhorstep1}\\
 \hline
 $\mathfrak{H}_{ i, i_1,i_2}(t_1, t_2)$ & \eqref{oct29eqn55} & The main hyperbolic parts in the second iteration\\
 \hline
  ${}^{2}\mathcal{E}_{k_1,n_1}$ & \eqref{seconditerindex} & The index set of $(k_2,n_2)$ for the third iteration\\
 \hline
 $Err^a_{ i, i_1,i_2}(t_1, t_2)$ & \eqref{oct11eqn187}--\eqref{oct11eqn43} & Error type terms in  the third iteration; estimated in Lemma     \ref{2021errhorstep2}  \\
 \hline
 $\mathfrak{H}_{ i, i_1,i_2,i_3}(t_1, t_2)$ & \eqref{oct29eqn61} & The main hyperbolic parts in the third iteration\\
 \hline
  ${}^{3}\mathcal{E}_{k_2,n_2}$ & \eqref{thirditerindex} &  The index set of $(k_3,n_3)$ for the fourth iteration\\
\hline 
$ Err^a_{ i, i_1,i_2,i_3}(t_1, t_2)$ & \eqref{oct12eqn2} &  Error type terms in  the fourth iteration; estimated in Lemma   \ref{2021errhorstep3}\\
\hline
$  \mathfrak{H}_{ i, i_1,i_2,i_3,i_4}(t_1, t_2)$ & \eqref{oct12eqn76} &  The main hyperbolic parts in the fourth iteration\\
\hline
  ${}^{4}\mathcal{E}_{k_3,n_3}$  & \eqref{forthiterindex} &  The index set of $(k_4,n_4)$ for the last iteration\\
\hline 
$  Err^i_{ i, i_1,i_2,i_3,i_4}(t_1, t_2)$ & \eqref{2022feb22eqn64} &  Error type terms in  the last iteration; estimated in Lemma   \ref{2022errhorstep5}\\
\hline
$ {}_{}^5\mathfrak{H}(t_1, t_2)$ & \eqref{2022feb22eqn31} &  The main hyperbolic parts in the final iteration;\\
& &  estimated directly without further iteration; see\eqref{oct12eqn98}. \\
\hline
$  \mathfrak{E}^{\mu, 4;\kappa }_{k,j,n}(t_1, t_2)$ & \eqref{oct3eqn111} & Further decomposition for $  \mathfrak{E}^{\mu, 4  }_{k,j,n}(t_1, t_2)$ based on the time resonance set.\\
\hline
$Err^{\mu,a}_{k,j;n}(t_1, t_2)$, $Err^{\mu,4;\kappa }_{k,j,n}(t_1, t_2)$ & \eqref{oct7eqn62} & Error terms in the first iteration for  \\
& & the elliptic parts; estimated  in Lemma \ref{error3part1} \\
\hline
$ {}_{b}^1\mathfrak{E}^{\mu,a}_{k,j;n}(t_1, t_2)$, $ {}_{b}^1 \mathfrak{E}^{\mu,4;\kappa }_{k,j;n}(t_1, t_2)$ &\eqref{oct29eqn82}, \eqref{2024oct28eqn1}  & The main terms  in the first iteration for  the elliptic parts\\
\hline
$ \mathfrak{HE}_{k_1,j_1;n_1}^{i,i_1, \mu_1}(t_1,t_2)$ & \eqref{2022feb25eqn61} & The hyperbolic-elliptic interaction of $ {}_{1}^1\mathfrak{EH}^{\mu,a}_{k,j;n}(t_1, t_2)$ and $ {}_{1}^1 \mathfrak{E}^{\mu,4;\kappa }_{k,j;n}(t_1, t_2) $  \\
\hline
$\mathfrak{EE}_{k_1,j_1;n_1}^{i,i_1, \mu_1}(t_1, t_2) $ & \eqref{2026may1eqn1} & The elliptic-elliptic interaction of $ {}_{1}^1\mathfrak{EH}^{\mu,a}_{k,j;n}(t_1, t_2)$ and $ {}_{1}^1 \mathfrak{E}^{\mu,4;\kappa }_{k,j;n}(t_1, t_2) $  \\
\hline
\end{tabular}%
}
\caption{Essential notations in section \ref{mainimprovedfull}.}\label{tablesection5}
\end{table}

\section{ISS Part II:   Proof of   Proposition \ref{bootstraplemma2}}\label{fullimproved}

 As in Section \ref{mainimprovedfull}, we rigorously prove that the desired estimate \eqref{oct29eqn70} from Proposition \ref{bootstraplemma2} also holds for a general class of coefficients. 

 Assume that $C(\cdot, \cdot):\R^2 \times \R^3\rightarrow \R^3$ is a smooth function s.t., the following estimate holds for any $  s \in [t_1, t_2]$,
\be\label{2024oct30eqn51}
 \big|C(  x_{\bot}, V(s))\big|   + |V(s)| \big|\nabla_{v}C( x_{\bot}, V(s))|_{v = V(s)} \big|\lesssim \mathcal{M}(C),\quad \nabla_{  x_{\bot}} C( x_{\bot}, V(s))=0.
\ee
Under the above assumptions on the coefficients, we prove the following estimate:
\be\label{2025oct16eqn15eqn3}
\begin{split}
&\big|\int_{t_1}^{t_2} C(  X_{\bot}(s), V(s)) \cdot   K(s, X(s), V(s)) d s \big|   \lesssim  \mathcal{M}(C)   2^{   (\gamma-5\epsilon) M_t } .
\end{split}
\ee
Direct computation shows that the coefficient $ ({V_1(s)}/{| V(s)|},   {V_2(s)}/{| V(s)|},{V_3(s)}/{| V(s)|}\big)$ appeared  in \eqref{oct29eqn70} satisfies the assumptions in \eqref{2024oct30eqn51} with the bound $\mathcal{M}(C)=1.$

Note that \eqref{2025oct16eqn15eqn2} and \eqref{2025oct16eqn15eqn3} provide estimates for the same quantity. The only difference lies in the bootstrap assumptions regarding the velocity characteristics and the coefficient $C$. Specifically, the assumption $\nabla_{ x_{\bot}} C(  x_{\bot}, V(t)) = 0$ in \eqref{2024oct30eqn51} reduces the number of terms in the iterative process. To maintain notational consistency, we adopt the notation from Section \ref{mainimprovedfull}.

\subsection{The first reduction}

  Define
\be\label{oct15eqn1} 
\begin{split}
\widetilde{\mathcal{E}}_0 =\widetilde{\mathcal{E}}'_1 =\widetilde{\mathcal{E}}_2&:=\big(\{( k,n):  k\in \Z_+,  n\in    [- (\alpha^{\star} +3\iota+ 60\epsilon ) M_t  ,2]\cap\Z,   \\ 
& \quad k+2n\geq  2(1/3-4\iota -130\epsilon)M_t, k+4n\geq  -2(6\iota+130\epsilon) M_t   \},\\  
  \widetilde{\mathcal{E}}_1&:=\{(  k,n):  k\in \Z_+,  n\in    [- (\alpha^{\star}  +3\iota+ 60\epsilon ) M_t,2]\cap\Z,   \\ 
 &\quad   k+2n\geq 2(2/3- 2\iota  -130\epsilon)M_t   ,    k+4n\geq  -2(6\iota+130\epsilon) M_t    \}, \\
    \widetilde{\mathcal{E}}_3=   \widetilde{\mathcal{E}}_4 &:=\{(    k,n):  k\in \Z_+, n\in    [-   ((1+3\iota)/2+50\epsilon) M_t,2]\cap\Z, \\ 
 & \quad  k+2n\geq  2(1/3-   \iota -130\epsilon)M_t, k+4n\geq   -(1/3 + 5\iota/2 +130\epsilon) M_t  \}.  
\end{split}
\ee

Recall the estimate \eqref{oct2eqn41}.   From the estimate \eqref{2022feb25eqn1} in Theorem \ref{maintheoremellipitic}  and    the estimates in  \eqref{2024oct8eqn1} in Theorem \ref{mainresultsfirstpart},  the following estimate holds for any $i\in \{0,2,3,4\}, (n,k)\in (\widetilde{\mathcal{E}}_i)^c,$ or $i=1, (n,k)\in (\widetilde{\mathcal{E}}'_1)^c$,
\[
 \big| \int_{t_1}^{t_2}    C( X_{\bot}(s), V(s)) \cdot   T_{k,j, n}^{\mu,i}(s,X(s), V(s)) ds \big|\lesssim  \mathcal{M}(C)  2^{   (\gamma-10\epsilon) M_t }  .
\]

Now, it remains to estimate the case  $i\in \{0,2,3,4\}, (n,k)\in  \widetilde{\mathcal{E}}_i ,$ or $i=1, (n,k)\in  \widetilde{\mathcal{E}}'_1  $.  For this case, we use the decomposition of the localized acceleration force in   \eqref{oct7eqn1}. Note that, from     the second estimate in  \eqref{2024oct8eqn1} in Theorem \ref{mainresultsfirstpart},     if $i=1, (n_1,k_1)\in (\widetilde{\mathcal{E}}_1)^c$,  we have 
\be
\big|  \int_{t_1}^{t_2}     C (  X_{\bot}(s), V(s))  \cdot K_{k,j;n}^{\mu,i } (s, X(s), V(s)) d s  \big|\lesssim    \mathcal{M}(C)  2^{   (\gamma-10\epsilon) M_t }. 
\ee

We emphasize that despite the range of $(n,k)$ for the case $i=1$ is improved for the hyperbolic parts, the range of $(n,k)$ remains same, i.e., $(n,k)\in  \widetilde{\mathcal{E}}'_1$ for the case $i=1$, for the elliptic parts.  

As in  \eqref{oct25eqn2} for the case $i\in\{0,1,2,3,4\}$, $(n,k)\in \widetilde{\mathcal{E}}_i$, we do integration by parts in  ``$s$'' once for the hyperbolic part. As a result, the obtained equality in  \eqref{oct25eqn41}  is still valid.  

The   proof of Proposition \ref{bootstraplemma2} is organized as follows.
\begin{enumerate}
\item[$\bullet$] In   section \ref{hyperbolipartISSpartII}, we estimate the hyperbolic type terms ${}_a^{1}\mathfrak{H}^{\mu,i}_{k,j;n}(t_1, t_2), $ $i\in\{0,1,2,3,4\}$, $a\in\{0,1\},$ in  \eqref{oct25eqn41}. 
\item[$\bullet$]In   section \ref{ellipticpartPartII}, we estimate the elliptic type terms $ \mathfrak{E}^{\mu,i}_{k,j;n}(t_1, t_2), i\in\{0,1,2,3,4\}$  in  \eqref{oct25eqn41}. 
\item[$\bullet$] In section \ref{errorcpartPartII}, we estimate the error terms $ {}_a^{1}Err^{\mu,i}_{k,j;n}(t_1, t_2), a\in\{0,1\}$, in \eqref{oct25eqn41},   and the error terms  produced in the ISS process of estimating the hyperbolic parts and the elliptic parts. 
\end{enumerate}

 \subsection{Estimating the hyperbolic parts}\label{hyperbolipartISSpartII}

Same as in section \ref{mainimprovedfull}, c.f., \eqref{2021dec24eqn1},   from the rough estimate of   the electromagnetic field \eqref{maintheoremroughest}   in Theorem \ref{maintheorem1part1}, thanks to the smoothing effect from the normal form transformation, the very large frequency case, e.g., $k\geq 50M_t$, can be safely discarded.  Therefore, it suffices to consider fixed $i\in\{0,1,2,3,4\},(k,n)\in \widetilde{\mathcal{E}}_i  $.     We will also use this observation in later argument. Hence, without further explanation, by paying the price of at most $2^{\epsilon M_t/1000}$,  we don't worry about the  summability issue with respect to frequency and it suffices to let the frequency scale variables, e.g., $k,k_1,k_2,$ etc., be fixed. 

In the following Lemma, we first  estimate  $ {}_0^{1}\mathfrak{H}^{\mu,i}_{k,j;n}(t_1, t_2) $.   

\begin{lemma}
Let $ i\in \{0, 1,2,3,4\},   ( n,k)\in \widetilde{\mathcal{E}}_i$. Under the assumption of Proposition \textup{\ref{bootstraplemma2}}, we have
\be\label{2024oct29eqn1}
\big| {}_0^{1}\mathfrak{H}^{\mu,i}_{k,j;n}(t_1, t_2)  \big|\lesssim \mathcal{M}(C) 2^{(\gamma-10\epsilon)M_t} . \
\ee
\end{lemma}
\begin{proof}

   Recall \eqref{nov12eqn61} and  the decomposition of the acceleration force in  \eqref{oct7eqn1}. 

   \medskip
\noindent \textbf{Step 1.}\quad The first reduction: contribution from the elliptic part. 
   \medskip

We first estimate the contribution of the elliptic part of the new introduced acceleration force in  \eqref{nov12eqn61}.   From the $L^\infty_{x}$-estimate of the elliptic parts  in     \eqref{2022feb25eqn1} in Theorem \ref{maintheoremellipitic} and   the $L^\infty_{x}$-estimate of the hyperbolic parts  in \eqref{2024oct8eqn1} in Theorem \ref{mainresultsfirstpart}, we have  
\be 
\begin{split}
&\sum_{\begin{subarray}{c}
k_1, j_1\in \Z_+, n_1\in [-M_t,2]\cap \Z\\
 \mu_1\in\{+,-\},a_1\in \{0,1,2,3\}  \\  
\end{subarray} }  \big| \int_{t_1}^{t_2} \int_{\R^3} e^{i X(s)\cdot \xi + i \mu s|\xi|} \\
&\quad \times  \mathfrak{E}^{\mu_1, a_1 }_{k_1,j_1;n_1} (s, X(s), V(s)) \cdot \mathcal{F}[ {}^1  \mathfrak{H} ](s, \xi,  X_{\bot}(s), V(s))  d \xi d s\big|\\
&\lesssim    2^{5\alpha^{\star}  M_t/3 +10\epsilon M_t}    2^{-\gamma M_t-n} 2^{  (\alpha^\star + 3\iota+130\epsilon)M_t- (k+2n)/2  } \mathcal{M}(C)\\
 &\lesssim 2^{(\gamma-10\epsilon)M_t} \mathcal{M}(C). \\
\end{split}
\ee
Similar to the obtained estimate \eqref{2024oct29eqn11}, from the above estimate, we have
\be\label{2024oct29eqn12}
\begin{split}
\big|{}_0^{1}\mathfrak{H}_{k,j;n}(t_1, t_2)\big|&\lesssim \sum_{\begin{subarray}{c}
k_1 \in \Z_+, j_1\in [0, (1+2\epsilon)M_t]\cap \Z,  \mu_1\in\{+,-\} \\ 
 i_1\in\{0,1,2,3,4\}, n_1\in [-M_t,2]\cap \Z, 
\end{subarray}}     |\mathfrak{H}_{ i, i_1}(t_1, t_2)| +2^{(\gamma-10\epsilon)M_t} \mathcal{M}(C),
\end{split}
\ee
where $\mathfrak{H}_{ i, i_1}(t_1, t_2)$ is defined in \eqref{2024oct29eqn11}. 

 Let
\be\label{firstindexsetlemm2}
\begin{split}
\forall i_1\in \{0,1,2\}, \quad {}^{1}\widetilde{\mathcal{E}}_{k,n}^{i_1,i}&:=\{(k_1,n_1): k_1 \in \Z_+, n_1\in [- (\alpha^{\star} +3\iota+ 60\epsilon ) M_t, 2]\cap \Z,\\    
 &\quad (k_1+2n_1)/2 - (k+2n)/2\geq  (1/3 
- \mathbf{1}_{i\in \{3,4\}}/6-10\iota   -300\epsilon) M_t\},\\
 \forall i_1\in \{3,4\},  {}^{1}\widetilde{\mathcal{E}}_{k,n}^{i_1,i} & :=\{(k_1,n_1): k_1 \in \Z_+, n_1\in [-  ((1+3\iota)/2+50\epsilon) M_t, 2]\cap \Z,  \\ 
 &\quad (k_1+2n_1)/2- (k+2n)/2\geq  ( 1/3
- \mathbf{1}_{i\in \{3,4\}}/6-7\iota   -200\epsilon) M_t\}. 
\end{split}
\ee

From  the $L^\infty_{x}$-estimate of the hyperbolic parts  in \eqref{2024oct8eqn1} in Theorem \ref{mainresultsfirstpart},  for  $i\in \{0,1, 2,3,4\},$ $(k_1, n_1)\notin  {}^{1}\widetilde{\mathcal{E}}_{k,n}^{i_1,i}$, we have 
\be 
\begin{split}
  \big| \mathfrak{H}_{ i, i_1}(t_1, t_2)\big|  
 & \lesssim  \mathcal{M}(C) 2^{-\gamma M_t-k-3n+170\epsilon M_t}   \big[ 2^{(\gamma-10\epsilon)M_t} + 2^{ (k_1+2n_1)/2+ (2/3+4\iota ) M_t }  \\
 &\quad \times \big(\mathbf{1}_{n_1\geq -  (\alpha^{\star} +3\iota+ 60\epsilon ) M_t, i_1\in\{0,1,2\}} +  \mathbf{1}_{n_1\geq  -  ((1+3\iota)/2+30\epsilon) M_t, i_1\in\{3,4\} }\big)\big]  \\
&\quad \times \big( 2^{(\gamma-10\epsilon)M_t} +  2^{(k+4n)/2+(1+6\iota )M_t}\mathbf{1}_{  i \in\{0,1,2\}  } +2^{(k+4n)/2+ (7 /6+5\iota/2)M_t}\mathbf{1}_{ i \in\{3,4\} }\big)  \\
 &\lesssim 2^{ (\gamma-10\epsilon)M_t}\mathcal{M}(C).\\
 \end{split}
\ee

   \medskip
\noindent \textbf{Step 2.}\quad The second iteration of smoothing. 
   \medskip

Now, it remains to   estimate  $  \mathfrak{H}_{ i, i_1}(t_1, t_2)$    for  the case $i\in\{0,1,2,3,4\}$, $(k,n)\in \widetilde{\mathcal{E}}_{i}$,  $ (k_1, n_1)\in  {}^{1}\widetilde{\mathcal{E}}_{k,n}^{i_1,i}$. For this case, we do integration by parts in time once. As a result, the equality in  \eqref{oct29eqn55} is still valid. 

 The estimate of the error type terms are postponed  to  Lemma \ref{errmainfull1}. Therefore, it suffices to estimate $  \mathfrak{H}_{ i, i_1,i_2}(t_1, t_2)$ in \eqref{oct29eqn55}. Let 
\be\label{2026may1eqn11}
\begin{split}
{}^{2}\widetilde{\mathcal{E}}_{k_1,n_1} :=\{(k_2,n_2) &: k_2 \in \Z_+, n_2\in [- (\alpha^{\star} +3\iota + 60\epsilon ) M_t, 2]\cap \Z,    \\
&\quad (k_2+2n_2)/2- (k_1+2n_1)/2  \geq  (1/3-19\iota)M_t  \}.\\
\end{split}
\ee

From  the $L^\infty_{x}$-estimate of the hyperbolic parts  in \eqref{2024oct8eqn1} in Theorem \ref{mainresultsfirstpart},   for any  $i, i_1,i_2\in \{0,1, 2, 3,4\},$ $(k,n)\in \widetilde{\mathcal{E}}_{i}$,  $ (k_1, n_1)\in  {}^{1}\widetilde{\mathcal{E}}_{k,n}^{i_1,i}$, $(k_2, n_2)\notin  {}^{2}\widetilde{\mathcal{E}}_{k_1,n_1} $,  we have 
\be\label{2021dec29eqn31}
\begin{split}
 | \mathfrak{H}_{ i, i_1,i_2}(t_1, t_2)| 
& \lesssim \mathcal{M}(C)2^{500\epsilon M_t}  \big( 2^{(1-\epsilon)M_t} + 2^{(k_2+2n_2)/2+  (2/3+4\iota) M_t  }\mathbf{1}_{n_2\geq  - (\alpha^{\star} +3\iota + 60\epsilon ) M_t } \big)   \\
&\quad  \times 2^{-2\gamma_2 M_t  } \big[2^{ 7M_t/6-(k_1+2n_1)/2 + 7M_t/6-(k+2n)/2 + 5\iota M_t}  \\
&\quad  +  2^{ (2/3+4\iota) M_t -(k_1+2n_1)/2 } 2^{7M_t/6+ 5\iota M_t/2-(k+4n)/2}  \big]\\
& \lesssim \mathcal{M}(C) 2^{(\gamma-10\epsilon)M_t}. 
\end{split}
\ee
 
   \medskip
\noindent \textbf{Step 3.}\quad The third iteration of smoothing. 
   \medskip

Now, it remains to estimate  $ \mathfrak{H}_{ i, i_1,i_2}(t_1, t_2)$    for  the case  $i, i_1,i_2\in \{0,1, 2, 3,4\},$ $(k,n)\in \widetilde{\mathcal{E}}_{i}$,  $ (k_1, n_1)\in  {}^{1}\widetilde{\mathcal{E}}_{k,n}^{i_1, i}$, $(k_2, n_2)\in  {}^{2}\widetilde{\mathcal{E}}_{k_1,n_1} $. For this case,  we do integration by parts in time once.  As a result, the equality  in  \eqref{oct29eqn61}  is still valid.

 The estimate of the error type terms are postponed  to  Lemma \ref{errortypefull6}. Therefore, it suffices to estimate $  \mathfrak{H}_{ i, i_1,i_2,i_3}(t_1, t_2)$ in \eqref{oct29eqn55}.  Let 
\be\label{2026may1eqn12}
\begin{split}
{}^{3}\widetilde{\mathcal{E}}_{k_2,n_2} :=\{(k_3,n_3)&:  k_3 \in \Z_+, n_3\in [- (\alpha^{\star} + 50\epsilon ) M_t, 2]\cap \Z,    \\
 & (k_3+2n_3)/2- (k_2+2n_2)/2\geq  (1/2-50\iota)M_t  \} .
 \end{split}
\ee

From  the $L^\infty_{x}$-estimate of the hyperbolic parts  in \eqref{2024oct8eqn1} in Theorem \ref{mainresultsfirstpart},    for any $ (k_3, n_3)\notin  {}^{3}\widetilde{\mathcal{E}}_{k_2,n_2}$, we have 
\[
\begin{split}
| \mathfrak{H}_{ i, i_1,i_2,i_3}(t_1, t_2)| & \lesssim \mathcal{M}(C)2^{700\epsilon M_t}   \big( 2^{(1-\epsilon)M_t} + 2^{(k_3+2n_3)/2+ (2/3+4\iota) M_t} \mathbf{1}_{n_3\geq - (\alpha^{\star} +3\iota + 60\epsilon ) M_t  } \big)  \\
 &\quad \times 2^{-3\gamma  M_t}  2^{2(2/3+4\iota)   M_t-(k_2+2n_2)/2}   2^{(7/3+ 12\iota )M_t -(k_1+2n_1)/2  -(k+2n)/2} \\
 &  \lesssim \mathcal{M}(C) 2^{(\gamma-10\epsilon)M_t}.  
 \end{split}
\]

   \medskip
\noindent \textbf{Step 4.}\quad The fourth iteration of smoothing. 
   \medskip

Now, it remains to estimate  $  \mathfrak{H}_{ i, i_1,i_2,i_3}(t_1, t_2)$ for any fixed
   $i, i_1,i_2,i_3\in \{0,1, 2, 3,4\},$ $(k,n)\in \widetilde{\mathcal{E}}_{i}$,  $ (k_1, n_1)\in  {}^{1}\widetilde{\mathcal{E}}_{k,n}^{i_1,i}$, $(k_2, n_2)\in  {}^{2}\widetilde{\mathcal{E}}_{k_1,n_1} $ $ (k_3, n_3)\in  {}^{3}\widetilde{\mathcal{E}}_{k_2,n_2}$.  For this case,  we do integration by parts in time once.  As a result, the equality  in  \eqref{oct12eqn76}  is still valid.

 The estimate of the error type terms are postponed  to  Lemma \ref{errortypefull8}.  From the rough estimate of the electromagnetic field \eqref{maintheoremroughest} in Theorem \ref{maintheorem1part1} and the estimate \eqref{2022feb25eqn1} in Theorem \ref{maintheoremellipitic} , we have 
 \[
 \begin{split}
 |  \mathfrak{H}_{ i, i_1,i_2,i_3,i_4}(t_1, t_2) |&\lesssim \mathcal{M}(C)   2^{(1+2\alpha^{\star})M_t } 2^{-4\gamma M_t} 2^{  2(2/3+4\iota) M_t-(k_2+2n_2)/2} \\
 &\quad \times   2^{ 2(2/3+4\iota)  M_t-(k_3+2n_3)/2}  2^{(7/3+15\iota )M_t -(k_1+2n_1)/2  -(k+2n)/2}  \\
 & \lesssim \mathcal{M}(C)2^{M_t/2}\\
 &  \lesssim \mathcal{M}(C) 2^{(\gamma-10\epsilon)M_t}.  \\
 \end{split} 
 \]
 Hence finishing the      desired  estimate of  ${}_0^{1}\mathfrak{H}^{\mu,i}_{k,j;n}(t_1, t_2) $ in \eqref{2024oct29eqn1}.  

\end{proof}

 \begin{lemma}
Let $ i\in \{0, 1,2,3,4\},   ( n,k)\in \widetilde{\mathcal{E}}_i$. Under the assumption of Proposition \textup{\ref{bootstraplemma2}}, we have
\be\label{2024oct29eqn31}
\big| {}_1^{1}\mathfrak{H}^{\mu,i}_{k,j;n}(t_1, t_2)  \big|\lesssim \mathcal{M}(C) 2^{(\gamma-10\epsilon)M_t} . \
\ee
\end{lemma}
\begin{proof}

Recall \eqref{nov12eqn61}.  We split $ {}_1^{1}\mathfrak{H}^{\mu,i}_{k,j;n}(t_1, t_2) $ into two parts as follows, 

\be\label{oct21eqn21}
{}_1^{1}\mathfrak{H}^{\mu,i}_{k,j;n}(t_1, t_2) =   {}_{1}^{1}\widetilde{\mathfrak{H}}^{\mu,i}_{k,j;n}(t_1, t_2) +{}_{1}^{1}Err\mathfrak{H}^{\mu,i}_{k,j;n}(t_1, t_2)  ,
\ee
where
\be\label{oct29eqn91}
\begin{split}
{}_{1}^{1}\widetilde{\mathfrak{H}}^{\mu,i}_{k,j;n}(t_1, t_2)&:=\int_{t_1}^{t_2} \int_{\R^3}\int_{\R^3} e^{i X(s)\cdot \xi   } 
  |\xi|^{-1}   ({ i \hat{V}(s)\cdot \xi + i \mu |\xi| })^{-1} \\
&\quad  \times \nabla_v \big( C (  X_{\bot}(s), V(s))\cdot
    \tilde{\varphi}_{k,j,n}^i (v,\xi,  V(s))\big)   \\
&\quad \cdot \mathcal{F}\big(K(s,\cdot, V(s)) f\big)(s, \xi, v)
d\xi d v ds,\\
{}_{1}^{1}Err\mathfrak{H}^{\mu,i}_{k,j;n}(t_1, t_2)&:=\int_{t_1}^{t_2} \int_{\R^3}\int_{\R^3} e^{i X(s)\cdot \xi   }  |\xi|^{-1}   ({ i \hat{V}(s)\cdot \xi + i \mu |\xi| })^{-1}\\
 &\quad \times ((\hat{v}-\hat{V}(s))\times B f)\cdot \nabla_v\big( C ( X_{\bot}(s), V(s))\cdot
    \tilde{\varphi}_{k,j,n}^i (v,\xi,  V(s))\big) d\xi d v d s.\\
\end{split}
\ee

The error term $ {}_{1}^{1}Err\mathfrak{H}^{\mu,i}_{k,j;n}(t_1, t_2)$ will be estimated in Lemma \ref{errortypefull}. The hyperbolic part  $ {}_{1}^{1}\widetilde{\mathfrak{H}}^{\mu,i}_{k,j;n}(t_1, t_2)  $  exhibits a structure analogous to a second-generation estimate of the elliptic parts. Both are estimated using the same method in Lemma \ref{errortypefullhyp2}.

 Therefore, from the estimate  \eqref{oct19eqn31}  in Lemma \ref{errortypefull} and the estimate  \eqref{2024oct30eqn52}  in Lemma \ref{errortypefullhyp2},   for any $ i\in \{0,1,2,3,4\},   ( n,k)\in \widetilde{\mathcal{E}}_i,$  $ \mu\in\{+,-\},$  we have 
\be\label{oct29eqn50}
\big|{}_1^{1}\mathfrak{H}^{\mu,i}_{k,j;n}(t_1, t_2)\big|   \lesssim \mathcal{M}(C) 2^{(\gamma-10\epsilon)M_t}.  
\ee
This finishing the proof of  the desired estimate \eqref{2024oct29eqn31}. 
\end{proof}

\subsection{Estimating the elliptic parts}\label{ellipticpartPartII}

As summarized in the following Lemma, the goal of this section is to estimate the elliptic parts, i.e., $ \mathfrak{E}^{\mu,i}_{k,j;n}(t_1, t_2), i\in\{0,1,2,3,4\} $, in \eqref{oct25eqn41}.   

\begin{lemma}\label{ellipticestpartII}
Let $i\in\{0,   2,3,4 \},     (k,n)\in \widetilde{\mathcal{E}}_i$ or $i=1, (k,n)\in \widetilde{\mathcal{E}}'_1$, $ j\in [0, (1+2\epsilon)M_t]\cap \Z_+$, $\mu\in\{+, -\}.$ Under the assumption of Proposition \textup{\ref{bootstraplemma2}}, we have
\be\label{2024oct29eqn111}
\big| \mathfrak{E}^{\mu,i}_{k,j;n}(t_1, t_2)  \big|\lesssim \mathcal{M}(C) 2^{(\gamma-9.9\epsilon)M_t} . \
\ee
\end{lemma}
\begin{proof}
Postponed to the end of this section, see section \ref{proofofellipticestpartII}.
\end{proof}

\subsubsection{Ruling out some trivial cases. }\label{trivialcasespartIIISS}
 
  Recall   \eqref{oct25eqn1}. For any $s\in [t_1,t_2]$, we use the following partition of unity,
\be\label{2026may1eqn21}
1= \sum_{l\in [-M_t, 2]\cap \Z}\varphi_{l;-M_t}(\tilde{v}-\tilde{V}(s)),
\ee
 to  decompose  $\mathfrak{E}^{\mu,i}_{k,j;n}(t_1, t_2) $ dyadically further based on the size of $\tilde{v}-\tilde{V}(s)$ as follows, 
  \be\label{2024oct29eqn41}
   \mathfrak{E}^{\mu,i}_{k,j;n}(t_1, t_2)=\sum_{l\in [-M_t, 2]\cap \Z}  \mathfrak{E}^{\mu,i;l}_{k,j;n}(t_1, t_2),
  \ee
where the formulas of  $ \mathfrak{E}^{\mu,i;l}_{k,j;n}(t_1, t_2)$ can be obtained from  \eqref{oct25eqn1} and \eqref{2022feb24eqn81} straightforwardly. Hence we omit them here.

As in Lemma \ref{trivialcaseell},  we first rule out some range of parameters of $k,n, j$ etc, in the following Lemma.
\begin{lemma}\label{ellfulltyp1}
Let $i\in\{0,   2,3,4 \},     (k,n)\in \widetilde{\mathcal{E}}_i$ or $i=1, (k,n)\in \widetilde{\mathcal{E}}'_1$, $ j\in [0, (1+2\epsilon)M_t]\cap \Z_+$, $\mu\in\{+, -\},$ $l\in[-M_t,2]\cap\Z$. If $k+2n\leq 2(1- \alpha^{\star} )M_t - 3 0\epsilon M_t $, or  $k+4n/3\leq M_t- \alpha^{\star}  M_t/3-30 \epsilon  M_t,$  or $ \min\{l,n\}\leq (1-2 \alpha^{\star} )M_t-30\epsilon M_t$, or $j\leq 3M_t/5-20\epsilon M_t,$ or $k\geq 2\alpha^{\star}M_t+20\epsilon M_t,$ we  have 
 \be\label{oct20eqn12}
|   \mathfrak{E}^{\mu,i;l}_{k,j;n}(t_1, t_2)|\lesssim \mathcal{M}(C) 2^{(\gamma-10\epsilon)M_t}.
\ee

\end{lemma}
\begin{proof}

Recall  \eqref{oct25eqn1} and \eqref{2024oct29eqn41}.
 Note that, after writing  $\mathfrak{E}^{\mu,i;l}_{k,j;n}(t_1, t_2)$ in terms of kernel,   doing integration by parts in $\xi$ along $V(s)$ direction and directions perpendicular to $V(s)$ for the kernel, and using the volume of support of $v$ and the estimate  \eqref{nov24eqn41}  if $|v|\geq 2^{(\alpha_s+\epsilon)M_s}$,  the following estimate holds 
 if $k+2n\leq 2(1- \alpha^{\star}  )M_t-30\epsilon M_t$, or $k+4n/3\leq (1-\alpha^{\star} /3)M_t-30\epsilon M_t $  or $\min\{l,n\}\leq (1-2\alpha^{\star} )M_t-30\epsilon M_t$,  or $j\leq 3M_t/5-20\epsilon M_t,$ 
 \be\label{oct20eqn11}
 \begin{split}
  |    \mathfrak{E}^{\mu,i;l}_{k,j;n}(t_1, t_2)|  &\lesssim \mathcal{M}(C)  2^{2k+ \min\{l,n\}+n+3\epsilon M_t}     \min\big\{2^{-j}  , 2^{-3k-2n}\min\{2^{j+2 \alpha^{\star}   M_t}, 2^{3j}\} \big\}\\
 &\lesssim \mathcal{M}(C)  2^{2k+ \min\{l,n\}+n+3\epsilon M_t}   \min\big\{  2^{-3k-2n +j+2 \alpha^{\star}   M_t}, (2^{-j})^{1/2} (2^{-3k-2n+j+2 \alpha^{\star}   M_t})^{1/2}, \\
 &\quad \big( 2^{-j}  \big)^{1/3}\big(2^{-3k-2n}\min\{2^{j+2 \alpha^{\star}   M_t}, 2^{3j}\}\big)^{2/3}, \big( 2^{-j}  \big)^{2/3}\big(2^{-3k-2n} 2^{2j+  \alpha^{\star}   M_t} \big)^{1/3} \big\}\\
  &\lesssim  \mathcal{M}(C)  2^{5\epsilon M_t}\min\big\{2^{-k+j+2\alpha^{\star}M_t}, 2^{(k+2n)/2+ \alpha^{\star} M_t}, 2^{5j/3}, 2^{k+4n/3+\alpha^{\star} M_t/3}, \\
  &\quad 2^{2\min\{l,n\}/3}\min\{2^{ M_t/3 + 4\alpha^{\star} M_t/3}\}\big\} \\
 & \lesssim \mathcal{M}(C) 2^{(\gamma-10\epsilon)M_t}.\\
 \end{split}
 \ee
 Hence finishing the proof of our desire estimate \eqref{oct20eqn12}.  
\end{proof}

Now, by using the singular weighted spacetime estimate \eqref{nov1eqn1}  in Lemma \ref{singularweigh}, we rule out the case  $(k+2n)/2-l\leq 5 M_t/12-2\iota M_t $.   More precisely, we have 

\begin{lemma}\label{firstreductiontrivialelli}
Let $i\in\{0,   2,3,4 \},     (k,n)\in \widetilde{\mathcal{E}}_i$ or $i=1, (k,n)\in \widetilde{\mathcal{E}}'_1$.    For any $k\in \Z, j\in [0, (1+2\epsilon)M_t]\cap\Z, n\in[-M_t, 2]\cap \Z,$ $l\in[-M_t,2]\cap\Z,$ s.t., $k+2n\geq  2(1- \alpha^{\star} )M_t - 3 0\epsilon M_t $,   $k+4n/3\geq M_t- \alpha^{\star}  M_t/3-30 \epsilon  M_t,$   $ \min\{l,n\}\geq (1-2 \alpha^{\star} )M_t-30\epsilon M_t$,   $j\geq 3M_t/5-20\epsilon M_t$, and $k\leq 2\alpha^{\star}M_t+20\epsilon M_t$. If $(k+2n)/2-l\leq 5 M_t/12-2\iota M_t $, then the following estimate holds,
\be\label{nov9eqn31}
|  \mathfrak{E}^{\mu,i;l}_{k,j;n}(t_1, t_2)|\lesssim  \mathcal{M}(C) 2^{(\gamma-10\epsilon)M_t}.
\ee
 
\end{lemma}
\begin{proof}

Recall  \eqref{oct25eqn1} and \eqref{2024oct29eqn41}. Since $k+2n\geq  2(1- \alpha^{\star} )M_t - 3 0\epsilon M_t$, we have $l\geq -M_t/12+\iota M_t$. For this case, we have $|v_{\bot}|/|v|\sim 2^l$.   From the cylindrical symmetry,  the volume of support of $v$, the estimate  \eqref{nov24eqn41}  if $|v_{\bot}|\geq 2^{(\alpha_s+\epsilon)M_s}$, and  the singular weighted space-time estimate  \eqref{nov1eqn1}  in Lemma \ref{singularweigh}, we have 
\be 
\begin{split}
|  \mathfrak{E}^{\mu,i;l}_{k,j;n}(t_1, t_2)| & \lesssim    \mathcal{M}(C) 2^{2k+n+\min\{l,n\} +\epsilon M_t}  \int_{t_1}^{t_2} \int_{\R^3}\int_{\R^3} (1+2^{k}|y\cdot \tilde{V}(s)|)^{-100}\\
&\quad \times (1+2^{k+n}|y\times  \tilde{V}(s)|)^{-100}  \psi_{[j+l-5, j+l+5]}(  v_{\bot} )    f(s, X(s)-y, v)  \psi_{j}(v) d y d v d s \\
& \lesssim   \mathcal{M}(C) 2^{\min\{l,n\}-n+5\epsilon M_t}\min\{2^{-k+3j+ 2l }, 2^{k+n-2(j+l)+j  } A_l(t_2) \}\\
&\lesssim \mathcal{M}(C) 2^{5\epsilon M_t} \big(2^{-k+3j+ 2l }\big)^{1/4}  \big(  2^{k+n-2(j+l)+j +\alpha^{\star}M_t +6\epsilon M_t} \big)^{3/4}\\
&\lesssim  \mathcal{M}(C)  2^{k/2+3n/4+3\alpha^{\star} M_t/4-l+ 10\epsilon M_t}\\
&\lesssim \mathcal{M}(C)  2^{(k+2n)/2-l+3\alpha^{\star} M_t/4-(1-2 \alpha^{\star} )M_t/4+ 30\epsilon M_t}  \lesssim  \mathcal{M}(C) 2^{(\gamma-10\epsilon)M_t}. 
\end{split}
\ee
 Hence finishing the proof of the desired estimate \eqref{nov9eqn31}.
\end{proof}

Now, it remains to   consider the case  $k+2n\geq  2(1- \alpha^{\star} )M_t - 3 0\epsilon M_t $,   $k+4n/3\geq M_t- \alpha^{\star}  M_t/3-30M_t,$       $ \min\{l,n\}\geq (1-2 \alpha^{\star} )M_t-30\epsilon M_t$,  $j\geq 3M_t/5-20\epsilon M_t$,    $k\leq 2\alpha^{\star}M_t+20\epsilon M_t$, and  $(k+2n)/2-l\geq 5 M_t/12-2\iota M_t $.

 If $i\in\{0,1,2,3\}$, then the good lower bounds of the phase in \eqref{2024oc27eqn81} and \eqref{2024oc27eqn82} are still valid. For the case  $i=4$,    based on the size of $\angle(v, V(s))$ and the size of $(\hat{v}-\hat{V}(s))\cdot \xi$,  we use the same decomposition in  \eqref{oct3eqn111} but with a different threshold  $ \bar{\kappa}:= -(5\alpha^{\star} -3)M_t-100\epsilon M_t$. As a result, for the case $\kappa \in ( \bar{\kappa}, 2]\cap \Z$, the associated phase of $  \mathfrak{E}^{\mu, 4;\kappa }_{k,j,n}(t_1, t_2)$  has a good lower bound, see \eqref{2024oct29eqn65}.

As in \eqref{2024oct29eqn41}, we have
\be\label{2024oct29eqn91}
\mathfrak{E}^{\mu, 4;\kappa }_{k,j,n}(t_1, t_2)= \sum_{l\in [n-2\epsilon M_t, n+2\epsilon M_t]\cap \Z}  \mathfrak{E}^{\mu,i;\kappa,l}_{k,j;n}(t_1, t_2).
\ee
In the above equality, the range of the index $l$ is $[n - 2\epsilon M_t, n + 2\epsilon M_t] \cap \mathbb{Z}$, not $[-M_t, 2] \cap \mathbb{Z}$, as a consequence of the cutoff function $\varphi^4_{j,n}(v, \zeta)$ (see \eqref{sep4eqn6}).

 In the following Lemma, we exclude  the cases $\kappa = \bar{\kappa}$ and $k + l + \kappa \le 2(1 - \alpha^*)M_t - 100\epsilon M_t$.
 \begin{lemma}\label{ellfulltyp3}
Let $\bar{\kappa}:= -(5\alpha^{\star} -3)M_t-100\epsilon M_t$, for any $     (k,n)\in \widetilde{\mathcal{E}}_4$, $j\in [0,(1+2\epsilon)M_t]\cap \Z, $ $l\in  [n-2\epsilon M_t, n+2\epsilon M_t]\cap \Z, $ s.t., $k+2n\geq  2(1- \alpha^{\star} )M_t - 3 0\epsilon M_t $,   $k+4n/3\geq M_t- \alpha^{\star}  M_t/3-30M_t,$   $ \min\{l,n\}\geq (1-2 \alpha^{\star} )M_t-30\epsilon M_t$,   $j\geq 3M_t/5-20\epsilon M_t$,  $k\leq 2\alpha^{\star}M_t+20\epsilon M_t$, and  $(k+2n)/2-l\geq 5 M_t/12-2\iota M_t $, we have 
\be\label{oct20eqn7} 
   \big|  \mathfrak{E}^{\mu,4;\bar{\kappa},l}_{k,j;n}(t_1, t_2)\big|   \lesssim \mathcal{M}(C)  2^{(\gamma-10\epsilon)M_t} .
\ee
Moreover, for any $\kappa\in (\bar{\kappa}, 2]\cap Z$, the following estimate holds if $k+l+\kappa\leq  2(1- \alpha^{\star}  )M_t-50\epsilon M_t, $ or $(2k+2l+3\min\{l,\kappa\})/4-2l\leq (1-3\alpha^{\star}M_t/4)-100\epsilon M_t$, or $\min\{l, \kappa\}\leq 7l/3-\iota M_t$, 
\be\label{nov11eqn231}
  \big|     \mathfrak{E}^{\mu,4;{\kappa},l}_{k,j;n}(t_1, t_2)\big|   \lesssim  \mathcal{M}(C) 2^{(\gamma-10\epsilon)M_t} .
\ee 
\end{lemma}
\begin{proof}
Recall  \eqref{oct3eqn111}  and  \eqref{2024oct29eqn91}.  Note that, after doing integration by parts with respect to the orthonormal frame $\{(\hat{v}-\hat{V}(s))/|\hat{v}-\hat{V}(s), \theta_1(v, s),\theta_2(v,s)\}$ (defined in  Lemma \ref{smallangest}) for the kernel,   from the volume of support of $v$ and the estimate  \eqref{nov24eqn41}  if $|  v_{\bot}|\geq 2^{(\alpha^\star+\epsilon)M_t}$, the following estimate holds, 
\[
  \begin{split}
 \big| &     \mathfrak{E}^{\mu,4;{\kappa},l}_{k,j;n}(t_1, t_2)\big|  
\lesssim    2^{2k+l +\min\{n,\kappa\} +\epsilon M_t} \int_{t_1}^{t_2} \int_{\R^3} \int_{\R^3}  \mathcal{M}(C) f(s, X(s)-y, v) \varphi_{l;-M_t}\big(\tilde{v}-\tilde{V}(s) \big) \\
&\quad \times \psi_j(v)( 1+ 2^{k+n}|y\cdot \theta_1(v, s)|)^{-100} (1+2^k|y\cdot \theta_2(v, s)|)^{-100}  \\
&\quad \times  (1+2^{k+\min\{n,\kappa\}}(y\cdot (\hat{v}-\hat{V}(s)/|\hat{v}-\hat{V}(s)|) ))^{-100} dy d v ds \\
 &\lesssim  \mathcal{M}(C) 2^{2k+l+\min\{n,\kappa\} +4\epsilon M_t} \min\{ 2^{-j}  ,  2^{-3k-n -\min\{n,\kappa\}}  2^{3j+2l}\}\\ 
&\lesssim 2^{14\epsilon M_t} \mathcal{M}(C) \min\{2^{(k+l+\kappa)/2+ \alpha^{\star} M_t},   2^{ \kappa/3   +5 \alpha^{\star} M_t/3}  \}.
 \end{split}
\]
Hence our desired estimate  \eqref{oct20eqn7} follows directly from the above estimate. Moreover, the desired estimate \eqref{nov11eqn231} holds if $k+l+\kappa\leq  2(1- \alpha^{\star}  )M_t-50\epsilon M_t $.

To prove  the desired estimate \eqref{nov11eqn231} for the other two cases, i.e., $(2k+2l+3\min\{l,\kappa\})/4-2l\leq (1-3\alpha^{\star}M_t/4)-100\epsilon M_t$, or $\min\{l, \kappa\}\leq 7l/3-\iota M_t$. It would be sufficient to consider the case $l \geq (\alpha^{\star}-\gamma+4\epsilon)M_t $, otherwise these two cases are covered by the restrictions $k+l+\kappa\leq  2(1- \alpha^{\star}  )M_t-50\epsilon M_t $ and $\kappa \geq \bar{\kappa}:= -(5\alpha^{\star} -3)M_t-100\epsilon M_t$.

  For the case $l \geq (\alpha^{\star}-\gamma+4\epsilon)M_t $, we have    $|  v_{\bot}|/|v|\sim 2^l$ since $|  V_{\bot}(s)|/|V(s)|\leq 2^{(\alpha^{\star}-\gamma )M_t}$. In particular, in view of  the estimate  \eqref{nov24eqn41}, it would be sufficient to consider the case $j+l\leq (\alpha^\star+\epsilon)M_t$. From the  first fact in   \eqref{2024feb13eqn1},  the first estimate of the Jacbobian in \eqref{jacobianthreshold}, the volume estimate \eqref{oct3eqn100},   and the estimate \eqref{singularweigh}  in Lemma \ref{singularweigh}, the following estimate holds for any $\kappa\in[\bar{\kappa},2]\cap\Z$ after dyadic localizing the size of  $\tilde{\theta}(v,s)\times  \big({(-X_2(s), X_1(s),0)}/{| X_{\bot}(s)|}\big) $ with the threshold $n$,
\be\label{thresholdelliptic}
\begin{split}
 \big|     \mathfrak{E}^{\mu,4;{\kappa},l}_{k,j;n}(t_1, t_2)\big| &   
 \lesssim \sum_{a\in[n,2]} \mathcal{M}(C) 2^{2k+l+\min\{n,\kappa\} +8\epsilon M_t} \min\{   2^{-3k-n -\min\{n,\kappa\}}  \min\{2^{3j+3n/2+3a/2}, 2^{3j+2l}  \},\\
&\quad \times 2^{-k-n-a+10\epsilon M_t} 2^{-2(j+l)+j+\alpha^{\star}M_t } \} \\
&\lesssim \sum_{a\in[n,2]} \mathcal{M}(C) 2^{2k+l+\min\{n,\kappa\}+18\epsilon M_t} \\
&\quad\times \min\big\{  (2^{-3k-n -\min\{n,\kappa\}} 2^{3j+3n/2+3a/2})^{1/4} (2^{-k-n-a } 2^{-2(j+l)+j+\alpha^{\star}M_t} )^{3/4}  ,\\ 
&\quad  (  2^{-k-n-a } 2^{-2(j+l)+j+\alpha^{\star}M_t })^{1/2} (2^{-3k-n -\min\{n,\kappa\}} (2^{3j+3n/2+3a/2})^{2/3}(2^{3j+2l })^{1/3} )^{1/2} \big\} \\
&\lesssim 2^{25\epsilon M_t} \mathcal{M}(C) \min\{  2^{(2k+2l+3\min\{l,\kappa\})/4-2l+3\alpha^{\star}M_t/4},   2^{3\alpha^{\star} M_t/2-7l/6+\min\{l,\kappa\}/2}  \}.
\end{split} 
\ee
Hence  the desired estimate \eqref{nov11eqn231} holds if $(2k+2l+3\min\{l,\kappa\})/4-2l\leq (1-3\alpha^{\star}M_t/4)-100\epsilon M_t$, or $\min\{l, \kappa\}\leq 7l/3-\iota M_t$. 
 
\end{proof} 

\subsubsection{The first iteration and the outline of proof}
 
For the remaining cases, we exploit the benefit of high  oscillation in time, we do  integration by parts in  ``$s$'' once for $ \mathfrak{E}^{\mu,i;l}_{k,j;n}(t_1, t_2), a\in\{0,1,2,3\},$ and $    \mathfrak{E}^{\mu,4;, \kappa, l}_{k,j;n}(t_1, t_2)$, $\kappa\in (\bar{\kappa}, 2]\cap\Z$.

 As a result, similar to   \eqref{2024oct23eqn11} and \eqref{2024oct23eqn12}, we have
 \be\label{2024oct31eqn81}
 \begin{split}
 \mathfrak{E}^{\mu,a;l}_{k,j,n}(t_1, t_2)& =  Err^{\mu,a;l}_{k,j;n}(t_1, t_2)  + \sum_{b=0,1}  {}_{b}^1\mathfrak{E}^{\mu,a;l}_{k,j;n}(t_1, t_2),\\
  \mathfrak{E}^{\mu, 4;\kappa, l}_{k,j,n}(t_1, t_2) & =  Err^{\mu,4;\kappa, l }_{k,j,n}(t_1, t_2)+ \sum_{b=0,1}  {}_{b}^1 \mathfrak{E}^{\mu,4;\kappa,l }_{k,j;n}(t_1, t_2).
 \end{split}
 \ee
For simplicity, we omit the detailed formulas of the above terms, which can easily be obtained from  
\eqref{oct7eqn62}, \eqref{oct29eqn82}, and \eqref{2024oct28eqn1}  with minor modification in the cutoff functions becuase of the new added $\varphi_{l;-M_t}(\tilde{v}-\tilde{V}(s))$ in  \eqref{2024oct29eqn41}.

The estimate of error type terms $ Err^{\mu,a;l}_{k,j;n}(t_1, t_2), a\in\{0,1,2,3\}$ and $Err^{\mu,4;\kappa, l }_{k,j,n}(t_1, t_2)$ are postponed to    Lemma \ref{ellfulltyp4}.  The estimate of $ {}_{1}^1\mathfrak{E}^{\mu,a;l}_{k,j;n}(t_1, t_2), a\in\{0,1,2,3\}$ and $  {}_{1}^1 \mathfrak{E}^{\mu,4;\kappa }_{k,j;n}(t_1, t_2)$ are postponed in Lemma \ref{errortypefullhyp2}.  

Now, we focus on the estimate of $ {}_{0}^1\mathfrak{E}^{\mu,a;l}_{k,j;n}(t_1, t_2), a\in\{0,1,2,3\}$ and $  {}_{0}^1 \mathfrak{E}^{\mu,4;\kappa }_{k,j;n}(t_1, t_2)$. For convenience, we refer these terms as $0$-part.

In contrast to the   estimate of $0$-part in Section \ref{mainpartsISSpart1seco}, the conditions of Proposition \ref{bootstraplemma2} preclude the use of cylindrical symmetry because the velocity characteristics can move along the $z$-axis.  The estimate of the $0$-part in  Proposition \ref{bootstraplemma2}  is much more complicated. We need to do a refined decomposition for the  $0$-part. Based on the size of $i$, we proceed in two steps as follows.

\medskip

\noindent \textbf{Step 1.} \quad The case $i\in\{0,1,2,3\}.$

\medskip

  For the term $ {}_{0}^1\mathfrak{E}^{\mu,a;l}_{k,j;n}(t_1, t_2)$ , we split it  further into two parts as follows, 
\be\label{nov9eqn62}
\begin{split}
 {}_{0}^1\mathfrak{E}^{\mu,a;l}_{k,j;n}(t_1, t_2)&=    {}_{0;0}^{\,\,1}\mathfrak{E}^{\mu,a;l}_{k,j;n}(t_1, t_2)+    {}_{0;1}^{\,\,1}\mathfrak{E}^{\mu,a;l}_{k,j;n}(t_1, t_2),\\
  {}_{0;0}^{\,\,1}\mathfrak{E}^{\mu,a;l}_{k,j;n}(t_1, t_2)& =  \int_{t_1}^{t_2}\int_{\R^3}   \int_{\R^3} e^{i X(s )\cdot \xi    }   \mathcal{F}\big[ K(s, \cdot, V(s) f(s, \cdot, v) \big]( \xi   )\\
 &\quad  \cdot   \nabla_v \big[  \mathcal{F}[ {}_{}^{1}\mathfrak{E}^{a} ]( X_{\bot}(s),   \xi, v, \zeta) \big]  d \xi d v d s,\\
  {}_{0;1}^{\,\,1}\mathfrak{E}^{\mu,a;l}_{k,j;n}(t_1, t_2)& =  \int_{t_1}^{t_2}  \int_{\R^3}  \int_{\R^3} e^{i X(s )\cdot \xi    }   \mathcal{F}\big[  \big((\hat{v}-\hat{V}(s))\times B(s,\cdot) \big) f(s, \cdot, v) \big]( \xi   ) \\
 &\quad \cdot   \nabla_v \big[ \mathcal{F}[ {}_{}^{1}\mathfrak{E}^{a} ](  X_{\bot}(s),   \xi, v, \zeta) \big]  d \xi d v d s.
 \end{split}
\ee 

The term   $  {}_{0;0}^{\,\,1}\mathfrak{E}^{\mu,a;l}_{k,j;n}(t_1, t_2)$ will be estimated in Lemma \ref{errortypefullhyp2} and the term $  {}_{0;1}^{\,\,1}\mathfrak{E}^{\mu,a;l}_{k,j;n}(t_1, t_2)$ will be estimated in  Lemma \ref{ellipticstep2good}.
\medskip

\noindent \textbf{Step 2.} \quad The case $i=4.$

\medskip

Since we have a much more complicated phase for the case $i=4$. The estimate of $ {}_{0}^1 \mathfrak{E}^{\mu,4;\kappa }_{k,j;n}(t_1, t_2) $ is more delicate.   We not only use the decomposition of the acceleration force used in \eqref{nov9eqn62} for the case $i\in\{0,1,2,3\}$ but also use the decomposition of the third component of the magnetic field in   \eqref{nov6eqn47}. As a result, we have 
\be\label{findecompellipticpart2024oct}
 {}_{0}^1 \mathfrak{E}^{\mu,4;\kappa,l }_{k,j;n}(t_1, t_2) :=     {}_{0;0}^{\,\,1} \mathfrak{E}^{\mu,4;\kappa ,l}_{k,j;n}(t_1, t_2)+     {}_{0;1}^{\,\,1} \mathfrak{E}^{\mu,4;\kappa,l }_{k,j;n}(t_1, t_2)+   {}_{0;2}^{\,\,1} \mathfrak{E}^{\mu,4;\kappa,l }_{k,j;n}(t_1, t_2),
\ee
where
\be\label{nov6eqn81}
 \begin{split}
  {}_{0;0}^{\,\,1} \mathfrak{E}^{\mu,4;\kappa ,l}_{k,j;n}(t_1, t_2) 
 &:= \int_{t_1}^{t_2}  \int_{\R^3}  \int_{\R^3} e^{i X(s  )\cdot \xi    } \varphi_{l;-M_t}(\tilde{v}-\tilde{V}(s))  \\
 &\quad \times   \big[ \big(\sum_{\begin{subarray}{c}
k_1\in \Z_+, \mu_1\in \{+,-\}\\
 n_1\in [-2M_t/15 , 2]\cap \Z
 \\ 
\end{subarray}}  \mathcal{F}\big( {}_{}^zT_{k_1,n_1}^{\mu_1}(B)(s,\cdot,V(s)) f(s,\cdot, v))(\xi)\\ 
 &\quad  +   \sum_{\begin{subarray}{c}
k_1\in \Z_+, \mu_1\in \{+,-\}\\
 n_1\in [-M_t,-2M_t/15 ]\cap \Z
\end{subarray}} \mathcal{F}\big(  {}_{}^zT_{k_1,n_1}^{\mu_1;1}(B)(s,\cdot ,V(s)) f(s,\cdot, v))(\xi) \big)  \\
 & \quad \times \big( (\hat{v}_2-\hat{V}_2(s)) \p_{v_1}   -(\hat{v}_1-\hat{V}_1(s))   \p_{v_2}\big) \\
 &\quad+  \mathcal{F}\big( f(s,\cdot, v)  K(s, \cdot, V(s) )  \big)(\xi) \cdot   \nabla_v 
 \big]  \mathcal{F}[ {}_{}^{1}\mathfrak{E}^{\kappa} ](  X_{\bot}(s),   \xi, v, \zeta) \big]  d \xi d v d s,
 \end{split}
\ee

\be\label{oct30eqn1}
\begin{split}
{}_{0;1}^{\,\,1}  \mathfrak{E}^{\mu,4;\kappa,l }_{k,j;n}(t_1, t_2)&:=\int_{t_1}^{t_2} \int_{\R^3}\int_{\R^3} e^{i X(s)\cdot \xi   }\varphi_{l;-M_t}(\tilde{v}-\tilde{V}(s))  \big[ \mathbf{P}_3\big((\hat{v}-\hat{V}(s))\times  \mathcal{F}[Bf](s,\xi,v) \big)\p_{v_3}  \\
 &\quad  - (\hat{v}_3-\hat{V}_3(s))  \mathbf{P}_2\big( \mathcal{F}[Bf](s,\xi,v)\big)\p_{v_1} + (\hat{v}_3-\hat{V}_3(s)) \\
  &\quad \times  \mathbf{P}_1\big(  \mathcal{F}[Bf](s,\xi,v) \big)\p_{v_2}   \big]\big(   \mathcal{F}[ {}_{}^{1}\mathfrak{E}^{\kappa} ](  X_{\bot}(s),   \xi, v, V(s)) \big)  d \xi d v ds,\\
  \end{split}
\ee

\be\label{nov11eqn61}
\begin{split}
{}_{0;2}^{\,\,1}  \mathfrak{E}^{\mu,4;\kappa,l }_{k,j;n}(t_1, t_2) &:=   \sum_{\begin{subarray}{c}
k_1\in \Z_+,   \mu_1\in \{+,-\}\\ 
n_1\in [-M_t, -2M_t/15 ]\cap \Z, 
\\ 
\end{subarray}}  \int_{t_1}^{t_2} \int_{\R^3}\int_{\R^3} e^{i X(s)\cdot \xi   }  \varphi_{l;-M_t}(\tilde{v}-\tilde{V}(s))   \\
&\quad \times   \mathcal{F}[  {}_{}^zT_{k_1,n_1}^{\mu_1;2}(B)(s,\cdot,  V(s)) f(s, \cdot, v)](\xi) \big( (\hat{v}_2-\hat{V}_2(s)) \p_{v_1}  \\
&\quad    -(\hat{v}_1-\hat{V}_1(s))   \p_{v_2} \big)\big]\big(  \mathcal{F}[ {}_{}^{1}\mathfrak{E}^{\kappa} ](  X_{\bot}(s),   \xi, v, V(s)) \big)  d \xi d v ds.\\
\end{split}
\ee

The estimate of  $  {}_{0;0}^{\,\,1} \mathfrak{E}^{\mu,4;\kappa }_{k,j;n}(t_1, t_2)$ is postponed to   Lemma  \ref{errortypefullhyp2}. The  estimate of $ {}_{0;1}^{\,\,1} \mathfrak{E}^{\mu,4;\kappa }_{k,j;n}(t_1, t_2)$ and $ {}_{0;2}^{\,\,1} \mathfrak{E}^{\mu,4;\kappa }_{k,j;n}(t_1, t_2)$ is postponed to Lemma \ref{ellipticstep1good}.

To sum up, the proof is structured as follows:
\begin{enumerate}
\item[$\bullet$] In Lemma \ref{ellipticstep2good}, we estimate  $  {}_{0;1}^{\,\,1}\mathfrak{E}^{\mu,a;l}_{k,j;n}(t_1, t_2)$,  $a\in\{0,1,2,3\}$.
\item[$\bullet$] In Lemma \ref{ellipticstep1good}, we estimate $   {}_{0;1}^{\,\,1} \mathfrak{E}^{\mu,4;\kappa }_{k,j;n}(t_1, t_2)$ and $    {}_{0;2}^{\,\,1} \mathfrak{E}^{\mu,4;\kappa }_{k,j;n}(t_1, t_2)$. 
\item[$\bullet$] In Lemma  \ref{errortypefullhyp2}, we estimate ${}_{1}^{1}\widetilde{\mathfrak{H}}^{\mu,i}_{k,j;n}(t_1, t_2),$ $  {}_{0;0}^{\,\,1}\mathfrak{E}^{\mu,a;l}_{k,j;n}(t_1, t_2)$, $    {}_{0;0}^{\,\,1} \mathfrak{E}^{\mu,4;\kappa,l }_{k,j;n}(t_1, t_2)$  ,  ${}_{1}^{1}\mathfrak{E}^{\mu,a;l}_{k,j;n}(t_1, t_2)$, and $  {}_{1}^{1} \mathfrak{E}^{\mu,4;\kappa,l }_{k,j;n}(t_1, t_2)$, $a\in\{0,1,2,3\}, \kappa\in (\bar{\kappa},2]\cap\Z.$
\end{enumerate}

\subsubsection{The estimate of $  {}_{0;1}^{\,\,1}\mathfrak{E}^{\mu,i;l}_{k,j;n}(t_1, t_2)$,  $i\in\{0,1,2,3\}$.}

\begin{lemma}\label{ellipticstep2good}
Let   $i\in\{0,   2,3  \},     (k,n)\in \widetilde{\mathcal{E}}_i$ or $i=1, (k,n)\in \widetilde{\mathcal{E}}'_1$.   For any  $\mu\in\{+, -\},    j\in [0,(1+2\epsilon)M_t]\cap \Z, $ $l\in[-j,2]\cap\Z,  $ s.t., $k+2n\geq  2(1- \alpha^{\star} )M_t - 3 0\epsilon M_t $,   $k+4n/3\geq M_t- \alpha^{\star}  M_t/3-30M_t,$   $ \min\{l,n\}\geq (1-2 \alpha^{\star} )M_t-30\epsilon M_t$,  $j\geq 3M_t/5-20\epsilon M_t$,  $k\leq 2\alpha^{\star}M_t+20\epsilon M_t$, and  $(k+2n)/2-l\geq 5 M_t/12-2\iota M_t $,  we have
\be\label{nov9eqn51}
\big|   {}_{0;1}^{\,\,1}\mathfrak{E}^{\mu,i;l}_{k,j;n}(t_1, t_2)\big| \lesssim \mathcal{M}(C) 2^{(\gamma-10 \epsilon)M_t}.
\ee
 
\end{lemma}

\begin{proof}

Recall  \eqref{nov9eqn62}.  After using the    decomposition of the magnetic field in  \eqref{2024oct30eqn2},  and doing integration by parts in $\xi$ along $V(t)$ direction and directions perpendicular to $V(t)$ for the kernel,  we have
\be\label{sep10eqn10}
\begin{split}
\big| {}_{0;1}^{\,\,1}\mathfrak{E}^{\mu,i;l}_{k,j;n}(t_1, t_2)\big| & \lesssim  \int_{t_1}^{t_2} \sum_{\begin{subarray}{c}
 (\tilde{m},\tilde{k},\tilde{j},\tilde{l})\in \mathcal{S}_1(t)\cup \mathcal{S}_2(t)
 \end{subarray} }\big| {}_{0;1}^{\,\,1}\mathfrak{E}^{\mu,i;l;\tilde{m}}_{k,j,n;\tilde{k}, \tilde{j}, \tilde{l}}(s)\big| d s,\\ 
\big| {}_{0;1}^{\,\,1}\mathfrak{E}^{\mu,i;l;\tilde{m}}_{k,j,n;\tilde{k}, \tilde{j}, \tilde{l}}(s )\big| &\lesssim \mathcal{M}(C) 2^{k-j-l+l+4\epsilon M_t}  \int_{\R^3}\int_{\R^3}   \varphi_{l;-M_t}\big(\tilde{v}-\tilde{V}(s) \big) \\ 
&\quad \times (1+2^k|y\cdot\tilde{V}(s)|)^{-100}(1+2^{k+n}|y\times \tilde{V}(s)|)^{-100} \\
&\quad \times  f(s, X(s)-y)\big|B_{\tilde{k};\tilde{j}, \tilde{l} }^{\tilde{m}} (s, X(s)- y,v) \big| d y d v,
\end{split}
\ee

Based on the range of $ (\tilde{m},\tilde{k},\tilde{j},\tilde{l})$, we proceed in two steps as follows. 

\medskip

\noindent \textbf{Step 1.}\quad  If $ (\tilde{m},\tilde{k},\tilde{j},\tilde{l})\in \mathcal{S}_1(t
).$

\medskip

From the above estimate  \eqref{sep10eqn10}, the conservation law  \eqref{conservationlaw},  the volume of support of $v$ and the estimate \eqref{nov24eqn41}  if $|v_{\bot}|\geq 2^{(\alpha_s+\epsilon)M_s}$,   we have  
\be\label{nov11eqn1}
  \big|  {}_{0;1}^{\,\,1}\mathfrak{E}^{\mu,i;l;\tilde{m}}_{k,j,n;\tilde{k}, \tilde{j}, \tilde{l}}(s )\big|\lesssim    \mathcal{M}(C)  2^{k-j+4\epsilon M_t} \min\{2^{-3k-2n}\min\{2^{3j+2l}, 2^{j+2\alpha^{\star} M_t}\}, 2^{-j}\}\| B_{\tilde{k};\tilde{j}, \tilde{l} }^{\tilde{m}} (s, \cdot) \|_{L^\infty_x}.
 \ee
 Therefore, from the above estimate and the  first estimate  in  \eqref{2024oct30eqn1}  in Proposition \ref{meanLinfest},  we have 
\be
\begin{split}
\sum_{(\tilde{m},\tilde{k},\tilde{j},\tilde{l})\in \mathcal{S}_1(t
)}  \big|  {}_{0;1}^{\,\,1}\mathfrak{E}^{\mu,i;l;\tilde{m}}_{k,j,n;\tilde{k}, \tilde{j}, \tilde{l}}(s )\big|&\lesssim      2^{2 \alpha^{\star} M_t +40\epsilon M_t +l -(k+2n)/2}\mathcal{M}(C)\\
&\lesssim 2^{(\gamma-10\epsilon)M_t}\mathcal{M}(C).
\end{split}
\ee

\medskip

\noindent \textbf{Step 2.}\quad  If $ (\tilde{m},\tilde{k},\tilde{j},\tilde{l})\in \mathcal{S}_2(t
).$

\medskip

Note that, by using the same argument, the obtained estimate  \eqref{nov11eqn1}  is still valid if we put $\| B_{\tilde{k};\tilde{j}, \tilde{l} }^{\tilde{m}} (t, \cdot) \|_{L^\infty_x}$ in $L^\infty_x$. 
Moreover, from \eqref{sep10eqn10}, if we use the Cauchy-Schwarz inequality,  the volume of support of $v$ and the estimate \eqref{nov24eqn41} if $|v_{\bot}|\geq 2^{(\alpha_s+\epsilon)M_s}$,   the conservation law  \eqref{conservationlaw}, the following estimate holds,
\be\label{nov10eqn21}
\begin{split}
\big|  {}_{0;1}^{\,\,1}\mathfrak{E}^{\mu,i;l;\tilde{m}}_{k,j,n;\tilde{k}, \tilde{j}, \tilde{l}}(s )\big| &\lesssim \| B_{\tilde{k};\tilde{j}, \tilde{l} }^{\tilde{m}} (s, \cdot) \|_{L^2_x} 2^{k-j+10\epsilon M_t} \big( 2^{j+2\alpha^{\star} M_t} \big)^{1/2} (2^{-3k-2n+j+2\alpha^{\star} M_t})^{1/2} \mathcal{M}(C)\\
&\lesssim 2^{2\alpha^{\star} M_t +10\epsilon M_t -(k+2n)/2}\| B_{\tilde{k};\tilde{j}, \tilde{l} }^{\tilde{m}} (s, \cdot) \|_{L^2_x}\mathcal{M}(C).\\
\end{split}
\ee
After combining the obtained estimates  \eqref{nov11eqn1}  and  \eqref{nov10eqn21}, the following estimate holds from  the second estimate in  \eqref{2024oct30eqn1} in Proposition  \ref{meanLinfest},  
\[ 
\begin{split}
\sum_{(\tilde{m},\tilde{k},\tilde{j},\tilde{l})\in \mathcal{S}_2(t
)}\big| {}_{0;1}^{\,\,1}\mathfrak{E}^{\mu,i;l;\tilde{m}}_{k,j,n;\tilde{k}, \tilde{j}, \tilde{l}}(s )\big|&\lesssim    2^{2\alpha^{\star}M_t +50\epsilon M_t }(2^{l-(k+2n)/2})^{1/2}(2^{-(k+2n)/2})^{1/2}\mathcal{M}(C)\\
&\lesssim 2^{2\alpha^{\star}M_t +50\epsilon M_t } (2^{-(1-\alpha^{\star})M_t +15\epsilon M_t} )^{1/2} (2^{-5M_t/12+2\iota M_t})^{1/2}\mathcal{M}(C)\\
&\lesssim \mathcal{M}(C) 2^{(\gamma-10\epsilon)M_t}.
\end{split}
\]
Hence finishing the proof of our desired estimate  \eqref{nov9eqn51}. 
\end{proof} 
\subsubsection{The estimate of $   {}_{0;1}^{\,\,1} \mathfrak{E}^{\mu,4;\kappa,l }_{k,j;n}(t_1, t_2)$ and $    {}_{0;2}^{\,\,1} \mathfrak{E}^{\mu,4;\kappa,l }_{k,j;n}(t_1, t_2)$. }
\begin{lemma}\label{ellipticstep1good}
Let $\bar{\kappa}:= -(5\alpha^{\star} -3)M_t-100\epsilon M_t$,  $  \mu\in\{+, -\},    (k,n)\in \widetilde{\mathcal{E}}_4$, $j\in [0,(1+2\epsilon)M_t]\cap \Z, $ $l\in[-j,2]\cap\Z, \kappa\in (\bar{\kappa}, 2]\cap \Z $ s.t., $k+2n\geq  2(1- \alpha^{\star} )M_t - 3 0\epsilon M_t $,    $k+4n/3\geq M_t- \alpha^{\star}  M_t/3-30M_t,$   $ \min\{l,n\}\geq (1-2 \alpha^{\star} )M_t-30\epsilon M_t$,  $j\geq 3M_t/5-20\epsilon M_t$,  $k\leq 2\alpha^{\star}M_t+20\epsilon M_t$,    $(k+2n)/2-l\geq 5 M_t/12-2\iota M_t $,  $k+l+\kappa\geq  2(1- \alpha^{\star}  )M_t-50\epsilon M_t, $      $(2k+2l+3\min\{l,\kappa\})/4-2l\geq (1-3\alpha^{\star}M_t/4)-100\epsilon M_t$, and $\min\{l, \kappa\}\geq 7l/3-\iota M_t$, we have
\be\label{nov28eqn80}
 \big| {}_{0;1}^{\,\,1} \mathfrak{E}^{\mu,4;\kappa,l }_{k,j;n}(t_1, t_2)\big| + \big|  {}_{0;2}^{\,\,1} \mathfrak{E}^{\mu,4;\kappa,l }_{k,j;n}(t_1, t_2)\big| \lesssim \mathcal{M}(C) 2^{(\gamma-10\epsilon)M_t}.
 \ee 
\end{lemma}
\begin{proof}

Recall the definition of ${}_{0;1}^{\,\,1} \mathfrak{E}^{\mu,4;\kappa }_{k,j;n}(t_1, t_2)$ and ${}_{0;2}^{\,\,1} \mathfrak{E}^{\mu,4;\kappa }_{k,j;n}(t_1, t_2)$ in \eqref{oct30eqn1} and   \eqref{nov11eqn61}, and the definition of $\mathcal{F}[ {}_{}^{1}\mathfrak{E}^{\kappa} ](   X_{\bot}(s),   \xi, v, \zeta)$ in  \eqref{2024oct28eqn11}.   Due to the cutoff function $\varphi^4_{j,n}(v,\xi)$ (see  \eqref{sep4eqn6}), we have $l\in[n-2\epsilon M_t, n+2\epsilon M_t]$.

 After  doing integration by parts in  $\xi$  with respect to the orthonormal frame $\{(\hat{v}-\hat{V}(s))/|\hat{v}-\hat{V}(s), \theta_1(v,s),\theta_2(v,s)\}$ (defined in  Lemma \ref{smallangest}) many times,   for the kernel and     using the decomposition of the magnetic field in  \eqref{2024oct30eqn2}  and  \eqref{nov6eqn47}, $ \forall i\in \{1,2\}, $ we have
\be\label{nov10eqn51} 
\begin{split}
 \big|{}_{0;i}^{\,\,1} \mathfrak{E}^{\mu,4;\kappa,l }_{k,j;n}(t_1, t_2)\big| & \lesssim  \int_{t_1}^{t_2} \sum_{\begin{subarray}{c}
 (\tilde{m},\tilde{k},\tilde{j},\tilde{l})\in \mathcal{S}_1(t)\cup \mathcal{S}_2(t)
 \end{subarray} } \big|   {}_{0;i}^{\,\,1} \mathfrak{E}^{\mu, 4; \kappa,l;\tilde{m}}_{k,j,n;\tilde{k}, \tilde{j}, \tilde{l}}(s )\big| d s,\\ 
\big|   {}_{0;1}^{\,\,1} \mathfrak{E}^{\mu, 4; \kappa,l;\tilde{m}}_{k,j,n;\tilde{k}, \tilde{j}, \tilde{l}}(s )\big|& \lesssim \int_{\R^3}\int_{\R^3} K_{k,n,\kappa}(  X_{\bot}(s), y, v, V(s))  (2^{l}+2^{(\alpha^{\star}-\gamma)M_t})\\
&\quad \times  \big|B_{\tilde{k};\tilde{j}, \tilde{l} }^{\tilde{m}} (s, X(s)- y) \big| f(s, X(s)-y, v) dy d v, \\ 
\big|    {}_{0;2}^{\,\,1} \mathfrak{E}^{\mu, 4; \kappa,l;\tilde{m}}_{k,j,n;\tilde{k}, \tilde{j}, \tilde{l}}(s ) \big| &\lesssim \int_{\R^3}\int_{\R^3} K_{k,n,\kappa}( X_{\bot}(s), y, v, V(s))  \big|{}_{}^zT_{\tilde{k},\tilde{j};\tilde{n},\tilde{l}}^{\mu;\tilde{m},2}(B)(s,x,\zeta)\big| \\
&\quad \times \mathbf{1}_{\tilde{n}\leq -2M_t/15+10\epsilon M_t}   f(s, X(s)-y, v) dy d v, \\ 
\end{split}
\ee
where the kernel $K_{k,n,\kappa}( X_{\bot}(s), y, v, V(s))$ satisfies the following estimate, 
\be
\begin{split}
  \big|K_{k,n,\kappa}(  X_{\bot}(s),y, v, V(s))\big| 
  &\lesssim  \mathcal{M}(C)  2^{k+l-n-j-\min\{l,\kappa\}+l+4\epsilon M_t} \\
  &\quad \times ( 1+ 2^{k+n}|y\cdot \theta_1(v, s)|)^{-N_0^3}  (1+2^k|y\cdot \theta_2(v, s)|)^{-N_0^3} \\ 
  &\quad \times   (1+2^{k+\min\{n,\kappa\} }(y\cdot (\hat{v}-\hat{V}(s)/|\hat{v}-\hat{V}(s)|) ))^{-N_0^3 } . 
\end{split}
\ee

On one hand, from the obtained estimate  \eqref{nov10eqn51}, the volume of support of $v$, and the conservation law \eqref{conservationlaw}, we have    
\be\label{nov11eqn31}
\begin{split}
   \big|     {}_{0;1}^{\,\,1} \mathfrak{E}^{\mu, 4; \kappa,l;\tilde{m}}_{k,j,n;\tilde{k}, \tilde{j}, \tilde{l}}(s ) \big| & \lesssim  2^{k+l-j-\min\{l,\kappa\} +4\epsilon M_t} \min\{2^{-3k-n- \min\{n,\kappa\} }\min\{ 2^{3j+2l}, 2^{j+2\alpha^{\star}M_t}\}, 2^{-j}\} \\
   &\quad \times (2^{l}+2^{(\alpha^{\star}-\gamma)M_t}) \big\|B_{\tilde{k};\tilde{j}, \tilde{l} }^{\tilde{m}} (s,\cdot) \big\|_{L^\infty_x}  \mathcal{M}(C) \\ 
   \big|    {}_{0;2}^{\,\,1} \mathfrak{E}^{\mu, 4; \kappa,l;\tilde{m}}_{k,j,n;\tilde{k}, \tilde{j}, \tilde{l}}(s ) \big|&\lesssim  2^{k+l-j-\min\{l,\kappa\} +4\epsilon M_t} \min\{2^{-3k-n- \min\{n,\kappa\} } \min\{ 2^{3j+2l}, 2^{j+2\alpha^{\star}M_t}\}, 2^{-j}\}   \\
   &\quad \times \big\|{}_{}^zT_{\tilde{k},\tilde{j};\tilde{n},\tilde{l}}^{\mu;\tilde{m},2}(B)(s,\cdot,\zeta)\big\|_{L^\infty_x} \mathbf{1}_{\tilde{n}\leq  -2M_t/15 } \mathcal{M}(C). 
\end{split}
\ee

On the other hand,  from the  first fact in   \eqref{2024feb13eqn1}, the estimate \eqref{nov10eqn51}, the cylindrical symmetry of the distribution function and the estimate \eqref{singularweigh}  in Lemma \ref{singularweigh}, the following estimate holds if $l\geq (\alpha^{\star}-\gamma)M_t+100\epsilon M_t,$
\be\label{nov11eqn32}
\begin{split}
   \int_{t_1}^{t_2} \big|     {}_{0;1}^{\,\,1} \mathfrak{E}^{\mu, 4; \kappa,l;\tilde{m}}_{k,j,n;\tilde{k}, \tilde{j}, \tilde{l}}(s ) \big|  ds &\lesssim  2^{k+l-j-\min\{l,\kappa\} +4\epsilon M_t} 2^{-k-2n+4\epsilon M_t} 2^{-2(j+l)+j +\alpha^{\star}M_t+6\epsilon M_t} \\
   &\quad \times  (2^{l}+2^{(\alpha^{\star}-\gamma)M_t}) \sup_{s\in[t_1, t_2]} \big\|B_{\tilde{k};\tilde{j}, \tilde{l} }^{\tilde{m}} (s,\cdot) \big\|_{L^\infty_x}\mathcal{M}(C)  ,\\
        \int_{t_1}^{t_2}   \big|    {}_{0;2}^{\,\,1} \mathfrak{E}^{\mu, 4; \kappa,l;\tilde{m}}_{k,j,n;\tilde{k}, \tilde{j}, \tilde{l}}(s )\big| ds  & \lesssim  2^{  -2n-2j- l-\min\{l,\kappa\} +\alpha^{\star}M_t+15\epsilon M_t}  \\
        &\quad \times  \sup_{s\in[t_1, t_2]} \big\|{}_{}^zT_{\tilde{k},\tilde{j};\tilde{n},\tilde{l}}^{\mu;\tilde{m},2}(B)(s,\cdot,\zeta)\big\|_{L^\infty_x} \mathbf{1}_{\tilde{n}\leq -2M_t/15 } \mathcal{M}(C).
\end{split}
\ee
In the above estimate, we used the fact that $|  v_{\bot}|/|v|\sim 2^l$ if $l\geq (\alpha^{\star}-\gamma)M_t+100\epsilon M_t. $

 Based on the range of $ (\tilde{m},\tilde{k},\tilde{j},\tilde{l})$, we proceed in two main steps  as follows. 

\medskip

\noindent \textbf{Step 1.}\quad  If $ (\tilde{m},\tilde{k},\tilde{j},\tilde{l})\in \mathcal{S}_1(t
).$

\medskip

If $l\leq -2M_t/15+3\iota M_t,$ then from the obtained estimate  \eqref{nov11eqn31},   the first estimate in   \eqref{2024oct30eqn1} in Proposition \ref{meanLinfest},  and  the estimate  \eqref{2024oct30eqn21}    in Proposition  \ref{set1goodPartP3B},  we have 
\be 
\begin{split}
\sum_{(\tilde{m},\tilde{k},\tilde{j},\tilde{l})\in \mathcal{S}_1(t
), i\in\{1,2\}}    \big|    {}_{0;i}^{\,\,1} \mathfrak{E}^{\mu, 4; \kappa,l;\tilde{m}}_{k,j,n;\tilde{k}, \tilde{j}, \tilde{l}}(s )\big| &    \lesssim  \mathcal{M}(C) 2^{2\alpha^{\star} M_t  + 150\epsilon M_t} 2^{2l-\min\{l,\kappa\}-(k+n+\min\{n,\kappa\})/2}   \\
&\quad \times  \big(  2^{  (\alpha^{\star}-\gamma)M_t} +2^{ -2M_t/15+3\iota M_t}\big)\\
 & \lesssim 2^{(\gamma-20\epsilon)M_t}\mathcal{M}(C). 
  \end{split}
\ee

If  $ l\geq -2M_t/15+3\iota M_t$, then  from the obtained estimates \eqref{nov11eqn31} and \eqref{nov11eqn32},  the first estimate in   \eqref{2024oct30eqn1} in Proposition \ref{meanLinfest},  and  the estimate  \eqref{2024oct30eqn21}    in Proposition  \ref{set1goodPartP3B}, and the fact that   $(2k+2l+3\min\{l,\kappa\})/4-2l\geq (1-3\alpha^{\star}M_t/4)-25\epsilon M_t$, we have
\be 
\begin{split}
\sum_{(\tilde{m},\tilde{k},\tilde{j},\tilde{l})\in \mathcal{S}_1(t
), i\in\{1,2\}}   \int_{t_1}^{t_2}  \big|    {}_{0;i}^{\,\,1} \mathfrak{E}^{\mu, 4; \kappa,l;\tilde{m}}_{k,j,n;\tilde{k}, \tilde{j}, \tilde{l}}(s )\big|  d s  & \lesssim  2^{2\alpha^{\star} M_t  + 150\epsilon M_t} \big( 2^{  -2l-2j -\min\{l,\kappa\} +\alpha^{\star}M_t+15\epsilon M_t}\big)^{1/2}\\
&\times \big(2^{-2k-2\min\{l,\kappa\}+2j+3l}\big)^{1/2}\mathcal{M}(C) \\
&\lesssim 2^{5\alpha^{\star}M_t/2+ 150\epsilon M_t-5l/2-(k+3\min\{l,\kappa\}/2-3l)} \mathcal{M}(C)\\
& \lesssim 2^{(\gamma-20\epsilon)M_t}\mathcal{M}(C).
  \end{split}
\ee

\medskip

\noindent \textbf{Step 2.}\quad   If $ (\tilde{m},\tilde{k},\tilde{j},\tilde{l})\in \mathcal{S}_2(t
).$

\medskip

From the Cauchy-Schwarz inequality, the obtained estimate  \eqref{nov10eqn51}, and the volume of support of $v$, we have
 \be\label{nov11eqn56}
 \begin{split}
  \big|    {}_{0;1}^{\,\,1} \mathfrak{E}^{\mu, 4; \kappa,l;\tilde{m}}_{k,j,n;\tilde{k}, \tilde{j}, \tilde{l}}(s )\big|&\lesssim  2^{k +l -j-\min\{l,\kappa\}+4\epsilon M_t} (2^{l}+2^{(\alpha^{\star}-\gamma)M_t}) \big( \min\{ 2^{3j+2l}, 2^{j+2\alpha^{\star}M_t}\}\big)^{1/2}    \\ 
   &\quad  \times    \big( 2^{-3k-n-\min\{n,\kappa\} }  \min\{ 2^{3j+2l}, 2^{j+2\alpha^{\star}M_t}\} \big)^{1/2}  \mathcal{M}(C) \big\|B_{\tilde{k};\tilde{j}, \tilde{l} }^{\tilde{m}} (s, \cdot) \big\|_{L^2}, \\
     \big|     {}_{0;2}^{\,\,1} \mathfrak{E}^{\mu, 4; \kappa,l;\tilde{m}}_{k,j,n;\tilde{k}, \tilde{j}, \tilde{l}}(s )\big|&\lesssim
     2^{k +l-j-\min\{l,\kappa\}+4\epsilon M_t}  \big( 2^{-3k-n-\min\{n,\kappa\} } \min\{ 2^{3j+2l}, 2^{j+2\alpha^{\star}M_t}\} \big)^{1/2}\\ 
     &\quad \times  \big( \min\{ 2^{3j+2l}, 2^{j+2\alpha^{\star}M_t}\}\big)^{1/2} \mathcal{M}(C)    \big\|{}_{}^zT_{\tilde{k},\tilde{j};\tilde{n},\tilde{l}}^{\mu;\tilde{m},2}(B)(s,x,\zeta)\big\|_{L^2} \mathbf{1}_{\tilde{n}\leq  -2M_t/15 }. 
 \end{split}
 \ee

 Based on the possible size of   $l$, we proceed in two sub-steps  as follows. 

 \medskip

\textbf{Step 2A.}\quad  If $l\leq -50\iota M_t$. 

\medskip

From the obtained estimates \eqref{nov11eqn31} and  \eqref{nov11eqn56},  the second  estimate in  \eqref{2024oct30eqn1}  in Proposition \ref{meanLinfest}, and  the second estimate in \eqref{2024oct30eqn21}  in Proposition  \ref{set1goodPartP3B},  for any $s\in [t_1, t_2]$,  we have 
\[
\begin{split}
\sum_{\begin{subarray}{c}
(\tilde{m},\tilde{k},\tilde{j},\tilde{l})\in \mathcal{S}_2(t
)\\ 
i\in\{1,2\}
\end{subarray}} & \big|   {}_{0;i}^{\,\,1} \mathfrak{E}^{\mu, 4; \kappa,l;\tilde{m}}_{k,j,n;\tilde{k}, \tilde{j}, \tilde{l}}(s ) \big|   \lesssim 2^{2\alpha^{\star} M_t+3l/2 +150\epsilon M_t -\min\{l,\kappa\} -(k+n+\min\{n,\kappa\})/2}  \\
 &\quad \times (2^{l}+2^{-2M_t/15})  \mathcal{M}(C) \\
 &\lesssim   \mathcal{M}(C)  2^{(3\alpha^{\star}-1)M_t }\big(2^{5l/2-7l/3+\iota M_t} + 2^{-2M_t/15} 2^{3l/2}( 2^{-(5\alpha^{\star}-3)M_t} )^{ 5/14 } (2^{-7l/3+\iota M_t})^{9/14} \big)\\
 & \lesssim \mathcal{M}(C) 2^{(\gamma-20\epsilon)M_t} . 
 \end{split}
\]
 
 \medskip

\textbf{Step 2B.}\quad     If $l\geq  -50\iota M_t$. 

 \medskip

Note that, for this case, we have $|  v_{\bot} |/|v|\sim 2^l$ and $\min\{l, \kappa\}\geq -120\iota M_t.$ In view of the estimate \eqref{nov24eqn41}, it would be sufficient to consider the case $j+l\leq (\alpha^{\star}+\epsilon)M_t.$   Recall the obtained estimate  \eqref{nov10eqn51}.    From 
   the  first fact in   \eqref{2024feb13eqn1},  the first estimate of the Jacbobian in \eqref{jacobianthreshold},    and the estimate \eqref{singularweigh}  in Lemma \ref{singularweigh},  after using the Cauchy-Schwarz inequality,   we have
 \[
 \begin{split}
  \big|    {}_{0;1}^{\,\,1} \mathfrak{E}^{\mu, 4; \kappa,l;\tilde{m}}_{k,j,n;\tilde{k}, \tilde{j}, \tilde{l}}(s )\big|&\lesssim    2^{k +2l -j-\min\{l,\kappa\}+20\epsilon M_t}   \big(   2^{3j+2l}  \big)^{1/2}     \big(2^{-k-2n -2(j+l)+j+\alpha^{\star}M_t } \big)^{1/2}  \mathcal{M}(C) \big\|B_{\tilde{k};\tilde{j}, \tilde{l} }^{\tilde{m}} (s, \cdot) \big\|_{L^2}  \\
   &\lesssim 2^{k/2+l+\alpha^{\star}M_t/2-\min\{l,\kappa\}+20\epsilon M_t}  \mathcal{M}(C) \big\|B_{\tilde{k};\tilde{j}, \tilde{l} }^{\tilde{m}} (s, \cdot) \big\|_{L^2},  \\
     \big|     {}_{0;2}^{\,\,1} \mathfrak{E}^{\mu, 4; \kappa,l;\tilde{m}}_{k,j,n;\tilde{k}, \tilde{j}, \tilde{l}}(s )\big|&\lesssim  2^{k/2+ \alpha^{\star}M_t/2-\min\{l,\kappa\}+20\epsilon M_t}  \mathcal{M}(C)     \big\|{}_{}^zT_{\tilde{k},\tilde{j};\tilde{n},\tilde{l}}^{\mu;\tilde{m},2}(B)(s,x,\zeta)\big\|_{L^2} \mathbf{1}_{\tilde{n}\leq  -2M_t/15 }. 
 \end{split}
 \]

 From the above estimate and the obtained estimate   \eqref{nov11eqn31}, the second  estimate in  \eqref{2024oct30eqn1}  in Proposition \ref{meanLinfest}, and   the second estimate in \ref{2024oct30eqn21}  in Proposition  \ref{set1goodPartP3B},   and the fact that   $(2k +3\min\{l,\kappa\})/4-3l/2\geq (1-3\alpha^{\star}M_t/4)-25\epsilon M_t$, we have
\[
\begin{split}
\sum_{ 
(\tilde{m},\tilde{k},\tilde{j},\tilde{l})\in \mathcal{S}_2(t
), i\in\{1,2\}} &  \int_{t_1}^{t_2}   \big|  {}_{0;i}^{\,\,1} \mathfrak{E}^{\mu, 4; \kappa,l;\tilde{m}}_{k,j,n;\tilde{k}, \tilde{j}, \tilde{l}}(s )  \big|  d s   \lesssim 2^{ \alpha^{\star} M_t +20\epsilon M_t }\big( 2^{k/2+l+\alpha^{\star}M_t/2-\min\{l,\kappa\}+20\epsilon M_t}  \big)^{1/2}  \\
&\quad \times  \big(      2^{-2k- 2\min\{l,\kappa\}+2j+3l}   \big)^{1/2}   \mathcal{M}(C) \\
& \lesssim 2^{ 9\alpha^{\star} M_t/4  +  50\epsilon M_t   -3\min\{l,\kappa\}/8-5l/4-3(k+3\min\{l,\kappa\}/2-3l)/4 }  \mathcal{M}(C)\\
&\lesssim 2^{(\gamma-20\epsilon)M_t}  \mathcal{M}(C). 
\end{split}
\]
Hence,  the desired estimate  \eqref{nov28eqn80}  holds from the above obtained estimates and   the estimate  \eqref{nov10eqn51}.
\end{proof}

\subsubsection{The estimate of all other main terms}\label{mainpartsISSpart2firs}

In the following lemma, we estimate the main terms  $  {}_{0;0}^{\,\,1} \mathfrak{E}^{\mu,4;\kappa,l }_{k,j;n}(t_1, t_2)$,  $  {}_{1}^1 \mathfrak{E}^{\mu,4;\kappa }_{k,j;n}(t_1, t_2)$, $  {}_{0;0}^{\,\,1} \mathfrak{E}^{\mu,a; l }_{k,j;n}(t_1, t_2) $, $ {}_{1}^1\mathfrak{E}^{\mu,a;l}_{k,j;n}(t_1, t_2)$ ($a \in {0, 1, 2, 3}$), and ${}_{1}^{1}\widetilde{\mathfrak{H}}^{\mu,i}_{k,j;n}(t_1, t_2)$. These terms are estimated together due to their shared structural features.

\begin{lemma}\label{errortypefullhyp2}
 Let $\bar{\kappa}:= -(5\alpha^{\star} -3)M_t-100\epsilon M_t$,  $  \mu\in\{+, -\},    (k,n)\in \widetilde{\mathcal{E}}_4$, $j\in [0,(1+2\epsilon)M_t]\cap \Z, $ $l\in[-M_t,2]\cap\Z, \kappa\in (\bar{\kappa}, 2]\cap \Z $ s.t., $k+2n\geq  2(1- \alpha^{\star} )M_t - 3 0\epsilon M_t $,   $k+4n/3\geq M_t- \alpha^{\star}  M_t/3-30M_t,$   $ \min\{l,n\}\geq (1-2 \alpha^{\star} )M_t-30\epsilon M_t$,  $j\geq 3M_t/5-20\epsilon M_t$,    $k+l+\kappa\geq  2(1- \alpha^{\star}  )M_t-100\epsilon M_t, $  $(2k+2l+3\min\{l,\kappa\})/4-2l\geq (1-3\alpha^{\star}M_t/4)-100\epsilon M_t$, and $\min\{l, \kappa\}\geq 7l/3-\iota M_t$,  we have
\be\label{nov12eqn21}
\big|       {}_{0;0}^{\,\,1} \mathfrak{E}^{\mu,4;\kappa,l }_{k,j;n}(t_1, t_2) \big| + \big|   {}_{1}^{ 1} \mathfrak{E}^{\mu,4;\kappa,l }_{k,j;n}(t_1, t_2)  \big|\lesssim 2^{(\gamma-10 \epsilon)M_t} \mathcal{M}(C).
\ee
Moreover, 
\begin{enumerate}
\item[(i)] For any $ i\in \{0,2,3 \},   ( n,k)\in \widetilde{\mathcal{E}}_i$ or $i=1, (n,k)\in \widetilde{\mathcal{E}}'_1,$  $ \mu\in\{+,-\},$ we have 
\be\label{2024oct31eqn510}
 |   {}_{0;0}^{\,\,1} \mathfrak{E}^{\mu,i; l }_{k,j;n}(t_1, t_2) |+\big|  {}_{1}^{ 1} \mathfrak{E}^{\mu,i; l }_{k,j;n}(t_1, t_2)  \big| \lesssim 2^{(\gamma-10 \epsilon)M_t} \mathcal{M}(C).
\ee
\item[(i)]For any $ i\in \{0,1,2,3,4 \},   ( n,k)\in \widetilde{\mathcal{E}}_i$ ,   $ \mu\in\{+,-\},$ we have 
\be\label{2024oct30eqn52}
 \big| {}_{1}^{1}\widetilde{\mathfrak{H}}^{\mu,i}_{k,j;n}(t_1, t_2)  \big|   \lesssim 2^{(\gamma-10 \epsilon)M_t} \mathcal{M}(C).
\ee
\end{enumerate}

\end{lemma}

\begin{proof}

Recall  \eqref{oct29eqn82}, \eqref{2024oct28eqn1}, \eqref{oct29eqn91},  \eqref{nov9eqn62}  and \eqref{nov6eqn81}. The singularity introduced when $\kappa < \min\{l, n\}$ in the terms  $   {}_{0;0}^{\,\,1} \mathfrak{E}^{\mu,4;\kappa,l }_{k,j;n}(t_1, t_2)  $ and $  {}_{1}^{ 1} \mathfrak{E}^{\mu,4;\kappa,l }_{k,j;n}(t_1, t_2) $  makes their estimation more difficult than that of ${}_{1}^{1}\widetilde{\mathfrak{H}}^{\mu,i}_{k,j;n}(t_1, t_2)$, $ {}_{0;0}^{\,\,1} \mathfrak{E}^{\mu,a; l }_{k,j;n}(t_1, t_2)$, and $ {}_{1}^{ 1} \mathfrak{E}^{\mu,a; l }_{k,j;n}(t_1, t_2)$ ($a \in {0, 1, 2, 3}$, $i \in {1, 2, 3, 4}$), which possess similar, less singular structures. 

We restrict our attention to a detailed proof of the more challenging case, specifically, estimate \eqref{nov12eqn21}. With minor modifications in  the proof of estimate  \eqref{nov12eqn21}, the desired estimates in   \eqref{2024oct31eqn510} and   \eqref{2024oct30eqn52} follows similarly.

 \medskip

\noindent \textbf{Step 1.}\quad The  first reduction.

 \medskip

Due to the cutoff function $\varphi^4_{j,n}(v,\xi)$ (see  \eqref{sep4eqn6}), we have $l\in[n-2\epsilon M_t, n+2\epsilon M_t]$. Recall \eqref{2024oct28eqn1}  and  \eqref{nov6eqn81}.  After localizing the frequency of new introduced acceleration force and using the  decomposition of localized acceleration force in \eqref{oct7eqn1},   we have
\be\label{oct19eqn41}
 \begin{split}
\big|       {}_{0;0}^{\,\,1} \mathfrak{E}^{\mu,4;\kappa,l }_{k,j;n}(t_1, t_2) \big| + \big|   {}_{1}^{ 1} \mathfrak{E}^{\mu,4;\kappa,l }_{k,j;n}(t_1, t_2)  \big|
 &\lesssim   \sum_{\begin{subarray}{c}
k_1\in\Z_+,   n_1\in [-M_t,2]\cap \Z\\ 
\mu_1\in \{+, -\}, i_1\in\{0,1,2,3,4\}  \\  
\end{subarray} }  \big|H^{\mu, \mu_1;i,i_1 }_{k,n;k_1,n_1}(t_1, t_2) \big| \\
&\quad + \big|M^{\mu, \mu_1;i,i_1 }_{k,n;k_1,n_1}(t_1, t_2) \big|+   \big|Err^{\mu, \mu_1;i,i_1 }_{k,n;k_1,n_1}(t_1, t_2) \big|,  
\end{split}
 \ee
 where
 \be\label{nov12eqn64}
 \begin{split}
  H^{\mu, \mu_1;i}_{k,n;k_1,n_1}(t_1, t_2) &=   \int_{t_1}^{t_2} \int_{\R^3}\int_{\R^3}  f(s, X(s)- y, v )\\
  &\quad \times  \mathfrak{H}^{\mu_1, i_1}_{k_1,j_1;n_1}(s,X(s) ,   V(s))\cdot  \mathcal{K}_{k,n,l}^{\mu;\kappa}( X_{\bot}(s),y, v, V(s) )  dy d v d s, \\ 
  Err^{\mu, \mu_1;i}_{k,n;k_1,n_1}(t_1, t_2)&=   \int_{t_1}^{t_2} \int_{\R^3}\int_{\R^3}  f(s, X(s)- y, v ) \big(   \|Ini_{k_1,j_1,n_1}^{\mu_1,i_1}(s, x, V(s))\|_{L^\infty_x}  \\ 
   &\quad +  \| \mathfrak{E}^{\mu_1, i_1}_{k_1,j_1;n_1} (s, x, V(s))\|_{L^\infty_x}\mathbf{1}_{i_1\in\{0,1,2,3\} } \big)\\
   &\quad \times  \big[ \big|\mathcal{K}_{k,n,l}^{\mu;\kappa}(  X_{\bot}(s),y, v, V(s) )\big|+  \big| \mathcal{M}_{k,n,l}^{\mu;\kappa}( X_{\bot}(s),y, v, V(s) )\big|\big] dy d v d s, \\ 
   M^{\mu, \mu_1;i,i_1 }_{k,n;k_1,n_1}(t_1, t_2)&:= \int_{t_1}^{t_2} \int_{\R^3}\int_{\R^3}  f(s, X(s)- y, v ) \\
   &\quad \times  \mathfrak{T}^{\mu_1, i_1}_{k_1,j_1;n_1}(s,X(s)- y,   V(s))\cdot  \mathcal{M}_{k,n,l}^{\mu;\kappa}( X_{\bot}(s),y, v, V(s) )dy d v d s,
\end{split}
\ee
where $  \mathfrak{T}^{\mu_1, i_1}_{k_1,j_1;n_1}(s,\cdot,   V(s))$ and the kernels $ \mathcal{K}_{k,n,l}^{\mu;\kappa}(  X_{\bot}(s),\cdot, v, V(s) )$ and $\mathcal{M}_{k,n,l}^{\mu;\kappa}( X_{\bot}(s),\cdot, v, V(s) ) $ are defined as follows, 
\be\label{nov12eqn63}
\begin{split}
\mathcal{K}_{k,n,l}^{\mu;\kappa}( X_{\bot}(s),y, v, V(s) )&:= \int_{\R^3} e^{i y\cdot \xi }\nabla_{\zeta}\big( \varphi_{l;-M_t}\big(\tilde{v}-\tilde{\zeta}  \big)   \mathcal{F}[ {}_{}^{1}\mathfrak{E}^{\kappa} ](   X_{\bot}(s),   \xi, v, \zeta) \big)\big|_{\zeta=V(s)} d \xi,\\
 \mathcal{M}_{k,n,l}^{\mu;\kappa}( X_{\bot}(s),y, v, V(s) )& :=  \int_{\R^3} e^{i y\cdot \xi}   \nabla_v \big[  \varphi_{l;-M_t}\big(\tilde{v}-\tilde{\zeta}  \big)   \mathcal{F}[ {}_{}^{1}\mathfrak{E}^{\kappa} ] ( X_{\bot}(s),   \xi, v, \zeta) \big]\big|_{\zeta= V(s)} d\xi,\\  \mathfrak{T}^{\mu_1, i_1}_{k_1,n_1,j_1}(s,x,   V(s))&:=\mathfrak{H}^{\mu_1, i_1}_{k_1,j_1;n_1}(s,x,   V(s)) +\big[ {}_{}^zT_{k_1,n_1}^{\mu_1}(B)(s,x,V(s)) \mathbf{1}_{n_1\geq -2M_t/15 }  \\ 
 &\quad  +  {}_{}^zT_{k_1,n_1}^{\mu_1;1}(B)(s,x,V(s))  \mathbf{1}_{n_1\leq -2M_t/15 } \big]    \big( \hat{v}_2-\hat{V}_2(s), -\hat{v}_1+\hat{V}_1(s), 0\big).
 \end{split}
 \ee

   Recall   \eqref{2024oct28eqn11}. After doing integration by parts in ``$\xi$''    with respect to the orthonormal frame  $\{(\hat{v}-\hat{V}(s))/|\hat{v}-\hat{V}(s), \theta_1(v,s),\theta_2(v,s)\}$ (defined in  Lemma  \ref{smallangest}) many times, we have 
\be\label{oct19eqn61}
\begin{split}
&2^j \big|    \mathcal{M}_{k,n,l}^{\mu;\kappa}( X_{\bot}(s),y, v, V(s) ) \big|+ 2^{\gamma M_t} \big|\mathcal{K}_{k,n,l}^{\mu;\kappa}(  X_{\bot}(s),y, v, V(s) )\big|\\
 &\lesssim 2^{k -\min\{l,\kappa\}+5\epsilon M_t} (1+2^k|y\cdot \theta_2(v, s)|)^{-N_0^3} ( 1+ 2^{k+n}|y\cdot \theta_1(v, s)|)^{-N_0^3}\\
&\quad  \times  (1+2^{k+\min\{n,\kappa\} }(y\cdot (\hat{v}-\hat{V}(s)/|\hat{v}-\hat{V}(s)|) ))^{-N_0^3} . 
 \end{split}
\ee 
Moreover, the following improved estimate holds for the projection of the kernels onto the $z$-axis, 
\be\label{nov12eqn41}
\begin{split}
 &2^j \big|  \mathbf{P}_3\big(  \mathcal{M}_{k,n,l}^{\mu;\kappa}( X_{\bot}(s),y, v, V(s) )\big) \big|+ 2^{\gamma M_t} \big| \mathbf{P}_3\big(  \mathcal{K}_{k,n,l}^{\mu;\kappa}( X_{\bot}(s),y, v, V(s) )\big)\big|\\
 &\lesssim 2^{k -\min\{l,\kappa\}+5\epsilon M_t} (2^l+2^{(\alpha^{\star}-\gamma )M_t})  ( 1+ 2^{k+n}|y\cdot \theta_1(v, s)|)^{- N_0^3 }\\
&\quad  \times   (1+2^k|y\cdot \theta_2(v, s)|)^{- N_0^3 } (1+2^{k+\min\{n,\kappa\} }(y\cdot (\hat{v}-\hat{V}(s)/|\hat{v}-\hat{V}(s)|) ))^{-N_0^3}  . 
 \end{split}
\ee
 
 We first estimate the error type term $Err^{\mu, \mu_1;i}_{k,n;k_1,n_1}(t_1, t_2)$.   From the above estimates of kernel,  the rough estimate of the elliptic part in \eqref{2022feb25eqn1} in  Theorem \ref{maintheoremellipitic}, and the facts that $\min\{l, \kappa\}\geq 7l/3-\iota M_t$ and $\kappa\geq \bar{\kappa}$, we have
\be\label{2026may1eqn31}
\begin{split}
\big|Err^{\mu, \mu_1;i}_{k,n;k_1,n_1}(t_1, t_2)\big| &  \lesssim 2^{10\epsilon M_t}   2^{5 \alpha^{\star}   M_t/3} 2^{-j-\min\{l,\kappa\}}  \min\{2^{k-j}, 2^{-3k-n-\min\{n,\kappa\}+ 3j+2l}\}  \mathcal{M}(C)\\
&\lesssim 2^{10\epsilon M_t}   2^{5 \alpha^{\star}   M_t/3}    2^{l-\min\{l,\kappa\}  } 2^{-(k+n+\min\{n,\kappa\})/2}\mathcal{M}(C)\\
&\lesssim 2^{60\epsilon M_t}   2^{5 \alpha^{\star}   M_t/3}   2^{l} (2^{-7l/3+\iota M_t})^{3/7} (2^{ (5\alpha^\star-3)M_t+ \iota M_t})^{4/7}    2^{-(1-\alpha^{\star})M_t  }\mathcal{M}(C)\\
& \lesssim 2^{(\gamma-10\epsilon)M_t} \mathcal{M}(C). \\
\end{split}
\ee

Now, we focus on  the other terms in \eqref{oct19eqn41}. From the  estimate of kernel $\mathcal{K}_{k,n,l}^{\mu;\kappa}( X_{\bot}(s),y, v, V(s) )$ in  \eqref{oct19eqn61}, the estimates  in \eqref{2024oct8eqn1} in Theorem \ref{mainresultsfirstpart},    the volume of support of $v$ and the estimate  \eqref{nov24eqn41}  if $| v_{\bot}|\geq 2^{(\alpha^\star+\epsilon)M_t}$, the following estimate holds if either $(k_1+2n_1)/2\leq  (k +n+ \min\{n, \kappa\})/2 + (1-\alpha^{\star}-10\iota )M_t-200\epsilon M_t$ or  $n_1\leq    -(\alpha^{\star}+3\iota+60\epsilon) M_t , $
\be\label{oct15eqn60}
\begin{split}
  \big|H^{\mu, \mu_1;i}_{k,n;k_1,n_1}(t_1, t_2)\big|  & \lesssim   \big(2^{(1-10\epsilon)M_t} + 2^{(k_1+2n_1)/2+ (\alpha^{\star}+3\iota + 130\epsilon)  M_t}\mathbf{1}_{n_1\geq  -(\alpha^{\star}+3\iota+60\epsilon) M_t  } \big)\\
&\quad \times  2^{ 10\epsilon M_t -\min\{l,\kappa\}} 2^{k- \gamma M_t}\min\{2^{-3k-n-\min\{n,\kappa\} + j+2\alpha^{\star}M_t}, 2^{-j}\}    \mathcal{M}(C) \\
&  \lesssim 2^{(\gamma-10\epsilon)M_t} \mathcal{M}(C).
\end{split}
\ee

Recall  \eqref{nov12eqn64}  and   \eqref{nov12eqn63}. From the estimates  in \eqref{2024oct8eqn1} in Theorem \ref{mainresultsfirstpart},   the estimates in   \eqref{nov5eqn1}     in Proposition \ref{goodpartprojmagn} with the choice of $ {n}=-2M_t/15$,  the  estimate of kernel $\mathcal{M}_{k,n,l}^{\mu;\kappa}(  X_{\bot}(s),y, v, V(s) )$ in \eqref{oct19eqn61},   the volume of support of $v$,  and the facts that $\min\{l, \kappa\}\geq 7l/3-\iota M_t$ and $\kappa\geq \bar{\kappa}$,  if $(k_1+2n_1)/2\leq (k+n +\min\{n,\kappa\}  )/2+  2M_t/15-10\iota M_t,$ or $n_1\leq - (\alpha^{\star}+3\iota+60\epsilon) M_t$, we have
\be\label{nov12eqn68}
\begin{split}
| M^{\mu, \mu_1;i,i_1 }_{k,n;k_1,n_1}(t_1, t_2)|&\lesssim      \mathcal{M}(C)  \big(2^{(1-10\epsilon)M_t} + 2^{(k_1+2n_1)/2+ (\alpha^{\star}+3\iota + 130\epsilon)  M_t}\mathbf{1}_{n_1\geq  -(\alpha^{\star}+3\iota+60\epsilon) M_t  } \big) 
\\
&\quad \times  2^{ 10\epsilon M_t -\min\{l,\kappa\} } 2^{k-j} (1+2^{ 2M_t/15} 2^{l}  )  \min\{2^{-3k-n-\min\{n,\kappa\} +3j+2l}, 2^{-j}\}   \\
&\lesssim   \mathcal{M}(C)  \big(2^{(1-10\epsilon)M_t} + 2^{(k_1+2n_1)/2+ (\alpha^{\star}+3\iota + 130\epsilon)  M_t}\mathbf{1}_{n_1\geq  -(\alpha^{\star}+3\iota+60\epsilon) M_t  } \big) \\
&\quad \times 2^{ 10\epsilon M_t  }2^{-(k+n+\min\{n,\kappa\})/2}\big[ 2^{l} (2^{-7l/3+\iota M_t})^{3/7} (2^{ (5\alpha^\star-3)M_t+ \iota M_t})^{4/7}  \\
&\quad +2^{ 2M_t/15} 2^{2l}(2^{-7l/3+\iota M_t})^{6/7} (2^{ (5\alpha^\star-3)M_t+ \iota M_t})^{1/7}\big]\\
& \lesssim 2^{(\gamma-10\epsilon)M_t} \mathcal{M}(C).
\end{split}
\ee
 
 \medskip

\noindent \textbf{Step 2.}\quad The second iteration of smoothing. 

 \medskip

 Now, it remains to estimate $ H^{\mu, \mu_1;i}_{k,n;k_1,n_1}(t_1, t_2)$ and $ M^{\mu, \mu_1;i}_{k,n;k_1,n_1}(t_1, t_2)$ for  the case $(k_1+2n_1)/2\geq  (k+n +\min\{n,\kappa\}  )/2+  2M_t/15-10\iota M_t$ and    $n_1\geq - (\alpha^{\star}+3\iota+60\epsilon) M_t$.  Recall  \eqref{nov12eqn64}.  Note that, on the Fourier side, we have
 \be
 \begin{split}
 H^{\mu, \mu_1;i}_{k,n;k_1,n_1}(t_1, t_2)
 &=  \int_{t_1}^{t_2}\int_{\R^3}  \int_{\R^3}\int_{\R^3} e^{i  X(s)\cdot \xi  - i s\hat{v}\cdot \eta + i s\mu_1 |\xi-\eta|  }    \mathcal{F}[\mathfrak{H}_{k_1,j_1;n_1}^{\mu_1,i_1}](s, \xi-\eta, V(s))   \\ 
 &\quad       \cdot  \nabla_{\zeta}\big(\varphi_{l;-M_t}\big(\tilde{v}-\tilde{\zeta}  \big)     \mathcal{F}[ {}_{}^{1}\mathfrak{E}^{\kappa} ]  (   X_{\bot}(s),   \eta, v, \zeta) \big)\big|_{\zeta=V(s)} \widehat{g}(s, \eta, v)      d \eta d \xi  d v   d s, \\  
  M^{\mu, \mu_1;i}_{k,n;k_1,n_1}(t_1, t_2)& =  \int_{t_1}^{t_2}\int_{\R^3}  \int_{\R^3}\int_{\R^3} e^{i  X(s)\cdot \xi  - i s\hat{v}\cdot \eta + i s\mu_1 |\xi-\eta|  } \mathcal{F}[\widetilde{\mathfrak{H}_{k_1,j_1;n_1}^{\mu_1,i_1}}] (s, \xi-\eta, v, V(s)) \\
  &\quad         \cdot  \nabla_{v}\big(  \varphi_{l;-M_t}\big(\tilde{v}-\tilde{\zeta}  \big)   \mathcal{F}[ {}_{}^{1}\mathfrak{E}^{\kappa} ](  X_{\bot}(s),   \xi, v, \zeta) \big)\big|_{\zeta=V(s)}   \widehat{g}(s, \eta, v) d \eta d \xi  d v   ds ,
 \end{split}
 \ee
 where
  \be\label{nov12eqn70}
  \begin{split}
  \mathcal{F}[\widetilde{\mathfrak{H}_{k_1,j_1;n_1}^{\mu_1,i_1}}](s, \xi,v,  V(s)) &: =      \mathcal{F}[ {\mathfrak{H}_{k_1,j_1;n_1}^{\mu_1,i_1}}](s, \xi , V(s))  + \big[\mathcal{F}_{x\rightarrow \xi}[{}_{}^zT_{k_1,n_1}^{\mu_1}(B)](s,\xi,V(s)) \\
 &\quad \times \mathbf{1}_{n_1\geq -2M_t/15  } +  \mathcal{F}_{x\rightarrow \xi}[{}_{}^zT_{k_1,n_1}^{\mu_1;1}(B)](s,\xi,V(s))  \mathbf{1}_{n_1\leq -2M_t/15 } \big]\\
 &\quad \times \big( \hat{v}_2-\hat{V}_2(s), -\hat{v}_1+\hat{V}_1(s), 0\big) .
  \end{split}
 \ee

Recall the equations satisfied by characteristics in  \eqref{backward}. After doing integration by parts in time, using the Vlasov equation and doing integration by parts in $v$  if $\p_t$ hits $\widehat{g}(t, \eta, v)$, and  using the decomposition  in  \eqref{oct7eqn1} for   the new introduced acceleration force,  for any $ U\in \{H, M\}$, the following decomposition holds
  \be\label{oct29eqn40}
 U^{\mu, \mu_1;i}_{k,n;k_1,n_1}(t_1, t_2) = \sum_{a=0,1,2,3,4} ErrU^a_{i,i_1}(t_1, t_2) +\sum_{\begin{subarray}{c}
  k_2\in \Z_+, n_2\in[-M_t,2]\cap \Z\\  
  \mu_2\in \{+,-\}, i_2\in\{0,1,2,3,4\}, b\in\{1,2\}\\  
  \end{subarray}} U^b_{i,i_1,i_2}(t_1,t_2),
\ee
where, for convenience in notation,  we suppressed the dependence of $U^b_{i,i_1,i_2}(t_1,t_2)$ with respect to $k,k_1, k_2$ etc., 
  \be\label{nov14eqn31}
  \begin{split}
  U^1_{i,i_1,i_2}(t_1,t_2)&:= \int_{t_1}^{t_2}\int_{(\R^3)^3}  e^{i  X(s)\cdot \xi  - i s\hat{v}\cdot \eta + i s \mu_1 |\xi-\eta|  }\widehat{g}(s, \eta, v) \\
  &\quad \times  \mathfrak{H}_{k_2,j_2,n_2}^{\mu_2,i_2}(s, X(s), V(s)) \cdot    \mathcal{F}[{}_{1}^{2}\mathfrak{E}U](s, \xi, \eta, v,X_{\bot}(s), V(s))   d \eta  d \xi  d v d s,\\
  U^2_{i,i_1,i_2}(t_1,t_2)&:= \int_{t_1}^{t_2}\int_{(\R^3)^3}  e^{i  X(s)\cdot \xi   + i s \mu_1 |\xi-\eta|  }     \mathcal{F}[{}_{2}^{2}\mathfrak{E}U](s, \xi, \eta, v,   X_{\bot}(s), V(s)) \\
  &\quad  \cdot    \mathcal{F}\big[\mathfrak{H}_{k_2,j_2,n_2}^{\mu_2,i_2}(s,\cdot , V(s))  {f}\big](s, \eta, v) d \eta  d \xi  d v d s,
  \end{split}
  \ee
  where
\be\label{oct27eqn51}
\begin{split}
 \mathcal{F}[{}_{1}^{2}\mathfrak{E}H](s, \xi, \eta, v, X_{\bot}(s), \zeta ) &:= \nabla_\zeta \big[ ( \hat{\zeta} \cdot \xi - \hat{v}\cdot \eta + \mu_1 |\xi-\eta|)^{-1} \\
&\quad \times   \mathcal{F}[\mathfrak{H}_{k_1,j_1;n_1}^{\mu_1,i_1}](s, \xi-\eta,\zeta) \cdot  \nabla_{\zeta}\big( \varphi_{l;-M_t}\big(\tilde{v}-\tilde{\zeta}  \big)  \mathcal{F}[ {}_{}^{1}\mathfrak{E}^{\kappa} ] (  X_{\bot}(s),   \eta, v, \zeta) \big)   \big],\\
\mathcal{F}[{}_{2}^{2}\mathfrak{E}H](s, \xi, \eta, v,   X_{\bot}(s), \zeta )&:= \nabla_v \big[ ( \hat{\zeta} \cdot \xi - \hat{v}\cdot \eta + \mu_1 |\xi-\eta|)^{-1}  \\
&\quad \times \mathcal{F}[\mathfrak{H}_{k_1,j_1;n_1}^{\mu_1,i_1}](s, \xi-\eta,\zeta) \cdot  \nabla_{\zeta}\big( \varphi_{l;-M_t}\big(\tilde{v}-\tilde{\zeta}  \big)  \mathcal{F}[ {}_{}^{1}\mathfrak{E}^{\kappa} ] (  _{\bot}X_{\bot}(s),   \eta, v, \zeta) \big)  \big],\\
 \mathcal{F}[{}_{1}^{2}\mathfrak{E}M](s, \xi, \eta, v,  X_{\bot}(s), \zeta)&:=   \nabla_\zeta \big[ ( \hat{\zeta} \cdot \xi - \hat{v}\cdot \eta + \mu_1 |\xi-\eta|)^{-1}\\
&\quad\times  \mathcal{F}[\widetilde{\mathfrak{H}_{k_1,j_1;n_1}^{\mu_1,i_1}}](s, \xi-\eta, v, \zeta)  \cdot  \nabla_{v}\big( \varphi_{l;-M_t}\big(\tilde{v}-\tilde{\zeta}  \big)  \mathcal{F}[ {}_{}^{1}\mathfrak{E}^{\kappa} ] (  X_{\bot}(s),   \eta, v, \zeta) \big)\big],\\
 \mathcal{F}[{}_{2}^{2}\mathfrak{E}M](s, \xi, \eta, v,   X_{\bot}(s),  \zeta)&:= \nabla_v \big[  ( \hat{\zeta} \cdot \xi - \hat{v}\cdot \eta + \mu_1 |\xi-\eta|)^{-1}\\
&\quad \times \mathcal{F}[\widetilde{\mathfrak{H}_{k_1,j_1;n_1}^{\mu_1,i_1}}](s, \xi-\eta, v, \zeta)  \cdot  \nabla_{v}\big(  \varphi_{l;-M_t}\big(\tilde{v}-\tilde{\zeta}  \big)  \mathcal{F}[ {}_{}^{1}\mathfrak{E}^{\kappa} ] ( X_{\bot}(s),   \eta, v, \zeta) \big)  \big]. \\
\end{split}
\ee
 
The error terms in \eqref{oct29eqn40} are given as follows,  
 \be\label{oct28eqn1}
 \begin{split}
 ErrH^0_{i,i_1}(t_1, t_2) & =\sum_{a=1,2} (-1)^a  \int_{\R^3}  \int_{\R^3}\int_{\R^3} e^{i  X(t_a)\cdot \xi  - i t_a\hat{v}\cdot \eta + i t_a\mu_1 |\xi-\eta|  }
\hat{g}(t_a, \eta, v) \\
&\quad \times  ( \hat{V}(t_a)\cdot \xi - \hat{v}\cdot \eta + \mu_1 |\xi-\eta|)^{-1}  \mathcal{F}[\mathfrak{H}_{k_1,j_1;n_1}^{\mu_1,i_1}](t_a, \xi-\eta, V(t_a))  \\
 &\quad       \cdot  \nabla_{\zeta}\big( \varphi_{l;-M_t}\big(\tilde{v}-\tilde{\zeta}  \big)  \mathcal{F}[ {}_{}^{1}\mathfrak{E}^{\kappa} ] (  X_{\bot}(t_a),   \eta, v, \zeta) \big) \big|_{\zeta=V(t_a)}  d \eta d \xi  d v ,\\
 \end{split}
\ee
\be\label{2024oct30eqn81}
\begin{split}
 ErrM^0_{i,i_1}(t_1, t_2) &=\sum_{a=1,2} (-1)^a  \int_{\R^3}  \int_{\R^3}\int_{\R^3} e^{i  X(t_a)\cdot \xi  - i t\hat{v}\cdot \eta + i t_a\mu_1 |\xi-\eta|  }
\hat{g}(t_a, \eta, v) \\
&\quad \times  ( \hat{V}(t_a)\cdot \xi - \hat{v}\cdot \eta + \mu_1 |\xi-\eta|)^{-1}  \mathcal{F}[\widetilde{\mathfrak{H}_{k_1,j_1;n_1}^{\mu_1,i_1}}](t_a, \xi-\eta,v,  V(t_a))\\
  &\quad   \cdot    \nabla_{ v  }\big(  \varphi_{l;-M_t}\big(\tilde{v}-\tilde{\zeta}  \big)  \mathcal{F}[ {}_{}^{1}\mathfrak{E}^{\kappa} ] (   X_{\bot}(t_a),   \eta, v, V(t_a)) \big)    d \eta d \xi  d v ,\\
 \end{split}
\ee

\be\label{2024oct30eqn82}
\begin{split}
  ErrH^1_{i,i_1}(t_1, t_2) &= \int_{t_1}^{t_2}\int_{\R^3}  \int_{\R^3}\int_{\R^3} e^{i  X(s)\cdot \xi  - i s\hat{v}\cdot \eta + i s \mu_1 |\xi-\eta|  } ( \hat{V}(s)\cdot \xi - \hat{v}\cdot \eta + \mu_1 |\xi-\eta|)^{-1} \\
  &\quad \times     \nabla_{\zeta}\big(  \varphi_{l;-M_t}\big(\tilde{v}-\tilde{\zeta}  \big)  \mathcal{F}[ {}_{}^{1}\mathfrak{E}^{\kappa} ] (   X_{\bot}(s),   \eta, v, \zeta) \big)  \big|_{\zeta=V(s)} \\
  &\quad \cdot  \p_s   \mathcal{F}[\mathfrak{H}_{k_1,j_1;n_1}^{\mu_1,i_1}](s , \xi-\eta, V(s))     \widehat{g}(s, \eta, v)      d \eta d \xi  d v d s ,\\
 \end{split}
\ee

 \be\label{nov13eqn11}
 \begin{split}
 ErrM^1_{i,i_1}(t_1, t_2) & = \int_{t_1}^{t_2}\int_{\R^3}  \int_{\R^3}\int_{\R^3} e^{i  X(s)\cdot \xi  - i s\hat{v}\cdot \eta + i s \mu_1 |\xi-\eta|  } ( \hat{V}(s )\cdot \xi - \hat{v}\cdot \eta + \mu_1 |\xi-\eta|)^{-1} \\
 &\quad \times      \nabla_{ v  }\big(  \varphi_{l;-M_t}\big(\tilde{v}-\tilde{\zeta}  \big)  \mathcal{F}[ {}_{}^{1}\mathfrak{E}^{\kappa} ] (   X_{\bot}(s),   \eta, v, \zeta) \big)  \big|_{\zeta=V(s)} \\
&\quad \cdot \p_s \mathcal{F}[\widetilde{\mathfrak{H}_{k_1,j_1;n_1}^{\mu_1,i_1}}]  (s, \xi-\eta,v,  V(s))  \widehat{g}(s, \eta, v)    d \eta d \xi  d v d s ,\\
 \end{split}
\ee

\be\label{2024oct30eqn83}
\begin{split}
ErrH^2_{i,i_1}(t_1, t_2) &  = \int_{t_1}^{t_2}\int_{\R^3}  \int_{\R^3}\int_{\R^3} e^{i  X(s)\cdot \xi    + i s \mu_1 |\xi-\eta|  } \mathcal{F}\big[(\hat{v}-\hat{V}(s))\times B f\big](s, \eta, v)\\
&\quad \cdot \nabla_v\big[  ( \hat{V}(s)\cdot \xi - \hat{v}\cdot \eta + \mu_1 |\xi-\eta|)^{-1} \mathcal{F}[\mathfrak{H}_{k_1,j_1;n_1}^{\mu_1,i_1}](s , \xi-\eta, V(s ))  \\
&\quad      \cdot  \nabla_{\zeta}\big(  \varphi_{l;-M_t}\big(\tilde{v}-\tilde{\zeta}  \big)  \mathcal{F}[ {}_{}^{1}\mathfrak{E}^{\kappa} ] (  X_{\bot}(s),   \eta, v, \zeta) \big)   \big|_{\zeta=V(s)} \big]   d \eta d \xi  d v d s,\\
 \end{split}
\ee

\be\label{2024oct30eqn84}
\begin{split}
ErrM^2_{i,i_1}(t_1, t_2) & = \int_{t_1}^{t_2}\int_{\R^3}  \int_{\R^3}\int_{\R^3} e^{i  X(s)\cdot \xi  - i s\hat{v}\cdot \eta + i s \mu_1 |\xi-\eta|  } \mathcal{F}\big[(\hat{v}-\hat{V}(s))\times B f\big](s, \eta, v)\\
&\quad \cdot \nabla_v\big[   ( \hat{V}(s )\cdot \xi - \hat{v}\cdot \eta + \mu_1 |\xi-\eta|)^{-1}    \mathcal{F}[\widetilde{\mathfrak{H}_{k_1,j_1;n_1}^{\mu_1,i_1}}]  (s , \xi-\eta,v,  V(s ))  \\
&\quad \cdot    \nabla_{ v  }\big(   \varphi_{l;-M_t}\big(\tilde{v}-\tilde{\zeta}  \big)  \mathcal{F}[ {}_{}^{1}\mathfrak{E}^{\kappa} ] (  X_{\bot}(s),   \eta, v, \zeta) \big)   \big|_{\zeta=V(s)} \big]  d \eta d \xi  d v d s,\\
 \end{split}
\ee

\be\label{2024oct30eqn85}
\begin{split}
ErrU^3_{i,i_1}(t_1, t_2) & =  \sum_{\begin{subarray}{c}
  k_2\in \Z_+, n_2\in[-M_t,2]\cap \Z\\ 
  \mu_2\in \{+,-\}, i_2\in\{0,1,2,3,4\}
  \end{subarray}} \int_{t_1}^{t_2}\int_{\R^3}  \int_{\R^3}\int_{\R^3} e^{i  X(s)\cdot \xi  - i s\hat{v}\cdot \eta + i s \mu_1 |\xi-\eta|  } \\
  &\quad\times \big(   Ini_{k_2,j_2,n_2}^{\mu_2,i_2}(s, X(s), V(s)) +  \mathfrak{E}^{\mu_2, i_2; l_2}_{k_2,j_2;n_2} (s, X(s), V(s))\mathbf{1}_{i_2\in\{0,1,2,3\} }  \big)\\
  & \quad \cdot     \mathcal{F}[{}_{1}^{2}\mathfrak{E}U](s, \xi, \eta, v,  X_{\bot}(s), V(s))  \widehat{g}(s, \eta, v)   d \eta d \xi  d v d s, \quad \forall U\in \{H, M\},\\
 \end{split}
\ee

\be\label{2024oct30eqn86}
\begin{split}
  ErrU^4_{i,i_1}(t_1, t_2) & =  \sum_{\begin{subarray}{c}
  k_2\in \Z_+, n_2\in[-M_t,2]\cap \Z\\ 
  \mu_2\in \{+,-\}, i_2\in \{0,1,2,3\}
  \end{subarray}}   \int_{t_1}^{t_2}\int_{(\R^3)^3}  e^{i  X(s)\cdot \xi   + i s \mu_1 |\xi-\eta|  } \mathcal{F}\big[ \big(      \mathbf{1}_{i_2\in\{0,1,2,3\} }    \\
   &\quad  \times     \mathfrak{E}^{\mu_2, i_2; l_2}_{k_2,j_2;n_2} (s, X(s), V(s))+  Ini_{k_2,j_2,n_2}^{\mu_2,i_2}(s, X(s), V(s)) \big)  {f}\big](s, \eta, v) \\
   &\quad    \cdot    \mathcal{F}[{}_{2}^{2}\mathfrak{E}U](s, \xi, \eta, v,   X_{\bot}(s), V(s))   d \eta  d \xi  d v d s, \quad \forall U\in \{H, M\}.  \\
\end{split}
\ee
 
 Estimates for the error terms $ErrU^a_{i,i_1}(t_1, t_2)$ ($a \in \{0, 1, 2, 3, 4\}$, $U \in \{H, M\}$) in \eqref{oct29eqn40} are deferred to Lemma \ref{errhypotype1}. 
We focus on the estimate of  $U^b_{i,i_1,i_2}(t_1,t_2), b\in \{1,2\} $ in  \eqref{nov14eqn31}.

We first estimate $U^1_{i,i_1,i_2}(t_1,t_2)$. Recall  \eqref{oct27eqn51}.    After writing $U^1_{i,i_1,i_2}(t_1,t_2)$ in terms of kernel, from  the estimates in \eqref{2024oct8eqn1} in Theorem \ref{mainresultsfirstpart},   if either $k_2+2n_2\leq k_1+2n_1 + 3M_t/5-30\iota M_t $ or $n_2\leq -  (\alpha^{\star}+3\iota+60\epsilon) M_t$, we have 
\[
\begin{split}
&\sum_{U\in \{H, M\}}|U^1_{i,i_1,i_2}(t_1,t_2)| \\
&\lesssim \mathcal{M}(C) 2^{-\gamma M_t+470\epsilon M_t}  2^{l-\min\{l,
\kappa\}} 2^{-(k+n+\min\{n,\kappa\})/2} \\
&\quad \times  \big(2^{ M_t} + 2^{(k_2+2n_2)/2+(\alpha^{\star} +3\iota) M_t}\mathbf{1}_{ n_2\geq - (\alpha^{\star}+3\iota+60\epsilon) M_t } \big)\\
&\quad  \times \big(2^{7M_t/6 +5\iota M_t/2 -(k_1+2n_1)/2} + 2^{  (\alpha^{\star} +3\iota)  M_t-(k_1+2n_1)/2-n}\big)  \\
  &\lesssim \mathcal{M}(C)\big[ 2^{(k_2+2n_2)/2 -(k_1+2n_1)/2 + 7M_t/10+14\iota M_t } \mathbf{1}_{ n_2\geq  (\alpha^{\star}+3\iota+60\epsilon) M_t } +2^{(\gamma-10\epsilon)M_t}\big] \\
  & \lesssim 2^{(\gamma-10\epsilon)M_t}\mathcal{M}(C).
  \end{split}
\]

Now, we estimate   $U^2_{i,i_1,i_2}(t_1,t_2)$. Recall \eqref{nov14eqn31}. After writing $U^2_{i,i_1,i_2}(t_1,t_2)$ in terms of kernel,  from  the estimates in \eqref{2024oct8eqn1} in Theorem \ref{mainresultsfirstpart},    if either $k_2+2n_2\leq k_1+2n_1 + 8M_t/15-40\iota M_t$ or $n_2\leq - (\alpha^{\star}+3\iota+60\epsilon)  M_t$, we have 
\[
\begin{split}
 \sum_{U\in \{H, M\}}|U^2_{i,i_1,i_2}(t_1,t_2)|&\lesssim \mathcal{M}(C)   \big(2^{ M_t} + 2^{(k_2+2n_2)/2+(\alpha^{\star} +3\iota) M_t}\mathbf{1}_{ n_2\geq - (\alpha^{\star}+3\iota+60\epsilon) M_t } \big) \\
 &\quad \times 2^{ \alpha^{\star}  M_t+180\epsilon M_t-(k_1+2n_1)/2}   (1+2^{ 2M_t/15} 2^{l}  ) \\
&\quad \times   2^{k -2\min\{l,\kappa\} -2j} \min\big\{2^{-3k-n-\min\{n,\kappa\} + 3j+2l}, 2^{-j} \big\}  \\
&\lesssim  \mathcal{M}(C)  (2^{ M_t} + 2^{(k_2+2n_2)/2+(\alpha^{\star} +3\iota) M_t}\mathbf{1}_{ n_2\geq - (\alpha^{\star}+3\iota+60\epsilon) M_t } )(1+2^{ 2M_t/15} 2^{l}  )  \\
  &\times  2^{(\alpha^{\star } +3\iota) M_t-(k_1+2n_1)/2}   2^{-5(k+l+\kappa)/4+ 2l-2\min\{l,\kappa\} +(5\kappa-3\min\{n,\kappa\})/4+200\epsilon M_t} \\
  &\lesssim 2^{(\gamma-10\epsilon)M_t }\mathcal{M}(C).
  \end{split}
\]
 
 \medskip

\noindent \textbf{Step 3.}\quad The third iteration of smoothing. 

 \medskip

Now, it suffices to consider the case $k_2+2n_2\geq  k_1+2n_1 + 8M_t/15-40\iota M_t$ and $n_2\geq   - (\alpha^{\star}+3\iota+60\epsilon)  M_t$. Recall  \eqref{nov14eqn31}. Note that, on the Fourier side, $\forall U\in \{H, M\},$ we have
\be\label{nov14eqn71}
\begin{split}
  U^1_{i,i_1,i_2}(t_1,t_2)&:= \int_{t_1}^{t_2}\int_{(\R^3)^4}  e^{i   \widetilde{\Phi}^1_{\mu_1, \mu_2}(\xi,\eta, \sigma;s,X(s),v)  } \widehat{g}(s, \eta, v)   \mathcal{F}[\mathfrak{H}_{k_2,j_2,n_2}^{\mu,i_2}](s, \sigma, V(s))\\
  &\quad    \cdot     \mathcal{F}[{}_{1}^{2}\mathfrak{E}U](s, \xi, \eta, v,  X_{\bot}(s), V(s))   d\sigma   d \eta d \xi  d v d s,  \\ 
   U^2_{i,i_1,i_2}(t_1,t_2)&:= \int_{t_1}^{t_2}\int_{(\R^3)^4}  e^{i  \widetilde{\Phi}^2_{\mu_1, \mu_2}(\xi,\eta, \sigma;s,X(s),v)}    \widehat{g}(s, \eta , v) \mathcal{F}[\mathfrak{H}_{k_2,j_2,n_2}^{\mu,i_2}](s, \sigma, V(s)) \\
   &\quad   \cdot   \mathcal{F}[{}_{2}^{2}\mathfrak{E}U](s, \xi, \eta+\sigma, v, X_{\bot}(s),  V(s))   d\sigma   d \eta d \xi  d v d s,\\ 
   \end{split}
\ee
where  the oscillation phases appear above are given as follows, 
\be\label{2024oct31eqn11}
\begin{split}
\widetilde{\Phi}^1_{\mu_1, \mu_2}(\xi,\eta, \sigma;s,X(s),v)&:=     X(s )\cdot ( \xi +\sigma) +  s(\mu_2 |\sigma| -  \hat{v}\cdot \eta + \mu_1 |\xi-\eta| ), \\  
\widetilde{\Phi}^2_{\mu_1, \mu_2}(\xi,\eta, \sigma;s,X(s),v)&:= X(s )\cdot\xi +   s (\mu_2 |\sigma| -  \hat{v}\cdot  \eta   + \mu_1 |\xi-\eta-\sigma| ).
\end{split}
\ee
 
 Note that, $\forall s\in [t_1,t_2],$ for the case we are considering, we have 
\be
\begin{split}
\big|\p_s\big(\widetilde{\Phi}^1_{\mu_1, \mu_2}(\xi,\eta, \sigma;s,X(s),v)\big) \big|&\sim    |\sigma|+ \mu_2\hat{V}(s)\cdot \sigma   ,\\
  \big| \p_s\big( \widetilde{\Phi}^2_{\mu_1, \mu_2}(\xi,\eta, \sigma;s,X(s),v)\big)\big|  &\sim \big|| \mu_2|\sigma| + \hat{v}\cdot \sigma |\big|   ,
\end{split}
\ee 

Therefore, to exploit high oscillation in time,  we do integration by parts in $s$ once. As a result, for $U\in \{H, M\}$,  $b\in\{1,2\},$ we have
\be\label{nov15eqn99}
\begin{split}
U^b_{i,i_1,i_2}(t_1,t_2)= \sum_{a=0,1,2,3}ErrU^{b;a}_{i,i_1,i_2}(t_1, t_2)  + \sum_{\begin{subarray}{c}
  k_3\in \Z_+, n_3\in[-M_t,2]\cap \Z\\
  \mu_3\in \{+,-\}, i_3\in\{0,1,2,3,4\}
  \end{subarray} } U^b_{i,i_1,i_2,i_3}(t_1,t_2),
  \end{split}
\ee
where
\be\label{nov12eqn81}
\begin{split}
U^b_{i,i_1,i_2,i_3}(t_1,t_2) &:= \int_{t_1}^{t_2}\int_{(\R^3)^3}    e^{i \widetilde{\Phi}^b_{\mu_1, \mu_2}(\xi,\eta, \sigma;s,X(s),v)  } \widehat{g}(s, \eta, v)\mathfrak{H}_{k_3,j_3,n_3}^{\mu_3,i_3}(s, X(s), V(s))\\
&\quad      \cdot \mathcal{F}[{}_{b}^{3}\mathfrak{E}U](s, \xi, \eta,\sigma, v,  X_{\bot}(s), V(t)) d \eta d \xi  d v d s,\\
&\\
\mathcal{F}[{}_{1}^{3}\mathfrak{E}U](s, \xi, \eta,\sigma,  v,  X_{\bot}(s), \zeta)&:= \nabla_\zeta \big[ \big(  \p_s(\widetilde{\Phi}^1_{\mu_1, \mu_2}(\xi,\eta, \sigma;s,X(s),v))  \big)^{-1} \\
&\quad \times \mathcal{F}[\mathfrak{H}_{k_2,j_2;n_2}^{\mu_2,i_2}](s , \sigma , \zeta)) \cdot     \mathcal{F}[{}_{1}^{2}\mathfrak{E}U](s, \xi, \eta, v,  X_{\bot}(s), \zeta)    \big],\\
&\\
\mathcal{F}[{}_{2}^{3}\mathfrak{E}U](s, \xi, \eta,\sigma,  v,  X_{\bot}(s),\zeta)&:= \nabla_\zeta \big[ \big(  \p_s( \widetilde{\Phi}^2_{\mu_1, \mu_2}(\xi,\eta, \sigma;s,X(s),v)) \big)^{-1} \\
&\quad \times  \mathcal{F}[\mathfrak{H}_{k_2,j_2;n_2}^{\mu_2,i_2}](s , \sigma , \zeta)) \cdot    \mathcal{F}[{}_{2}^{2}\mathfrak{E}U](s, \xi, \eta+\sigma, v, X_{\bot}(s), \zeta)  \big],
\end{split}
\ee

The error terms in \eqref{nov15eqn99} are defined as follows, 
 \be\label{oct29eqn6}
 \begin{split}
 ErrU^{1;0}_{i,i_1,i_2}(t_1, t_2) & =\sum_{a=1,2} (-1)^a  \int_{\R^3}  \int_{\R^3}\int_{\R^3} e^{i   \widetilde{\Phi}^1_{\mu_1, \mu_2}(\xi,\eta, \sigma;t_a, X(t_a),v)   }\\
 &\quad \times      \mathcal{F}[\mathfrak{H}_{k_2,j_2,n_2}^{\mu,i_2}](t_a, \sigma, V(t_a))  \cdot         \mathcal{F}[{}_{1}^{2}\mathfrak{E}U](t_a, \xi, \eta, v,   X_{\bot}(t_a), V(t_a))    \\
  &\quad\times \big(  \p_s(\widetilde{\Phi}^1_{\mu_1, \mu_2}(\xi,\eta, \sigma;t_a,X(t_a),v))  \big)^{-1}  \widehat{g}(t_a, \eta, v) d \eta d \xi  d v, \\
 \end{split}
\ee

\be\label{2024oct30eqn100}
 \begin{split}
 ErrU^{2;0}_{i,i_1,i_2}(t_1, t_2) & =\sum_{a=1,2} (-1)^a  \int_{\R^3}  \int_{\R^3}\int_{\R^3} e^{i   \widetilde{\Phi}^2_{ \mu_1, \mu_2}(\xi,\eta, \sigma;t_a, X(t_a),v)   } \\
  &\quad  \times       \mathcal{F}[\mathfrak{H}_{k_2,j_2,n_2}^{\mu,i_2}](t_a, \sigma, V(t_a))  \cdot      \mathcal{F}[{}_{2}^{2}\mathfrak{E}U](t_a, \xi, \eta+\sigma, v,    X_{\bot}(t_a), V(t_a))  \\
 &\quad \times  \big(  \p_s(\widetilde{\Phi}^2_{\mu_1, \mu_2}(\xi,\eta, \sigma;t_a,X(t_a),v))  \big)^{-1}  \widehat{g}(t_a, \eta, v)  
  d \eta d \xi  d v, \\
 \end{split}
\ee

\be\label{2024oct30eqn101}
 \begin{split}
  ErrU^{1;1}_{i,i_1,i_2}(t_1, t_2)&: = \int_{t_1}^{t_2}\int_{\R^3}  \int_{\R^3}\int_{\R^3} e^{i  \widetilde{\Phi}^1_{ \mu_1, \mu_2}(\xi,\eta, \sigma; s , X(s ),v)   } \big(  \p_s(\widetilde{\Phi}^1_{\mu_1, \mu_2}(\xi,\eta, \sigma;s,X(s),v))  \big)^{-1}\\
  &\quad \times   \widehat{g}(s , \eta, v)  \big[ \p_s        \mathcal{F}[\mathfrak{H}_{k_2,j_2,n_2}^{\mu,i_2}](s , \sigma, V(s)) \cdot      \mathcal{F}[{}_{1}^{2}\mathfrak{E}U](s , \xi, \eta, v,  X_{\bot}(s),  V(s ))\\
 &\quad  +        \mathcal{F}[\mathfrak{H}_{k_2,j_2,n_2}^{\mu,i_2}](s, \sigma, V(s )) \cdot       \p_s  \mathcal{F}[{}_{1}^{2}\mathfrak{E}U](s , \xi, \eta, v,    X_{\bot}(s), V(s )) \big]   d \eta d \xi  d v d s,\\
 \end{split}
\ee

\be\label{2024oct30eqn102}
 \begin{split}
 ErrU^{2 ;1}_{i,i_1,i_2}(t_1, t_2) &= \int_{t_1}^{t_2}\int_{\R^3}  \int_{\R^3}\int_{\R^3} e^{i   \widetilde{\Phi}^2_{ \mu_1, \mu_2}(\xi,\eta, \sigma;s, X(s ),v)   }   \big(  \p_s(\widetilde{\Phi}^2_{\mu_1, \mu_2}(\xi,\eta, \sigma;s,X(s),v))  \big)^{-1} \\
 &\quad \times   \widehat{g}(s, \eta, v)\big[ \p_s    \mathcal{F}[\mathfrak{H}_{k_2,j_2,n_2}^{\mu,i_2}](s , \sigma, V(s )) \cdot        \mathcal{F}[{}_{2}^{2}\mathfrak{E}U](s , \xi, \eta+\sigma, v,  X_{\bot}(s), V(s))  \\
 &\quad+        \mathcal{F}[\mathfrak{H}_{k_2,j_2,n_2}^{\mu,i_2}](s , \sigma, V(s)) \cdot       \p_s   \mathcal{F}[{}_{2}^{2}\mathfrak{E}U](s , \xi, \eta+\sigma, v,  X_{\bot}(s), V(s)) \big]   d \eta d \xi  d v d s,\\
 \end{split}
\ee

\be\label{nov15eqn40}
\begin{split}
ErrU^{1;2}_{i,i_1,i_2}(t_1, t_2)& = \int_{t_1}^{t_2}\int_{\R^3}  \int_{\R^3}\int_{\R^3} e^{i \widetilde{\Phi}^1_{ \mu_1, \mu_2}(\xi,\eta, \sigma;s , X(s ),v)  } \big(  \p_s(\widetilde{\Phi}^1_{\mu_1, \mu_2}(\xi,\eta, \sigma;s,X(s),v))  \big)^{-1}  \\
&\quad \times         \mathcal{F}[\mathfrak{H}_{k_2,j_2,n_2}^{\mu,i_2}](s , \sigma, V(s )) \cdot     \mathcal{F}[{}_{1}^{2}\mathfrak{E}U](s , \xi, \eta, v,   X_{\bot}(s), V(s ))    \\
&\quad \times   \p_s \widehat{g}(s, \eta, v)   d \eta d \xi  d v d s,\\
 \end{split}
\ee
 
 \be\label{2024oct30eqn103}
 \begin{split}
ErrU^{2;2}_{i,i_1,i_2}(t_1, t_2)& = \int_{t_1}^{t_2}\int_{\R^3}  \int_{\R^3}\int_{\R^3} e^{i \widetilde{\Phi}^2_{ \mu_1, \mu_2}(\xi,\eta, \sigma;s, X(s),v)  }   \big(  \p_s(\widetilde{\Phi}^2_{\mu_1, \mu_2}(\xi,\eta, \sigma;s,X(s),v))  \big)^{-1}   \\
&\quad \times      \mathcal{F}[\mathfrak{H}_{k_2,j_2,n_2}^{\mu,i_2}](s , \sigma, V(s)) \cdot      \mathcal{F}[{}_{2}^{2}\mathfrak{E}U](s, \xi, \eta+\sigma, v,   X_{\bot}(s), V(s ))  \\
&\quad\times   \p_s \widehat{g}(s, \eta, v)   d \eta d \xi  d v d s,\\
 \end{split}
\ee
 
 \be\label{2024oct30eqn104}
 \begin{split}
ErrU^{b;3}_{i,i_1,i_2}(t_1, t_2)&=\sum_{\begin{subarray}{c}
  k_3\in \Z_+, n_3\in[-M_t,2]\cap \Z\\ 
  \mu_3\in \{+,-\}, i_3\in\{0,1,2,3,4\}
  \end{subarray} } \int_{t_1}^{t_2} \int_{\R^3}  \int_{\R^3}\int_{\R^3}   e^{i   \widetilde{\Phi}^b_{ \mu_1, \mu_2}(\xi,\eta, \sigma;s , X(s ),v)  } \widehat{g}(s, \eta, v)\\
  &\quad \times  \big(     Ini_{k_3,j_3,n_3}^{\mu_3,i_3}(s, X(s), V(s))  +   \mathfrak{E}^{\mu_3, i_3 }_{k_3,j_3;n_3} (s, X(s), V(s))  \\
  &\quad  \times  \mathbf{1}_{i_3\in\{0,1,2,3\} }        \big) \cdot  \mathcal{F}[{}_{b}^{3}\mathfrak{E}U](s, \xi, \eta,\sigma, v, X_{\bot}(s),  V(s))     d \eta d \xi  d v d s.\\
  \end{split}
\ee

The error terms in \eqref{nov15eqn99} will be estimated in   Lemma \ref{ellerrstep2}. At this stage, we focus on the  estimate of main terms    $U^b_{i,i_1,i_2,i_3}(t_1,t_2),b\in \{1,2\}, U\in \{H, M\}$. 

 Recall \eqref{nov12eqn81}. After writing $  U^b_{i,i_1,i_2,i_3}(t_1,t_2)$ in terms of kernel,  from  the estimates in \eqref{2024oct8eqn1} in Theorem \ref{mainresultsfirstpart},     for any $U\in \{H, M\},$ if either $k_3+2n_3\leq k_2+2n_2 + 13M_t/15-50\iota M_t $ or $n_3\leq - (\alpha^{\star}+3\iota+60\epsilon) M_t $, then we have 
 \[
 \begin{split}
 \big| U^1_{i,i_1,i_2,i_3}(t_1,t_2) \big|&\lesssim \mathcal{M}(C) \big(2^{ M_t} + 2^{(k_3+2n_3)/2+(\alpha^{\star} +3\iota) M_t}\mathbf{1}_{ n_3\geq - (\alpha^{\star}+3\iota+60\epsilon) M_t } \big)\\
 &\quad \times     2^{-2\gamma  M_t+800\epsilon M_t } 2^{ 2(\alpha^{\star}+3\iota) M_t-(k_2+2n_2)/2} 2^{7M_t/6 +5\iota M_t/2-(k_1+2n_1)/2} \\
 &\quad \times   2^{k-\min\{l,
\kappa\}-j}  \min\big\{2^{-3k-n-\min\{n,\kappa\} + 3j+2l}, 2^{-j} \big\}\\
& \lesssim  \mathcal{M}(C)  2^{(\gamma-10\epsilon)M_t}, \\
&\\
 \big| U^2_{i,i_1,i_2,i_3}(t_1,t_2) \big| &\lesssim \mathcal{M}(C) (2^{ M_t} + 2^{(k_3+2n_3)/2+(\alpha^{\star} +3\iota) M_t}\mathbf{1}_{ n_3\geq - (\alpha^{\star}+3\iota+60\epsilon) M_t } )  2^{- \gamma  M_t +800\epsilon M_t }   \\
 &\quad \times 2^{ (\alpha^{\star}+3\iota) M_t-(k_2+2n_2)/2 +7M_t/6 +5\iota M_t/2 -(k_1+2n_1)/2}  2^{k -2\min\{l,\kappa\} -2j} \\
 &\quad \times  \min\big\{2^{-3k-n-\min\{n,\kappa\} + 3j+2l}, 2^{-j} \big\}\\
 &\lesssim \mathcal{M}(C)  \big(2^{ M_t} + 2^{(k_3+2n_3)/2+  (\alpha^{\star} +3\iota)  M_t} \mathbf{1}_{ n_3\geq -\alpha^{\star} M_t-50\epsilon M_t} \big)\\
 &\quad \times    2^{ (\alpha^{\star} +3\iota) M_t -(k_2+2n_2)/2 + M_t/6+5\iota M_t/2 -(k_1+2n_1)/2}  \\
 &\quad \times 2^{-5(k+l+\kappa)/4+ 2l -2\min\{l,\kappa\} +(5\kappa-3\min\{n,\kappa\})/4+800\epsilon M_t}\\
 & \lesssim  \mathcal{M}(C) 2^{(\gamma-10\epsilon)M_t}. 
 \end{split}
 \]

\medskip

\noindent \textbf{Step 4.}\quad The fourth (last) iteration of smoothing. 

 \medskip

Now, it suffices   to consider the case  $k_3+2n_3\geq k_2+2n_2 + 13M_t/15 -50\iota M_t$ and  $n_3\geq -(\alpha^{\star}+3\iota+60\epsilon) M_t $.   Recall  \eqref{nov12eqn81}. For this case, note that, on the Fourier side, $ \forall U\in \{H, M\}, b\in \{1,2\}, $ we have
\be
\begin{split}
  U^b_{i,i_1,i_2,i_3}(t_1,t_2)&:= \int_{t_1}^{t_2}\int_{(\R^3)^4}  e^{i \widetilde{\Phi}^b_{\mu_1,\mu_2,\mu_3}(\xi, \eta, \sigma,\kappa;s,X(s), v)   } \widehat{g}(s, \eta, v)   \mathfrak{H}_{k_3,j_3,n_3}^{\mu_3,i_3}(s, \kappa, V(s))\\
  &\quad     \cdot   \mathcal{F}[{}_{b}^{3}\mathfrak{E}U](s, \xi, \eta,\sigma, v,  X_{\bot}(s),  V(s))    d \kappa  d\sigma   d \eta d \xi  d v d s, \\
   \widetilde{\Phi}^b_{\mu_1,\mu_2,\mu_3}(\xi, \eta, \sigma,\kappa;s,x, v)&:= \widetilde{\Phi}^b_{\mu_1, \mu_2}(\xi,\eta, \sigma;s,x,v) + x \cdot \kappa + s \mu_3 |\kappa|
\end{split}
\ee
where $\widetilde{\Phi}^b_{\mu_1, \mu_2}(\xi,\eta, \sigma;s,x,v), b\in\{1,2\},$ are defined in \eqref{2024oct31eqn11}.  Note  that, for the case we are considering, we have
\be
\begin{split}
\p_s\big(    \widetilde{\Phi}^b_{ \mu_1, \mu_2,\mu_3}(\xi,\eta, \sigma,\kappa;s, X(s),v)\big)&:= \p_s\big( \widetilde{\Phi}^b_{\mu_1, \mu_2}(\xi,\eta, \sigma;s,X(s),v)\big) + \hat{V}(s)\cdot \kappa + \mu_3|\kappa| ,\\
 \Longrightarrow &\quad | \p_s\big(    \widetilde{\Phi}^b_{ \mu_1, \mu_2,\mu_3}(\xi,\eta, \sigma,\kappa;s, X(s),v)\big)|   \sim |  \hat{V}(s)\cdot \kappa + \mu_3|\kappa||.
 \end{split}
\ee

To exploit high oscillation in time,  we do integration by parts in $s$ once. As a result, for $U\in \{H, M\}, b\in \{1,2\}$,  we have
\be\label{nov15eqn96}
U^b_{i,i_1,i_2,i_3}(t_1,t_2)= \sum_{a=0,1,2 }ErrU^{b;a}_{i,i_1,i_2,i_3}(t_1, t_2)  + LastU^b_{i,i_1,i_2,i_3}(t_1,t_2),
\ee
where
\be\label{nov12eqn90}
\begin{split}
LastU^b_{i,i_1,i_2,i_3}(t_1,t_2)&= \int_{t_1}^{t_2}\int_{(\R^3)^5}  e^{i  \widetilde{\Phi}^b_{\mu_1,\mu_2,\mu_3}(\xi, \eta, \sigma,\kappa;s,X(s), v) } \widehat{g}(s, \eta, v) K (s, X(s), V(s))  \\
&\quad     \cdot     \mathcal{F}[{}_{b}^{4}\mathfrak{E}U](s, \xi, \eta,\sigma, v,   X_{\bot}(s),  V(s)) d \kappa  d\sigma   d \eta d \xi  d v d s, \\
 \mathcal{F}[{}_{b}^{4}\mathfrak{E}U](s, \xi, \eta,\sigma, v,   X_{\bot}(s), \zeta) &  : = \nabla_{\zeta}\big[ \big(  \p_s\big(    \widetilde{\Phi}^b_{ \mu_1, \mu_2,\mu_3}(\xi,\eta, \sigma,\kappa;s, X(s),v)\big) \big)^{-1} \\
 &\quad \times  \mathcal{F}[\mathfrak{H}_{k_3,j_3,n_3}^{\mu_3,i_3}](s, \kappa, \zeta)  \cdot   \mathcal{F}[{}_{b}^{3}\mathfrak{E}U](s, \xi, \eta,\sigma, v,   X_{\bot}(s),\zeta )  \big] ,  \\
 \end{split}
\ee

The error terms in \eqref{nov15eqn96} are given as follows
 
  \be\label{nov12eqn91}
  \begin{split}
  ErrU^{b;0}_{i,i_1,i_2,i_3}(t_1, t_2)&:= \sum_{a=1,2}(-1)^{a-1} \int_{ \R^3 }\int_{ \R^3 }\int_{ \R^3 }\int_{ \R^3 }\int_{ \R^3 }  e^{i  \Phi_{\mu_1,\mu_2,\mu_3}(\xi, \eta, \sigma,\kappa;t_a,X(t_a))   }  \widehat{g}(t_a, \eta, v)  \\
  &\quad \times \big(  \p_s\big(    \widetilde{\Phi}^b_{ \mu_1, \mu_2,\mu_3}(\xi,\eta, \sigma,\kappa;t_a, X(t_a),v)\big) \big)^{-1}  \mathcal{F}[\mathfrak{H}_{k_3,j_3,n_3}^{\mu_3,i_3}](t_a, \kappa, V(t_a)) \\
  &\quad   \cdot    \mathcal{F}[{}_{b}^{3}\mathfrak{E}U] (t_a, \xi, \eta,\sigma, v,   X_{\bot}(t_a), V(t_a))  d \kappa  d\sigma   d \eta d \xi  d v, \\
 \end{split}
\ee
\be\label{2024oct31eqn21}
  \begin{split}
 ErrU^{b;1}_{i,i_1,i_2,i_3}(t_1, t_2)&:=  \int_{t_1}^{t_2}\int_{ \R^3 }\int_{ \R^3 }\int_{ \R^3 }\int_{ \R^3 }\int_{ \R^3 } e^{i   \widetilde{\Phi}^b_{\mu_1,\mu_2,\mu_3}(\xi, \eta, \sigma,\kappa;s,X(s), v)   }   \p_s \widehat{g}(s, \eta, v)\\
 &\quad\times   \big(  \p_s\big(    \widetilde{\Phi}^b_{ \mu_1, \mu_2,\mu_3}(\xi,\eta, \sigma,\kappa;s, X(s),v)\big) \big)^{-1}\mathcal{F}[\mathfrak{H}_{k_3,j_3,n_3}^{\mu_3,i_3}](s , \kappa, V(s))\\
 &\quad      \cdot   \mathcal{F}[{}_{b}^{3}\mathfrak{E}U] (s , \xi, \eta,\sigma, v,  X_{\bot}(s),  V(s))  d \kappa  d\sigma   d \eta d \xi  d v d s,\\
 \end{split}
\ee

 \be\label{2024oct31eqn22}
  \begin{split}
 ErrU^{b;2}_{i,i_1,i_2,i_3}(t_1, t_2)&:=  \int_{t_1}^{t_2} \int_{ \R^3 }\int_{ \R^3 }\int_{ \R^3 }\int_{ \R^3 }\int_{ \R^3 }   e^{i   \widetilde{\Phi}^b_{\mu_1,\mu_2,\mu_3}(\xi, \eta, \sigma,\kappa;s,X(s), v)   } \\
 &\quad \times   \widehat{g}(s, \eta, v) \big(  \p_s\big(    \widetilde{\Phi}^b_{ \mu_1, \mu_2,\mu_3}(\xi,\eta, \sigma,\kappa;s, X(s),v)\big) \big)^{-1}\\
 &\quad \times \big[ \p_s\mathcal{F}[\mathfrak{H}_{k_3,j_3,n_3}^{\mu_3,i_3}](s, \kappa, V(s))  \cdot    \mathcal{F}[{}_{b}^{3}\mathfrak{E}U](s , \xi, \eta,\sigma, v,   X_{\bot}(s), V(s))\\
 &\quad + \mathcal{F}[\mathfrak{H}_{k_3,j_3,n_3}^{\mu_3,i_3}](s , \kappa, V(s))  \cdot   \p_s   \mathcal{F}[{}_{b}^{3}\mathfrak{E}U](s, \xi, \eta,\sigma, v,  X_{\bot}(s),   V(s)) \big]    d \kappa  d\sigma   d \eta d \xi  d v d s.\\
  \end{split}
\ee

The error terms in  \eqref{nov15eqn96} will be estimated in Lemma \ref{ellerrstep3}. At this stage, we focus on the estimate of the main parts $LastU^b_{i,i_1,i_2,i_3}(t_1,t_2),b\in \{1,2\}, U\in \{H, M\}$, $b\in\{1,2\}$.

Recall \eqref{nov12eqn90}.  After writing  $ LastU^b_{i,i_1,i_2,i_3}(t_1,t_2), b\in \{1,2\},$ in terms of kernel, from the estimates in \eqref{2024oct8eqn1} in Theorem \ref{mainresultsfirstpart},  and the rough estimate of the electromagnetic field in  \eqref{maintheoremroughest} in Theorem  \ref{mainresultsfirstpart},   we have 
\[
\begin{split}
&\big| LastU^1_{i,i_1,i_2,i_3}(t_1,t_2) \big|\\
&\lesssim  \mathcal{M}(C) 2^{(1+2\alpha^{\star}+ \iota)M_t} 2^{-3\gamma M_t} 2^{2(\alpha^{\star} + 3\iota) M_t -(k_3+2n_3)/2}  2^{2 (\alpha^{\star} + 3\iota) M_t -(k_2+2n_2)/2} \\
&\quad  \times 2^{7M_t/6+5\iota M_t/2-(k_1+2n_1)/2} 2^{k  - j - \min\{l, \kappa\} } \min\{ 2^{-3k-n-\min\{n,\kappa\}}2^{3j+2l},  2^{-j}\}    \\
& \lesssim \mathcal{M}(C)    2^{19M_t/6+ 22\iota M_t-(k_3+2n_3)/2-(k_2+2n_2)/2-(k_1+2n_1)/2}  2^{l-\min\{l,\kappa\}}  2^{-(k+n+\min\{n,\kappa\})/2}\\
&\lesssim \mathcal{M}(C) 2^{11M_t/15 + 100\iota  M_t}    \lesssim \mathcal{M}(C) 2^{(\gamma-10\epsilon)M_t }  , \\
&\\
&\big| LastU^2_{i,i_1,i_2,i_3}(t_1,t_2) \big|\\
&\lesssim   2^{(1+  2\alpha^{\star}+400\epsilon)M_t} 2^{-2\gamma M_t} 2^{ (\alpha^{\star} + 3\iota)  M_t -(k_3+2n_3)/2}  2^{2 (\alpha^{\star} + 3\iota) M_t  -(k_2+2n_2)/2}\\
 &\quad \times 2^{7M_t/6+5\iota M_t/2-(k_1+2n_1)/2} 2^{k   -2 j - 2\min\{l, \kappa\} }   \min\{ 2^{-3k-n-\min\{n,\kappa\}}2^{3j+2l},  2^{-j}\}   \mathcal{M}(C)\\
 &  \lesssim \mathcal{M}(C)    2^{21M_t/6+ 17\iota M_t-(k_3+2n_3)/2-(k_2+2n_2)/2-(k_1+2n_1)/2}\\
&\quad \times   2^{-5(k+l+\kappa)/4+ 2l -2\min\{l,\kappa\} +(5\kappa-3\min\{n,\kappa\})/4}\\
&\lesssim 2^{(\gamma-10\epsilon)M_t } \mathcal{M}(C). 
\end{split}
\]
Hence, finishing the proof of our desired estimate  \eqref{nov12eqn21}.  
\end{proof}

\subsubsection{Proof of Lemma \ref{ellipticestpartII}}\label{proofofellipticestpartII}

Recall the decompositions in \eqref{2024oct29eqn41}, \eqref{oct3eqn111},  and \eqref{2024oct29eqn91}.  From the estimate \eqref{nov9eqn31} in Lemma \ref{firstreductiontrivialelli},  the estimate  \eqref{oct20eqn7}  and \eqref{nov11eqn231}  in Lemma \ref{ellfulltyp3},  we know that a set of parameters  and the threshold case for the other set of parameters are ruled out. For the other cases, we use the formulas of first iteration in \eqref{2024oct23eqn11} and \eqref{2024oct23eqn12} and the decomposition of $0$-parts in \eqref{nov9eqn62} and \eqref{findecompellipticpart2024oct}.

 From  the estimate \eqref{nov9eqn51} in Lemma \ref{ellipticstep2good},   the estimate  \eqref{nov28eqn80}  in Lemma \ref{ellipticstep1good},  and the estimates  \eqref{nov12eqn21},      \eqref{2024oct31eqn510}   in Lemma \ref{errortypefullhyp2}, and the estimate  \eqref{oct29eqn80}  in Lemma \ref{ellfulltyp4} for the error terms,  we have 
\be
 \sum_{a=0,1,2,3}\big|     \mathfrak{E}^{\mu,a;l}_{k,j,n}(t_1, t_2) \big| +\sum_{\kappa\in (\bar{\kappa}, 2]\cap \Z} \big| \mathfrak{E}^{\mu,4;\kappa, l}_{k,j,n}(t_1, t_2)\big|  \lesssim     \mathcal{M}(C)  2^{(\gamma-10\epsilon) M_t}.
\ee
 Hence finishing the estimate of the elliptic part, i.e.,   the desired estimate  \eqref{2024oct29eqn111}  in Lemma \ref{ellipticestpartII} holds.

\subsection{Estimating  error type terms}\label{errorcpartPartII}

In the following Lemma, we first estimate the error terms created in the first reduction, see \eqref{oct25eqn41}. 

\begin{lemma}\label{errmainfull0}
Let $ i\in \{0, 1,2,3,4\},   ( n,k)\in \widetilde{\mathcal{E}}_i$. Under the assumption of Proposition \textup{\ref{bootstraplemma2}}, we have
\be\label{oct15eqn22}
\sum_{j\in \Z_+}\sum_{a=0,1} \big|{}_a^{1}Err^{\mu,i}_{k,j;n}(t_1, t_2)\big| \lesssim   2^{  M_t/2} \mathcal{M}(C). 
\ee
\end{lemma}
\begin{proof}

Recall  \eqref{oct7eqn31}.  Due to the assumption  $\nabla_{  x_{\bot} } C(  X_{\bot}(t), V(t))=0 $ of the coefficient $C$ in \eqref{2024oct30eqn51}, we know that  ${}_1^{1}Err^i_{k,j;n}(t_1, t_2)=0.$ Hence, it's sufficient to estimate  ${}_0^{1}Err^i_{k,j;n}(t_1, t_2)$.

 From    the estimates in \eqref{2024oct8eqn1} in Theorem \ref{mainresultsfirstpart}, we have 
\[
\sum_{i=0, 1,2,3,4}\big|{}_0^{1}Err^i_{k,j;n}(t_1, t_2)\big|\lesssim \mathcal{M}(C)  2^{   (\alpha^\star + 3\iota + 140\epsilon) M_t- (k+2n)/2} \lesssim 2^{  M_t/2} \mathcal{M}(C). 
\]
 Hence finishing the proof of our desired estimate  \eqref{oct15eqn22}. 
\end{proof}

The rest of this section is naturally divided into two steps as follows. 
\begin{enumerate}
\item[$\bullet$]
In section \ref{errorhyperbolicPartII}, we mainly estimate the error type terms created in the ISS process of estimating the hyperbolic parts ${}_a^{1}\mathfrak{H}^{\mu,i}_{k,j;n}(t_1, t_2), a\in\{0,1\},i\in\{0,1,2,3,4\}$ in   \eqref{oct25eqn41}. 
\item[$\bullet$]  In section \ref{errorellipticPartIIISS}, we mainly estimate the error type terms created in the ISS process of estimating the elliptic parts $ \mathfrak{E}^{\mu,i}_{k,j;n}(t_1, t_2), i\in\{0,1,2,3,4\}$ in   \eqref{oct25eqn41}. 
\end{enumerate}
\subsubsection{Error terms in estimating the hyperbolic parts}\label{errorhyperbolicPartII}

We first estimate the error terms created in  estimating ${}_0^{1}\mathfrak{H}^{\mu,i}_{k,j;n}(t_1, t_2)$. For error terms in  \eqref{oct29eqn55}, the following lemma holds.

\begin{lemma}\label{errmainfull1}
Let $i\in\{0,1,2,3,4\}$,  $(k,n)\in \widetilde{\mathcal{E}}_{i}$,  $ (k_1, n_1)\in  {}^{1}\widetilde{\mathcal{E}}_{k,n}^{i_1,i}$.  Under the assumption of Proposition \textup{\ref{bootstraplemma2}}, we have  
\be\label{oct26eqn76}
  \sum_{a\in\{0,1,2 \}} \sum_{i_1\in\{ 0,1,2,3,4 \}} \big| Err^a_{ i, i_1}(t_1, t_2)\big|   \lesssim     2^{(\gamma -10\epsilon)M_t} \mathcal{M}(C) . 
\ee
\end{lemma}
\begin{proof}

Based on the size of $a$, we proceed in steps as follows.

\medskip

\noindent \textbf{Step 1.}\quad The estimate of  $Err^0_{ i, i_1}(t_1, t_2)$.

\medskip

Recall  \eqref{oct8eqn1}. After writing $ Err^0_{ i, i_1}(t_1, t_2)$ in terms of kernel,  from   the estimates in \eqref{2024oct8eqn1} in Theorem \ref{mainresultsfirstpart}, we have 
\be\label{oct26eqn1}
\begin{split}
| Err^0_{ i, i_1}(t_1, t_2)|&\lesssim \mathcal{M}(C) 2^{ (\alpha^{\star} + 3\iota) M_t -(k_1+2n_1)/2} 2^{-\gamma M_t} 2^{(7/6+3\iota)M_t-(k+2n)/2}\\
& \lesssim    2^{(\gamma -10\epsilon)M_t} \mathcal{M}(C) .
\end{split}
\ee

\medskip

\noindent \textbf{Step 2.}\quad The estimate of  $Err^1_{ i, i_1}(t_1, t_2)$.

\medskip

Recall \eqref{oct10eqn77} and  \eqref{oct7eqn90}. We observe that the third component 
of $\mathcal{F}[  {}^1 \mathfrak{H}](s, \xi,  X_{\bot}(s), V(s))$ is better than $\mathbf{P}\big( \mathcal{F}[  {}^1 \mathfrak{H}] (s, \xi,   X_{\bot}(s), V(s))\big)$.   Thanks to this observation, from the pointwise estimates in \eqref{oct7eqn21}   and \eqref{2024oct27eqn61} in Lemma \ref{secondhypohori} and   the estimates in \eqref{2024oct8eqn1} in Theorem \ref{mainresultsfirstpart}, we have   
\be\label{oct26eqn72}
\begin{split}
& \big|\int_{t_1}^{t_2} \int_{\R^3}\int_{\R^3} e^{i X(s)\cdot (\xi+\eta) + i \mu s|\xi| + i \mu_1 s|\eta|}   
  (  \hat{V}(s)\cdot (\xi+\eta) +   \mu |\xi| +  \mu_1 |\eta|)^{-1} \\
  &\quad \times \p_s  \mathcal{F}[\mathfrak{H}_{k_1,j_1;n_1}^{\mu_1,i_1}](s, \eta , V(s))  \cdot   \mathcal{F}[  {}^1 \mathfrak{H}] (s, \xi,  X_{\bot}(s), V(s))   d \xi d\eta  d s\big| \\
& \lesssim  \mathcal{M}(C)  \big(    2^{7\alpha^{\star} M_t/3-(k_1+2n_1)/2}    2^{-\gamma M_t} 2^{(7 M_t/6+3\iota)M_t-(k+2n)/2}  \\
&\quad +  2^{ (1+\alpha^\star )M_t-(k_1+2n_1)/2}    2^{-\gamma M_t} 2^{ (\alpha^{\star} + 3\iota)M_t-(k+2n)/2} \big)\\
&\lesssim  \mathcal{M}(C)  2^{(\gamma-10\epsilon)M_t}. \\
\end{split}
\ee 

It remains to consider the case when $\p_s$ hits $   \mathcal{F}[  {}^1 \mathfrak{H}](s, \xi,  X_{\bot}(s), V(s)) $. For this case, we observe that $\mathbf{P}\big(  \mathcal{F}[\mathfrak{H}_{k_1,j_1;n_1}^{\mu_1,i_1}](s, \eta ,   V(s))\big) $ is better than $  \mathcal{F}[\mathfrak{H}_{k_1,j_1;n_1}^{\mu_1,i_1}](s, \eta , V(s)) $, see \eqref{2022feb22eqn81}. In particular, for the case $i=3,4$, in which we have $|\tilde{v}-\tilde{\zeta}|\lesssim 2^{n+\epsilon M_t}$, we know that the symbol of $\mathbf{P}\big( \mathcal{F}[\mathfrak{H}_{k_1,j_1;n_1}^{\mu_1,i_1}](s, \eta , V(s))\big) $ is better than $ \mathcal{F}[\mathfrak{H}_{k_1,j_1;n_1}^{\mu_1,i_1}](s, \eta , V(s)) $ by a factor of $\max\{\frac{ \zeta_{\bot}}{|\zeta|},\frac{ v_{\bot}}{|v|}, \frac{ \xi_{\bot}}{|\xi|} \}$, which is less that $2^{\epsilon M_t}\max\{2^n, \frac{  \zeta_{\bot}}{|\zeta|}\}$. Therefore, similar to the obtained estimate   in \eqref{2024oct8eqn1} in Theorem \ref{mainresultsfirstpart} for  $i=3,4$, we have 
\[
\begin{split}
&\big\|\int_{\R^3} e^{ix\cdot \xi + i \mu s|\xi|} \mathbf{P}\big( \mathcal{F}[\mathfrak{H}_{k_1,j_1;n_1}^{\mu_1,i_1}]\big) d \xi\big\|_{L^\infty_x}\\
&\lesssim 2^{\epsilon M_t} \max\{2^n, \frac{|  \zeta_{\bot}|}{|\zeta|}\} \big[2^{(1-19\epsilon) M_{t^\star} } + 2^{128\epsilon M_{t^\star} }  \mathbf{1}_{n\geq   -(1/2+3\iota/2 + 40\epsilon)M_{t^{\star}} } \\ 
&\quad \times  \min\{2^{(k+2n)/2 +  {\alpha}^{\star} M_{t^\star}}, 2^{(k+4n)/2 +(7/6+5\iota/2)M_{t^\star}}\} \big]. 
\end{split}
\]

Recall  \eqref{oct15eqn1}  and  \eqref{firstindexsetlemm2}  for the definition of index sets $\widetilde{\mathcal{E}}_{i}$ and $   {}^{1}\widetilde{\mathcal{E}}_{k,n}^{i_1,i}$. From the above estimate,  the pointwise estimates  in \eqref{oct7eqn21}   and \eqref{2024oct27eqn61} in Lemma \ref{secondhypohori} and   the estimates in \eqref{2024oct8eqn1} in Theorem \ref{mainresultsfirstpart}, we have 
\be\label{oct26eqn71}
\begin{split}
&\big|\int_{t_1}^{t_2} \int_{\R^3}\int_{\R^3} e^{i X(s)\cdot (\xi+\eta) + i \mu s|\xi| + i \mu_1 s |\eta|}   
  (  \hat{V}(s)\cdot (\xi+\eta) +   \mu |\xi| +  \mu_1 |\eta|)^{-1} \\
  &\quad \times \mathcal{F}[\mathfrak{H}_{k_1,j_1;n_1}^{\mu_1,i_1}](s, \eta , V(s))  \cdot   \p_s  \mathcal{F}[  {}^1 \mathfrak{H}] (s, \xi,    X_{\bot}(s), V(s)) d \xi d\eta  d s\big| \\
&\lesssim \mathcal{M}(C) 2^{  (\alpha^{\star} + 3\iota)M_t+200\epsilon M_t-(k_1+2n_1)/2}  2^{-\gamma M_t-n} \\
&\quad \times 2^{ (1+\alpha^{\star})M_t-(k+2n)/2} \big(1+ 2^{\max\{n, (\alpha^{\star}-\gamma)M_t \}+\epsilon M_t}\mathbf{1}_{i=3,4}  \big)  \\
&\lesssim   \mathcal{M}(C) 2^{(\gamma-20\epsilon)M_t}. 
\end{split}
\ee
Recall  \eqref{oct10eqn77}. After combining the above two estimates  \eqref{oct26eqn72}  and  \eqref{oct26eqn71}, we have 
\be
| Err^1_{ i, i_1}(t_1, t_2)\big| \lesssim     \mathcal{M}(C)2^{(\gamma -20\epsilon)M_t}.
\ee

\medskip

\noindent \textbf{Step 3.}\quad The estimate of  $Err^2_{ i, i_1}(t_1, t_2)$.

\medskip

Recall \eqref{oct10eqn93}  and   \eqref{oct10eqn96}. 
Similar to the obtained estimate  \eqref{2021dec29eqn31}, from the rough $L^\infty_x$-type  estimate of the elliptic part in 
\eqref{2022feb25eqn1} in  Theorem \ref{maintheoremellipitic}    and  the estimates in \eqref{2024oct8eqn1} in Theorem \ref{mainresultsfirstpart}, we have   
\be\label{oct26eqn3}
\begin{split}
| Err^2_{ i, i_1}(t_1, t_2)\big|& \lesssim  2^{  5 \alpha^{\star}  M_t/3+300\epsilon M_t} 2^{-2\gamma M_t}  \big[2^{ 7M_t/6 + 5\iota M_t/2 -(k_1+2n_1)/2 + 7M_t/6+ 5\iota M_t/2-(k+2n)/2} \\
&\quad   +  2^{ (\alpha^{\star} + 3\iota) M_t -(k_1+2n_1)/2 }  2^{7M_t/6+ 5\iota M_t/2-(k+4n)/2}      \big] \mathcal{M}(C)  \\
 &\lesssim  \mathcal{M}(C) 2^{ 5\alpha^{\star} M_t/3- 7M_t/15+20\iota M_t}\\
 &\lesssim  \mathcal{M}(C)2^{(\gamma-10\epsilon)M_t}. 
 \end{split}
\ee
Hence finishing the proof of our desired estimate  \eqref{oct26eqn76}. 
\end{proof}

 As summarized in the following lemma,  for error terms in  \eqref{oct29eqn61}, we have 

\begin{lemma}\label{errortypefull6}
Let $i\in\{0,1,2,3,4\}$, $(k,n)\in \widetilde{\mathcal{E}}_{i}$,  $ (k_1, n_1)\in  {}^{1}\widetilde{\mathcal{E}}_{k,n}^{i_1,i}$, $(k_2, n_2)\in {}^{2}\widetilde{\mathcal{E}}_{k_1,n_1}.$  Under the assumption of Proposition \textup{\ref{bootstraplemma2}}, we have  
\be\label{oct26eqn80}
  \sum_{a\in\{0,1,2 \}} \sum_{i_1,i_2\in\{0,1,2,3,4\}} \big| Err^a_{ i, i_1,i_2}(t_1, t_2) \big|   \lesssim  \mathcal{M}(C) 2^{(\gamma-10\epsilon )M_t}. 
\ee
\end{lemma}
\begin{proof}

Based on the size of $a$, we proceed in steps as follows.

\medskip

\noindent \textbf{Step 1.}\quad The estimate of  $Err^0_{ i, i_1,i_2}(t_1, t_2)$.

\medskip

Recall  \eqref{oct11eqn187}  and  \eqref{seconditerindex}. After writing $ Err^0_{ i, i_1, i_2}(t_1, t_2)$ in terms of kernel,  from  the estimates in \eqref{2024oct8eqn1} in Theorem \ref{mainresultsfirstpart}, we have
\be\label{oct26eqn6}
\begin{split}
| Err^0_{ i, i_1,i_2}(t_1, t_2)|&\lesssim 2^{ (\alpha^{\star} + 3\iota) M_t +200\epsilon M_t -(k_2+2n_2)/2}  2^{-2\gamma M_t}\\
&\quad \times  \big[2^{ 7M_t/6+ 5\iota M_t/2-(k_1+2n_1)/2 + 7M_t/6+5\iota M_t/2-(k+2n)/2}  \\
&\quad +  2^{ (2/3+4\iota) M_t -(k_1+2n_1)/2 }  2^{7M_t/6+ 5\iota M_t/2 -(k+4n)/2}   \big] \mathcal{M}(C)\\
&  \lesssim  \mathcal{M}(C) 2^{(\gamma-10\epsilon )M_t}  .
\end{split}
\ee

\medskip

\noindent \textbf{Step 2.}\quad The estimate of  $Err^1_{ i, i_1,i_2}(t_1, t_2)$.

\medskip

Recall \eqref{oct11eqn42}. From the estimates in   \eqref{oct7eqn21}  and  \eqref{2024oct27eqn61}  in   Lemma \ref{secondhypohori} and   the estimates in \eqref{2024oct8eqn1} in Theorem \ref{mainresultsfirstpart}, we have  
\be 
\begin{split}
&| Err^1_{ i, i_1, i_2}(t_1, t_2)\big|\\
 &\lesssim 2^{6\iota} \mathcal{M}(C) \big[ 2^{ (1+\alpha^{\star})  M_t -(k_2+2n_2)/2} 2^{-2\gamma M_t } 2^{7M_t/6 -(k_1+2n_1)/2}2^{7M_t/6 -(k+2n)/2}\\
&\quad + 2^{  (\alpha^{\star} + 3\iota)  M_t -(k_2+2n_2)/2}  2^{-2\gamma M_t }  2^{-\min\{n,n_1\} } 2^{  (1+\alpha^{\star}) M_t+7M_t/6 -(k_1+2n_1)/2  -(k+2n)/2}   \big]\\
&\lesssim \mathcal{M}(C) 2^{ M_t/2 + 16M_t/9-M_t/3-M_t/2-5M_t/6+10 0\iota M_t} \\
&\lesssim \mathcal{M}(C) 2^{(\gamma-20\epsilon)M_t}. 
\end{split}
\ee

\medskip

\noindent \textbf{Step 3.}\quad The estimate of  $Err^2_{ i, i_1,i_2}(t_1, t_2)$.

\medskip

Recall  \eqref{oct11eqn43} and  \eqref{oct10eqn96}.   From the rough $L^\infty_x$-type  estimate of the elliptic part in 
\eqref{2022feb25eqn1} in  Theorem \ref{maintheoremellipitic} and   the estimates in \eqref{2024oct8eqn1} in Theorem \ref{mainresultsfirstpart}, we have   
 \be
\begin{split}
| Err^2_{ i, i_1, i_2}(t_1, t_2)\big| & \lesssim \mathcal{M}(C) 2^{ 5 \alpha^{\star}   M_t/3+6\iota M_t}  2^{-3\gamma  M_t}2^{2 (\alpha^{\star} + 3\iota) M_t-(k_2+2n_2)/2}  \\
&\quad \times 2^{(7/6+10\epsilon)M_t-(k_1+2n_1)/2}  2^{(7/6+10\epsilon)M_t-(k+2n)/2} \\
& \lesssim \mathcal{M}(C) 2^{16M_t/9 -M_t/3-M_t/2-5M_t/6+100\iota M_t} \\
&\lesssim \mathcal{M}(C) 2^{(\gamma-20\epsilon)M_t}. 
\end{split}
\ee
Hence finishing the proof of the desired estimate  \eqref{oct26eqn80}. 
\end{proof}
For error terms in   \eqref{oct12eqn76},  the following lemma holds. 
\begin{lemma}\label{errortypefull8}
Let $i\in\{0,1,2,3,4 \}$, $(k,n)\in \widetilde{\mathcal{E}}_{i}$,  $ (k_1, n_1)\in  {}^{1}\widetilde{\mathcal{E}}_{k,n}^{i_1,i}$, $(k_2, n_2)\in {}^{2}\widetilde{\mathcal{E}}_{k_1,n_1}, (k_3, n_3)\in {}^{3}\widetilde{\mathcal{E}}_{k_2,n_2}.$ Under the assumption of Proposition \textup{\ref{bootstraplemma2}}, we have  
\be\label{oct26eqn86}
  \sum_{a\in\{0,1,2  \}} \sum_{i_1,i_2,i_3\in\{0,1,2,3,4\}} \big| Err^a_{ i, i_1,i_2,i_3}(t_1, t_2) \big|   \lesssim  \mathcal{M}(C) 2^{(\gamma-20\epsilon)M_t}.
\ee
\end{lemma}
\begin{proof}

Based on the size of $a$, we proceed in steps as follows.

\medskip

\noindent \textbf{Step 1.}\quad The estimate of  $Err^0_{ i, i_1,i_2,i_3}(t_1, t_2) $.

\medskip

Recall  \eqref{oct12eqn2}. After writing $ Err^0_{ i, i_1, i_2, i_3}(t_1, t_2)$ in terms of kernel, from  the estimates in \eqref{2024oct8eqn1} in Theorem \ref{mainresultsfirstpart},  we have
\be 
\begin{split}
| Err^0_{ i, i_1,i_2,i_3}(t_1, t_2)|&\lesssim 2^{  (\alpha^{\star} + 3\iota)  M_t + 6\iota M_t -(k_3+2n_3)/2}   2^{-3\gamma  M_t}\\
&\quad \times  2^{2  (\alpha^{\star} + 3\iota)  M_t+ 7M_t/3 -(k_2+2n_2)/2 -(k_1+2n_1)/2 -(k+2n)/2} \mathcal{M}(C)\\
&\lesssim  \mathcal{M}(C) 2^{(\gamma-20\epsilon)M_t} .\\
\end{split}
\ee

\medskip

\noindent \textbf{Step 2.}\quad The estimate of  $Err^1_{ i, i_1,i_2,i_3}(t_1, t_2) $.

\medskip

  Recall  \eqref{oct12eqn2}. From the estimates in   \eqref{oct7eqn21}  and  \eqref{2024oct27eqn61}  in   Lemma \ref{secondhypohori} and   the estimates in \eqref{2024oct8eqn1} in Theorem \ref{mainresultsfirstpart}, we have  
\be 
\begin{split}
| Err^1_{ i, i_1, i_2}(t_1, t_2)\big| & \lesssim \mathcal{M}(C)   2^{30\iota M_t}  2^{(1+\alpha^{\star}) M_t -(k_3+2n_3)/2} 2^{-3\gamma M_t }\\
&\quad \times  2^{6\alpha^{\star} M_t -(k_2+2n_2)/2  -(k_1+2n_1)/2  -(k+2n)/2}\\
&\lesssim \mathcal{M}(C) 2^{5M_t/3 + 5\alpha^{\star} M_t/3-M_t/3-M_t/2-5M_t/6-4M_t/3+200\iota M_t} \\
&\lesssim  \mathcal{M}(C) 2^{(\gamma-20\epsilon)M_t}. \\
\end{split}
\ee

\medskip

\noindent \textbf{Step 3.}\quad The estimate of  $Err^2_{ i, i_1,i_2,i_3}(t_1, t_2) $.

\medskip

  Recall  \eqref{oct12eqn2}. From the rough $L^\infty_x$-type  estimate of the elliptic part in 
\eqref{2022feb25eqn1} in  Theorem \ref{maintheoremellipitic} and   the estimates in \eqref{2024oct8eqn1} in Theorem \ref{mainresultsfirstpart}, we have 
 \[
 \begin{split}
 | Err^2_{ i, i_1,i_2,i_3}(t_1, t_2) |&\lesssim \mathcal{M}(C)   2^{5\alpha^{\star}M_t/3+10\epsilon M_{t} } 2^{-4\gamma M_t} 2^{  2(2/3+4\iota) M_t-(k_2+2n_2)/2} \\
 &\quad \times   2^{ 2(2/3+4\iota)  M_t-(k_3+2n_3)/2}  2^{(7/3+15\iota )M_t -(k_1+2n_1)/2  -(k+2n)/2}  \\
 & \lesssim \mathcal{M}(C)2^{M_t/2}.  \\
 \end{split} 
 \]
Hence finishing the proof of the desired estimate  \eqref{oct26eqn86}. 
\end{proof}

 As summarized in the following lemma,  for error terms in  \eqref{oct21eqn21}, we have
 
 \begin{lemma}\label{errortypefull}
Let $ i\in \{0,1,2,3,4\},   ( n,k)\in \widetilde{\mathcal{E}}_i,$  $ \mu\in\{+,-\}.$ Under the assumption of Proposition \textup{\ref{bootstraplemma2}}, we have  
\be\label{oct19eqn31}
 \big|  {}_{1}^{1}Err\mathfrak{H}^{\mu,i}_{k,j;n}(t_1, t_2)\big|     \lesssim \mathcal{M}(C) 2^{(\gamma-10 \epsilon)M_t}.
\ee

\end{lemma}

\begin{proof}
Recall  \eqref{oct29eqn91}. After using the decomposition of the magnetic field in  \eqref{2024oct30eqn2},   we have
\[
\begin{split}
 {}_{1}^{1}Err\mathfrak{H}^{\mu,i}_{k,j;n}(t_1, t_2)&:=   \sum_{l\in [-M_t, 2]\cap \Z}   {}_{1}^{1}Err\mathfrak{H}^{\mu,i;l}_{k,j;n}(t_1, t_2), \\
  {}_{1}^{1}Err\mathfrak{H}^{\mu,i;l}_{k,j;n}(t_1, t_2) &:=\sum_{  {(\tilde{m},\tilde{k},\tilde{j},\tilde{l})\in \mathcal{S}_1(t)\cup\mathcal{S}_2(t)  } }      Err^{\mu,i;l;\tilde{m} }_{k,n; \tilde{k},\tilde{j}, \tilde{l}}(t_1, t_2), 
\end{split} 
\]
where
 \be\label{oct21eqn22}
 \begin{split}
  {}_{1}^{1}Err\mathfrak{H}^{\mu,i;l}_{k,j;n}(t_1, t_2)&:=\int_{t_1}^{t_2} \int_{\R^3}\int_{\R^3} f(s, X(s)- y, v )  \\
  &\quad \times      \mathcal{K}^{ \mu,i;l }_{k,n,j}(  X_{\bot}(s), y, v,  V(s))\cdot \big[(\hat{v}- \hat{V}(s))\times  B  (s, X(s)- y)\big]dy d v d s, \\
   Err^{\mu,i;l;\tilde{m} }_{k,n; \tilde{k},\tilde{j}, \tilde{l}}(t_1, t_2)&:=  \int_{t_1}^{t_2} \int_{\R^3}\int_{\R^3} f(s, X(s)- y, v )    \\
  &\quad \times \mathcal{K}^{ \mu,i;l }_{k,n,j}( X_{\bot}(s), y, v,  V(s))\cdot \big[(\hat{v}- \hat{V}(s))\times  B_{\tilde{k};\tilde{j}, \tilde{l} }^{\tilde{m}} (s, X(s)- y)\big] dy d v d s,
  \end{split}
\ee
where the kernel $ \mathcal{K}^{\mu,i;l}_{k,n,j}(  X_{\bot}(s),y, v, V(s)),\mu\in\{+,-\}, i\in \{0,1,\cdots, 4\},   $ are defined as follows, 
  \be\label{oct19eqn42}
\begin{split}
 \mathcal{K}^{\mu,i;l}_{k,n,j}( X_{\bot}(s),y, v, V(s)) & : = \int_{\R^3} e^{i y \cdot \xi }    \mathcal{F}[\mathcal{K}_{k,j,n}^{\mu,i;l}](   X_{\bot}(s),   \xi, v, V(s))
  d\xi,\\
 \mathcal{F}[\mathcal{K}_{k,j,n}^{\mu,i;l}](  X_{\bot}(s),   \xi, v, \zeta) &:=   |\xi|^{-1}({ i \hat{\zeta} \cdot \xi + i \mu |\xi| })^{-1}  {  \varphi_{n;-M_t}\big(\tilde{\zeta}\times \tilde{\xi} \big) \psi_k(\xi) }    \varphi_{l;-M_t}\big(\tilde{v}-\tilde{\zeta}  \big) \\
 &\quad \times \nabla_v \big(  C(  X_{\bot}(s), \zeta)\cdot \tilde{\varphi}_{k,j,n}^i (v,\xi,  \zeta) \big) .
 \end{split}
 \ee

Based on the size of $i$, we proceed in steps as follows. 

 \medskip 

\noindent \textbf{Step 1.}\quad  The estimate of $  {}_{1}^{1}Err\mathfrak{H}^{\mu,1;l}_{k,j;n}(t_1, t_2)$.

 \medskip

Recall the definition of the set $  \widetilde{\mathcal{E}}_1$ in  \eqref{oct15eqn1}. After writing  $H^{\mu ;1 ;\tilde{m} }_{k,n; \tilde{k},\tilde{j}, \tilde{l}}(t_1, t_2)$ in terms of kernel and doing integration by parts in $\xi$ along $V(s)$ direction and directions perpendicular to $V(s)$ many times for the kernel, the following estimate holds after using the Cauchy-Schwarz inequality, the volume of support of $v$ and the estimate  \eqref{nov24eqn41} if $|v|\geq 2^{(\alpha_s+\epsilon)M_s}$,
\be\label{oct19eqn10}
\begin{split}
 \big| Err^{\mu,1;l;\tilde{m} }_{k,n; \tilde{k},\tilde{j}, \tilde{l}}(t_1, t_2)\big| & \lesssim    \int_{t_1}^{t_2} 2^{ k } 2^{-j-l+l+4\epsilon M_t} \mathcal{M}(C)  \big( \min\{  2^{-j}, 2^{-3k-2n}  \min\{ 2^{3j+2l}, 2^{j+2\alpha_s M_s} \}  \} \big)^{1/2}  \\
 &\quad \times \big(\int_{\R^3} \min\{2^{j+2\alpha_t M_t} , 2^{3j+2l}\}  | B_{\tilde{k};\tilde{j}, \tilde{l} }^{\tilde{m}} (t, y)|^2 dy  \big)^{1/2} d s \\
& \lesssim \sup_{s\in [t_1,t_2] } \|B_{\tilde{k};\tilde{j}, \tilde{l} }^{\tilde{m}} (s,\cdot)\|_{L^2_x} 2^{4\epsilon M_t} \min\{2^{-(k+2n)/2 } \min\{2^{2\alpha^{\star}  M_t}, 2^{2j+2l} \}, 2^{ k +l}\}\mathcal{M}(C).\\
\end{split}
\ee
Moreover, if we put the localized magnetic field in $L^\infty$, then we have
\be\label{nov28eqn76}
\begin{split}
\big|Err^{\mu,1;l;\tilde{m} }_{k,n; \tilde{k},\tilde{j}, \tilde{l}}(t_1, t_2)\big| &\lesssim \sup_{s\in [t_1,t_2] } \|B_{\tilde{k};\tilde{j}, \tilde{l} }^{\tilde{m}} (s,\cdot)\|_{L^\infty_x}  2^{ k -j+4\epsilon M_t} \mathcal{M}(C)\\
&\quad \times \min\{ 2^{-3k-2n} \min\{ 2^{3j+2l}, 2^{j+2\alpha^{\star} M_t} \},  2^{-j} \}\\ 
&\lesssim \sup_{s\in [t_1,t_2] } \|B_{\tilde{k};\tilde{j}, \tilde{l} }^{\tilde{m}} (s,\cdot)\|_{L^\infty_x}  2^{ k -j+4\epsilon M_t}\big(2^{-3k-2n}   2^{3j+2l}\big)^{1/2}  \big( 2^{-j}\big)^{1/2}\mathcal{M}(C)\\
&\lesssim  \sup_{s\in [t_1,t_2] } \|B_{\tilde{k};\tilde{j}, \tilde{l} }^{\tilde{m}} (s,\cdot)\|_{L^\infty_x}  2^{ -(k+2n)/2+l+4\epsilon M_t}\mathcal{M}(C). 
\end{split}
\ee
 
 If $ (\tilde{m},\tilde{k},\tilde{j},\tilde{l})\in \mathcal{S}_1(t)$, from the above estimate and the estimate \eqref{2024oct30eqn1} in Proposition \ref{meanLinfest}, we have
\be\label{oct19eqn16}
\begin{split}
\sum_{  {(\tilde{m},\tilde{k},\tilde{j},\tilde{l})\in \mathcal{S}_1(t)}}\big|Err^{\mu,1;l;\tilde{m} }_{k,n; \tilde{k},\tilde{j}, \tilde{l}}(t_1, t_2)\big|&\lesssim 2^{2\alpha^{\star} M_t+20\epsilon M_t -(k+2n)/2 +l}\mathcal{M}(C)\\
&\lesssim 2^{ 2M_t/3+4\iota M_t+2 00\epsilon M_t}\mathcal{M}(C)\\
&\lesssim 2^{(\gamma-20 \epsilon)M_t}\mathcal{M}(C). 
\end{split}
\ee

If $ (\tilde{m},\tilde{k},\tilde{j},\tilde{l})\in \mathcal{S}_2(t)$, from the obtained estimates  \eqref{oct19eqn10}  and  \eqref{nov28eqn76} and the estimate  \eqref{2024oct30eqn1}  in Proposition \ref{meanLinfest}, we have
\be\label{oct19eqn17}
\begin{split}
\sum_{  {(\tilde{m},\tilde{k},\tilde{j},\tilde{l})\in \mathcal{S}_2(t)}}\big| Err^{\mu,1;l;\tilde{m} }_{k,n; \tilde{k},\tilde{j}, \tilde{l}}(t_1, t_2)\big| &\lesssim \sum_{  {(\tilde{m},\tilde{k},\tilde{j},\tilde{l})\in \mathcal{S}_2(t)}}\mathcal{M}(C)  \big( \sup_{s\in [t_1,t_2] } \|B_{\tilde{k};\tilde{j}, \tilde{l} }^{\tilde{m}} (t,\cdot)\|_{L^2_x} \big)^{1/2}\\
&\quad \times  \big( \sup_{s\in [t_1,t_2] } \|B_{\tilde{k};\tilde{j}, \tilde{l} }^{\tilde{m}} (t,\cdot)\|_{L^\infty_x} \big)^{1/2}2^{\alpha^{\star } M_t -(k+2n)/2+20\epsilon M_t +l/2}
\\
& \lesssim 2^{2\alpha^{\star} M_t+200\epsilon M_t -(k+2n)/2}\mathcal{M}(C)\\
& \lesssim 2^{(\gamma-20\epsilon)M_t}\mathcal{M}(C).  
\end{split}
\ee

  \medskip 

\noindent \textbf{Step 2.}\quad  The estimate of $   {}_{1}^{1}Err\mathfrak{H}^{\mu,i;l}_{k,j;n}(t_1, t_2), i\in\{0,2,3,4\}$.

 \medskip

Recall the cutoff functions $\varphi_{j,n}^a(\cdot, \cdot)$ in  \eqref{sep4eqn6}. For the case we are considering, we either have $l\leq   \vartheta^\star_0:=\max\{n+\epsilon M_t, -10^{-3} M_t\} +10  $ or $j\leq   (1/2+3\iota + 55\epsilon ) M_{t^{\star  }} +10. $

Based on the possible size of $j$, $l$, and $n$,  we proceed in three sub-steps as follows.

  \medskip 

  \textbf{Step 2A.}\quad  If $  j\leq   (1/2+3\iota + 55\epsilon ) M_{t^{\star  }} +10$, i.e., the case $i=0.$

  \medskip

Similar to the obtained estimate \eqref{oct19eqn10}, from the conservation law \eqref{conservationlaw},  we have 
\be\label{oct19eqn18}
\begin{split}
\big|   {}_{1}^{1}Err\mathfrak{H}^{\mu,0;l}_{k,j;n}(t_1, t_2) \big| & \lesssim   \sup_{t\in [t_1,t_2] } \|B  (t,\cdot)\|_{L^2_x} 2^{4\epsilon M_t}  2^{-(k+2n)/2 +2j+2l }  \mathcal{M}(C)\\
& \lesssim 2^{3M_t/4} \mathcal{M}(C) \lesssim 2^{(\gamma-20\epsilon)M_t}\mathcal{M}(C). 
\end{split}
\ee

  \medskip 

  \textbf{Step 2B.}\quad  If $l\leq    \vartheta^\star_0:=\max\{n+\epsilon M_t, -10^{-3} M_t\} +10 $ and $n\leq-10^{-3} M_t-\epsilon M_t.$

    \medskip

For this case, we have $l\leq -10^{-3}M_t=-10 \iota M_t.$ Recall   \eqref{oct15eqn1}. Similar to the obtained estimates  \eqref{oct19eqn16}  and  \eqref{oct19eqn17}, we have
\be\label{oct19eqn19}
\begin{split}
\big|    {}_{1}^{1}Err\mathfrak{H}^{\mu,i;l}_{k,j;n}(t_1, t_2) \big| &\lesssim \sum_{  {(\tilde{m},\tilde{k},\tilde{j},\tilde{l})\in \mathcal{S}_1(t)\cup\mathcal{S}_2(t)  } }\big|Err^{\mu,i;l;\tilde{m} }_{k,n; \tilde{k},\tilde{j}, \tilde{l}}(t_1, t_2)\big|  \\
& \lesssim \big(2^{2\alpha^{\star} M_t+30\epsilon M_t -(k+2n)/2 +l}  
 + 2^{2\alpha^{\star} M_t+40\epsilon M_t -(k+2n)/2 +l/2} \big)\mathcal{M}(C)\\
 & \lesssim 2^{M_t- \iota M_t} \mathcal{M}(C) \lesssim 2^{(\gamma-20\epsilon)M_t}\mathcal{M}(C). 
 \end{split}
\ee

  \medskip 

  \textbf{Step 2C.}\quad   If $l\leq    \vartheta^\star_0:=\max\{n+\epsilon M_t, -10^{-3} M_t\} +10 $ and $n\geq-10^{-3} M_t-\epsilon M_t.$

  \medskip

For this case, we have $l\leq n +\epsilon M_t$.   Similar to the obtained estimate  \eqref{oct19eqn10}, the following estimate holds if $k+l\leq (1-20\epsilon)M_t,$
\be\label{oct19eqn20}
\begin{split}
\big|     {}_{1}^{1}Err\mathfrak{H}^{\mu,i;l}_{k,j;n}(t_1, t_2) \big| &  \lesssim   \sup_{s\in [t_1,t_2] } \|B  (t,\cdot)\|_{L^2_x} 2^{k+l+4\epsilon M_t} \mathcal{M}(C)\\
& \lesssim 2^{(\gamma-10\epsilon)M_t}\mathcal{M}(C). 
\end{split}
\ee
If $k+l\geq (1-20\epsilon)M_t$, then similar to the obtained estimates  \eqref{oct19eqn16}  and  \eqref{oct19eqn17}, we have
\be\label{oct19eqn21}
\begin{split}
\big|   {}_{1}^{1}Err\mathfrak{H}^{\mu,i;l}_{k,j;n}(t_1, t_2) \big| & \lesssim \sum_{  {(\tilde{m},\tilde{k},\tilde{j},\tilde{l})\in \mathcal{S}_1(t)\cup\mathcal{S}_2(t)  } }\big|Err^{\mu,i;l;\tilde{m} }_{k,n; \tilde{k},\tilde{j}, \tilde{l}}(t_1, t_2)\big|  \\
 &\lesssim 
\big( 2^{2\alpha^{\star} M_t+30\epsilon M_t -(k+2n)/2 +l}  + 2^{2\alpha^{\star} M_t+40\epsilon M_t -(k+2n)/2 +l/2}\big) \mathcal{M}(C) \\
& \lesssim  2^{2\alpha^{\star}  M_t -M_t/2 + 11\iota M_t} \mathcal{M}(C) \\
& \lesssim 2^{(\gamma-10\epsilon)M_t}\mathcal{M}(C). 
\end{split}
\ee
To sum up, in whichever case,  our desired estimate  \eqref{oct19eqn31}  holds from the obtained estimates  \eqref{oct19eqn16}--\eqref{oct19eqn21}.   
\end{proof}

\subsubsection{Error terms in estimating the elliptic parts}\label{errorellipticPartIIISS}

As summarized in the following Lemma, for the error terms in  \eqref{2024oct31eqn81} (see also \eqref{oct7eqn62}), we have  
\begin{lemma}\label{ellfulltyp4}
 Let  $\bar{\kappa}:= -(5\alpha^{\star} -3)M_t-100\epsilon M_t$, $i\in\{0,   2,3,4 \},     (k,n)\in \widetilde{\mathcal{E}}_i$ or $i=1, (k,n)\in \widetilde{\mathcal{E}}'_1$. For any  $ \mu\in\{+, -\}$, $j\in [0,(1+2\epsilon)M_t]\cap \Z, $ $l\in[-M_t,2]\cap\Z, \kappa\in (\bar{\kappa}, 2]\cap \Z $ s.t., $k+2n\geq  2(1- \alpha^{\star} )M_t - 3 0\epsilon M_t $, $k+4n/3\geq M_t- \alpha^{\star}  M_t/3-30M_t,$   $ \min\{l,n\}\geq (1-2 \alpha^{\star} )M_t-30\epsilon M_t$,  $j\geq 3M_t/5-20\epsilon M_t$,     $k+l+\kappa\geq  2(1- \alpha^{\star}  )M_t-50\epsilon M_t, $  $(2k+2l+3\min\{l,\kappa\})/4-2l\geq (1-3\alpha^{\star}M_t/4)-100\epsilon M_t$, and $\min\{l, \kappa\}\geq 7l/3-\iota M_t$,   we have 
\be\label{oct29eqn80}
  \big|   Err^{\mu,a;l}_{k,j;n}(t_1, t_2) \big| \mathbf{1}_{a\in\{0,1,2,3\} } + \sum_{\kappa\in (\bar{\kappa}, 2 ]\cap \Z}  \big|     Err^{\mu,4;\kappa, l}_{k,j;n}(t_1, t_2) \big|  \lesssim   2^{(\gamma-10\epsilon)M_t}  \mathcal{M}(C).
\ee
\end{lemma}
\begin{proof}

 Recall  \eqref{oct7eqn62}. Note that  $\nabla_{ x_{\bot} } C(  X_{\bot}(t), V(t)))=0$.  After writing  $ Err^{\mu,a;l}_{k,j;n}(t_1, t_2), a\in \{0,1,2,3\},$ in terms of kernel,   doing integration by parts in $\xi$ along $V(s)$ direction and directions perpendicular to $V(s)$ for the kernel, and using the volume of support of $v$ and the estimate  \eqref{nov24eqn41}  if $|v_{\bot}|\geq 2^{(\alpha^\star+\epsilon)M_t}$,  the following estimate holds for any $a\in \{0,1,2,3\},$
\be
\begin{split}
 \big| Err^{\mu,a;l}_{k,j;n}(t_1, t_2)\big|&\lesssim   2^{ k+20\epsilon M_t} \min\{2^{-j},  2^{-3k-2n+j+2\alpha^{\star}M_t}\} \mathcal{M}(C) \\
 &\lesssim 2^{(\alpha^{\star}+20\epsilon)M_t-(k+2n)/2} \mathcal{M}(C)\\
 & \lesssim  2^{ 3\alpha^{\star} M_t/4 } \mathcal{M}(C). \\
 \end{split}
\ee

It remains to consider the case $i=4$.  Due to the cutoff function $\varphi^4_{j,n}(v,\xi)$ (see   \eqref{sep4eqn6}), we have $l\in[n-2\epsilon M_t, n+2\epsilon M_t]$.   Note that, after doing integration by parts with respect to the orthonormral frame $\{(\hat{v}-\hat{V}(t))/|\hat{v}-\hat{V}(t), \theta_1(v, t),\theta_2(v, t)\}$(defined in   Lemma \ref{smallangest}), we have 
\be 
\begin{split}
\sum_{\kappa\in (\bar{\kappa}, 2 ]\cap \Z} \big|   Err^{\mu,4;\kappa, l}_{k,j;n}(t_1, t_2) \big|& \lesssim  \sum_{\kappa\in (\bar{\kappa}, 2 ]\cap \Z} 2^{ k+20\epsilon M_t} \min\{2^{-j}  , 2^{-3k-n- {\kappa}+j+2 \alpha^{\star} M_t}\} \mathcal{M}(C) \\
&\lesssim 2^{(\alpha^{\star}+20\epsilon)M_t-(k+n+\kappa)/2}  \mathcal{M}(C)  \\
& \lesssim 2^{(\gamma-10\epsilon)M_t} \mathcal{M}(C). 
\end{split}
\ee
Hence finishing the proof of our desired estimate  \eqref{oct29eqn80}.
\end{proof}

As summarized in the following Lemma, for the error terms in    \eqref{oct29eqn40}, we have 
\begin{lemma}\label{errhypotype1}
  Let $\bar{\kappa}:= -(5\alpha^{\star} -3)M_t-100\epsilon M_t$,  $  \mu\in\{+, -\},  k_1\in \Z_+, n_1\in [-M_t ,2]\cap \Z,   (k,n)\in \widetilde{\mathcal{E}}_4$, $j\in [0,(1+2\epsilon)M_t]\cap \Z, $ $l\in[-M_t,2]\cap\Z, \kappa\in (\bar{\kappa}, 2]\cap \Z $, s.t., $k+2n\geq  2(1- \alpha^{\star} )M_t - 3 0\epsilon M_t $,    $k+4n/3\geq M_t- \alpha^{\star}  M_t/3-30M_t,$   $ \min\{l,n\}\geq (1-2 \alpha^{\star} )M_t-30\epsilon M_t$,   $j\geq 3M_t/5-20\epsilon M_t$,    $k+l+\kappa\geq  2(1- \alpha^{\star}  )M_t-50\epsilon M_t, $  $(2k+2l+3\min\{l,\kappa\})/4-2l\geq (1-3\alpha^{\star}M_t/4)-100\epsilon M_t$,   $\min\{l, \kappa\}\geq 7l/3-\iota M_t$,  $(k_1+2n_1)/2\geq  (k+n +\min\{n,\kappa\}  )/2+  2M_t/15-10\iota M_t$ and    $n_1\geq - (\alpha^{\star}+3\iota+60\epsilon) M_t$,  we have   
\be\label{oct27eqn81}
\sum_{a\in\{0,1,2,3,4\}}\sum_{U\in \{H, M\}}|ErrU^a_{i,i_1}(t_1, t_2)|\lesssim \mathcal{M}(C) 2^{(\gamma-10\epsilon)M_t}. 
\ee 
\end{lemma}
\begin{proof}

Based on the size of $a$, we proceed in five steps as follows.

\medskip

\noindent \textbf{Step 1.}\quad   The estimate of $ErrU^0_{i,i_1}(t_1, t_2), U\in \{H, M\}.$

\medskip

Recall  \eqref{oct28eqn1} and  \eqref{2024oct30eqn81}.   After writing $   ErrU^0_{i,i_1}(t_1, t_2), U\in\{H, M\},$ in terms of kernel,  from   the estimates in \eqref{2024oct8eqn1} in Theorem \ref{mainresultsfirstpart}, we have 
\be\label{oct29eqn1}
\begin{split}
 |ErrU^0_{i,i_1}(t_1, t_2)| & \lesssim \mathcal{M}(C) 2^{ (\alpha^{\star}+4\iota) M_t- (k_1+2n_1)/2 }\\ 
&\quad \times \big[  2^{   k-\min\{l,\kappa\} -\gamma M_t } \min\{2^{-3k-n-\min\{n,\kappa\} + j+2\alpha^{\star}M_t}, 2^{-j}\} \\
&\quad + (1+2^{l+2M_t/15}) 2^{k-\min\{l, \kappa\}-j}\min\{2^{-3k-l-\min\{l,\kappa\}+3j+2l}, 2^{-j}\} \big]  \\
&\lesssim \mathcal{M}(C) 2^{(\gamma-10\epsilon)M_t}.
\end{split}
\ee

\medskip

\noindent \textbf{Step 2.}\quad The estimate of $ErrU^1_{i,i_1}(t_1, t_2), U\in \{H, M\}.$

\medskip

Recall \eqref{2024oct30eqn82}  and  \eqref{nov13eqn11}.  After writing $   ErrU^1_{i,i_1}(t_1, t_2),U\in\{H, M\},$ in terms of kernel, from the estimates in   \eqref{oct7eqn21}  and  \eqref{2024oct27eqn61}  in   Lemma \ref{secondhypohori} and   the estimates in \eqref{2024oct8eqn1} in Theorem \ref{mainresultsfirstpart}, we have  
 \be\label{oct29eqn2}
 \begin{split}
 |ErrU^1_{i,i_1}(t_1, t_2)| &\lesssim   \mathcal{M}(C)\big( 2^{ (1+ \alpha^{\star})M_t  +l} + 2^{7\alpha^{\star}  M_t/3} \big) 2^{\iota M_t-(k_1+2n_1)/2}\\
 &\quad \times  \big[  2^{   k-\min\{l,\kappa\} -\gamma M_t } \min\{2^{-3k-n-\min\{n,\kappa\} + j+2\alpha^{\star}M_t}, 2^{-j}\} \\
&\quad + (1+2^{l+2M_t/15}) 2^{k-\min\{l, \kappa\}-j}\min\{2^{-3k-l-\min\{l,\kappa\}+3j+2l}, 2^{-j}\} \big]\\
&\lesssim   \mathcal{M}(C)  2^{(\gamma-10\epsilon)M_t} .
\end{split}
\ee

\medskip

\noindent \textbf{Step 3.}\quad   The estimate of $ErrU^2_{i,i_1}(t_1, t_2), U\in \{H, M\}.$

\medskip

Recall  \eqref{2024oct30eqn83}  and   \eqref{2024oct30eqn84}. Similar to the obtained estimates  \eqref{oct19eqn61}  and  \eqref{nov12eqn41}  for the kernels, after   using the Cauchy-Schwarz inequality and the estimate     \eqref{nov24eqn41}  if $|  v_{\bot}|\geq 2^{(\alpha^\star+\epsilon)M_t}$, for any $U\in\{H, M\},$  we have
\be\label{oct29eqn3}
\begin{split}
 |ErrU^2_{i,i_1}(t_1, t_2)| &  \lesssim \mathcal{M}(C) 2^{(\alpha^{\star}+5\iota) M_t -(k_1+2n_1)/2} 2^{k-j+l-2\min\{l,\kappa\}}   \\
&\quad \times (2^{-j  } + 2^{ -\gamma M_t} )  \big(2^{-3k-l-\min\{l,\kappa\}+ j +2(\alpha^\star+\epsilon)M_t } \big)^{1/2} \big(2^{3j+2l} \big)^{1/2} \\
& \lesssim \mathcal{M}(C)  2^{ 2\alpha^{\star} M_t+ 10\iota M_t +2l-(k_1+2n_1)/2 - (k +l+\min\{l,\kappa\} )/2 -2\min\{l,\kappa\} } \\
&  \lesssim \mathcal{M}(C) 2^{(\gamma-10\epsilon)M_t}  . 
\end{split}
\ee

\medskip

\noindent \textbf{Step 4.}\quad  The estimate of $ErrU^3_{i,i_1}(t_1, t_2), U\in \{H, M\}.$

\medskip

Recall  \eqref{2024oct30eqn85}  and  \eqref{oct27eqn51}.  
 After writing $   ErrU^3_{i,i_1}(t_1, t_2)$ in terms of kernel, from the rough $L^\infty_x$-type  estimate of the elliptic part in 
\eqref{2022feb25eqn1} in  Theorem \ref{maintheoremellipitic} and   the estimates in \eqref{2024oct8eqn1} in Theorem \ref{mainresultsfirstpart}, we have 
\be\label{oct29eqn4}
\begin{split}
|ErrU^3_{i,i_1}(t_1, t_2)| & \lesssim    2^{5\alpha^{\star} M_t/3 + 5\iota M_t }   2^{-\gamma  M_t+  7M_t/6 -(k_1+2n_1)/2}\\
&\quad \times  2^{l-\min\{l,\kappa\}}   2^{-(k+n+\min\{n,\kappa\})/2} \mathcal{M}(C)\\
 &\lesssim \mathcal{M}(C) 2^{(\gamma-10\epsilon)M_t}. 
 \end{split}
 \ee

\medskip

\noindent \textbf{Step 5.}\quad The estimate of $ErrU^4_{i,i_1}(t_1, t_2), U\in \{H, M\}.$

\medskip

 Recall  \eqref{2024oct30eqn86} and  \eqref{oct27eqn51}.   After writing $   ErrU^4_{i,i_1}(t_1, t_2)$ in terms of kernel,  from the rough $L^\infty_x$-type  estimate of the elliptic part in 
\eqref{2022feb25eqn1} in  Theorem \ref{maintheoremellipitic} and   the estimates in \eqref{2024oct8eqn1} in Theorem \ref{mainresultsfirstpart}, we have 
    \be\label{nov14eqn97}
    \begin{split}
     & |ErrU^4_{i,i_1}(t_1, t_2)|\\
      &\lesssim  \mathcal{M}(C)  2^{5\alpha^{\star} M_t/3 +5\iota \epsilon M_t } 2^{ (\alpha^\star+3\iota)M_t -(k_1+2n_1)/2} \\
      &\quad \times \min\big\{2^{-3k-n-\min\{n,\kappa\} + 3j+2l}, 2^{-j} \big\}  2^{k -2\min\{l,\kappa\} -2j} \\
    &   \lesssim  \mathcal{M}(C) 2^{5\alpha^{\star} M_t/3 +5\iota \epsilon M_t }  2^{ (\alpha^\star+3\iota)M_t -(k_1+2n_1)/2}   2^{-5(k+l+\kappa)/4+2l-2\min\{l,\kappa\} +(5\kappa-3\min\{n,\kappa\})/4}  \\
 &\lesssim  \mathcal{M}(C) 2^{(\gamma-10\epsilon)M_t}. 
 \end{split}
 \ee
  Hence, our desired estimate  \eqref{oct27eqn81}  holds from the obtained estimates  \eqref{oct29eqn1}--\eqref{nov14eqn97}. 
\end{proof}

For error terms in  \eqref{nov15eqn99}, we have 
 
\begin{lemma}\label{ellerrstep2}
   Let $\bar{\kappa}:= -(5\alpha^{\star} -3)M_t-100\epsilon M_t$,  $  \mu\in\{+, -\},  k_1, k_2\in \Z_+, n_1, n_2\in [-M_t ,2]\cap \Z,   (k,n)\in \widetilde{\mathcal{E}}_4$, $j\in [0,(1+2\epsilon)M_t]\cap \Z, $ $l\in[-M_t,2]\cap\Z, \kappa\in (\bar{\kappa}, 2]\cap \Z $, s.t., $k+2n\geq  2(1- \alpha^{\star} )M_t - 3 0\epsilon M_t $,    $k+4n/3\geq M_t- \alpha^{\star}  M_t/3-30M_t,$   $ \min\{l,n\}\geq (1-2 \alpha^{\star} )M_t-30\epsilon M_t$,  $j\geq 3M_t/5-20\epsilon M_t$,     $k+l+\kappa\geq  2(1- \alpha^{\star}  )M_t-50\epsilon M_t, $  $(2k+2l+3\min\{l,\kappa\})/4-2l\geq (1-3\alpha^{\star}M_t/4)-100\epsilon M_t$,   $\min\{l, \kappa\}\geq 7l/3-\iota M_t$,  $(k_1+2n_1)/2\geq (k+n +\min\{n,\kappa\}  )/2+   2M_t/15-10\iota M_t,$  $k_2+2n_2\geq  k_1+2n_1    + 8M_t/15-40\iota M_t$, and $n_1, n_2\geq   -(\alpha^{\star}+3\iota+60\epsilon) M_t$,  we have 
\be\label{nov15eqn30}
 \sum_{a=0,1,2,3}\sum_{b=1,2}\big| ErrU^{b;a}_{i,i_1,i_2}(t_1, t_2)\big|\lesssim  \mathcal{M}(C)2^{(\gamma-10\epsilon)M_t}. 
 \ee
\end{lemma}
\begin{proof}

We proceed in steps as follows.

\medskip

\noindent \textbf{Step 1.}\quad   The estimate of $ErrU^{b;0}_{i,i_1,i_2}(t_1, t_2), b\in\{1,2\}, U\in \{H, M\}.$

\medskip

Recall  \eqref{oct29eqn6}and  \eqref{2024oct30eqn100}. After writing $ErrU^{b;0}_{i,i_1,i_2}(t_1, t_2), b\in\{0,1\}, U\in \{H, M\}$ in terms of kernels,   from   the estimates in \eqref{2024oct8eqn1} in Theorem \ref{mainresultsfirstpart}, we have 
\be\label{nov15eqn1}
\begin{split}
\big| ErrU^{1;0}_{i,i_1,i_2}(t_1, t_2)\big|  &\lesssim 2^{(\alpha^{\star}+4\iota) M_t  -(k_2+2n_2)/2} 2^{-\gamma M_t} \min\big\{2^{-3k-n-\min\{n,\kappa\} + 3j+2l}, 2^{-j} \big\} \\
&\quad \times 2^{k-\min\{l,
\kappa\}-\gamma M_t} \big(2^{7M_t/6 -(k_1+2n_1)/2} + 2^{ \alpha^{\star} M_t -(k_1+2n_1)/2-n}\big)\mathcal{M}(C)      \\
& \lesssim 2^{(\gamma-10\epsilon)M_t} \mathcal{M}(C), \\
\big| ErrU^{2;0}_{i,i_1,i_2}(t_1, t_2)\big| & \lesssim   2^{2(\alpha^{\star}+4\iota) M_t     -(k_2+2n_2)/2 -(k_1+2n_1)/2 }  2^{k -2\min\{l,\kappa\} -2j}  \\
&\quad \times  \min\big\{2^{-3k-n-\min\{n,\kappa\} + 3j+2l}, 2^{-j} \big\}\mathcal{M}(C)\\
& \lesssim   2^{-5(k+l+\min\{n,\kappa\})/4+  2l-2\min\{l,\kappa\} +(5\kappa-3\min\{n,\kappa\})/4} \\
&\quad \times 2^{ 2 (\alpha^{\star}+4\iota) M_t    -(k_2+2n_2)/2-(k_1+2n_1)/2}  \mathcal{M}(C)\\
&\lesssim \mathcal{M}(C)2^{(\gamma-10\epsilon)M_t}. 
\end{split}
\ee

\medskip

\noindent \textbf{Step 2.}\quad The estimate of $ErrU^{b;1}_{i,i_1,i_2}(t_1, t_2), b\in\{1,2\}, U\in \{H, M\}.$

\medskip

Recall  \eqref{2024oct30eqn101}  and \eqref{2024oct30eqn102}.  After writing $  ErrU^{2;0}_{i,i_1,i_2}(t_1, t_2),U\in\{H, M\},$ in terms of kernel, from the estimates in   \eqref{oct7eqn21}  and  \eqref{2024oct27eqn61}  in   Lemma \ref{secondhypohori} and   the estimates in \eqref{2024oct8eqn1} in Theorem \ref{mainresultsfirstpart}, we have  
\be\label{nov15eqn5}
\begin{split}
\big|ErrU^{1;1}_{i,i_1,i_2}(t_1, t_2)\big|& \lesssim  \mathcal{M}(C)  2^{(1+\alpha^{\star})M_t }   2^{  -(k_2+2n_2)/2} 2^{-\gamma M_t}  2^{7M_t/6 +3\iota M_t-(k_1+2n_1)/2} \\
&\quad \times 2^{l-\min\{l,
\kappa\}}   2^{-(k+n+\min\{n,\kappa\})/2} \\
&\lesssim \mathcal{M}(C)2^{(\gamma-10\epsilon)M_t}, \\
&\\
\big|ErrU^{2;1}_{i,i_1,i_2}(t_1, t_2)\big| &\lesssim  \mathcal{M}(C)    2^{(1+\alpha^{\star})M_t  }    2^{   -(k_2+2n_2)/2}    2^{(\alpha^{\star}+4\iota)M_t  -(k_1+2n_1)/2}   \\
&\quad \times   2^{k -2\min\{l,\kappa\} -2j}   \min\big\{2^{-3k-n-\min\{n,\kappa\} + 3j+2l}, 2^{-j} \big\}\\
& \lesssim \mathcal{M}(C)2^{( \gamma -10\epsilon)M_t}.
\end{split}
\ee

\medskip

\noindent \textbf{Step 3.}\quad  The estimate of $ErrU^{b;2}_{i,i_1,i_2}(t_1, t_2), b\in\{1,2\}, U\in \{H, M\}.$

\medskip

Recall  \eqref{nov15eqn40} and  \eqref{2024oct30eqn103}. We  first use the equation satisfied by the distribution function in Vlasov equation  \eqref{mainequation} and then  do integration by parts in $v$ for $ErrU^{1;2}_{i,i_1,i_2}(t_1, t_2)$ and $ErrU^{2 ;2}_{i,i_1,i_2}(t_1, t_2). $ As a result,  from  the estimates in \eqref{2024oct8eqn1} in Theorem \ref{mainresultsfirstpart}, we have 
\be\label{nov15eqn21}
\begin{split}
 \big| ErrU^{1 ;2}_{i,i_1,i_2}(t_1, t_2)\big| 
 & \lesssim \big|  \int_{t_1}^{t_2}\int_{\R^3}  \int_{\R^3}\int_{\R^3} e^{i \widetilde{\Phi}^1_{ \mu_1, \mu_2}(\xi,\eta, \sigma; s , X(s ),v)  + i s\hat{v}\cdot \eta  } \\
&\quad \times  \mathcal{F}\big[ (E+\hat{v}\times B)f  \big](s, \eta, v)\cdot \nabla_v \big[      \big(  \p_s(\widetilde{\Phi}^1_{\mu_1, \mu_2}(\xi,\eta, \sigma;s,X(s),v))  \big)^{-1} \\
&\quad \times       \mathcal{F}[\mathfrak{H}_{k_2,j_2,n_2}^{\mu,i_2}](s , \sigma, V(s )) \cdot     \mathcal{F}[{}_{1}^{2}\mathfrak{E}U](s , \xi, \eta, v,   X_{\bot}(s), V(s ))\big]  d \eta d \xi  d v d s\big|\\
&\lesssim   \mathcal{M}(C)  2^{ (\alpha^{\star}+6\iota) M_t     -(k_2+2n_2)/2} 2^{-\gamma M_t}  2^{7M_t/6 -(k_1+2n_1)/2}\\
&\quad \times  2^{k-2j -2\min\{l, \kappa\} }
  \big( 2^{3j+2l}\big)^{1/2} \big( 2^{-3k-n-\min\{n,\kappa\}}2^{j+2(\alpha^\star + \epsilon)M_t}\big)^{1/2} \\
&\lesssim \mathcal{M}(C)     2^{2 \alpha^{\star} M_t +7\iota  M_t-(k_2+2n_2)/2} 2^{ M_t/6 -(k_1+2n_1)/2} 2^{ l -2\min\{l, \kappa\} } 2^{-(k+n+\min\{n,\kappa\})/2}\\
& \lesssim  \mathcal{M}(C)2^{(\gamma-10\epsilon)M_t}. 
\end{split}
\ee
Similarly, we have
 \be\label{nov15eqn22}
 \begin{split}
 \big| ErrU^{2 ;2}_{i,i_1,i_2}& (t_1, t_2)\big|  
 \lesssim  \mathcal{M}(C)  2^{ 2 (\alpha^{\star}+4\iota)M_t -(k_2+2n_2)/2   -(k_1+2n_1)/2}\\
 &\quad\times  2^{k-3j -3\min\{l, \kappa\} }
 \big( 2^{-3k-n-\min\{n,\kappa\}}2^{3j+2l}\big)^{1/2}  \big( 2^{3j+2l}\big)^{1/2}\\
  &\lesssim \mathcal{M}(C)     2^{  2 (\alpha^{\star}+4\iota)M_t  -(k_2+2n_2)/2  -(k_1+2n_1)/2} 2^{2 l-3\min\{l, \kappa\} } 2^{-(k+n+\min\{n,\kappa\})/2}\\
  & \lesssim  \mathcal{M}(C)2^{(\gamma-10\epsilon)M_t}. 
  \end{split}
\ee

\medskip

\noindent \textbf{Step 4.}\quad   The estimate of $ErrU^{b;3}_{i,i_1,i_2}(t_1, t_2), b\in\{1,2\}, U\in \{H, M\}.$

\medskip

Recall  \eqref{2024oct30eqn104}.  From   the rough $L^\infty_x$-type  estimate of the elliptic part in 
\eqref{2022feb25eqn1} in  Theorem \ref{maintheoremellipitic} and   the estimates in \eqref{2024oct8eqn1} in Theorem \ref{mainresultsfirstpart}, we have 
 \be\label{nov15eqn25}
\begin{split}
\big| ErrU^{1 ;3}_{i,i_1,i_2}(t_1, t_2)\big| & \lesssim \mathcal{M}(C) 2^{  5\alpha^{\star} M_t/3+4\iota M_t}   2^{-2\gamma  M_t} 2^{ 2(\alpha^{\star}+3\iota) M_t-(k_2+2n_2)/2}  2^{7M_t/6 -(k_1+2n_1)/2} \\
 &\quad \times  2^{k-\min\{l,
\kappa\}-j} \min\big\{2^{-3k-n-\min\{n,\kappa\} + 3j+2l}, 2^{-j} \big\} \\
&\lesssim  2^{(\gamma-10\epsilon)M_t} \mathcal{M}(C),\\
&\\
\big| ErrU^{2 ;3}_{i,i_1,i_2}(t_1, t_2)\big| &  \lesssim \mathcal{M}(C)  2^{ 5\alpha^{\star}  M_t/3+4\iota M_t}  2^{- \gamma  M_t} 2^{ (\alpha^{\star}+3\iota) M_t -(k_2+2n_2)/2 +7M_t/6 -(k_1+2n_1)/2} \\
 &\quad \times   2^{k -2\min\{l,\kappa\} -2j} \min\big\{2^{-3k-n-\min\{n,\kappa\} + 3j+2l}, 2^{-j} \big\} \\
 & \lesssim 2^{(\gamma -10\epsilon)M_t} \mathcal{M}(C). \\
\end{split}
\ee
To sum up, our desired estimate \eqref{nov15eqn30}  holds from the obtained estimates \eqref{nov15eqn1}--\eqref{nov15eqn25}. 
\end{proof}

For error terms in  \eqref{nov15eqn96}, we have 
 
\begin{lemma}\label{ellerrstep3}
  Let $\bar{\kappa}:= -(5\alpha^{\star} -3)M_t-100\epsilon M_t$,  $  \mu\in\{+, -\},  k_1, k_2, k_3\in \Z_+, n_1, n_2, n_3\in [-M_t ,2]\cap \Z,   (k,n)\in \widetilde{\mathcal{E}}_2$, $j\in [0,(1+2\epsilon)M_t]\cap \Z, $ $l\in[-M_t,2]\cap\Z, \kappa\in (\bar{\kappa}, 2]\cap \Z $, s.t., $k+2n\geq  2(1- \alpha^{\star} )M_t - 3 0\epsilon M_t $,   $k+4n/3\geq M_t- \alpha^{\star}  M_t/3-30M_t,$   $ \min\{l,n\}\geq (1-2 \alpha^{\star} )M_t-30\epsilon M_t$,  $j\geq 3M_t/5-20\epsilon M_t$,      $k+l+\kappa\geq  2(1- \alpha^{\star}  )M_t-50\epsilon M_t, $  $(2k+2l+3\min\{l,\kappa\})/4-2l\geq (1-3\alpha^{\star}M_t/4)-100\epsilon M_t$,   $\min\{l, \kappa\}\geq 7l/3-\iota M_t$,  $(k_1+2n_1)/2\geq (k+n +\min\{n,\kappa\}  )/2+  2M_t/15-10\iota M_t,$  $k_2+2n_2\geq  k_1+2n_1 + 8M_t/15-40\iota M_t$, $k_3+2n_3\geq k_2+2n_2 + 13M_t/15 -50\iota M_t$, and $n_1, n_2, n_3\geq   -(\alpha^{\star}+3\iota +60\epsilon) M_t$,  we have 
 \be\label{nov15eqn81}
 \sum_{a=0,1,2 }\sum_{b=0,1} \big| ErrU^{b;a}_{i,i_1,i_2,i_3}(t_1, t_2) \big| \lesssim  \mathcal{M}(C)2^{(\gamma-10\epsilon)M_t}. 
 \ee
 \end{lemma}

\begin{proof}

We proceed in steps as follows.

\medskip

\noindent \textbf{Step 1.}\quad  The estimate of $ ErrU^{b;0}_{i,i_1,i_2,i_3}(t_1, t_2), b\in \{1,2\}, U\in \{H, M\}. $

\medskip

Recall \eqref{nov12eqn91}.   After writing  $  ErrU^{b;0}_{i,i_1,i_2,i_3}(t_1, t_2)$ in terms of kernel,  from the estimates in \eqref{2024oct8eqn1} in Theorem \ref{mainresultsfirstpart}, we have
 \be\label{nov15eqn51}
 \begin{split}
  \big|  ErrU^{1;0}_{i,i_1,i_2,i_3}(t_1, t_2)\big| 
  &\lesssim   \mathcal{M}(C) 2^{(\alpha^{\star}+3\iota) M_t -(k_3+2n_3)/2}  2^{-2\gamma  M_t} 2^{2(\alpha^{\star}+3\iota) M_t-(k_2+2n_2)/2}  \\
& \quad  \times    2^{7M_t/6+3\iota M_t -(k_1+2n_1)/2} 2^{k-\min\{l,
\kappa\}-j} \min\big\{2^{-3k-n-\min\{n,\kappa\} + 3j+2l}, 2^{-j} \big\}  \\
& \lesssim    \mathcal{M}(C)2^{(\gamma-10\epsilon)M_t}, \\
&\\
  \big|  ErrU^{2;0}_{i,i_1,i_2,i_3}(t_1, t_2)\big| 
  &\lesssim     2^{(\alpha^{\star}+3\iota)M_t -(k_3+2n_3)/2} 2^{- \gamma  M_t} 2^{(\alpha^{\star}+3\iota)M_t-(k_2+2n_2)/2 +7M_t/6 +3\iota M_t-(k_1+2n_1)/2}  \\
 &\quad \times  2^{k -2\min\{l,\kappa\} -2j} \min\big\{2^{-3k-n-\min\{n,\kappa\} + 3j+2l}, 2^{-j} \big\}   \mathcal{M}(C)  \\
 & \lesssim     \mathcal{M}(C) 2^{(2\alpha^{\star}+8\iota)M_t+ M_t/6-(k_3+2n_3)/2  -(k_2+2n_2)/2  -(k_1+2n_1)/2}  \\
 &\quad \times    2^{-5(k+l+\kappa)/4+ 2l-2\min\{l,\kappa\} +(5\kappa-3\min\{n,\kappa\})/4} \\
 & \lesssim  \mathcal{M}(C)2^{(\gamma-10\epsilon)M_t}. 
\end{split}
 \ee

\medskip

\noindent \textbf{Step 2.}\quad  The estimate of $ ErrU^{b;1}_{i,i_1,i_2,i_3}(t_1, t_2), b\in \{1,2\}, U\in \{H, M\}. $

\medskip

Recall \eqref{2024oct31eqn21}.   We  first use the equation satisfied by the distribution function in Vlasov equation  \eqref{mainequation}   and then  do integration by parts in $v$ for $ ErrU^{b;1}_{i,i_1,i_2,i_3}(t_1, t_2), b\in\{1,2\}.$ As a result,     from the estimates in \eqref{2024oct8eqn1} in Theorem \ref{mainresultsfirstpart}, we have
\be\label{nov15eqn54}
\begin{split}
 \big|ErrU^{1;1}_{i,i_1,i_2,i_3}(t_1, t_2)\big|  
&\lesssim   \mathcal{M}(C)   2^{ (\alpha^{\star}+7\iota) M_t   -(k_3+2n_3)/2} 2^{2(\alpha^{\star}+3\iota) M_t -(k_2+2n_2)/2} 2^{-2\gamma  M_t}  \\
 &\quad \times 2^{7M_t/6 -(k_1+2n_1)/2} 2^{k -2j -2\min\{l, \kappa\} } \big( 2^{-3k-n-\min\{n,\kappa\}}2^{3j+2l}\big)^{1/2} \big( 2^{3j+2l}\big)^{1/2}  \\  
 &\lesssim  \mathcal{M}(C)2^{(\gamma-10\epsilon)M_t},\\
 &\\
 \big|ErrU^{2;1}_{i,i_1,i_2,i_3}(t_1, t_2)\big|   &\lesssim   \mathcal{M}(C)   2^{ 2 (\alpha^{\star}+3\iota) M_t +5\iota M_t-(k_3+2n_3)/2 -(k_2+2n_2)/2}  2^{-\gamma  M_t} \\
  &\quad  \times  2^{7M_t/6 -(k_1+2n_1)/2} 2^{k -3j -3\min\{l, \kappa\} } \big( 2^{-3k-n-\min\{n,\kappa\}}2^{3j+2l}\big)^{1/2} \big( 2^{3j+2l}\big)^{1/2} \\
  &   \lesssim    \mathcal{M}(C)2^{(\gamma-10\epsilon)M_t}. \\
  \end{split}
\ee

\medskip

\noindent \textbf{Step 3.}\quad   The estimate of $ ErrU^{b;2}_{i,i_1,i_2,i_3}(t_1, t_2), b\in \{1,2\}, U\in \{H, M\}. $

\medskip

Recall  \eqref{nov12eqn91}. After writing $  ErrU^{b;2}_{i,i_1,i_2}(t_1, t_2),U\in\{H, M\},$ in terms of kernel,    from the estimates in   \eqref{oct7eqn21}  and  \eqref{2024oct27eqn61}  in   Lemma \ref{secondhypohori} and   the estimates in \eqref{2024oct8eqn1} in Theorem \ref{mainresultsfirstpart}, we have   
\be\label{nov15eqn82}
\begin{split}
 \big|ErrU^{1 ;2}_{i,i_1,i_2,i_3}(t_1, t_2)\big|  
&\lesssim   \mathcal{M}(C)   2^{ (1+\alpha^{\star} + 4 \iota )M_t-(k_3+2n_3)/2} 2^{ 2(\alpha^{\star}+3\iota) M_t -(k_2+2n_2)/2} 2^{-2\gamma  M_t}  \\
 &\quad \times 2^{7M_t/6 -(k_1+2n_1)/2}   2^{k  - j - \min\{l, \kappa\} } \min\{ 2^{-3k-n-\min\{n,\kappa\}}2^{3j+2l},  2^{-j}\}\\  &   \lesssim \mathcal{M}(C)     2^{ 3\alpha^{\star}M_t+  M_t/6+10\iota M_t-(k_3+2n_3)/2-(k_2+2n_2)/2}\\
   &\quad  \times  2^{ -(k_1+2n_1)/2+ l-\min\{l,\kappa\}-(k+n+\min\{n,\kappa\})/2 } \\
   & \lesssim    \mathcal{M}(C)2^{(\gamma-10\epsilon)M_t},  \\
   &\\
    \big|ErrU^{2 ;2}_{i,i_1,i_2,i_3}(t_1, t_2)\big|&\lesssim   \mathcal{M}(C)   2^{ (1+\alpha^{\star}+4\iota )M_t-(k_3+2n_3)/2} 2^{ (\alpha^{\star}+3\iota) M_t -(k_2+2n_2)/2}  2^{k -2j -2\min\{l, \kappa\} } \\
    &\quad  \times 2^{- \gamma M_t} 2^{7M_t/6 -(k_1+2n_1)/2}   \min\{2^{-3k-n-\min\{n,\kappa\}}2^{3j+2l}, 2^{-j}\}   \\
    &   \lesssim \mathcal{M}(C)     2^{ 2\alpha^{\star} M_t+  7M_t/6+ 10\iota M_t-(k_3+2n_3)/2-(k_2+2n_2)/2} 2^{ -(k_1+2n_1)/2 }\\
    & \quad \times    2^{-5(k+l+\kappa)/4+  2l-2\min\{l,\kappa\} +(5\kappa-3\min\{n,\kappa\})/4}\\
    &  \lesssim   \mathcal{M}(C)2^{(\gamma-10\epsilon)M_t}.\\
   \end{split} 
\ee
 
To sum up, our desired estimate  \eqref{nov15eqn81} holds after combining the obtained estimates  \eqref{nov15eqn51}--\eqref{nov15eqn82}. 
\end{proof}

 Equation \eqref{2025oct16eqn15eqn2} and \eqref{2025oct16eqn15eqn3} provide point estimates for the same physical quantity, differing only with respect to the bootstrap assumptions governing velocity characteristics and the coefficient $C$. For notational consistency with Section \ref{mainimprovedfull}, we adhere to its established convention. Supplementary notations employed herein—excluding those listed in Table \ref{tablesection5}—are systematically delineated in Table \ref{tablesection6}.
\begin{table}[H] 
\centering
\resizebox{\columnwidth}{!}{%
\begin{tabular}{ |c|c|c|c| } 
 \hline
 Notation & Definition &  Remarks \\
 \hline
 $\widetilde{\mathcal{E}}_i, \widetilde{\mathcal{E}}'_1$ & \eqref{oct15eqn1}  & The   set of $(k,n)$ in the first iteration for the hyperbolic parts;\\ 
 & &  Providing  lower bounds for $k+2n$ and $n$. \\
 \hline
 ${}^{1}\widetilde{\mathcal{E}}_{k,n}^{i_1,i} $ & \eqref{firstindexsetlemm2} &  The   set of $(k_1,n_1)$ in the second iteration  for the hyperbolic parts\\
 \hline
  $ {}^{2}\widetilde{\mathcal{E}}_{k_1,n_1}  $ & \eqref{2026may1eqn11} & The   set of $(k_2,n_2)$ in the forth iteration  for the hyperbolic parts\\
 \hline
 ${}^{3}\widetilde{\mathcal{E}}_{k_2,n_2}  $ & \eqref{2026may1eqn12} & The   set of $(k_3,n_3)$ in the third iteration  for the hyperbolic parts\\
 \hline
$ {}_{1}^{1}\widetilde{\mathfrak{H}}^{\mu,i}_{k,j;n}(t_1, t_2)$, ${}_{1}^{1}Err\mathfrak{H}^{\mu,i}_{k,j;n}(t_1, t_2) $  & \eqref{oct29eqn91} & Two components of ${}_1^{1}\mathfrak{H}^{\mu,i}_{k,j;n}(t_1, t_2)$, see \eqref{nov12eqn61}. \\
 \hline
 $  \mathfrak{E}^{\mu,i;l}_{k,j;n}(t_1, t_2), \mathfrak{E}^{\mu,i;\kappa,l}_{k,j;n}(t_1, t_2)$ & \eqref{2024oct29eqn41}, \eqref{2024oct29eqn91} & Further dyadic localization  based on the angle\\ 
 $ {}_{b}^1\mathfrak{E}^{\mu,a;l}_{k,j;n}(t_1, t_2)$, ${}_{b}^1 \mathfrak{E}^{\mu,4;\kappa,l }_{k,j;n}(t_1, t_2)$ & \eqref{2024oct31eqn81} &  between $v$ and $V(s)$, see \eqref{2026may1eqn21} \\
 \hline
 $  {}_{0;a}^{\,\,1}\mathfrak{E}^{\mu,a;l}_{k,j;n}(t_1, t_2)$, $a\in\{0,1\}$ & \eqref{nov9eqn62} & Two components of  $ {}_{0}^{ 1}\mathfrak{E}^{\mu,a;l}_{k,j;n}(t_1, t_2)$\\
 \hline
 $ {}_{0;b}^{\,\,1} \mathfrak{E}^{\mu,4;\kappa ,l}_{k,j;n}(t_1, t_2)$, $b\in\{0,1,2\}$ & \eqref{findecompellipticpart2024oct} & Three components of  $ {}_{0}^1 \mathfrak{E}^{\mu,4;\kappa,l }_{k,j;n}(t_1, t_2)$; the refined decomposition\\ 
 & &  of $\mathbf{P}_3(B)$ in \eqref{nov6eqn47}  was employed\\
 \hline
 $Err^{\mu, \mu_1;i,i_1 }_{k,n;k_1,n_1}(t_1, t_2) $ & \eqref{nov12eqn64} & The contribution from the new introduced acceleration\\
 & &  force without the hyperbolic parts, see \eqref{oct7eqn1}; estimated in \eqref{2026may1eqn31}.  \\
 \hline
  $ H^{\mu, \mu_1;i}_{k,n;k_1,n_1}(t_1, t_2)$   & \eqref{nov12eqn64}, \eqref{nov12eqn63} & New Type-I term created in the first iteration for the elliptic parts\\
$  M^{\mu, \mu_1;i,i_1 }_{k,n;k_1,n_1}(t_1, t_2)$ & & New Type-II term created in the first iteration for the elliptic parts\\
 \hline
 $ErrU^a_{i,i_1}(t_1, t_2)$ & & Error terms in the second iteration; estimated in Lemma \ref{errhypotype1}\\
 $U^1_{i,i_1,i_2}(t_1,t_2)$, $U\in \{H, M\}$ & \eqref{nov14eqn31} & New Type-I term created in the second iteration\\
 $U^2_{i,i_1,i_2}(t_1,t_2)$ & & New Type-II term created in the second iteration\\
 \hline
 $ErrU^{b;a}_{i,i_1,i_2}(t_1, t_2)$ & & Error terms in the third iteration; estimated in       Lemma \ref{ellerrstep2}\\
 $U^1_{i,i_1,i_2,i_3}(t_1,t_2)$, $U\in \{H, M\}$ & \eqref{nov15eqn99} & New Type-I term created in the third iteration\\
 $U^2_{i,i_1,i_2,i_3}(t_1,t_2)$ & & New Type-II term created in the third iteration\\
 \hline
 $ErrU^{b;a}_{i,i_1,i_2,i_3}(t_1, t_2) $ & \eqref{nov12eqn91}--\eqref{2024oct31eqn22} & Error terms in the fourth(last)  iteration; estimated in Lemma \ref{ellerrstep3}\\
 $LastU^b_{i,i_1,i_2,i_3}(t_1,t_2)$, $U\in \{H, M\}$ & \eqref{nov12eqn90} &  The main term    in the fourth(last) iteration\\
 \hline
\end{tabular}%
}
\caption{Essential notations in section \ref{fullimproved}.}\label{tablesection6}
\end{table}


\begin{thebibliography}{99} 

  
 
 

\bibitem{ChristKlainerman}  D.  Christodoulou,    S.  Klainerman. The Global Nonlinear Stability of the Minkowski Space. Princeton Mathematical Series, 41. Princeton University Press, Princeton, NJ, \textbf{1993}. 

 
 \bibitem{GerMasSha09}
P.~Germain, N.~Masmoudi,   J.~Shatah.
\newblock Global solutions for 3{D} quadratic {S}chr\"{o}dinger equations.
\newblock {\, Int. Math. Res. Not. IMRN}, (3)(\textbf{2009}), 414--432.


\bibitem{GerMasSha12}
P. Germain,  N. Masmoudi,  J. Shatah. Global solutions for the gravity water waves equation in dimension $3$. \textit{Ann. of Math.} (2) 175(\textbf{2012}), no. 2, 691–-754.

 \bibitem{glasseys11}  R.  Glassey, J. Schaeffer.   On symmetric solutions of the relativistic Vlasov-Poisson system.\textit{Commun. Math. Phys.},
101(\textbf{1985}), 459-473. 


\bibitem{glasseys1} R. Glassey, J. Schaeffer. The two and one half dimensional relativistic Vlasov-Maxwell system, \textit{Comm. Math. Phys.}, 185(\textbf{1997}), 257--284.
 
\bibitem{glasseys3} R. Glassey,   J. Schaeffer.    The relativistic Vlasov–Maxwell system in two space
dimensions: Part I. \textit{Arch. Ration. Mech. Anal.}, 141(\textbf{1998}), 331–354.
  



\bibitem{glasseys2} R. Glassey, J. Schaeffer.  The relativistic Vlasov–Maxwell system in two space
dimensions: Part II.  \textit{Arch. Ration. Mech. Anal.}, 141(\textbf{1998}), 355–374.




\bibitem{glassey6}  R.  Glassey,  J. Schaeffer. On global symmetric solutions to the relativistic Vlasov-Poisson equation in three space dimensions, \textit{Math. Meth. Appl. Sic.}, 24 ($\mathbf{2001}$), 143--157.


\bibitem{glassey3} R. Glassey, W. Strauss. Singularity formation in a collisionless plasma could only occur
at large velocities, \textit{Arch. Rat. Mech. Anal.}, 92(\textbf{1986}), 59--90.


 
 

\bibitem{glassey2} R. Glassey, W. Strauss.  Absence of shocks in an initially dilute collisionless plasma, \textit{Comm. Math. Phys.}, 113(\textbf{1987}), 191--208. 

 
 
  \bibitem{IP1} A.   Ionescu,   B. Pausader, The Euler-Poisson system in 2D: global stability of the constant equilibrium
solution, Int. Math. Res. Not. 2013(\textbf{2013}), 761–826.

\bibitem{IP2} A.   Ionescu, B. Pausader, Global solutions of quasilinear systems of Klein–Gordon equations in 3D,
J. Eur. Math. Soc., 16(\textbf{2014}), 2355–2431.

\bibitem{IP3} A.  Ionescu, B. Pausader, On the global regularity for a Wave-Klein-Gordon coupled system, Acta
Math. Sin. (Engl. Ser.) 35 (Special Issue in honor of Carlos Kenig on his 65th birthday)(\textbf{2019}), 933–986.

\bibitem{IP} A. Ionescu,  B. Pausader.  The Einstein-Klein-Gordon coupled system: Global stability of the Minkowski solution, \textit{Ann. Math. Stud.,} Vol 403,  \textbf{2022}.


 \bibitem{Kla83}
S.~Klainerman. On ``almost global" solutions to quasilinear wave equations in three space dimensions. Comm. Pure Appl. Math. 36(\textbf{1983}), no. 3, 325–344.


\bibitem{Kl85}
S.~Klainerman.
\newblock Uniform decay estimates and the {L}orentz invariance of the classical
  wave equation.
\newblock {\, Comm. Pure Appl. Math.}, Vol 38, no.3(\textbf{1985}), 321--332.

  

\bibitem{Kla86}
S.~Klainerman.
\newblock The null condition and global existence to nonlinear wave equations.
\newblock In {Nonlinear systems of partial differential equations in
  applied mathematics, {P}art 1 ({S}anta {F}e, {N}.{M}., 1984)}, volume~23 of
  {  Lectures in Appl. Math.}, pages 293--326. Amer. Math. Soc., Providence,
  RI,  \textbf{1986}.


 

\bibitem{Klainerman3} S. Klainerman, G. Staffilani. A new approach to study the Vlasov-Maxwell system, \textit{Comm. Pure Appl. Anal.}, 1(\textbf{2002}), no. 1, 103--125.


\bibitem{kunze} M. Kunze. Yet another criterion for global existence in the $3D$ relativistic Vlasov-Maxwell system, \textit{J. Differential Equations}, 259(\textbf{2015}), no. 9, 4413--4442.

 

\bibitem{Lions} P.-L. Lions and B. Perthame. Propagation of moments and regularity for the $3$-dimensional Vlasov-Poisson system, \textit{Invent. Math.}, 105(\textbf{1991}), 415--430.


\bibitem{luk2} J. Luk, R. Strain. A new continuation criterion for the relativistic Vlasov-Maxwell system, \textit{Comm. Math. Phys.} 331(\textbf{2014}), no. 3, 1005--1027.

\bibitem{luk} J. Luk, R. Strain. Strichartz estimates and moment bounds for the relativistic Vlasov-Maxwell system, \textit{Arch. Rat. Mech. Anal.}, 219(\textbf{2016}), no. 1, 1--120. 


 

 


 

\bibitem{pallard3} C. Pallard. A refined existence criterion for the relativistic Vlasov-Maxwell system, \textit{Comm. Math. Sci.}, 13(\textbf{2015}), no. 2, 347--354.

\bibitem{patel} N. Patel. Three new results on continuation criteria for the $3D$ relativistic Vlasov-Maxwell system, \textit{J. Differential Equations}., 264(\textbf{2018}), no. 3, 1841--1885.

\bibitem{Pfaffelmoser} K. Pfaffelmoser. Global classical solutions for the Vlasov-Poisson system in three dimensions for general initial data, \textit{J. Diff. Eqns.,} 95(\textbf{1992}), 281--303.



\bibitem{rein} G. Rein, A.D. Rendall. Global existence of solutions of the spherically symmetric Vlasov-Einstein system with small initial data, \textit{Comm. Math. Phys.} 150(\textbf{1992}), no. 3, 561--583.

\bibitem{schaeffer2} J. Schaffer. Global existence of smooth solutions to the Vlasov-Poisson system in three dimensions, \textit{Comm. P.D.E}, 16(\textbf{1991}), 1313--1335. 

\bibitem{schaeffer1} J. Schaeffer. A small data theorem for collisionless plasma that includes high velocity particles, \textit{Indiana Univ. Math. J.}, 53(\textbf{2004}), no.1, 1--34.

 
 \bibitem{alfonso} R. Sospedra-Alfonso, R. Illner. Classical solvability of the relativistic Vlasov-Maxwell system with bounded spatial density, \textit{Math. Methods Appl. Sci.}, 33(\textbf{2010}), no. 6, 751--757. 

 

\bibitem{wang} X. Wang.   Propagation of  regularity and long time behavior of   $3D$ massive relativistic transport equation II: Vlasov-Maxwell system,  \textit{Comm. Math. Phys.},  Vol 389(\textbf{2022}), no.2, pp 715--812. 
 
  \bibitem{wang2} X. Wang.  Global solution of the   3D relativistic Vlasov-Poisson system for a class of large data,     \textit{J. Stat. Phys.}, vol. 190, art. no. 162, \textbf{2023}.  
 
  \bibitem{wang3} X. Wang.   Remarks on the Large data global solutions of $3D$ RVP  system and $3D$ RVM system, \textit{DCDS-A}, vol. 43, no. 10, 3796--3829, \textbf{2023}.

 \bibitem{PartII} X. Wang.  Large data global solution of the 3D RVM system with cylindrical symmetry II: pointwise estimates, \textit{preprint.}


\end{thebibliography}
\end{document}